\documentclass[a4paper,10pt,reqno,american]{amsart}  
\usepackage[utf8]{inputenc}
\usepackage[T1]{fontenc}
\usepackage{amssymb}
\usepackage{graphicx}
\usepackage{amsmath,amsthm}
\usepackage{amsfonts,amssymb,enumerate}
\usepackage{url,paralist}
\usepackage{mathtools}
\usepackage{cancel}
\usepackage[arrow,curve,matrix,tips,2cell,rotate]{xy}  
  \SelectTips{eu}{10} \UseTips
  \UseAllTwocells
\usepackage{lscape}
 
\usepackage{nccmath}
\usepackage{stackengine}

\usepackage{tikz}
\usepackage[colorlinks=true,urlcolor=blue,linkcolor=black,citecolor=magenta]{hyperref}
\usepackage{color}
\usepackage{enumerate,amssymb}
\usepackage{colortbl}

\makeindex

\theoremstyle{plain}
\newtheorem{theorem}{Theorem}[section]

\newtheorem{lemma}[theorem]{Lemma}
\newtheorem{claim}[theorem]{Claim}
\newtheorem{corollary}[theorem]{Corollary}
\newtheorem{proposition}[theorem]{Proposition}

\newtheorem*{theorem*}{Theorem}

\theoremstyle{definition}
\newtheorem{definition}[theorem]{Definition}

\newtheorem{example}[theorem]{Example}
\newtheorem{remark}[theorem]{Remark}

\newcommand{\BIGOP}[1]{\mathop{\mathchoice%
{\raise-0.22em\hbox{\huge $#1$}}%
{\raise-0.05em\hbox{\Large $#1$}}{\hbox{\large $#1$}}{#1}}}
\newcommand{\bigtimes}{\BIGOP{\times}}

\newcommand\RP{\mathbb{R}{\rm P}}
\newcommand{\R}{\mathbb{R}}

\newcommand{\C}{\mathbb{C}}
\newcommand{\OO}{\mathrm{O}}
\newcommand{\UU}{\mathrm{U}}
\newcommand{\epicy}{\operatorname{ecy}}
\newcommand{\Ce}{\operatorname{Ce}}
\newcommand{\cepicy}{\operatorname{cecy}}
\newcommand{\N}{\mathbb{N}}
\newcommand{\Z}{\mathbb{Z}}
\newcommand{\F}{\mathbb{F}}
\newcommand{\B}{\mathrm{B}}
\newcommand{\BB}{\mathcal{B}}
\newcommand{\PP}{\mathcal{P}}
\newcommand{\EE}{\mathcal{E}}
\newcommand{\EEE}{\mathrm{E}}
\newcommand{\Sy}{\mathcal{S}}
\newcommand{\im}{\operatorname{im}}
\newcommand{\pt}{\mathrm{pt}}
\newcommand{\IIII}{\mathbf{1}}
\newcommand{\GL}{\operatorname{GL}}
\newcommand{\U}{\mathrm{L}}
\newcommand\Sym{\mathfrak S}
\newcommand\CC{\mathcal C}
\newcommand\OOO{\mathcal O}
\newcommand\DD{\mathcal D}
\newcommand\End{\mathcal End}
\newcommand{\dimaff}{\dim_{\affine}}
\newcommand{\basis}{\mathcal{B}}
\newcommand{\basisa}{\mathcal{B}_a}
\newcommand{\basisi}{\mathcal{B}_i}
\newcommand{\conf}{\operatorname{F}} % TODO : Maybe return to Conf
\newcommand{\Sp}{\operatorname{Pe}}
\newcommand{\ev}{\operatorname{ev}}
\newcommand{\AAA}{\mathrm{A}}
\newcommand{\III}{\mathrm{I}}
\newcommand{\TF}{\mathrm{F}}
\newcommand{\TD}{\mathrm{D}}
\newcommand{\AF}{\mathcal{F}}
\newcommand{\AD}{\mathcal{D}}
\newcommand{\TT}{\mathcal{T}}
\newcommand{\Salg}{\mathrm{Sym}}
\newcommand{\Q}{\mathcal{Q}}
\newcommand{\spann}{\operatorname{span}}

\newcommand{\id}{\operatorname{id}}
\newcommand{\inter}{\operatorname{int}}

\newcommand{\res}{\operatorname{res}}
\newcommand{\colim}{\operatorname{colim}}

\newcommand{\affine}{\operatorname{aff}}
\newcommand{\hght}{\operatorname{height}}
\newcommand{\sgn}{\operatorname{sgn}}
\newcommand{\Th}{\operatorname{Th}}
\renewcommand{\emptyset}{\varnothing}
\newcommand{\Hung}{\mbox{Hu\hspace{-1pt}’\hspace{-.6pt}ng}}
\newcommand{\NguyenHuuViet}{\mbox{Nguy\^en H\~u\hspace{-1pt}’\hspace{-.6pt}u Vi\^et}}

\newcommand{\Ctop}{\operatorname{\mathrm Top}}   % Category of compactly generated topological spaces 
\newcommand{\Ctoppt}{\operatorname{\mathrm Top_{pt}}} % Category of compactly generated topological spaces with  the base point
\newcommand{\Op}{\operatorname{\mathrm Op}} % Category of operads
\newcommand{\CWtop}{\operatorname{\mathrm Top_{cw}}} % Category of CW complexes
\newcommand{\mor}{\operatorname{\mathrm Mor}}

\begin{document}
% -------------------------------------------------------------------%

% -------------------------------------------------------------------%
\title[Equivariant Cohomology of Configuration Spaces mod~2]{Equivariant Cohomology of Configuration Spaces mod~2: The State of the Art}

% -------------------------------------------------------------------%

% -------------------------------------------------------------------%
\author[Blagojevi\'c]{Pavle V. M. Blagojevi\'{c}} 
%\thanks{The research of Pavle V. M. Blagojevi\'{c} was supported by the grant ON 174024 of the Serbian Ministry of Education and Science, Frederick R. Cohen was supported by the University of Rochester, Wolfgang L\"uck was supported by the ERC Advanced Grant ``KL2MG-interactions'' (ID 662400), and 
%G\"unter M. Ziegler was supported by the German Science Foundation DFG via the Collaborative Research Center TRR~109 ``Discretization in Geometry and Dynamics.''
% This material is based upon work supported by the National Science Foundation under Grant DMS-1440140 while Pavle V. M. Blagojevi\'{c} was in residence at the Mathematical Sciences Research Institute in Berkeley, California, during the fall of 2017. 
% }
\address{Mathemati\v{c}ki Institut SANU, Knez Mihailova 36, 11001 Beograd, Serbia}
\email{pavleb@mi.sanu.ac.rs} 
\curraddr{\sc Institut f\"ur Mathematik, Freie Universit\"at Berlin, Arnimallee 2, 14195 Berlin, Germany}
\email{blagojevic@math.fu-berlin.de}
\author[Cohen]{Frederick R. Cohen}
%\thanks{F.C.\ was supported by the University of Rochester.}
\address{Department of Mathematics, University of Rochester, Rochester, NY 14625, USA}
\email{cohf@math.rochester.edu} 
\author[Crabb]{Michael C. Crabb} 
%\thanks{M.C. was supported by}
\address{Institute of Mathematics, University of Aberdeen, Aberdeen AB24 3UE, UK}
\email{m.crabb@abdn.ac.uk}
\author[L\"uck]{Wolfgang L\"uck} 
%\thanks{W.L.\ was supported by an ERC Advanced Grant.}
\address{Mathematisches Institut der Universit\"at Bonn, Endenicher Allee 60, 53115 Bonn, Germany} 
\email{wolfgang.lueck@him.uni-bonn.de} 
\author[Ziegler]{G\"unter M. Ziegler} 
%\thanks{G.M.Z.\ was supported by the German Science Foundation DFG via the Collaborative Research Center TRR~109 ``Discretization in Geometry and Dynamics''}  
\address{Institut f\"ur Mathematik, Freie Universit\"at Berlin, Arnimallee 2, 14195 Berlin, Germany} 
\email{ziegler@math.fu-berlin.de}

\date{}%
\renewcommand\medskip{\vspace{2pt}}%{\smallskip}%
% -------------------------------------------------------------------% 

\begin{abstract}
The equivariant cohomology of the classical configuration space $\conf(\R^d,n)$ 
has been been of great interest and has been studied intensively starting with
the classical papers by Artin (1925/1947) on the theory of braids,
by Fox and Neuwirth (1962), Fadell and Neuwirth (1962), and Arnol'd (1969).
We give a brief treatment of the subject from the beginnings to recent developments.
However, we focus on the mod $2$ equivariant cohomology algebras of the classical configuration space $\conf(\R^d,n)$,
as described in an influential paper by {\Hung} (1990).
We show with a new, detailed proof that his main result is correct, 
but that the arguments that were given by {\Hung} on the way to his result are not, 
as are some of the intermediate results in his paper.

This invalidates a paper by three of the present authors, Blagojevi\'c, L\"uck \& Ziegler (2016), 
who used a claimed intermediate result from {\Hung} (1990) in order to derive lower bounds 
for the existence of $k$-regular and $\ell$-skew embeddings.
Using our new proof for {\Hung}'s main result, we get new lower bounds for existence of highly regular embeddings:
Some of them agree with the previously claimed bounds, some are weaker.
\end{abstract}

\maketitle

\tableofcontents

% -------------------------------------------------------------------%

%%%%%%%%%%%%%%%%%%%%%%%%%%%%%%%%%%%%%%%%%%%%%%%%%%%%%%%%%%%%%%%%%%%%%%%%%%%%%%%%%%%%%
%%%%%%%%%%%%%%%%%%%%%%%%%%%%%%%%%%%%%%%%%%%%%%%%%%%%%%%%%%%%%%%%%%%%%%%%%%%%%%%%%%%%%
\section*{Preface}
%%%%%%%%%%%%%%%%%%%%%%%%%%%%%%%%%%%%%%%%%%%%%%%%%%%%%%%%%%%%%%%%%%%%%%%%%%%%%%%%%%%%%
%%%%%%%%%%%%%%%%%%%%%%%%%%%%%%%%%%%%%%%%%%%%%%%%%%%%%%%%%%%%%%%%%%%%%%%%%%%%%%%%%%%%%
The systematic study of the ordered configuration space\index{ordered configuration space} 
\[
\conf(M,n):=
\{
(x_1,\dots,x_n) \in M^n : x_i\neq x_j \text{ for all }1\leq i<j\leq n
\}
\]
of all ordered $n$-tuples of distinct points on a manifold $M$ started in 1962 with the work of Fadell \& Neuwirth \cite{FadellNeuwirth1962} and Fox \& Neuwirth \cite{FoxNeuwirth1962}, with prehistory going back to the work of Artin \cite{Artin1925,Artin1947-1,Artin1947-2}. 
Soon after, Arnol'd, in his seminal work \cite{Arnold1969} from 1969, gave a description of the integral cohomology ring\index{integral cohomology ring} of the ordered configuration space $\conf(\R^2,n)$.
From that point on the topology of the ordered configuration spaces was studied very intensively from many aspects, while finding applications in diverse problems, theories and even different fields of Mathematics and beyond, notably in Physics. 

Each configuration space $\conf(M,n)$ is equipped with a natural free action of the symmetric group\index{free action of the symmetric group} on $n$ letters $\Sym_n$,
given by the permutation of points.
The associated orbit space $\conf(M,n)/\Sym_n$, called the unordered configuration space\index{unordered configuration space}, is an important and challenging object to study. 
(The free action of the symmetric group is also an essential ingredient of the little cubes operad structure to be discussed later.)

In his influential 1970 paper \cite{Fuks1970} using fundamental new ideas, Fuks  gave a description of the cohomology algebra of the unordered configuration space of $n$ distinct points in the plane $H^*(\conf(\R^2,n)/\Sym_n;\F_2)$ as an image of the cohomology $H^*(\B\OO(n) ;\F_2)$.
In the course of study of infinite and iterated loop spaces objects of the same homotopy type as the configuration space $\conf(\R^d,n)$ were invented
by Boardman \& Vogt \cite{Boardman1973} and adapted in a beautiful way by May \cite[Sec.\,4]{May1972} for the definition of an important structure that we now call the little cubes operad\index{little cubes}; see Section \ref{sub : operads}.
Frederick Cohen, in his 1976 contribution \cite{Cohen1976LNM533}, gave the first descriptions of the cohomology of the unordered configuration space $\conf(\R^d,n)/\Sym_n$, for $n$ a prime, 
with trivial coefficients (including the ring structure) and 
with twisted coefficients (including the relevant module structure) \cite[Thm.\,5.2 and Thm.\,5.3]{Cohen1976LNM533}.

The homology of the unordered configuration space\index{homology of the unordered configuration space} for points on a smooth manifold $M$ has been determined in 1989 by B\"odigheimer, F. Cohen \& Taylor~\cite{Bodigheimer1989} in the case when $M$ is odd-dimensional and coefficients are in an arbitrary field, and in the case when $M$ is even-dimensional and coefficients are in a field of characteristic~$2$.
More precisely, for even-dimensional manifolds they computed the homology of the unordered configuration space of $M$ with coefficients in the field twisted by the sign representation.
These results were given in terms of Cohen's computation for the case $M =\R^n$.
Some further results, for an even-dimensional orientable closed manifold and the rationals as the field of coefficients, were obtained by F\'elix \& Thomas~\cite{Felix2000}.

\medskip
\NguyenHuuViet\ \Hung, 
in a series of papers \cite{Hung1981, Hung1982, Hung1987, Hung1990} from 1981 until 1990, studied the mod~$2$ cohomology algebra of the symmetric group $\Sym_n$ and of the unordered configuration space $\conf(\R^d,n)/\Sym_n$ in the case when $n$ is a power of two.
The key paper in this series, \cite{Hung1990}, which contained detailed proofs for all results announced in \cite{Hung1982}, was apparently finished in August of 1982, but after some delays (described in \cite[Footnote 1 on p.\,286]{Hung1990}) was published only in 1990.
The central idea was  

\medskip
\begin{compactitem}[\ ---]

\item to consider a natural embedding  
\[
\xymatrix{
\Sp(\R^d,2^m)\ar[rr]^-{\epicy_{d,2^m}}& & \conf(\R^d,2^m)
}
\]
of the product of spheres $\Sp(\R^d,2^m)=(S^{d-1})^{2^m-1}$ into the configuration space $\conf(\R^d,2^m)$, which turns out to be equivariant with respect to the action of a Sylow $2$-subgroup $\Sy_{2^m}$ of the symmetric group $\Sym_{2^m}$,
\medskip
	
\item to describe the cohomology ring $H^*(\Sp(\R^d,2^m)/\Sy_{2^m};\F_2)$ of the quotient space $\Sp(\R^d,2^m)/\Sy_{2^m}$ using the homeomorphism
	\[
	\Sp(\R^d,2^{m+1})/\Sy_{2^{m+1}}\cong \big(\Sp(\R^d,2^{m})/\Sy_{2^{m}}\times \Sp(\R^d,2^{m})/\Sy_{2^{m}} \big)\times_{\Z_2}S^{d-1},
	\]
	via an inductive computation, and finally 
	\medskip
	
\item to prove that the induced homomorphism in cohomology	
\begin{multline*}
\xymatrix{
H^*(\conf(\R^d,2^m)/\Sym_{2^m};\F_2)\ar[rr]^-{(\id/\Sym_{2^m})^*} & &H^*(\conf(\R^d,2^m)/\Sy_{2^m};\F_2)}\\
\xymatrix{	\ar[rr]^-{(\epicy_{d,2^m}/\Sy_{2^m})^*} & &H^*(\Sp(\R^d,2^m)/\Sy_{2^m};\F_2)}
\end{multline*}
is a monomorphism.
Here $(\id/\Sym_{2^m})^*$ is directly a monomorphism since $\Sy_{2^m}$ is a $2$-Sylow subgroup of $\Sym_{2^m}$ and cohomology is considered with $\F_2$ coefficients.
Thus, the main difficulty lies in proving that $(\epicy_{d,2^m}/\Sy_{2^m})^*$ is a injective.
\end{compactitem}
In this way the cohomology ring $H^*(\conf(\R^d,2^m)/\Sym_{2^m};\F_2)$ of the unordered configuration space\index{unordered configuration space} $\conf(\R^d,2^m)/\Sym_{2^m}$ could be seen as a subring of the, now known, ring $H^*(\Sp(\R^d,2^m)/\Sy_{2^m};\F_2)$.
Further on,

\medskip
This series of papers, and in particular the paper \cite{Hung1990}, feature extended and substantial calculations. It turned out to be important and influential.
It was quoted, and its main result was used, in quite a number of papers since then, such as Vassiliev's 1988 and 1998 papers on braid group cohomologies and algorithm complexity \cite{Vassiliev1988} and $r$-neighborly embeddings of manifolds \cite{Vassiliev1998}, Crabb's 2012 survey on the topological Tverberg theorem and related topics \cite{Crabb2012},
Karasev \& Landweber's 2012 paper on higher topological complexity of spheres \cite{Karasev2012}, Karasev \& Volovikov's 2013 paper on the waist of the sphere theorem for maps to manifolds \cite{Karasev2013}, Matschke's 2014 paper on a parameterized Borsuk--Ulam--Bourgin--Yang--Volovikov paper \cite{Matschke2014}, as well as Karasev, Hubard \& Aronov's 2014 paper on the ``spicy chicken theorem''~\cite{Karasev2014}.

\medskip
None of these papers mentioned the fact that -- as we will document in Section~\ref{subsec : injectivity} of the present {work} -- {\Hung}'s proof for his main result  
\cite[Thm.\,2.3]{Hung1982} \cite[Thm.\,3.1]{Hung1990} is incorrect, as are some of his intermediate and follow-up results.
This does not jeopardize the papers listed above, as {\Hung}'s main result, the injectivity of the composition $(\epicy_{d,2^m}/\Sy_{2^m})^*\circ (\id/\Sym_{2^m})^*$, holds, as we will demonstrate -- by a new, entirely different, homotopy-theoretic proof -- in Section \ref{subsec : proof of injectivity} of this paper.

\medskip
In contrast to the above works, the 2016 paper of Blagojevi\'c, L\"uck \& Ziegler \cite{Blagojevic2016-01} -- by three of the present authors -- did not only quote {\Hung}'s papers, but it also used some of {\Hung}'s intermediate results in an essential way, specifically the decomposition of the equivariant cohomology claimed in \cite[(4.7), page 279]{Hung1990}.
Our computations in \cite{Blagojevic2016-01} based on this led to results that are not consistent with some of Crabb's computations related to \cite{Crabb2012}.
This led to our discovery of the substantial mistakes in {\Hung}'s paper, including the fact that the decomposition of \cite[(4.7), p.~279]{Hung1990} is not correct, which also invalidates the main results of \cite{Blagojevic2016-01} and two minor follow-up corollaries given in \cite{Blagojevic2016-02}.

\medskip
Thus the second main purpose of the present paper is to correct our work in \cite{Blagojevic2016-02} and in \cite{Blagojevic2016-01},
by presenting alternative arguments, based on the corrected proof for {\Hung}'s theorem, towards estimates for the dimensions of $k$-regular embeddings\index{$k$-regular embedding} and their relatives.
The results we get are in some cases weaker than what we had claimed before, in other cases we recreate the previously-claimed results in full.
\bigskip

\noindent
This {text} is organized as follows. (See below for a summary of notations as well as for definitions and background.)  
\medskip
\begin{compactitem}[ \ ---]

\item In Section \ref{sec : Ptolemaic epicycles embedding} we describe the $\Sy_{2^m}$-equivariant embedding 
	\[
	\epicy_{d,2^m}\colon \ \Sp(\R^d,2^m)\longrightarrow \conf(\R^d,2^m)
	\]
	of the $(d-1)(2^m-1)$-dimensional manifold $\Sp(\R^d,2^m)\cong (S^{d-1})^{2^m-1}$ into the classical configuration space $\conf(\R^d,2^m)$.
	Furthermore, we relate the embedding $\epicy_{d,2^m}$ with the structural map of the little cubes operad.
	\medskip

\item In Section \ref{sec : Equivariant cohomology of epicicles} we study the $\Sy_{2^m}$-equivariant cohomology $H^*_{\Sy_{2^m}}(\Sp(\R^d,2^m);\F_2)$ using the Serre spectral sequence associated to the fiber bundle 
	\[
	\xymatrix{
	X\times X\ar[r] & (X\times X)\times_{\Z_2}\EEE\Z_2\ar[r] & \B\Z_2.
	}
	\] 
	The highlight of that section is the proof of the decomposition of the cohomology given in Theorem \ref{th: cohomology of Sp}:
	\begin{multline*}
	\qquad	H^*_{\Sy_{2^m}}(\Sp(\R^d,2^m);\F_2)	
	\cong   \\
	\F_2[V_{m,1},\ldots,V_{m,m}]/\langle V_{m,1}^d,\ldots, V_{m,m}^d\rangle \oplus \III^*(\R^d,2^m),
	\end{multline*}
	where $\III^*(\R^d,2^m)$ is an ideal, and $\deg(V_{m,r})=2^{r-1}$ for $1\leq r\leq m$.
	\medskip
	
\item In Section \ref{sec : injectivity} we discuss the claim that the induced homomorphism in cohomology
	\begin{multline*}
	\xymatrix{   
	(\epicy_{d,2^m}/\Sy_{2^m})^*\colon H^*(\conf(\R^d,2^m)/\Sy_{2^m};\F_2)\ar[r]& H^*(\Sp(\R^d,2^m)/\Sy_{2^m};\F_2)
	}
	\end{multline*}
	is a monomorphism, or equivalently that the homomorphism
	\begin{multline*}
	\xymatrix{ \qquad 
	(\epicy_{d,2^m}/\Sy_{2^m})^*\circ(\id/\Sym_{2^m})^*\colon H^*(\conf(\R^d,2^m)/\Sym_{2^m};\F_2)\ar[r] & }\\
	\xymatrix{ & H^*(\Sp(\R^d,2^m)/\Sy_{2^m};\F_2)}
	\end{multline*}
	is a monomorphism.
	In Section \ref{subsec : injectivity} we present the proof for injectivity of $(\epicy_{d,2^m}/\Sy_{2^m})^*$ given by {\Hung} in \cite[Thm.\,3.1]{Hung1990} and document several critical gaps that invalidate this proof.
	In particular, the failure of decomposition \cite[(4.7)]{Hung1990} will be illustrated by a counterexample in Claim \ref{claim-02}.
	The new inductive proof of the injectivity of $(\epicy_{d,2^m}/\Sy_{2^m})^*$, or $(\epicy_{d,2^m}/\Sy_{2^m})^*\circ(\id/\Sym_{2^m})^*$, is given in Section \ref{subsec : proof of injectivity}.
	More precisely, for the inductive step, using the presentation of homology of the configuration space via Araki--Kudo--Dyer--Lashof homology operations\index{Araki--Kudo--Dyer--Lashof homology operations}, we prove that the structural map of the little cubes operad induces, now in homology, 
	an epimorphism
	\begin{multline*}
	\xymatrix{
	(\mu_{d,2^{m}})_*\colon
	H_*((\CC_d(2^{m-1})/\Sym_{2^{m-1}} \times \CC_d(2^{m-1})/\Sym_{2^{m-1}})\times_{\Z_2}\CC_d(2);\F_2)\ar[r] &}\\
	\xymatrix{
	& H_*(\CC_d(2^{m})/\Sym_{2^{m}};\F_2);
	}	
	\end{multline*}
	see Theorem \ref{th : surjection }.
	\medskip

\item Additionally in Section \ref{subsec : an unexpected corollary}, motivated  by the results of Atiyah \cite{Atiyah1966} and Giusti, Salvatore \& Sinha \cite{Giusti2012} we prove, as an interesting fact, that the homology of the space of all finite subsets of $\R^d$ with addition of a base point and appropriately defined multiplication, is a polynomial ring.
	\medskip

\item In Section \ref{sec : regular embeddings}, based on the results of the previous sections, we explain the induced gaps in the results given by three of the present authors in \cite[Thm.\,2.1, Thm.\,3.1, Thm.\,4.1]{Blagojevic2016-01} and \cite[Thm.\,5.1, Thm.\,6.1]{Blagojevic2016-02} and correct all of them.	
	In particular, corrected lower bounds for the existence of $k$-regular, $\ell$-skew and $k$-regular-$\ell$-skew embeddings of an Euclidean space are given; see Theorem \ref{th : Correctionog T 2.1}, Theorem \ref{th : Correctionog T 3.1} and Theorem \ref{thm : Correction of T.4.1}.
	\medskip

\item In Section \ref{sec : more bounds for regular embeddings}, using novel techniques for the computation of Stiefel--Whitney classes\index{Stiefel--Whitney classes}, we get the Key Lemma \ref{lem : key} and from this derive still stronger lower bounds for the existence of $k$-regular, $\ell$-skew, $k$-regular-$\ell$-skew embeddings as well as for their complex analogues. 
	They are summarized in Theorem \ref{th : additional bounds} and Theorem \ref{th : additional complex bounds}.
\end{compactitem}

\bigskip
\noindent
In Sections \ref{sec : regular embeddings} and \ref{sec : more bounds for regular embeddings} we in particular present extensive calculations with characteristic classes of the vector bundles associated with the natural permutation and the standard representation of the symmetric group $\Sym_n$ over the unordered configuration space\index{unordered configuration space} $\conf(\R^d,n)/\Sym_n$.
These are the vector bundles
\[
\xymatrix{ \xi_{\R^d,n} \colon \quad \R^n\ar[r] &  \conf(\R^d,n)\times_{\Sym_n}\R^n\ar[r] &\conf(\R^d,n)/\Sym_n,}
\]
and 
\[\xymatrix{ \zeta_{\R^d,n} \colon \quad W_n\ar[r] &  \conf(\R^d,n)\times_{\Sym_k}W_n\ar[r] &\conf(\R^d,n)/\Sym_n,}\]
where 
\[
W_n=\{(a_1,\ldots, a_n)\in\R^n : a_1+\dots+a_n=0\}
\] 
denotes the standard representation of $\Sym_n$.
These vector bundles have been studied intensively over the years.
For example, a particular result that can be deduced from \cite[Thm.\,2.10]{Hung1990} about the $(d-1)$st power of the top Stiefel--Whitney class of the vector bundle $\zeta_{\R^d,n}$ was rediscovered, extended and reproved by many authors. 
Going back in time, already in 1970 it was known, by the work of Cohen \cite[Thm.\,8.2]{Cohen1976LNM533}, which was published only in 1976, that the Euler class\index{Euler class} of $\zeta_{\R^d,n}^{\oplus (d-1)}$ does not vanish if $n$ is a prime.
At the same time Fuks in \cite{Fuks1970} showed that the vector bundle $\xi_{\R^2,n}^{\oplus 2}$ is trivial and furthermore that $w_{n-1}(\xi_{\R^2,n})\neq 0$ if and only if $n$ is a power of $2$.
In 1978 while working on the existence of $k$-regular embedding Cohen \& Handel \cite[Thm.\,3.1]{Cohen1978} evaluated Stiefel--Whitney classes\index{Stiefel--Whitney classes} of the vector bundle $\zeta_{\R^2,n}$. 
One year later, Chisholm, in his follow-up paper \cite[Lem.\,3]{Chisholm1978}, computed Stiefel--Whitney classes of the vector bundle $\zeta_{\R^d,n}$ in the case when $d$ is a power of~$2$.
Gromov, in his seminal work on the waist of a sphere\index{waist of a sphere}, sketched an argument that the top Stiefel--Whitney class of $\zeta_{\R^d,n}^{\oplus (d-1)}$ does not vanish in the case when $n$ is a power of~$2$ \cite[Lem.\,5.1]{Gromov2003}.
Three different proofs for the fact that the Euler class of $\zeta_{\R^d,n}^{\oplus (d-1)}$ does not vanish if and only if $n$ is a power of a prime were given by three groups of authors: by Karasev, Hubard \& Aronov in \cite[Thm.\,1.10]{Karasev2014}, by Blagojevi\'c \& Ziegler in \cite[Thm.\,1.2]{Blagojevic2014}, and by Crabb in \cite[Prop.\,5.1]{Crabb2012}.
Furthermore, the stable order of the vector bundle $\xi_{\R^d,n}$ was analyzed already in 1978 by Cohen, Mahowald \& Milgram in \cite[Thm.\,1]{CohenMachowaldMilgram1978} in the case when $d=2$.
It was completely determined, for all $d\geq 2$, in 1983 by Cohen, Cohen, Kuhn \&  Neisendorfer \cite[Thm.\,1.1]{CohenCohen1983}.

\medskip 
We would like to point out that in these lecture notes we treat a number of quite different problems, for which we perform extended and rather complex computations, which use a variety of different tools and methods.
In order to make this accessible, and for reasons of completeness, in the following we present many classical proofs with all details, accompanied with all relevant references, rather than just quoting them; see for example Section \ref{sec : appendix}.
For the same reason we start with the list of notations used in this paper, some of which we have already used in the overview.

\medskip
The research of Pavle V. M. Blagojevi\'{c} was supported by the grant ON 174024 of the Serbian Ministry of Education and Science, Frederick R. Cohen was supported by the University of Rochester, Wolfgang L\"uck was supported by the ERC Advanced Grant ``KL2MG-interactions'' (ID 662400), and 
the work of G\"unter M. Ziegler was supported by the German Science Foundation DFG via the Collaborative Research Center TRR~109 ``Discretization in Geometry and Dynamics,''
the Berlin Mathematical School and the Excellence Cluster MATH+.
This material is based upon work supported by the National Science Foundation 
under Grant DMS-1440140
while Pavle V. M. Blagojevi\'{c} and G\"unter M. Ziegler were in residence at
the Mathematical Sciences Research Institute in Berkeley, California, during the fall of 2017. 
\smallskip

The authors would like to express their gratitude to Peter Landweber for a great number of excellent comments and recommendations, and to Roman Karasev for useful discussions about different topics related to the content of the book.
We also thank Evgeniya Lagoda and Tatiana Levinson for careful reading of the  manuscript.

\vspace{\baselineskip}
\begin{flushright}\noindent
Aberdeen, Belgrade,  Berlin,   \hfill {\it Pavle V. M. Blagojevi\'c, Frederick R. Cohen,}\\
Bonn, Rochester  \hfill {\it Michael C. Crabb, Wolfgang L\"uck and}\\
May 2021 \hfill {\it G\"unter M. Ziegler}\\
\end{flushright}

%%%%%%%%%%%%%%%%%%%%%%%%%%%%%%%%%%%%%%%%%%%%%%%%%%%%%%%%%%%%%%%%%%%%%%%%%%%%%%%%%%%%%
%%%%%%%%%%%%%%%%%%%%%%%%%%%%%%%%%%%%%%%%%%%%%%%%%%%%%%%%%%%%%%%%%%%%%%%%%%%%%%%%%%%%%

\section*{Notations used} % in the paper}

{
\begin{compactitem}[ \ $\triangleright$]
%----------
\item[{\bf Groups:}]
\vspace{3mm}%
%----------
\item $C_G(H)$ \,:\, the centralizer of the subgroup\index{centralizer of a subgroup} $H$ in the group $G$,
\item $N_G(H)$ \,:\, the normalizer of the subgroup\index{normalizer of a subgroup} $H$ in the group $G$,
\item $W_G(H)$ \,:\, the Weyl group of the subgroup\index{Weyl group} $H$ in the group $G$ is the quotient group  $W_G(H)=N_G(H)/H$,
\item $\Sym_{n}$ \,:\, the symmetric group\index{symmetric group} on $n$ letters,
\item $\Sym_{2^m}$ \,:\, the symmetric group of the point set of the group $\Z_2^{\oplus m}$, or the vector space $\F_2^{\oplus m}$,
\item $\EE_m$ \,:\, the elementary abelian group\index{elementary abelian group}, isomorphic to $\Z_2^{\oplus m}$, that is regularly embedded, via translations of $\F_2^{\oplus m}$, into $\Sym_{2^m}$, 
\item $\Sy_{2^m}$ \,:\, the Sylow $2$-subgroup of $\Sym_{2^m}$\index{Sylow $2$-subgroup}, isomorphic to $\Z_2^{\wr m}:=\Z_2\wr \Z_2\wr\dots\wr\Z_2$ ($m$ times), that contains $\EE_m$, and acts freely on $\Sp(\R^d,2^m)$,
\item $\GL_m(\F_2)$ \,:\, the general linear group of the $\F_2$ vector space $\F_2^{\oplus m}$,
\item $\U_m(\F_2)$ \,:\, the Sylow $2$-subgroup of $\GL_m(\F_2)$ of all lower triangular matrices with $1$'s on the main diagonal,
\item $\mathrm{U}_m(\F_2)$ \,:\, the Sylow $2$-subgroup of $\GL_m(\F_2)$ of all upper triangular matrices with $1$'s on the main diagonal,
\item $\OO(n)$ \,:\, the orthogonal group\index{orthogonal group},
\item $\UU(n)$ \,:\, the unitary group\index{unitary group},
\item $\BB_n$ \,:\, Artin's braid group\index{braid group} on $n$ strings,
\item $\PP_n$ \,:\, Artin's pure braid group\index{pure braid group} on $n$ strings.
%----------
\vspace{3mm}%
\item[{\bf Group representations:}]
\vspace{3mm}%
\item $\R^n$ \,:\, is a real $n$-dimensional $\Sym_{n}$-representation,
\item $\C^n$  \,:\, is a complex $n$-dimensional $\Sym_{n}$-representation,
\item $W_n:=\{(a_1,\ldots,a_n)\in\R^n : a_1+\cdots+a_n=0\}$ is a real $(n-1)$-dimensional $\Sym_{n}$-representation and an irreducible $\Sym_{n}$-subrepresentation of $\R^n$,
\item $W_n^{\C}:=\{(b_1,\ldots,b_n)\in\C^n : b_1+\cdots+b_n=0\}$ is a complex $(n-1)$-dimensional $\Sym_{n}$-representation, and an irreducible $\Sym_{n}$-subrepresentation of $\C^n$,
%----------
\vspace{3mm}%
\item[{\bf Algebras:}]
\vspace{3mm}%
%----------
\item $\mathfrak{H}_m:=\F_2[x_1,\ldots,x_m]^{\U_m(\F_2)}$,
\item $\mathfrak{D}_m:=\F_2[x_1,\ldots,x_m]^{\GL_m(\F_2)}$ the Dickson algebra.
%----------
\vspace{3mm}%
\item[{\bf Sets:}]
\vspace{3mm}%
%----------
\item ${[n]}_{~}:=\{1,2,\ldots,n\}$,
\item ${[n]_1}:=\{1,2,\ldots,\tfrac{n}2\}$ for even $n$,
\item ${[n]_2}:=\{\tfrac{n}2+1,\tfrac{n}2+2,\ldots,n\}$ for even $n$.
%----------
\vspace{3mm}%
\item[{\bf Categories:}]
\vspace{3mm}%
%----------
\item $\Ctop$  \,:\,  the category of compactly generated weak Hausdorff spaces\index{category of compactly generated weak Hausdorff spaces} with continuous maps as morphisms,
\item  $\Ctoppt$  \,:\,  the category of compactly generated weak Hausdorff spaces with non-degenerate base points\index{category of compactly generated weak Hausdorff spaces with non-degenerate base points} and with continuous maps that preserve base points as morphisms,
\item $\CWtop$  \,:\,  the category of $CW$-complexes\index{category of $CW$-complexes} with continuous maps as morphisms.
%----------
\vspace{3mm}%
\item[{\bf Spaces:}]
\vspace{3mm}%
%----------
\item $\conf(X,n)$  \,:\, the ordered configuration space\index{ordered configuration space} of $n$ pairwise distinct points in the space $X$,
\item $\conf(X,n)/\Sym_n$  \,:\, the unordered configuration space\index{unordered configuration space} of $n$ pairwise distinct points in the space $X$,
\item $\Sp(\R^d,2^m)$  \,:\,  the Ptolemaic epicycles space\index{Ptolemaic epicycles space} $(S^{d-1})^{2^m-1}$,
\item $\Ce(\R^d,2^m)$  \,:\,  the little cubes epicycles space\index{little cubes epicycles space},
\item $\CC_d(n)$  \,:\,  the space of ordered $n$-tuples of interior disjoint little $d$-cubes\index{space of ordered $n$-tuples of interior disjoint little $d$-cubes},
\item $\Th(\xi$) \,:\, the Thom space of the vector bundle $\xi$.
%----------
\vspace{3mm}%
\item[{\bf Maps:}]
\vspace{3mm}%
%----------
\item $\epicy_{d,2^m}\colon\Sp(\R^d,2^m)\longrightarrow \conf(\R^d,2^m)$  \,:\,  the  $\Sy_{2^m}$-equivariant Ptolemaic epicycles embedding\index{Ptolemaic epicycles embedding},
\item $\cepicy_{d,2^m}\colon\Ce(\R^d,2^m)\longrightarrow\CC_d(2^m)$  \,:\,  the $\Sy_{2^m}$-equivariant little cubes epicycle embedding\index{little cubes epicycle embedding},
\item $\rho_{d,2^m}\colon \Sp(\R^d,2^m)/\Sy_{2^m}\longrightarrow\conf(\R^d,2^m)/\Sym_{2^m}$  \,:\, the composition $(\id/\Sym_{2^m})\circ(\epicy_{d,2^m}/\Sy_{2^m})$,
\item $\kappa_{d,2^m}\colon \Sp(\R^d,2^m)\longrightarrow \Sp(\R^{\infty},2^m)$  \,:\,  the $\Sy_{2^m}$-equivariant map induced map by the inclusion $\R^d\longrightarrow\R^{\infty}$, $x\longmapsto (x,0,0,\ldots)$ where $x\in\R^d$,
\item $\iota_{d,n}\colon \conf(\R^d,n)\longrightarrow \conf(\R^{\infty},n)$  \,:\, the $\Sym_n$-equivariant map induced map by the inclusion $\R^d\longrightarrow\R^{\infty}$, $x\longmapsto (x,0,0,\ldots)$ where $x\in\R^d$,
\item $\ev_{d,n}\colon\CC_d(n)\longrightarrow\conf(\R^d,n)$  \,:\, the evaluation at the centres of cubes map,
\item $\beta\colon \Z_2^{\oplus m}\longrightarrow [2^m]$  \,:\, the bijection  given by $(i_1,\ldots,i_m)\longmapsto 1+\sum_{j=1}^m 2^{m-j}i_j$,
\item $\alpha\colon\N\longrightarrow\N$  \,:\,  $\alpha(k)$ is the number of $1$s in the binary presentation of $k$,
\item $\epsilon\colon\N\longrightarrow\N$  \,:\,  $\epsilon(k)$   is the remainder of $k$ modulo $2$,
\item $\gamma\colon\N\longrightarrow\N$  \,:\,  $\gamma(k)=\lfloor\log_2k\rfloor+1$  is the minimal integer such that $2^{\gamma(k)}\geq k+1$.
%----------
\vspace{3mm}%
\item[{\bf Vector bundles\index{vector bundle}:}]
\vspace{3mm}%
%----------
\item $\xymatrix@1{ \xi_{X,k} \colon \  \R^k\ \ar[r]&\   \conf(X,k)\times_{\Sym_k}\R^k\ \ar[r]&\ \conf(X,k)/\Sym_k,}$
%the vector bundle  associated to the natural real permutation representation $\R^k$ of the symmetric group $\Sym_k$,
\item $\xymatrix@1{ \zeta_{X,k} \colon \ W_k\ \ar[r]&\   \conf(X,k)\times_{\Sym_k}W_k\ \ar[r]&\ \conf(X,k)/\Sym_k,}$
%the vector bundle  associated to the standard representation  $W_k=\{(a_1,\ldots, a_k)\in\R^k : a_1+\dots+a_k=0\}$ of the symmetric group $\Sym_k$,
\item $\xymatrix@1{ \tau_{X,k} \colon \  \R\ \ar[r]&\   \conf(X,k)/\Sym_k\times\R\ \ar[r]&\ \conf(X,k)/\Sym_k,}$
\item $\xymatrix@1{\xi_{X,k}^{\C} \colon \   \C^k\ \ar[r]&\  \conf(X,k)\times_{\Sym_k}\C^k\ \ar[r]&\    \conf(X,k)/\Sym_k,}$
\item $\xymatrix@1{\zeta_{X,k}^{\C} \colon \  W_k^C \ \ar[r]&\  \conf(X,k)\times_{\Sym_k}W_k^{\C}   \ \ar[r]&\  \conf(X,k)/\Sym_k,}$
\item $\xymatrix@1{ \tau_{X,k}^{\C} \colon \ \C   \ \ar[r]&\ \conf(X,k)/\Sym_k\times\C     \ \ar[r]&\  \conf(X,k)/\Sym_k,}$
%the vector bundle  associated to the trivial $1$-dimensional representation of the symmetric group $\Sym_k$,
\item $\xymatrix@1{ \xi_{d,k} \colon \  \R^k\ \ar[r]&\   \CC_d(k)\times_{\Sym_k}\R^k\ \ar[r]&\ \CC_d(k)/\Sym_k,}$ 
%the vector bundle  associated to the natural permutation representation $\R^k$ of the symmetric group $\Sym_k$.
\item $\xymatrix@1{ \zeta_{d,k} \colon \   W_k\ \ar[r]&\   \CC_d(k)\times_{\Sym_k}W_k\ \ar[r]&\ \CC_d(k)/\Sym_k,}$
\item $\xymatrix@1{ \tau_{d,k} \colon \  \R\ \ar[r]&\   \CC_d(k)/\Sym_k\times\R\ \ar[r]&\ \CC_d(k)/\Sym_k,}$
\item $\xymatrix@1{\lambda_{d,m} \colon \ W_{2^m}\ \ar[r]&\    \Sp(\R^d,2^m)\times_{\Sy_{2^m}} W_{2^m}\ \ar[r]&\ \Sp(\R^d,2^m)/\Sy_{2^m},}$
\item $\xymatrix@1{\tau_{d,m} \colon \R\ \ar[r]&\    \Sp(\R^d,2^m)/\Sy_{2^m}\times\R\ \ar[r]&\ \Sp(\R^d,2^m)/\Sy_{2^m},}$
\item $\xymatrix@1{ \gamma_{k} \colon \  \R^k\ \ar[r]&\   \EEE\OO(k)\times_{\OO(k)}\R^k\ \ar[r]&\ \B\OO(k),}$ 
\item $\xymatrix@1{ \gamma_{k}^{\C} \colon \  \C^k\ \ar[r]&\   \EEE\mathrm{U}(k)\times_{\mathrm{U}(k)}\R^k\ \ar[r]&\ \B\mathrm{U}(k),}$ 
\item $\xymatrix@1{ \xi_{k} \colon \  \R^k\ \ar[r]&\   \EEE\Sym_k\times_{\Sym_k}\R^k\ \ar[r]&\ \B\Sym_k,}$ 
\item $\xymatrix@1{ \eta_{2^m} \colon \ \R^{2^m}\ \ar[r]&\   \EEE\Sy_{2^m}\times_{\Sy_{2^m}}\R^{2^m}\ \ar[r]&\ \B\Sy_{2^m},}$ 
\item $\xymatrix@1{ \nu_{2^m} \colon \  \R^{2^m}\ \ar[r]&\   \EEE\EE_{m}\times_{\EE_{m}}\R^{2^m}\ \ar[r]&\ \B\EE_{m},}$ 
\item $\xymatrix@1{ \theta_{2^m} \colon \  \R^{2^m}\ \ar[r]&\   \EEE(\Sym_{2^{m-1}}^2)\times_{\Sym_{2^{m-1}}^2}\R^{2^m}\ \ar[r]&\ \B(\Sym_{2^{m-1}}^2),}$ 
\item $\xymatrix@1{ \omega_{2^m} \colon \  \R^{2^m}\ \ar[r]&\   \EEE(\Sy_{2^{m-1}}^2)\times_{\Sy_{2^{m-1}}^2}\R^{2^m}\ \ar[r]&\ \B(\Sy_{2^{m-1}}^2).}$ 
\end{compactitem}
 }

%%%%%%%%%%%%%%%%%%%%%%%%%%%%%%%%%%%%%%%%%%%%%%%%%%%%%%%%%%%%%%%%%%%%%%%%%%%%%%%%%%%%%
%%%%%%%%%%%%%%%%%%%%%%%%%%%%%%%%%%%%%%%%%%%%%%%%%%%%%%%%%%%%%%%%%%%%%%%%%%%%%%%%%%%%%
\section{Snapshots from the history}
\label{sec:history}
%%%%%%%%%%%%%%%%%%%%%%%%%%%%%%%%%%%%%%%%%%%%%%%%%%%%%%%%%%%%%%%%%%%%%%%%%%%%%%%%%%%%%
%%%%%%%%%%%%%%%%%%%%%%%%%%%%%%%%%%%%%%%%%%%%%%%%%%%%%%%%%%%%%%%%%%%%%%%%%%%%%%%%%%%%%

In many ways it is much harder to write accurately about the complete history of a subject than to make a contribution to its development. 
Even with the possibility to take a peek into the past such an endeavor is impossible to complete.
For these reasons
we give only a brief overview of the study of configuration spaces from the perspective of the contents of this book, choosing both seminal contributions as well as some particular topics to focus on.
Our presentation by no means aims to be, or~could be, complete.

\medskip
We begin our story by introducing the object we study; that is, we answer the question: {\em What is hiding under the name configuration space?} 

\medskip
Let $X$ be a topological space.
The {\bf ordered configuration space}\index{ordered configuration space} of all $n$-tuples of distinct points on $X$ is the following subspace of the product space $X^n$:
\[
\conf(X,n):=
\{
(x_1,\dots,x_n) \in X^n : x_i\neq x_j \text{ for all }1\leq i<j\leq n
\}.
\]
The ordered configuration space $\conf(X,n)$ is endowed with a natural free (left) action of the symmetric group on $n$ letters $\Sym_n$, given by the permutation of points, that is
\[
\pi\cdot (x_1,\dots,x_n)=(x_{\pi(1)},\dots,x_{\pi(n)}),
\]
where $\pi\in\Sym_n$ and $(x_1,\dots ,x_n)\in \conf(X,n)$.
The associated orbit space $\conf(X,n)/\Sym_n$ is called the {\bf unordered configuration space}\index{unordered configuration space} of $n$ distinct points on $X$.

\medskip
The official history of configuration spaces begins in 1925 with {\em  Theorie der Z\"opfe}, the fundamental work of Emil Artin \cite{Artin1925}, while the prehistory goes back to a work of Adolf Hurwitz \cite{Hurwitz1891} from 1891 and a 1930s work of Oscar Zariski \cite{Zariski1937}.
The braids, the oldest of gadgets of man, now transformed into a beautiful algebraic object, became the starting point for the intensive research branching over areas, sciences and decades not ever losing on intensity or its fundamental importance.

%-------------------------------------------------
\subsection{The braid group\index{braid group}}
%-------------------------------------------------

In his seminal paper \cite{Artin1925} from 1925 Artin introduced the notion, of what we would call today a \textbf{geometric braid}\index{braid}\index{geometric braid}, as a collection of $n$ disjoint arcs $(\beta_1,\dots,\beta_n)$ connecting two collections of $n$ pairwise distinct points $(x_1,\dots,x_n)$ and $(x_1',\dots,x_n')$ placed on the planes $z=0$ and $z=1$ of the $3$-dimensional Euclidean space $\R^3$ respectively.
In addition, the arcs $(\beta_1,\dots,\beta_n)$ need to satisfy the following two properties:
\begin{compactitem}[ \ ---]
\item $\beta_1(0)=x_1,\dots,\beta_n(0)=x_n$ and $\beta_1(1)=x_{i_1}',\dots,\beta_n(1)=x_{i_n}'$ with the index sets $\{i_1,\dots,i_n\}$ and $\{1,\dots,n\}$ coinciding, and 
\item $\beta_{i}(t)$ belongs to the plane $z=t$ for every $0\leq t\leq 1$ and every $1\leq i\leq n$.	
\end{compactitem}
Here, an arc is assumed to be a continuous injection of the segment $[0,1]$ into $\R^3$.
Classically the arcs of the braid $\beta_1,\dots,\beta_n$ are also called strings of the braid\index{strings of the braid}. 
For an illustration of a geometric braid see Figure \ref{fig:braid-01}.

\begin{figure}[ht]
\centering
\includegraphics[scale=1]{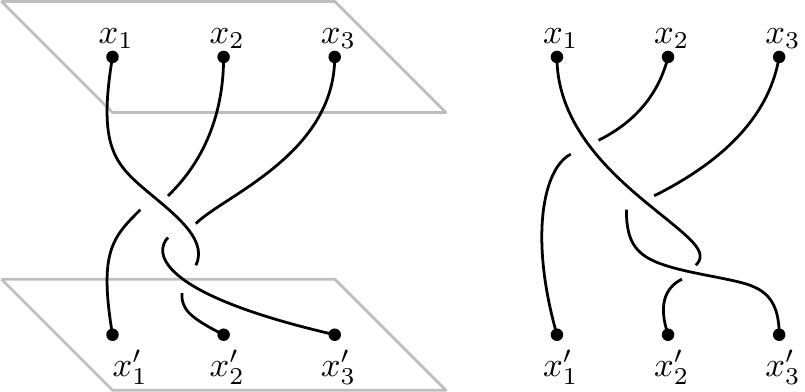}
\caption{\small A geometric braid and its projection to the affine plane spanned by the end points of the strings.}
\label{fig:braid-01}
\end{figure}

\medskip
Considering braids up to a homotopy (through the space of geometric braids), rather than how they are geometrically presented, Artin was able to define a natural operation on the set $\BB_n$ of homotopy classes of all braids between two fixed collections of points.
Artin \cite{Artin1925} said that two braids are equivalent (homotopic for us) if one braid can be deformed into another without self-intersections --- {\em Zwei solche Z\"opfe hei{\ss}en \"aquivalent oder k\"urzer gleich, wenn sie sich ineinander ohne Selbstdurchdringung deformieren lassen}.
Concatenation of braids and the shrinking procedure applied to the representatives of homotopy classes of braids nicely fit with the homotopy relation and as such a well defined operation on $\BB_n$ was introduced.   
This operation was associative and with the obvious neutral element --- the homotopy class of the trivial braid $\tau$ --- the collection of line segments $([x_1,x_1'],\dots,[x_n,x_n'])$ as strings.
Furthermore, the elementary braids $\sigma_i$, for $1\leq i\leq n-1$, were introduced as braids identical to the trivial braid in all strings except the $i$-th string crosses ``over'' $(i+1)$-th string.
Here ``over'' refers to a side, half-space, of the affine plan spanned by the end points of the strings.
An illustration of the plane projections of representatives of the braids $\tau$, $\sigma_i$ and its inverse $\sigma_i^{-1}$ are given in Figure \ref{fig:braid-02}.

\medskip
The most quoted result of Artin's paper \cite[Satz 1]{Artin1925} is the presentation of the {\bf braid group}\index{presentation of the braid group} $\BB_n$ on $n$ strings of the following form:
\[
\BB_n= \Big \langle \sigma_1,\dots,\sigma_{n-1} :
\begin{array}{ll}
\sigma_i\sigma_j=\sigma_j\sigma_i, &\text{ for }1\leq i<j-1\leq n-2\\
\sigma_i\sigma_{i+1}\sigma_i=\sigma_{i+1},\sigma_i\sigma_{i+1}&\text{ for }1\leq i\leq n-2
\end{array}
\Big \rangle,
\]
now known as Artin's presentation of the braid group\index{Artin's presentation of the braid group}.
An alternative and more formal argument for this result was further developed by Artin in his paper \cite{Artin1947-1} from 1947.
In particular, a relation of $s$-isotopy between braids is introduced and it is shown that two braids are $s$-isotopic if and only if they are homotopic (through the space of geometric braids); see \cite[Thm.\,8]{Artin1947-1}.

\begin{figure}[ht]
\centering
\includegraphics[scale=0.75]{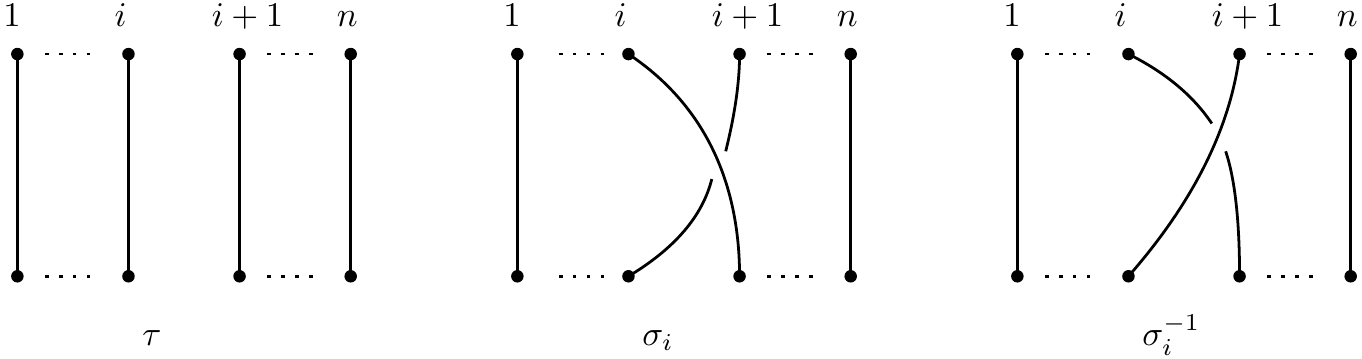}
\caption{\small The representatives of the braids $\tau$, $\sigma_i$ and $\sigma_i^{-1}$.}
\label{fig:braid-02}
\end{figure}

Furthermore, Artin in \cite[Sec.\,4]{Artin1925} defines a group homomorphism $\mathfrak{a}_n\colon \BB_n\longrightarrow\Sym_n$ which fits into the following short exact sequence of groups:
\[
\xymatrix{
1\ \ar[r] & \ \PP_n\  \ar[r] & \ \BB_n \ \ar[r]^-{\mathfrak{a}_n} & \ \Sym_n \ \ar[r] & \ 1. 
}
\]    
The normal subgroup $\PP_n:=\ker(\mathfrak{a}_n)$ of the braid group $\BB_n$ is called the {\bf pure braid group}\index{pure braid group} on $n$ strings.
Geometrically, $\PP_n$ is the set of all representatives of geometric braids $(\beta_1,\dots,\beta_n)$ which have the property that 
$\beta_1(0)=x_1,\dots,\beta_n(0)=x_n$ and $\beta_1(1)=x_{1}',\dots,\beta_n(1)=x_{n}'$.

\medskip
For more aspects in the study of braid groups consult the classical monograph {\em Braids, links, and mapping class groups} by Joan Birman \cite{Birman1974}.

 %-------------------------------------------------
\subsection{The fundamental sequence of fibrations}
%-------------------------------------------------

The work of Edward Fadell \& Lee Neuwirth \cite{FadellNeuwirth1962}  from 1962 gave birth to the name {\em configuration space}\index{ordered configuration space} for the space of all ordered collections of pairwise distinct points on a topological space. 
The central results of this fundamental work is the construction of the so-called fundamental sequence of fibrations of configuration spaces of a manifold.

\medskip
Let $M$ be a topological, connected manifold without a boundary of dimension at least $2$.
For an integer $m\geq 1$ let $Q_m=\{q_1,\dots,q_m\}$ be some fixed collection of $m$ distinct points on $M$, and in particular let $Q_0=\emptyset$.
In addition, let  $n\geq 1$ be an integer.
Consider a family of configuration spaces of the punctured manifold $M{\setminus}Q_m$ defined by
\[
\conf_{m,n}(M):=\conf(M{\setminus}Q_m,n) 
\]
with the corresponding projections 
\[
\xymatrix{
p_{m,n,r}\colon \conf_{m,n}(M) \ \ar[r] & \ M{\setminus}Q_m, \qquad (x_1,\dots x_n) \ \ar@{|->}[r] & \ (x_1,\dots,x_r)
}
\]
for $1\leq r\leq n-1$ and $(x_1,\dots x_n)\in \conf_{m,n}(M)$.
Observe that $\conf(M,n)=\conf_{0,n}(M)$.
Fadell \& Neuwirth, in \cite[Thm.\,1 and Thm.\,3]{FadellNeuwirth1962}, for $n\geq 2$ showed that $p_{m,n,1}$ is a locally trivial fibration with fibre $\conf_{m+r,n-r}(M)$.
In addition, they proved that the fibration $p_{m,n,1}$ admits a (continuous) cross-section\index{cross-section}.

\medskip
Using the sequence of fibrations\index{fibration} 
\[
{\small
\xymatrix@C=1.7em{
\conf_{n-1,1}(M)\ \ar[r] &\ \conf_{n-2,2}(M)\ar[r] \ \ar[d] & \ \cdots \ \ar[r]  &\  \conf_{2,n-2}(M)\ \ar[r]\ar[d] & \ \conf_{1,n-1}(M) \ \ar[r]\ar[d] & \ \conf_{0,n}(M)\ar[d]\\
			        & \conf_{n-2,1} (M)			    &               &      \conf_{2,1} (M)         &   \conf_{1,1} (M) &  \conf_{0,1}(M),
}
}
\]
and associated long exact sequences in homotopy, in combination with the existence of the corresponding cross-sections, they obtained various descriptions of homotopy groups of ordered configuration spaces\index{homotopy groups of ordered configuration space}.
For example, they showed that for $d\geq 2$ and $i\geq 2$ there exist isomorphisms
\[
\pi_i(\conf(\R^d,n))\cong \bigoplus_{k=1}^{n-1}\pi_i(\underbrace{S^{d-1}\vee\dots\vee S^{d-1}}_{k}),
\]
and 
\[
\pi_i(\conf(S^{2d-1},n))\cong \pi_i(S^{2d-1}) \oplus\bigoplus_{k=1}^{n-2}\pi_i(\underbrace{S^{2d-2}\vee\dots\vee S^{2d-2}}_{k}).
\]
In particular, they obtained that $\conf(\R^2,n)$ is an Eilenberg--Mac Lane space\index{Eilenberg--Mac Lane space} $\mathrm{K}(\PP_n,1)$.
Hence, the (co)homology of the pure braid group\index{pure braid group} $\PP_n$ coincides with the corresponding (co)homology of the configuration space $\conf(\R^2,n)$.
Consult \cite[Cor.\,2.1, Cor.\,5.1]{FadellNeuwirth1962}.

\medskip
Furthermore, almost at the same time,  Fadell in \cite[Thm.\,1 and Thm.\,2]{Fadell1962} gave additional descriptions of homotopy groups of configuration spaces for $d\geq 4$:
\[
\pi_i(\conf(S^{d},n))\cong
\pi_i(V_{2}(\R^{d+1}))\oplus\bigoplus_{k=1}^{n-2}\pi_i(\underbrace{S^{d-1}\vee\dots\vee S^{d-1}}_{k}),
\]
and 
\[
\pi_i(\conf(\RP^{d-1},n))\cong
\pi_i(V_{2}(\R^{d+1}))\oplus\bigoplus_{k=1}^{n-2}\pi_i(\underbrace{S^{d-1}\vee\dots\vee S^{d-1}}_{2k+1}).
\]
Here $V_{2}(\R^{d+1})$ denotes the Stiefel manifold\index{Stiefel manifold} of orthogonal $2$-frames in $\R^{d+1}$.

\medskip
For a detailed exposition of these results and much more about topology of configuration spaces consult the book {\em Geometry and Topology of Configuration Spaces} by Fadell \& Husseini \cite{FadellHusseini2001}.

%-------------------------------------------------
\subsection{Artin's presentation of $\BB_n$ and $\pi_1(\conf(\R^2,n))$}
%-------------------------------------------------

The seminal work of Fadell \&  Neuwirth \cite{FadellNeuwirth1962} we discussed in the previous section intertwines in an essential way with the work of Ralph Fox with Lee Neuwirth \cite{FoxNeuwirth1962} from the same year.
As explained at the beginning of their papers they aimed at ``a straightforward derivation'' of Artin's presentation of the braid group\index{braid group} $\BB_n$.

\medskip
Since the braid group $\BB_n$ can also be seen as the fundamental group of the unordered configuration space of $n$ distinct points in the plane, as already pointed out by Artin,  Fox and Neuwirth aimed to use classical knowledge for presentation of fundamental groups to obtain exactly Artin's presentation as a presentation of  $\pi_1(\conf(\R^2,n)/\Sym_n)$. 

\medskip
More precisely, using the lexicographic ordering on the coordinates\index{lexicographic ordering on the coordinates} of $\R^2$ the power set $(\R^2)^n$ can be stratified by a family of convex cones.
For example, for $n=7$ the stratum defined by the symbol 
\[
\theta = (3<5=1<6\ \underline{\underline{\vee}}\ 4 \ \underline{\underline{\vee}}\ 2=7)
\]
is the convex cone of $(\R^2)^7$ given by
\begin{multline*}
	\Big\{ (p_1,\dots,p_7)=(x_1,y_1,\dots,x_7,y_7)\in (\R^2)^7 : \\
\begin{array}{lllllllllllll}
	x_3 & < & x_5 & = & x_1 & < & x_6 & = & x_4 & = & x_2 & = & x_7 \\
	    &   & y_5 & = & y_1 &   & y_6 & < & y_4 & < & y_2 & = & y_7
\end{array}
\Big\}.
\end{multline*}
For an illustration of a point inside $\theta$ see Figure \ref{fig:braid-03}.
The action of the symmetric group $\Sym_n$ on $(\R^2)^n$ induces a stratum-wise action  -- translations given by action send strata to strata homeomorphically.    
Furthermore, such a stratification induces a regular $\Sym_n$-invariant cell complex structure on the one-point compactification $S^{2n}$ of the ambient $(\R^2)^n$.

\begin{figure}[ht]
\centering
\includegraphics[scale=0.9]{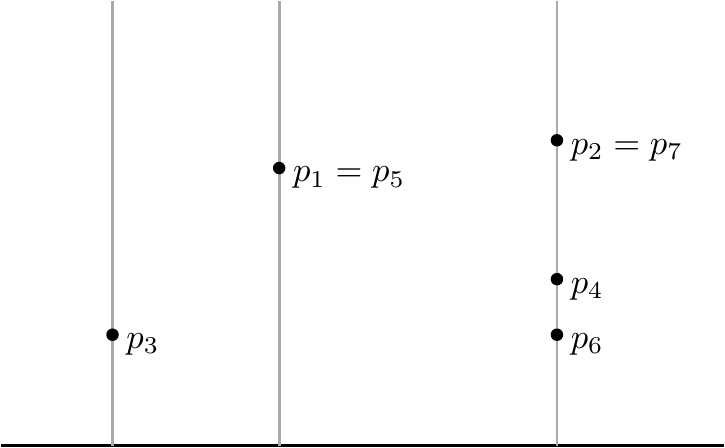}
\caption{\small A configuration that corresponds to a point in the stratum $(3<5=1<6\ \underline{\underline{\vee}}\ 4 \ \underline{\underline{\vee}}\ 2=7)$.}
\label{fig:braid-03}
\end{figure}

\medskip
Consider the $\Sym_n$-CW-subcomplex $\Delta$ of $S^{2n}$ given by all the cells whose symbols have at least one sign of equality ``$=$''.
In particular such a cell is induced by the stratum $\theta$.
Then the configuration space $\conf(\R^2,n)$ is the complement $S^{2n}{\setminus}\Delta$ of the $(2n-2)$-dimensional $\Sym_n$-CW-subcomplex $\Delta$ inside the $(2n)$-dimensional $\Sym_n$-CW-complex $S^{2n}$.
Furthermore, the unordered configuration space can also be seen as the complement
\[
\conf(\R^2,n)/\Sym_n= (S^{2n}/\Sym_n) \, {\setminus} \, (\Delta/\Sym_n)
\]
of the $(2n-2)$-dimensional CW-subcomplex $\Delta/\Sym_n$ inside the $(2n)$-dimensional  CW-complex $S^{2n}/\Sym_n$.
In particular, the CW-complex $S^{2n}/\Sym_n$ has one maximal $(2n)$-cell given by the symbol (representative) $(1<2<\dots<n)$.

\medskip
In such a situation, a presentation of the fundamental group of a complement of a cell complex, in our case $\pi_1(\conf(\R^2,n)/\Sym_n)= \pi_1((S^{2n}/\Sym_n) \, {\setminus} \, (\Delta/\Sym_n))$, can be obtained from:
\begin{compactitem}[\ ---]
\item the specified generators contained in the complement indexed by all $(2n-1)$-cells in the boundary of the $(2n)$-cell, and 
\item the relations corresponding to $(2n-2)$-cells contained in the complement induced from theirs coboundaries in a specific way. 
\end{compactitem}
Knowing this technical gadget Fox \& Neuwirth ``only'' needed to list generators and identify the corresponding relations, as the ``recipe'' suggested.
They did this in a beautiful and geometrically clear way in \cite[Sec.\,7]{FoxNeuwirth1962} showing that, at the end, the braid group $\BB_n\cong \pi_1(\conf(\R^2,n)/\Sym_n)$ can be given via Artin's presentation.
 
 \medskip
 Furthermore,  Fadell \&  Neuwirth noticed that the unordered configuration space $\conf(\R^2,n)/\Sym_n$ is an Eilenberg--Mac Lane space\index{Eilenberg--Mac Lane space} $\mathrm{K}(\BB_n,1)$.
 This implies, in particular, that the braid group\index{braid group} $\BB_n$ has no elements of finite order \cite[Cor.\,1]{FoxNeuwirth1962}.

\medskip
In parallel, Fadell \& James Van Buskirk in \cite{FadellVanBuskirk1961}, based on the work of Wei-Liang Chow \cite{Chow1948},  offered an argument of different flavour for the fact that $\BB_n\cong \pi_1(\conf(\R^2,n)/\Sym_n)$ can be described via Artin's presentation.

\medskip 
The idea of Fox \& Neuwirth to stratify the ambient $(\R^2)^n$ of the configuration space $\conf(\R^2,n)$ into cones motivated Anders Bj\"orner \& G\"unter M. Ziegler \cite{Bjorner1992} to use stratifications in construction of cell complex models\index{cell complex model} for general complements of subspace arrangements; it was essential in the work of Pavle Blagojevi\'c \& Ziegler \cite{Blagojevic2014}.

%-------------------------------------------------
\subsection{The cohomology  ring  $H^*(\conf(\R^2,n);\Z)$}
%-------------------------------------------------
 
The next section in our story about configuration spaces is dedicated to the seminal work of Vladimir Igorovich Arnol'd \cite{Arnold1969} from 1969.
He showed that the integral cohomology ring\index{integral cohomology ring} of the configuration space $\conf(\R^2,n)$ can be presented as a quotient of the exterior algebra (over the ring of integers) as follows:
\[
H^*(\conf(\R^2,n);\Z)\cong
\frac{\Lambda \big(\omega_{i,j} : 1\leq i < j\leq n \big)}{(\omega_{i,j}\omega_{j,k}+\omega_{j,k}\omega_{i,k}+\omega_{i,k}\omega_{j,k} : 1\leq i < j < k \leq n)}=:A(n).
\]
In particular, he described an additive basis\index{additive basis} of the cohomology $H^*(\conf(\R^2,n);\Z)$ and showed that the Poincar\'e polynomial\index{Poincar\'e polynomial} of  $\conf(\R^2,n);\Z)$ is given by
\[
p_{\conf(\R^2,n)}(t)=(1+t)(2+t)\cdots (1+(n-1)t).
\]
See for example \cite[Thm.\,1, Cor.\,1, Cor.\,3]{Arnold1969}.

\medskip 
The presented proof of these results was based on a masterful use of the Serre spectral sequence\index{Serre spectral sequence}. 
According to Arnol'd, Dmitry Fuks made a substantial contribution to this proof.
More precisely, Arnol'd considered the projection
\[
p\colon \conf(\R^2,n)\longrightarrow \conf(\R^2,n-1), \qquad (z_1,\dots,z_n)\longmapsto (z_1,\dots,z_{n-1})
\]
and showed that it is a fiber bundle which additionally admits a continuous cross-section.
Furthermore, in  \cite[Lem.\,1]{Arnold1969} he showed that the fundamental group of the base space -- the pure braid group\index{pure braid group} $\pi_1(\conf(\R^2,n-1))\cong\PP_{n-1}$ -- acts trivially on the cohomology of the fiber 
\[
\R^2{\setminus}\{z_1,\dots,z_{n-1}\} \simeq \underbrace{S^1\vee\dots\vee S^1}_{n-1}.
\] 
Here the point $(z_1,\dots,z_{n-1})\in \conf(\R^2,n-1)$ is assumed to be fixed.
With these ingredients the $E_2$-term of the Serre spectral sequence associated to the fiber bundle~$p$ is:
\begin{align*}
E^{r,s}_2 &= H^r(\conf(\R^2,n-1); \mathcal{H}^s(\R^2{\setminus}\{z_1,\dots,z_{n-1}\};\Z))\\
&\cong 	H^r(\conf(\R^2,n-1);\Z)\otimes H^s(\R^2{\setminus}\{z_1,\dots,z_{n-1}\};\Z).
\end{align*}
This spectral sequence converges to the cohomology of the base space of the fiber bundle -- in this case $H^*(\conf(\R^2,n);\Z)$.
In the description of the $E_2$-term only the triviality of the action of $\pi_1(\conf(\R^2,n-1))$ on the cohomology of the fiber is used.
Now the existence of the cross-section of the fiber bundle $p$ implies that the only possible non-zero differential $\partial_2$ has to vanish.
Consequently, the spectral sequence is completely determined and collapses at the $E_2$-term, that is $E^{r,s}_2\cong E^{r,s}_{\infty}$ for all integers $r$ and $s$.
Since the cohomology groups of the fiber $H^*(\R^2{\setminus}\{z_1,\dots,z_{n-1}\};\Z)$ are free abelian groups (free $\Z$-modules) in every dimension using induction on $n$ we see that all entries of the $E_{\infty}$-term are free Abelian groups. 
Hence, the spectral sequence has no extension problem and the additive structure of the cohomology of the configuration space $H^*(\conf(\R^2,n);\Z)$ is completely determined.
For example, one can write
\[
H^*(\conf(\R^2,n);\Z)\cong
H^*(S^1\times (S^1\vee S^1 )\times\dots\times (\underbrace{S^1\vee\dots\vee S^1}_{n-1});\Z).
\]

\medskip 
{\em What about the ring structure on the cohomology $H^*(\conf(\R^2,n);\Z)$?} 
To answer this question Arnol'd brought into play a beautiful new idea which was to be used over and over again for years to come.
He defined the map $\varphi\colon A(n)\longrightarrow H^*(\conf(\R^2,n);\R)$,
between the quotient of the exterior algebra and the de Rham cohomology\index{De Rham cohomology} of the configuration space $\conf(\R^2,n)$, by sending the algebra generator $\omega_{i,j}$, $1\leq i<j\leq n$, to the cohomology class of the logarithmic differential form
\[
w_{i,j}:=\frac{1}{2\pi i}\cdot\frac{dz_i-dz_j}{z_i-z_j}=\frac{d\log(z_i-z_j)}{2\pi i} .
\]
The map turns out to be well defined and, in particular, a ring homomorphism. 
Now, since the cohomology classes of $w_{i,j}$, for all $1\leq i<j\leq n$, are integral, the ring homomorphism  $\varphi$ factors as follows:
\[
\xymatrix@1{
A(n) \  \ar[rr]^-{\varphi}\ar[rd]^{\varphi'} & & \ H^*(\conf(\R^2,n);\R)\\
& H^*(\conf(\R^2,n);\Z)\ar[ru]^{\psi} . &
}
\]
Here, the ring homomorphism $\psi$ is induced by the coefficient inclusion $\Z\longrightarrow\R$. 
Finally, combining the knowledge of additive structures of $A(n)$ and the cohomology $H^*(\conf(\R^2,n);\Z)$, Arnol'd proved that the ring homomorphism $\varphi'$ is actually a ring isomorphism.

%-------------------------------------------------
\subsection{The cohomology of the braid group\index{braid group} $\BB_n$}
%-------------------------------------------------

A year later, in 1970 Arnol'd published yet another breakthrough paper \cite{Arnold1970-01} in which he studied the cohomology, now of the unordered configuration space\index{unordered configuration space} $\conf(\R^2,n)/\Sym_n$.

\medskip
In order to argue the importance of understanding the topology of the unordered configuration space he gave a list of various important incarnations of the unordered configuration space, like the space of all monic polynomials of degree $n$ in the polynomial ring $\C[z]$ (algebraic functions) without multiple roots, the space of all hyperelliptic curves of degree $n$ (see \cite{Arnold1968-01}, \cite{Varcenko1969}), and the set of regular values of the mapping $\Sigma^{1n}$ (see~\cite{Arnold1968-02}).

\medskip
Using the fact that the unordered configuration space $\conf(\R^2,n)/\Sym_n$ is a $\mathrm{K}(\BB_n,1)$-space he utilized the isomorphism $H^i(\BB_n;\Z)\cong H^i(\conf(\R^2,n)/\Sym_n;\Z)$ (with trivial integer coefficients), to obtain the following fundamental facts about the additive cohomology structure of the braid group.

\begin{theorem}[Finiteness, Repetition and  Stability theorem\index{Finiteness, Repetition and  Stability theorem}]
Let $n\geq 1$ be an integer.
\begin{compactenum}[\rm \ (1)]
\item The cohomology groups of the braid group $\BB_n$ are all finite, except $H^0(\BB_n;\Z)\cong H^1(\BB_n;\Z)\cong\Z$. Furthermore, $H^i(\BB_n;\Z)=0$ for all $i\geq n$.
\item For all integers $n\geq 1$ and $i\geq 0$ there is an isomorphism 
\[
H^i(\BB_{2n+1};\Z)\cong H^i(\BB_{2n};\Z).
\]
\item For all integers $n\geq 1$ and $i\geq 0$ with the property that $n\geq 2i-2$ there is an isomorphism
\[
H^i(\BB_{n};\Z)\cong H^i(\BB_{2i-2};\Z).
\]
\end{compactenum}

\end{theorem}

Furthermore, by direct computations Arnol'd completed the following table of cohomologies for the first ten braid groups $\BB_2,\dots,\BB_{11}$.

\medskip
{\small\hspace{3pt}
\begin{center}
\begin{tabular}{ |r | c c c c c c c c c c  | } 
 \hline 
   								$i=$	 & \hspace{3pt}  $0$   & \hspace{3pt}  $1$  & \hspace{3pt} $2$  & \hspace{3pt} $3$  & \hspace{3pt} $4$  & \hspace{3pt} $5$  & \hspace{3pt} $6$  & \hspace{3pt} $7$  & \hspace{3pt} $8$  & \hspace{3pt} $9$  \\  [0.7ex] \hline
 $H^i(\BB_2)\cong H^i(\BB_3) $ & \hspace{3pt}$\Z$ & \hspace{3pt}$\Z$ & \hspace{3pt}$0$ & \hspace{3pt}$0$ & \hspace{3pt}$0$ & \hspace{3pt}$0$ & \hspace{3pt}$0$ & \hspace{3pt}$0$ & \hspace{3pt}$0$ & \hspace{3pt}$0$\\  [0.5ex]
 $H^i(\BB_4)\cong H^i(\BB_5) $ & \hspace{3pt}$\Z$ & \hspace{3pt}$\Z$ & \hspace{3pt}$0$ & \hspace{3pt}$\Z_2$ & \hspace{3pt}$0$ & \hspace{3pt}$0$ & \hspace{3pt}$0$ & \hspace{3pt}$0$ & \hspace{3pt}$0$ & \hspace{3pt}$0$\\ [0.5ex]
 $H^i(\BB_6)\cong H^i(\BB_7) $ & \hspace{3pt}$\Z$ & \hspace{3pt}$\Z$ & \hspace{3pt}$0$ & \hspace{3pt}$\Z_2$ & \hspace{3pt}$\Z_2$ & \hspace{3pt}$\Z_3$ & \hspace{3pt}$0$ & \hspace{3pt}$0$ & \hspace{3pt}$0$ & \hspace{3pt}$0$\\  [0.5ex] 
 $H^i(\BB_8)\cong H^i(\BB_9)$ & \hspace{3pt}$\Z$ & \hspace{3pt}$\Z$ & \hspace{3pt}$0$ & \hspace{3pt}$\Z_2$ & \hspace{3pt}$\Z_2$ & \hspace{3pt}$\Z_6$ & \hspace{3pt}$\Z_3$ & \hspace{3pt}$\Z_2$ & \hspace{3pt}$0$ & \hspace{3pt}$0$\\  [0.5ex]
 $H^i(\BB_{10})\cong H^i(\BB_{11}) $ & \hspace{3pt}$\Z$ & \hspace{3pt}$\Z$ & \hspace{3pt}$0$ & \hspace{3pt}$\Z_2$ & \hspace{3pt}$\Z_2$ & \hspace{3pt}$\Z_6$ & \hspace{3pt}$\Z_6$ or $\Z_3$ & \hspace{3pt}$\Z_2$ or $0$ & \hspace{3pt}$\Z_2$ & \hspace{3pt}$\Z_5$\\    [0.5ex]
 \hline
\end{tabular}
\end{center} 
}

%-------------------------------------------------
\subsection{The cohomology ring $H^*(\BB_n;\F_2)$}
%-------------------------------------------------

In parallel with the work of Arnol'd, a flood of new ideas was presented by Dmitry Borisovich Fuks in his seminal paper \cite{Fuks1970}, which aimed to describe the cohomology ring of the braid group\index{braid group} $H^*(\BB_n;\F_2)$. 

\medskip
The approach Fuks used differed from the one applied by Arnol'd.
The initial idea was to consider the sequence of group embeddings
\begin{equation}
\label{intro:seq01}	
\xymatrix{
\BB_n \ \ar[r]^-{\mathfrak{a}_n} & \ \Sym_n\ar[r] \ &\ \OO(n)\ar[r]& \ \mathrm{U}(n).
}
\end{equation} 
The sequence of homomorphism \eqref{intro:seq01} induces the corresponding sequence of continuous maps between the classifying spaces\index{classifying space}:
\begin{equation}
\label{intro:seq02}	
\xymatrix{
\conf(\R^2,n)/\Sym_n \cong \B \BB_n \ \ar[r]^-{\B\mathfrak{a}_n}  & \ \B\Sym_n\ar[r] &\ \B\OO(n)\ \ar[r]&\ \B\mathrm{U}(n).
}
\end{equation}
The classifying space $\B \BB_n$ can be modeled with $\conf(\R^2,n)/\Sym_n$, because $\conf(\R^2,n)/\Sym_n$ is a $\mathrm{K}(\B \BB_n,1)$-space. 
The sequence of continuous maps \eqref{intro:seq02} induces the following sequence of morphisms of (pull-back) real vector bundles $\xi_{\R^2,n}\longrightarrow\xi_n\longrightarrow\gamma_n$, that~is,
\begin{equation}
\label{intro:seq03}
\xymatrix{
\conf(\R^2,n)\times_{\Sym_n}\R^n \ \ar[r]\ar[d] &\ \EEE \Sym_n\times_{\Sym_n}\R^n\ \ar[r]\ar[d] &\  \EEE\OO(n)\times_{\OO(n)}\R^n\ar[d]\\
\conf(\R^2,n)/\Sym_n \ \ar[r]  & \  \B\Sym_n \ \ar[r] & \ \B\OO(n).
}
\end{equation}	 
In addition, the following morphisms between complex vector bundles $\xi_{\R^2,n}^{C}\longrightarrow\gamma_n^{C}$ can also be induced, that is
\begin{equation}
\label{intro:seq0w}
\xymatrix{
\conf(\R^2,n)\times_{\Sym_n}\C^n \ \ar[r]\ar[d] & \ \EEE\mathrm{U}(n)\times_{\mathrm{U}(n)}\C^n\ar[d]\\ 
\conf(\R^2,n)/\Sym_n \ \ar[r] &\ \B\mathrm{U}(n).
}
\end{equation}
Now, the basic results of Fuks' paper \cite{Fuks1970} can be stated as follows.

\begin{theorem}
\label{thm : Fuks 01}
Let $n\geq 1$ be an integer.
\begin{compactenum}[\rm \ (1)]
\item \label{thm : Fuks 01-01} The homomorphism in cohomology
\[
\xymatrix{
H^*(\B\OO(n);\F_2) \ \ar[r] & \ H^*(\conf(\R^2,n)/\Sym_n;\F_2),
}
\]
induced by the bundle morphism \eqref{intro:seq03}, is an epimorphism. 
In other words, the cohomology (algebra) ring $H^*(\BB_n;\F_2)\cong H^*(\conf(\R^2,n)/\Sym_n;\F_2)$ is generated by the Stiefel--Whitney classes of the vector bundle $\xi_{\R^2,n}$, that is 
\begin{multline}
\label{intro:seq05}
H^*(\BB_n;\F_2)\cong H^*(\conf(\R^2,n)/\Sym_n;\F_2)\cong\\
\F_2[w_1(\xi_{\R^2,n}),\dots, w_{n-1}(\xi_{\R^2,n})]/I_n,	
\end{multline} 
where  
\[
\xymatrix{
I_n=\ker \big(H^*(\B\OO(n);\F_2) \ \ar[r] & \  H^*(\conf(\R^2,n)/\Sym_n;\F_2)\big).
}
\]

\item \label{thm : Fuks 01-02}  The homomorphism in cohomology
\begin{equation}
\label{intro:seq06}	
\xymatrix{
H^*(\B\mathrm{U}(n);\F_2) \ \ar[r] & \ H^*(\conf(\R^2,n)/\Sym_n;\F_2),
}
\end{equation}
is the zero homomorphism in all positive degrees.
\end{compactenum}
	
\end{theorem}
Note that the bundle $\xi_{\R^2,n}$ has a non-vanishing cross section and consequently $w_{n}(\xi_{\R^2,n})=0$ does not appear in \eqref{intro:seq05}.
Furthermore, the zero homomorphism \eqref{intro:seq06} factors through the non-zero homomorphism $H^*(\B\mathrm{U}(n);\F_2)\longrightarrow H^*(\B\Sym_n;\F_2)$.

\medskip
{\em How did Fuks obtain these results?}
First, he realised that the lexicographic stratification\index{lexicographic stratification} of $(\R^2)^n$, introduced by Fox \& Neuwirth in \cite{FoxNeuwirth1962}, can be used to obtain a non-regular cell complex model for the one-point compactification $\widehat{\conf(\R^2,n)/\Sym_n}$ of the unordered configuration space\index{unordered configuration space} $\conf(\R^2,n)/\Sym_n$.
In particular, the cell complex model\index{cell complex model} has one $0$-cell, the infinity point, and for every integer partition $(n_1,\dots,n_k)$ of the integer $n=n_1+\dots+n_k$ a cell $e(n_1,\dots,n_k)$ of dimension $n+k$.
Furthermore, he computed the boundary operator for all associated generators in the cellular chain complex, that is
\[
\partial e(n_1,\dots,n_k)=
\sum_{i=1}^k{{n_i+n_{i+1}}\choose{n_i}}e(n_1,\dots,n_{i-1},n_i+n_{i+1},\dots,n_k).
\]
(Here, with the usual abuse of notation, $e(n_1,\dots,n_k)$ also denotes the corresponding generator in the chain group $C_{n+k}(\widehat{\conf(\R^2,n)/\Sym_n};\F_2)$.)
Then the Poincar\'e duality\index{Poincar\'e duality} isomorphism for relative homology manifolds, as in \cite[Thm.\,70.2]{Munkres1984}:
\[
H^*(\conf(\R^2,n)/\Sym_n;\F_2)\cong \widetilde{H}_{2n-*}(\widehat{\conf(\R^2,n)/\Sym_n};\F_2)
\]
allowed Fuks to go back and forth between the cohomology $H^*(\conf(\R^2,n)/\Sym_n;\F_2)$ and the homology $H_*(\widehat{\conf(\R^2,n)/\Sym_n};\F_2)$ of the explicitly given CW-complex for $\widehat{\conf(\R^2,n)/\Sym_n}$.
In this way he was able, for example, to show the following results.

\begin{theorem}
\label{thm : Fuks 02}
Let $n\geq 1$ and $k\geq 0$ be integers.
\begin{compactenum}[\rm \ (1)]
\item The dimension of the $\F_2$ vector space 
\[
H^k(\BB_n;\F_2)\cong H^k(\conf(\R^2,n)/\Sym_n;\F_2)
\] 
is equal to the number of representations of the integer $n$ as a sum of $n-k$ powers of $2$, that is the number of sets $\{i_1,\dots,i_{n-k}\}$ of non-negative integers such that $k=2^{i_1}+\dots+2^{i_{n-k}}$.
\item The group 
\[
H^{n-1}(\BB_n;\F_2)\cong H^{n-1}(\conf(\R^2,n)/\Sym_n;\F_2) 
\] 
does not vanish if and only if $n$ is power of two.
\end{compactenum}

\end{theorem}

\medskip
For every integer $n\geq 1$ there is the natural inclusion homomorphism of groups $\varphi_n\colon\BB_n\longrightarrow\BB_{n+1}$ induced by extending a collection of $n$ strings with a trivial $(n+1)$st string.
The approach Fuks employed allowed him also to prove the following stability results.

\begin{theorem}
\label{thm : Fuks 03}
Let $n\geq 1$ be an integer.
\begin{compactenum}[\rm \ (1)]
\item The homomorphism $\varphi_n^*\colon H^*(\BB_{n+1};\F_2)\longrightarrow	  H^*(\BB_{n};\F_2)$ is an epimorphism.
\item If $n$ is even, then the homomorphism $\varphi_n^*$ is an isomorphism.
\end{compactenum}
\end{theorem}

For an additional presentation of results by Fuks consult the famous book of Vassiliev {\em Complements of Discriminants of Smooth Maps: Topology and Applications}, \cite[Ch.\,I]{Vassiliev1992}.

 %-------------------------------------------------
\subsection{Cohomology of braid spaces}
%-------------------------------------------------  
 
In 1973, in a paper of Frederick Cohen \cite{CohenF1973-1} titled {\em Cohomology of braid spaces}, came an announcement, a teaser, for the landmark Springer Lecture Notes in Mathematics volume 533, {\em The Homology of Iterated Loop Spaces}, written 
%authored 
by Cohen, Thomas Lada \& J. Peter May \cite{SLNM533-1976}.
Cohen presented two theorems \cite[Thm.\,1 and Thm.\,2]{CohenF1973-1} and outlined a proof, with details appearing in \cite{Cohen1976LNM533}, as an auxiliary tool on the road towards the homology of $\CC_d$-spaces. 
Both the results of \cite{CohenF1973-1} and \cite{Cohen1976LNM533} and even more the proof methods played a key role in various applications over the years.
\medskip

{\em What was announced in \cite{CohenF1973-1} and then proved in \cite{Cohen1976LNM533}, or in other words what is the cohomology of braid spaces?}
In this article under the name of a braid space was hidden an unordered configuration space $\conf(M,n)/\Sym_n$ of a manifold $M$ of dimension at least~$2$.
Motivated by the fundamental work of May \cite{May1970} \cite{May1972} related to the study of iterated loop spaces and corresponding homology operations, Cohen computed specific cohomologies of the unordered configurations space $\conf(\R^d,n)/\Sym_n$ in the case when $n=p$ is a prime.

\medskip
To be more precise let us fix an odd prime $p$, and integers $d\geq 2$ and $q\geq 0$.
Furthermore, let $\mathcal{F}_p(q)$ denotes the $\Sym_p$-module defined on the ground vector space $\F_p$ by $\pi\cdot x=(-1)^{q\cdot\sgn(\pi) }x$.
The cohomology Cohen considered is the cohomology of the cochain complex
\begin{equation}
\label{intro:cochaincomplex}	
\hom_{\Sym_p} \big(C_*\conf(\R^d,p);\mathcal{F}_p(q)\big)
\end{equation}
of all $\Sym_p$-equivariant cochains\index{equivariant cochain} in the cochain complex\index{singular cochain complex} $\hom\big(C_*\conf(\R^d,p);\mathcal{F}_p(q)\big)$.
Here $C_*\conf(\R^d,p)$ denotes the singular chain complex of the ordered configuration space\index{ordered configuration space} $\conf(\R^d,p)$ with the natural structure of a free $\Sym_p$-module, or $\Z[\Sym_p]$-module, inherited from the free $\Sym_p$-action on $\conf(\R^d,p)$.

\medskip
In the case when $q$ is even and since the action of $\Sym_p$ on $\conf(\R^d,p)$ is proper the cohomology of the cochain complex \eqref{intro:cochaincomplex} coincides with the (usual) cohomology 
\[
H^*(\conf(\R^d,p)/\Sym_p;\F_p)
\]
of the unordered configuration space $\conf(\R^d,p)/\Sym_p$ with (trivial) coefficients in the field $\F_p$.
Furthermore, in this case the cohomology has a structure of a ring.
In the case when $d\geq 2$ the cohomology of the cochain complex \eqref{intro:cochaincomplex} coincides with the cohomology 
\[
H^*(\conf(\R^d,p)/\Sym_p;\mathcal{F}_p(q))
\]
of the unordered configuration space\index{unordered configuration space} $\conf(\R^d,p)/\Sym_p$ with local (twisted) coefficients in the  $\Sym_p$-module $\mathcal{F}_p(q)$.
For simplicity, by an abuse of notation we denote the cohomology of the cochain complex \eqref{intro:cochaincomplex} always by $H^*(\conf(\R^d,p)/\Sym_p;\mathcal{F}_p(q))$. 
Note that in the case when $q$ is odd the cohomology $H^*(\conf(\R^d,p)/\Sym_p;\mathcal{F}_p(q))$ does not have a structure of a ring, only an additional structure of $\F_p$-module.

\medskip
Before we state the main results announced in \cite{CohenF1973-1}, and proved in all the details in  \cite{Cohen1976LNM533}, we need to set the stage.

\medskip
For integers $n\geq 2$ and $d\geq 1$ let  $\iota_{d,n}\colon \conf(\R^d,n)\longrightarrow \conf(\R^{\infty},n)$ denote the $\Sym_n$-equivariant continuous map induced by the inclusion $\R^d\longrightarrow\R^{\infty}$, $x\longmapsto (x,0,0,\ldots)$ where $x\in\R^d$.
It induces the following morphism of fibrations:
\[
\xymatrix{
 \conf(\R^d,n) \ \ar[rr]^-{\iota_{d,n}}\ar[d] & & \  \conf(\R^{\infty},n)\ar[d]\\
  \conf(\R^d,n)/\Sym_n\ \ar[rr]^-{\iota_{d,n}/\Sym_n} & & \ \conf(\R^{\infty},n)/\Sym_n.
}
\]
Since the configuration space $\conf(\R^{\infty},n)$ is contractible and equipped with a free $\Sym_n$-action the orbit space $\conf(\R^{\infty},n)/\Sym_n$ is a model for $\B\Sym_n$.
In particular, $H^*(\Sym_n)\cong H^*(\conf(\R^{\infty},n)/\Sym_n)$ with any appropriately defined coefficients. 

\medskip
Let $A$ and $B$ be connected $\Z_{\geq 0}$-graded $\F_p$-algebras, where {\em connected} refers to $A_0\cong B_0\cong \F_p$. 
The $\sqcap$-product of $A$ and $B$ is the connected graded $\F_p$-algebras $A\sqcap B$ given for an integer $n\geq 0$ by
\[
(A\sqcap B)_n:=
\begin{cases}
	\F_p, 			&n=0,\\
	A_n\times B_n,  &n\geq 1.
\end{cases}
\]
The product structure on $A\sqcap B$ is specified by $A_s\cdot B_r=0$ for all $s\geq 1$ and $r\geq 1$, and by the requirement that both projection maps $A\sqcap B\longrightarrow A$ and $A\sqcap B\longrightarrow B$ be algebra homomorphisms. 
Consult also  \cite[pp.\,245-246]{Cohen1976LNM533}.

\medskip
Now we can present the following result of Cohen \cite[Thm.\,1]{CohenF1973-1},  \cite[Thm.\,5.2]{Cohen1976LNM533}.

\begin{theorem}\label{cohen:theorem1}
Let $d\geq 2$ be an integer, $p$ an odd prime and $q$ an even integer.
Then 
\[
H^*(\conf(\R^d,p)/\Sym_p;\mathcal{F}_p(q))
\cong
A_d \sqcap\im (\iota_{d,n}/\Sym_n)^*
\]
as a connected $\F_p$-algebra.
Here the $\F_p$-algebra $\im (\iota_{d,n}/\Sym_n)^*$ is given by 
\begin{eqnarray*}
\im (\iota_{d,n}/\Sym_n)^*&\cong
&
H^*(\Sym_p;\mathcal{F}_p(q))/\ker (\iota_{d,n}/\Sym_n)^*\\
&\cong &
H^*(\Sym_p;\mathcal{F}_p(q))/H^{\geq (d-1)(p-1)+1}(\Sym_p;\mathcal{F}_p(q)),	
\end{eqnarray*}
and the graded $\F_p$-algebra $A_d$ by
\[
A_d
=
\begin{cases}
\Lambda (a), & d\text{ even},\\
\F_p,		 & d\text{ odd}.
\end{cases}
\]
The element $a$ is of degree $d-1$, $\Lambda (a)$ is the exterior algebra generated by $a$, and $\F_p$ denotes the trivial connected $\F_p$-algebra.
\end{theorem}

Implicitly, we said that $\ker (\iota_{d,n}/\Sym_n)^*$ is the ideal of $H^*(\Sym_p;\mathcal{F}_p(q))$ consisting of all elements of the ring of degree $\geq (d-1)(p-1)+1$.

\medskip
It is important to mention that the cohomology ring of the symmetric group\index{cohomology ring} $\Sym_p$ with trivial $\F_p$ coefficients was already known at that time \cite[p.\,158]{May1970}.  
Concretely, for $q$ even
\[
H^*(\Sym_p;\F_p)\cong H^*(\Sym_p;\mathcal{F}_p(q))\cong \Lambda(b)\otimes \F_p[\beta b],
\]
where $b$ is an element of degree $2(p-1)-1$, $\Lambda (b)$ is the exterior algebra generated by $b$, $\beta b$ is the Bockstein\index{Bockstein homomorphism} of $b$, and $\F_p[\beta b]$ the polynomial algebra generated by $\beta b$.

\medskip
Next we give the following result of Cohen \cite[Thm.\,2]{CohenF1973-1},  \cite[Thm.\,5.3]{Cohen1976LNM533}.

\begin{theorem}\label{cohen:theorem2}
Let $d\geq 2$ be an integer, $p$ an odd prime and $q$ an odd integer.
Then 
\[
H^*(\conf(\R^d,p)/\Sym_p;\mathcal{F}_p(q))
\cong
M_d \oplus\im (\iota_{d,n}/\Sym_n)^*
\]
as an $\F_p$-vector space, or as an $H^*(\Sym_p;\F_p)$-module, or as an $H^*(\conf(\R^d,p)/\Sym_p;\F_p)$-module.
Here the $\F_p$-vector space, or $H^*(\Sym_p;\F_p)$-module, or $H^*(\conf(\R^d,p)/\Sym_p;\F_p)$-module, $\im (\iota_{d,n}/\Sym_n)^*$ is given by
\begin{eqnarray*}
\im (\iota_{d,n}/\Sym_n)^*&\cong
&
H^*(\Sym_p;\mathcal{F}_p(q))/\ker (\iota_{d,n}/\Sym_n)^*\\
&\cong &
H^*(\Sym_p;\mathcal{F}_p(q))/H^{\geq (d-1)(p-1)+1}(\Sym_p;\mathcal{F}_p(q)),	
\end{eqnarray*}
and the $\F_p$-vector space, or the $H^*(\Sym_p;\F_p)$-module, or $H^*(\conf(\R^d,p)/\Sym_p;\F_p)$-module, $M_d$ is determined by
\[
M_d
=
\begin{cases}
0 , 								& d\text{ even},\\
\F_p=\langle\lambda\rangle ,		 & d\text{ odd and }\deg(\lambda)=\tfrac{1}{2}(d-1)(p-1) .
\end{cases}
\]
The $H^*(\Sym_p;\F_p)$-module structure on $M_d$ is trivial, which means that the generator $\lambda$ is annihilated by all elements of positive degree of the ring $H^*(\Sym_p;\F_p)$.
\end{theorem}

Again, implicitly we have that $\ker (\iota_{d,n}/\Sym_n)^*$ is the $H^*(\Sym_p;\F_p)$-submodule of $H^*(\Sym_p;\mathcal{F}_p(q))$ generated by all elements of degree $\geq (d-1)(p-1)+1$.

\begin{figure}[ht]
\centering
\includegraphics[scale=0.7]{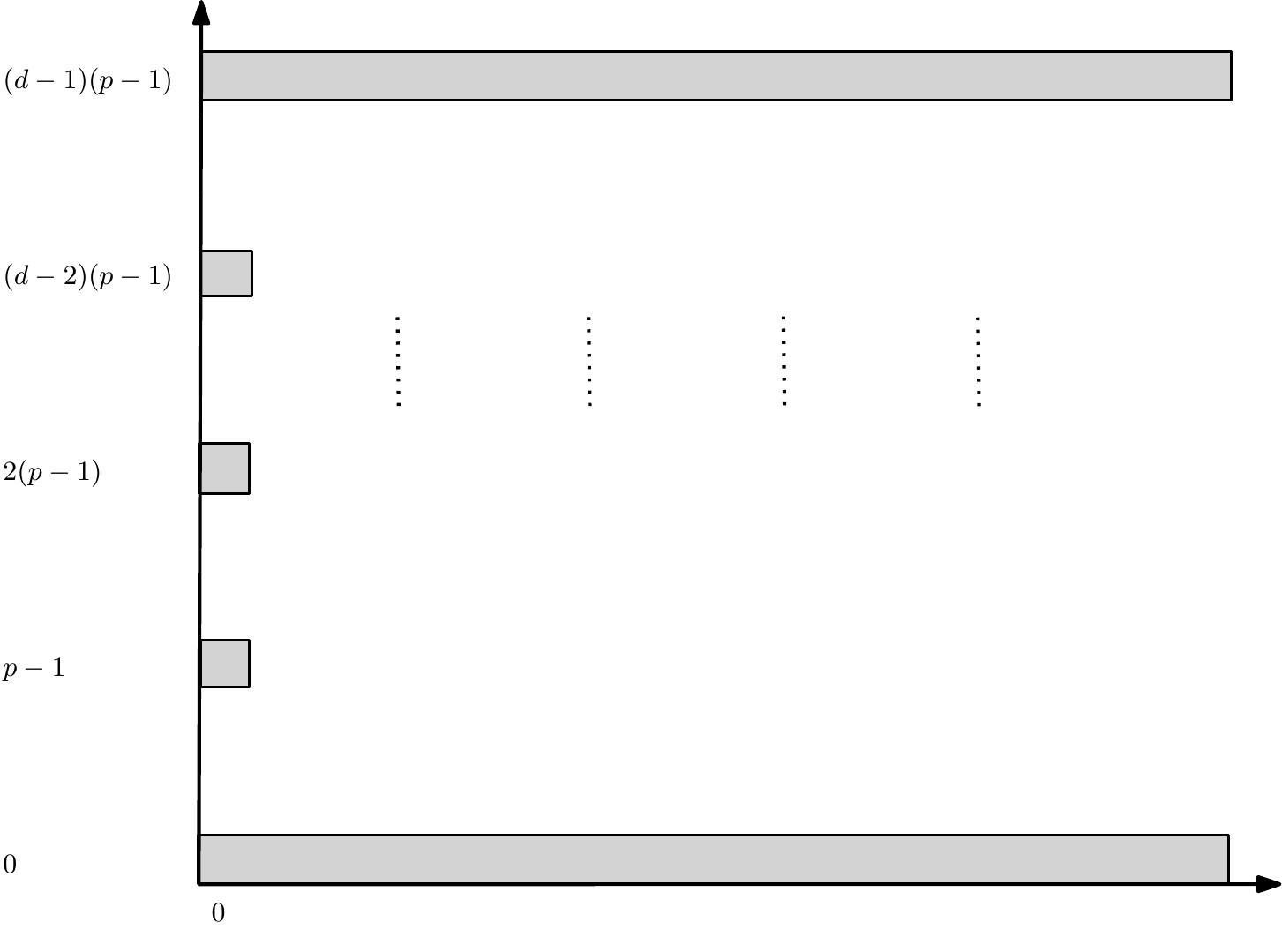}
\caption{\small The shape of $E_2$-term of the Serre spectral sequence associated to the fiber bundle \eqref{cohen:fiber_bundle_2}.}
\label{fig:spectralsequence}
\end{figure}

\medskip
The methods used in the proofs of Theorem \ref{cohen:theorem1} and Theorem \ref{cohen:theorem2} are at least as important as the results themselves.
The proofs are presented on more than 50 pages in \cite[Sec.\,5-11]{Cohen1976LNM533}. 
One of the technical highlights is the so-called \emph{Vanishing Theorem}\index{Vanishing Theorem} \cite[Thm.\,8.2]{Cohen1976LNM533}, which was also announced in \cite[Thm.\,4]{CohenF1973-1}. 
It gives a description of the cohomology Serre spectral sequences\index{Serre spectral sequence}, with coefficients in the appropriately interpreted $\Sym_p$-module associated with fiber bundles
\begin{equation}\label{cohen:fiber_bundle_1}
\xymatrix{
\conf(\R^d,p)\ \ar[r] &\ \conf(\R^d,p)\times_{\Sym_p}\EEE\Sym_p \ \ar[r] & \ \B \Sym_p
}	
\end{equation}
and
\begin{equation}\label{cohen:fiber_bundle_2}
\xymatrix{
\conf(\R^d,p)\ \ar[r] &\ \conf(\R^d,p)\times_{\Z_p}\EEE\Sym_p \ \ar[r] &\ \B \Z_p.
}
\end{equation}
Here the cyclic group $\Z_p$ is the Sylow $p$-subgroup\index{Sylow $p$-subgroup} of the symmetric group $\Sym_p$ generated by the cyclic shift, or in other words by the $p$-cycle $(12\dots p)$.
An illustration of the $E_2\cong E_{(d-1)(p-1)+1}$-term of the Serre spectral sequence with trivial $\F_p$ coefficients associated to the fiber bundle \eqref{cohen:fiber_bundle_2} shaped by the Vanishing theorem is given in Figure~\ref{fig:spectralsequence}.%

\medskip
In this section, so far, we touched only the achievements of \cite{Cohen1976LNM533} announced in \cite{CohenF1973-1}, but as the title {\em The homology of $\mathcal{C}_{n+1}$-spaces} indicates, the paper has much more to offer; see also the announcement \cite{CohenF1973-2}.
In the context of the previously presented results we only point out that the homology of the unordered configuration space $\conf(\R^d,n)/\Sym_n$, with arbitrary number of particles $n\geq 1$, and coefficients in the field $\F_p$, can be recovered from \cite[Thm.\,3.1]{Cohen1976LNM533}.
In addition, a recipe for the coalgebra structure on $\F_p$-homology of $\conf(\R^d,n)/\Sym_n$ was given by showing that $H_*(\conf(\R^d,n)/\Sym_n;\F_p)$ injects into the coalgebra $H_*(\Omega^d S^d;\F_p)$.

\medskip
Now, it was natural to ask: {\em How much further can these computations be extended? For which classes of spaces are  similar formulas true?}

%-------------------------------------------------
\subsection{Homology of unordered configuration spaces}
%-------------------------------------------------  

The breakthroughs made in 1970s first by May \cite{May1970, May1972, May1976LNM533-01}, and then by F. Cohen \cite{CohenF1973-1, CohenF1973-2, Cohen1976LNM533}, Dusa McDuff \cite{McDuff1975, McDuff1977}, Victor Snaith \cite{Snaith1974}, Graeme Segal \cite{Segal1973}, Cohen, Mark Mahowald \& James Milgram in \cite[Thm.\,1]{CohenMachowaldMilgram1978}, and Cohen, May \& Laurence R. Taylor \cite{Cohen1978-2, Cohen1979}, opened a pathway for applications of homotopy methods in the study of topology of configuration spaces. 
%
%\medskip
The next decade brought more excitement of different flavor with the work of Cohen, Ralph Cohen, Nicholas Kuhn \& Joseph Neisendorfer \cite{CohenCohen1983}, Cohen, May \& Taylor \cite{Cohen1984-2}, Jeffrey Caruso, Cohen, May \& Taylor \cite{Cohen1984-1}, followed by the results of Carl-Friedrich B\"odigheimer, Cohen \& Taylor \cite{Bodigheimer1989} and B\"odigheimer \& Cohen \cite{Bodigheimer1987-2}.

\medskip
The highlight of 1970s and 1980s in the study of configuration spaces, from the perspective of this book, are the results given in the paper \cite{Bodigheimer1989}.
For this reason, we give a (simplified) presentation of how B\"odigheimer, Cohen \& Taylor computed homologies of unordered configuration spaces\index{unordered configuration space} of manifolds. 

\medskip
In this section we consider configuration spaces of smooth, compact connected manifolds $M$ of (fixed) dimension $d\geq 2$. 
Furthermore, by $\F$ we denote the field $\F_p$ with prime number $p$ of elements, or a field of characteristic zero.
Let $n\geq 1$ be an integer, and assume that in the case when $\F$ is not the field with two elements $\F_2$ the sum $d+n$ is odd.

\medskip
The objective is to compute the homology of the unordered configuration spaces of the manifold $M$ with coefficients in the field~$\F$.
The main idea is to describe the graded vector space $H_*(\conf(M,k)/\Sym_k;\F)$ as a part of the homology of a much larger space, namely the quotient space 
\[
C(M;X):=\Big( \coprod_{k\geq 1} \conf(M,k)\times_{\Sym_k}X^k\Big)/\approx,
\]
where $X$ is a CW-complex with the base point $\pt\in X$, and the equivalence relation $\approx$ is generated by $(m_1,\dots,m_k;x_1,\dots,x_k)\approx (m_1,\dots,m_{k-1};x_1,\dots,x_{k-1})$ if $x_k=\pt$.
The computation is done in several steps.

\medskip
In the first step, based on a result from \cite{Cohen1976LNM533} and proceeding by an induction on the number of handles in a handle decomposition of $M$, the homology of $C(M;S^n)$ is described in terms of homologies of iterated loop spaces of spheres as follows; see \cite[Thm.\,A]{Bodigheimer1989}.

\begin{theorem}
\label{th-bct-thm-A}
There is an isomorphism of graded vector spaces
%\[
%\xymatrix{
%\theta\colon H_*(C(M;S^n);\F)\ \ar[r] & \ \bigotimes_{i=0}^d H_*(\Omega^{d-i}S^{d+n};\F)^{\otimes\dim(H_i(M;\F))}.
%}
%\]	
\[
\theta\colon H_*(C(M;S^n);\F)\longrightarrow  \bigotimes_{i=0}^d H_*(\Omega^{d-i}S^{d+n};\F)^{\otimes\dim(H_i(M;\F))}.
\]
\end{theorem}

\noindent
It is important to point out that, as an artefact of the proof, the isomorphism $\theta$ depends on the choice of a handle decomposition of $M$.
On the other hand, the isomorphism is natural for embeddings which preserve the handle decomposition.

\medskip
The next step is a more delicate one.
We can say that it gives us a ``filtration refinement''\index{filtration refinement} of the isomorphism $\theta$.
Indeed, the space $C(M;S^n)$ can be naturally filtered by the number of points in a configuration.
More precisely, let 
\begin{multline*}
\TF_k C(M;S^n):=\im\Big(\coprod_{0\leq m\leq k}\conf(M,k)\times_{\Sym_m}(S^n)^m\longrightarrow \\
\coprod_{m\geq 0}\conf(M,k)\times_{\Sym_m}(S^n)^m\longrightarrow
\Big(\coprod_{m\geq 0}\conf(M,k)\times_{\Sym_m}(S^n)^m\Big)/_{\approx}\Big),	
\end{multline*}
with the first map being the obvious inclusion and the second map the identification map.
In this way we have the filtration\index{filtration} of $C(M;S^n)$:
\begin{multline*}
\emptyset=\TF_{-1}C(M;S^n) \subseteq \TF_0 C(M;S^n)\subseteq \TF_1C(M;S^n)\subseteq\cdots\\
\subseteq \TF_{k-1}C(M;S^n)\subseteq \TF_{k}C(M;S^n) \subseteq \cdots ,
\end{multline*}
where each consecutive pair of spaces $(\TF_{k}C(M;S^n),\TF_{k-1}C(M;S^n))$ is an NDR-pair; consult \cite[Prop.\,2.6]{May1972}. 
According to the work of Segal \cite{Segal1973}, Cohen \cite{Cohen1983-2} and B\"odigheimer \cite{Bodigheimer1987}, the filtration stably splits.
In particular,
\[
 \widetilde{H}_*(C(M;S^n);\F)\cong \bigoplus  \widetilde{H}_*(\TF_{k}C(M;S^n)/\TF_{k-1}C(M;S^n);\F).
\]
On the other hand $\CC(H_*,M,S^n):=\bigotimes_{i=0}^d H_*(\Omega^{d-i}S^{d+n};\F)^{\otimes\dim(H_i(M;\F))}$
is an algebra with each generator equipped with a weight --- as described in the language of Araki--Kudo--Dyer--Lashof homology operations\index{Araki--Kudo--Dyer--Lashof homology operations} much earlier by Sh{\^o}r{\^o} Araki \& Tatsuji Kudo \cite{Araki1956} in the case $\F=\F_2$, by Eldon Dyer \& Richard Lashof \cite{Dyer1962} for $\F=\F_p$, and by May \cite{May1976LNM533-01} using the framework of $E_{\infty}$-operads.
The weight assignment induces the so-called product filtration\index{product filtration} on the algebra $\CC(H_*,M,S^n)$:
\begin{multline*}
0=\AF_0\CC(H_*,M,S^n)\subseteq \AF_1\CC(H_*,M;S^n)\subseteq\cdots \\
\subseteq \AF_{k-1}\CC(H_*,M,S^n)\subseteq \AF_{k}\CC(H_*,M,S^n)\subseteq\cdots	.
\end{multline*} 
It turns out that these two filtrations agree under the isomorphism $\theta$ of Theorem \ref{th-bct-thm-A}. 
In other words, the following theorem holds \cite[Thm.\,B]{Bodigheimer1989}.
\begin{theorem}
\label{th-bct-thm-b}
There are isomorphisms of graded vector spaces

\[
\theta_k \colon H_*( \TF_{k}C(M;S^n);\F)\longrightarrow \AF_{k}\CC(H_*,M,S^n),
\]
such that for every $k\geq 0$ the following diagram commutes:
\[
\xymatrix{
H_*( \TF_{k}C(M;S^n);\F)\ \ar[r]^-{\theta_k}\ar[d] &\ \AF_{k}\CC(H_*,M,S^n)\ar[d]\\
H_*(C(M;S^n);\F)\ \ar[r]^-{\theta} & \ \CC(H_*,M,S^n).
}
\]
The left vertical homomorphism in the diagram is induced by the inclusion of spaces $\TF_{k}C(M;S^n)\subseteq C(M;S^n)$, while the right vertical homomorphism is the inclusion  homomorphism $\AF_{k}\CC(H_*,M,S^n)\subseteq \CC(H_*,M,S^n)$.
\end{theorem}

\medskip
In the final step we consider the successive quotients: 
\begin{align*}
\TD_kC(M;S^n)&:=\TF_{k}C(M;S^n)/\TF_{k-1}C(M;S^n),\\
\AD_kC(M;S^n)&:=\AF_{k}\CC(H_*,M,S^n)/\AF_{k-1}\CC(H_*,M,S^n).	
\end{align*}
From Theorem \ref{th-bct-thm-b} it follows directly that the family of isomorphisms $\theta_k$ induce the sequence of isomorphisms
%\[
%\xymatrix{\overline{\theta}_k \colon H_*( \TD_{k}C(M;S^n);\F)\ \ar[r] & \ \AD_{k}\CC(H_*,M,S^n).}
%\]
\[
\overline{\theta}_k \colon H_*( \TD_{k}C(M;S^n);\F)\longrightarrow \AD_{k}\CC(H_*,M,S^n).
\]
Next, consider the vector bundle $\xi_{M,k}$ given by
\[
\xymatrix{\R^k \ \ar[r] & \ \conf(M,k)\times_{\Sym_k}\R^k \ \ar[r] & \ \conf(M,k)/\Sym_k.}
\]
It is not hard to see that the space $\TD_kC(M;S^n)$ is the Thom space\index{Thom space} \cite[Sec,\,18]{Milnor1974} of the Whitney power vector bundle $\xi_{M,k}^{\oplus n}$.
Consequently, applying the Thom isomorphism theorem\index{Thom isomorphism theorem} \cite[Cor.\,10.7 and Lem.\,18.2]{Milnor1974} to the following result, a description of the homology of the unordered configuration space of the manifold $M$ can be obtained; see \cite[Thm.\,C]{Bodigheimer1989}. 
\begin{theorem}
\label{th-bct-thm-c}
There is an isomorphism of graded vector spaces
	\[
	H_{*-kn}(\conf(M,k)/\Sym_k;\F)\cong \AD_{k}\CC(H_*,M,S^n).
	\]
\end{theorem}

\medskip
In the case when $M=\R^d$ the last result is used in Section \ref{subsec : proof of injectivity}, more precisely in the proof of Theorem \ref{th : injection }.
For relevant details see Corollary \ref{cor : homology of configuration space}.

\medskip
For $n$ odd and $d$ even a similar result was deduced by a slight modification of the coefficients.
Instead of coefficients in the field $\F$ with trivial $\Sym_k$-module structure one considers the local coefficient system given on the field $\F$ by $\pi\cdot a=(-1)^{\sgn(\pi)}a$ for $\pi\in\Sym_k$ and $a\in\F$.

\medskip
The isomorphism of Theorem \ref{th-bct-thm-c}, in combination with understanding of Araki--Kudo--Dyer--Lashof homology operations\index{Araki--Kudo--Dyer--Lashof homology operations}, allows one to do explicit computations of the homology of unordered configuration spaces\index{homology of the unordered configuration space} of manifolds.
For example, B\"odigheimer, Cohen \& Taylor illustrated such computations in \cite[Sec.\,5]{Bodigheimer1989} and in particular computed the dimensions of the vector spaces $H_{*}(\conf(S^2,k)/\Sym_k;\F)$ for $k\leq 10$:

{%\small
\begin{center}\hspace{4pt}
\begin{tabular}{ |r | c c c c c c c c c c  | } 
 \hline
   								$k=$	 & \hspace{4pt}  $1$   & \hspace{4pt}  $2$  & \hspace{4pt} $3$  & \hspace{4pt} $4$  & \hspace{4pt} $5$  & \hspace{4pt} $6$  & \hspace{4pt} $7$  & \hspace{4pt} $8$  & \hspace{4pt} $9$  & \hspace{4pt} $10$ \\ [0.7ex]\hline 
$\dim (H_{0}(\conf(S^2,k)/\Sym_k))$ & \hspace{4pt}$1$ & \hspace{4pt}$1$ & \hspace{4pt}$1$ & \hspace{4pt}$1$ & \hspace{4pt}$1$ & \hspace{4pt}$1$ & \hspace{4pt}$1$ & \hspace{4pt}$1$ & \hspace{4pt}$1$ & \hspace{4pt}$1$\\ [0.5ex]
$\dim (H_{1}(\conf(S^2,k)/\Sym_k))$ & \hspace{4pt}$0$ & \hspace{4pt}$1$ & \hspace{4pt}$1$ & \hspace{4pt}$1$ & \hspace{4pt}$1$ & \hspace{4pt}$1$ & \hspace{4pt}$1$ & \hspace{4pt}$1$ & \hspace{4pt}$1$ & \hspace{4pt}$1$\\[0.5ex]
$\dim (H_{2}(\conf(S^2,k)/\Sym_k))$ & \hspace{4pt}$1$ & \hspace{4pt}$1$ & \hspace{4pt}$1$ & \hspace{4pt}$2$ & \hspace{4pt}$2$ & \hspace{4pt}$2$ & \hspace{4pt}$2$ & \hspace{4pt}$2$ & \hspace{4pt}$2$ & \hspace{4pt}$2$\\  [0.5ex]
 $\dim (H_{3}(\conf(S^2,k)/\Sym_k))$&  &  & \hspace{4pt}$1$ & \hspace{4pt}$2$ & \hspace{4pt}$2$ & \hspace{4pt}$3$ & \hspace{4pt}$3$ & \hspace{4pt}$3$ & \hspace{4pt}$3$ & \hspace{4pt}$3$\\ 
 $\dim (H_{4}(\conf(S^2,k)/\Sym_k))$&   &   &   &   &  \hspace{4pt}$1$ &  \hspace{4pt}$2$ & \hspace{4pt}$2$   & \hspace{4pt}$3$   & \hspace{4pt}$3$ & \hspace{4pt}$3$\\   [0.5ex]
$\dim (H_{5}(\conf(S^2,k)/\Sym_k))$&   &   &   &   &    \hspace{4pt}$1$ &   \hspace{4pt}$1$ & \hspace{4pt}$2$   & \hspace{4pt}$3$   & \hspace{4pt}$3$ & \hspace{4pt}$4$\\   [0.5ex]
$\dim( H_{6}(\conf(S^2,k)/\Sym_k))$&   &   &   &   &    &   & \hspace{4pt}$1$   & \hspace{4pt}$2$   & \hspace{4pt}$3$ & \hspace{4pt}$4$\\   [0.5ex]
$\dim (H_{7}(\conf(S^2,k)/\Sym_k))$&   &   &   &   &    &   &    & \hspace{4pt}$1$   & \hspace{4pt}$2$ & \hspace{4pt}$3$\\   [0.5ex]
$\dim( H_{8}(\conf(S^2,k)/\Sym_k))$&   &   &   &   &    &   &    &     & \hspace{4pt}$1$ & \hspace{4pt}$2$\\   [0.5ex]
$\dim( H_{9}(\conf(S^2,k)/\Sym_k))$&   &   &   &   &    &   &    &     & \hspace{4pt}$1$ & \hspace{4pt}$1$\\   [0.5ex]
 \hline
\end{tabular}
\end{center} 
} 

\medskip
Finally, in the case when the sum $n+d$ is even and the coefficients are taken in the field of rational numbers, B\"odigheimer \& Cohen in \cite{Bodigheimer1987-2} computed the cohomology  $H^*(C(M_g;S^{2n});\mathbb{Q})$ for $n\geq 1$, as a $\mathbb{Q}$ vector space.
Here $M_g$ denotes the open manifold obtained by deleting a point from an orientable surface of genus $g$. 
In this way they demonstrated the importance of the parity assumption on the sum $n+d$ for the results of \cite{Bodigheimer1989}. 
More precisely, they showed in \cite[Thm.\,A]{Bodigheimer1987-2} that
\[
H^*(C(M_g;S^{2n});\mathbb{Q})\cong
\mathbb{Q}[v,u_1,\dots,u_{2g}]
\otimes 
H_*(\Lambda (w,z_1,\dots,z_{2g}),\partial),
\]
where $\deg(v)=2n$, $\deg(u_1)=\dots=\deg(u_{2g})=4n+2$,  $\deg(w)=4n+1$, $\deg(z_1)=\dots=\deg(z_{2g})=2n+1$, and the differential $\partial$ on $\Lambda (w,z_1,\dots,z_{2g})$ is given by $\partial w= 2(z_1z_2+\dots+z_{2g-1}z_{2g})$, and $\partial z_1=\cdots=\partial z_{2g}=0$.

%%%%%%%%%%%%%%%%%%%%%%%%%%%%%%%%%%%%%%%%%%%%%%%%%%%%%%%%%%%%%%%%%%%%%%%%%%%%%%%%%%%%%
%%%%%%%%%%%%%%%%%%%%%%%%%%%%%%%%%%%%%%%%%%%%%%%%%%%%%%%%%%%%%%%%%%%%%%%%%%%%%%%%%%%%%
\section{The Ptolemaic epicycles embedding}
\label{sec : Ptolemaic epicycles embedding}
%%%%%%%%%%%%%%%%%%%%%%%%%%%%%%%%%%%%%%%%%%%%%%%%%%%%%%%%%%%%%%%%%%%%%%%%%%%%%%%%%%%%%
%%%%%%%%%%%%%%%%%%%%%%%%%%%%%%%%%%%%%%%%%%%%%%%%%%%%%%%%%%%%%%%%%%%%%%%%%%%%%%%%%%%%%

In this section we follow the work of {\Hung} \cite[Sec.\,2]{Hung1982} \cite[Sec.\,1 and Sec.\,2]{Hung1990}.
Using the analogy with the structural map of the little cubes operad, we introduce and study an embedding of a product of spheres into the ordered configuration space of a Euclidean space.

%-----
\medskip
Let $M$ be a topological space.
The {\bf ordered configuration space\index{ordered configuration space}} of $n$ pairwise distinct points on the space $M$ is the following space:
\[
\conf(M,n):=
\{
(x_1,\dots,x_n) \in X^n : x_i\neq x_j \text{ for all }1\leq i<j\leq n
\},
\]
equipped with the subspace topology.
The symmetric group\index{symmetric group} $\Sym_n$ acts (from right) freely on the configuration space by
\[
\pi\cdot (x_1,\ldots,x_n):=(x_{\pi(1)},\ldots,x_{\pi(n)}),
\]
where $\pi\in\Sym_n$ and $(x_1,\ldots,x_n)\in \conf(M,n)$.

%-----
\medskip
Now let $m\geq 0$ be an integer and $n=2^m$.
We define a bijection $\beta\colon \Z_2^{\oplus m}\longrightarrow [2^m]$ as follows:
%\[
%\xymatrix{
%(i_1,\ldots,i_m)\  \ar@{|->}[r]& \ 1+\sum_{j=1}^m 2^{m-j}i_j.
%}
%\]
\[
(i_1,\ldots,i_m)\longrightarrow 1+\sum_{j=1}^m 2^{m-j}i_j.
\]
In particular, we have:
%\[
%\xymatrix{
%(0,0,\ldots,0,0)\  \ar@{|->}[r]& \ 1, 	& (0,1,\ldots,1,1) \ \ar@{|->}[r]& \ 2^{m-1},}
%\]
%\[
%\xymatrix{
%(1,0,\ldots,0,0) \ \ar@{|->}[r]& \ 2^{m-1}+1, 	& (1,1,\ldots,1,1) \  \ar@{|->}[r]& \ 2^{m}.
%}
%\]
\[
\begin{array}{l l l}
(0,0,\ldots,0,0) \longmapsto 1, 			&	& (0,1,\ldots,1,1) \longmapsto 2^{m-1},\\
(1,0,\ldots,0,0) \longmapsto 2^{m-1}+1,	&   	& (1,1,\ldots,1,1) \longmapsto 2^{m}.
\end{array}
\]
The symmetric group $\Sym_{2^m}$ for us is the group of permutations of the set $\Z_2^{\oplus m}=[2^m]$, where the last equality (set identification) is given via the bijection $\beta$.
\eject

%%-----
\begin{definition}
	\label{def : Epicycles}
	Let $d\geq 2$ be an integer or $d=\infty$, and let $m\geq 0$ be also an integer.
	Furthermore, let $S^{d-1}$ denote the unit sphere in $\R^d$ with the base point $*:=(1,0,\dots,0)\in S^{d-1}$.
	Fix a real number $0<\varepsilon<\tfrac13$.
	\begin{compactitem}[ \ ---]
	\item The space $\Sp(\R^d,2^m)$, a product of spheres,
	\item its embedding into the configuration space
	\[
	\epicy_{d,2^m}\colon\Sp(\R^d,2^m)\longrightarrow \conf(\R^d,2^m),
	\]
	which is called the {\bf Ptolemaic epicycles embedding\index{Ptolemaic epicycles embedding}},	 and
	\item the group $\Sy_{2^m}$ that acts on the space $\Sp(\R^d,2^m)$,
	\end{compactitem}
	are defined inductively as follows. 
	\begin{compactenum}[\rm \ (1)]
	
	%---------------------
	\item If $m=0$ then $\Sp(\R^d,1):=\{\pt\}$ is a point, and 
	\[
	\epicy_{d,1}\colon \Sp(\R^d,1)\longrightarrow \conf(\R^d,1),\qquad\qquad \pt\longmapsto 0\in (\R^d)^1.
	\]
	The group $\Sy_1:=1$ acts on both spaces $\Sp(\R^d,1)$ and $\conf(\R^d,1)$ trivially.
	Since $\Sy_1=1=\Sym_1$ is the trivial group, the map $\epicy_{d,1}$ is an $\Sy_1$-equivariant map.
	
	%---------------------
	\item If $m=1$ then $\Sp(\R^d,2):=S^{d-1}= (\Sp(\R^d,1)\times \Sp(\R^d,1))\times S^{d-1}$ is a $(d-1)$-sphere, and 
	\[
	\epicy_{d,2}\colon \Sp(\R^d,2)\longrightarrow \conf(\R^d,2),\qquad\qquad x\longmapsto (x,-x).
	\]
	The group $\Sy_2:= (\Sy_1\times \Sy_1)\rtimes \Z_2 =\Sy_1\wr\Z_2\cong \Z_2$ acts on $\Sp(\R^d,2)$ antipodally.
	The groups $\Sy_2$ and $\Sym_2$ are isomorphic via the unique isomorphism $\iota_1\colon \Sy_2\longrightarrow\Sym_2$.
	Hence, $\conf(\R^d,2)$ is a $\Sy_2$-space where the $\Sy_2$-action on $\conf(\R^d,2)$ is induced via the isomorphism $\iota_1$.
	Consequently, $\epicy_{d,2}$ is an $\Sy_2$-equivariant map.	

	%---------------------
	\item Let us now assume that for $m=k$ we have defined
	\begin{compactitem}[ \ ---]
	\item the space 
	\[
	\qquad\qquad\Sp(\R^d,2^k)= (\Sp(\R^d,2^{k-1})\times\Sp(\R^d,2^{k-1}))\times S^{d-1}=(S^{d-1})^{2^k-1},
	\] 
	
	%---------------------
	\item the embedding of the spaces
	\[
	\epicy_{d,2^k}\colon \Sp(\R^d,2^k)\longrightarrow \conf(\R^d,2^k),
	\]
	
	%---------------------
	\item the group embedding 
		\[ 
		\iota_k\colon \Sy_{2^k}\longrightarrow\Sym_{2^k}
		\]
		such that $\iota_k(\Sy_{2^k})$ is a Sylow $2$-subgroup of $\Sym_{2^k}$, and
	
	%---------------------
	\item the action of the group $\Sy_{2^k}$ on $\Sp(\R^d,2^k)$ in such a way that $\epicy_{d,2^k}$ is an $\Sy_{2^k}$-equivariant map, assuming that the action of $\Sy_{2^k}$ on the configurations space $\conf(\R^d,2^k)$ is given via $\iota_k$.
	\end{compactitem}
	
	\medskip
	\noindent
	For convenience we denote the coordinate functions of the embedding $\epicy_{d,2^k}$~by%
	\[
	\epicy_{d,2^k}(y)=(\epicy_{d,2^k}^{1}(y),\ldots,\epicy_{d,2^k}^{2^k}(y))\in \conf(\R^d,2^k),
	\]
	where $y\in\Sp(\R^d,2^k)$.
	That is, $\epicy_{d,2^k}^{i}\colon \Sp(\R^d,2^k)\longrightarrow \R^d$ for $1\leq i\leq 2^k$.
		
	%---------------------
	\item
    Let $m=k+1$, then we define 
   \begin{compactitem}[ \ ---]
   
    %----
    \item the space 
	\begin{align*}
		\Sp(\R^d,2^{k+1}) &:= (\Sp(\R^d,2^{k})\times \Sp(\R^d,2^{k}))\times S^{d-1}\\
		&~=(S^{d-1})^{2^{k+1}-1},
	\end{align*}
	%----
	\item the group
	\begin{align*}
 	\Sy_{2^{k+1}}&:= (\Sy_{2^k}\times \Sy_{2^k})\rtimes \Z_2\\
 	             &~= \Sy_{2^k}\wr\Z_2=\Z_2\wr\cdots\wr\Z_2 \qquad	 (k+1 \text{ times}),
 	\end{align*}

	%----
	\item the action of the group $\Sy_{2^{k+1}}$ on $\Sp(\R^d,2^{k+1})$ by
	\begin{eqnarray*}
	(h_1,h_2)\cdot(y_1,y_2,x) &:=& ( h_1\cdot y_1, h_2\cdot y_2,x),\\	
	(h_1,h_2,\omega)\cdot(y_1,y_2,x) &:=& ( h_2\cdot y_2, h_1\cdot y_1,-x),
	\end{eqnarray*}
 	where $(h_1,h_2)\in \Sy_{2^k}\times \Sy_{2^k}\subseteq (\Sy_{2^k}\times \Sy_{2^k})\rtimes \Z_2$, $\omega$ is the generator of $\Z_2\subseteq  (\Sy_{2^k}\times \Sy_{2^k})\rtimes \Z_2$, and  $(y_1,y_2,x)\in  (\Sp(\R^d,2^{k})\times \Sp(\R^d,2^{k}))\times S^{d-1}$,\medskip
 	
 	%----
 	\item the Ptolemaic epicycles embedding 
 	\[
 	\epicy_{d,2^{k+1}}\colon \Sp(\R^d,2^{k+1})\longrightarrow \conf(\R^d,2^{k+1})
 	\] 
 	by
	\begin{multline*}
	\qquad\qquad
		\epicy_{d,2^{k+1}}(y_1,y_2,x):=(x+\varepsilon\, \epicy_{d,2^k}^{1}(y_1),\ldots,x+\varepsilon\, \epicy_{d,2^k}^{2^k}(y_1),\\
		 -x+\varepsilon\, \epicy_{d,2^k}^{1}(y_2),\ldots,-x+\varepsilon\, \epicy_{d,2^k}^{2^k}(y_2)),
	\end{multline*}

	%----	 
	\item the embedding $\iota_{k+1}\colon \Sy_{2^{k+1}}\longrightarrow\Sym_{2^{k+1}}$ that is defined by
	\begin{eqnarray*}
	\iota_{k+1}(h_1,h_2)(i) &:= &\begin{cases}
		\iota_k(h_1)(i), & 1\leq i\leq 2^k,\\
		\iota_k(h_2)(i-2^k)+2^k, & 2^k+1\leq i\leq 2^{k+1},
	\end{cases}\\
	\iota_{k+1}(\omega)(i)&:= & \begin{cases}
		i+2^k, & 1\leq i\leq 2^k,\\
		i-2^k, & 2^k+1\leq i\leq 2^{k+1}.
	\end{cases}
	\end{eqnarray*}
This, in particular, means that the subgroup $\Sy_{2^k}\times 1$ permutes elements of $[2^{k+1}]_1$  while keeping elements of $[2^{k+1}]_2$ fixed, the subgroup $1\times \Sy_{2^k}$ on the other hand permutes elements of $[2^{k+1}]_2$ and fixes elements of $[n]_1$.
The subgroup generated by $\omega$ interchanges the blocks $[2^{k+1}]_1$ and $[2^{k+1}]_2$.	 
In addition $\iota_{k+1}(\Sy_{2^{k+1}})$ is a Sylow $2$-subgroup of $\Sym_{2^{k+1}}$.
For an illustration of the embedding $\epicy_{2,2}\colon  \Sp(\R^2,2)\longrightarrow \conf(\R^2,2)$ see Figure \ref{fig:epicy-embedding}.
\end{compactitem}	

\medskip		
\noindent
Then $\epicy_{d,2^{k+1}}$ is an $\Sy_{2^{k+1}}$-equivariant map if the action of $\Sy_{2^{k+1}}$ on $\conf(\R^d,2^{k+1})$ is given by
\begin{eqnarray*}
(h_1,h_2)\cdot(z_1,\ldots,z_{2^{k+1}})&:=&  (h_1\cdot  (z_{1},\ldots,z_{2^{k}}),h_2\cdot (z_{2^k+1},\ldots,z_{2^{k+1}})),\\
(h_1,h_2,\omega)\cdot(z_1,\ldots,z_{2^{k+1}})&:=& (h_2\cdot (z_{2^k+1},\ldots,z_{2^{k+1}}), h_1\cdot  (z_{1},\ldots,z_{2^{k}})),
\end{eqnarray*}
for $(h_1,h_2)\in \Sy_{2^k}\times \Sy_{2^k}$, $\omega$ the generator of $\Z_2$, and $(z_1,\ldots,z_{2^{k+1}})\in \conf(\R^d,2^{k+1})$.
In other words, the action of $\Sy_{2^{k+1}}$ on $\conf(\R^d,2^{k+1})$ is given via the embedding $\iota_{k+1}$.
\end{compactenum}
\end{definition}

\begin{figure}[ht]
\centering
\includegraphics[scale=0.75]{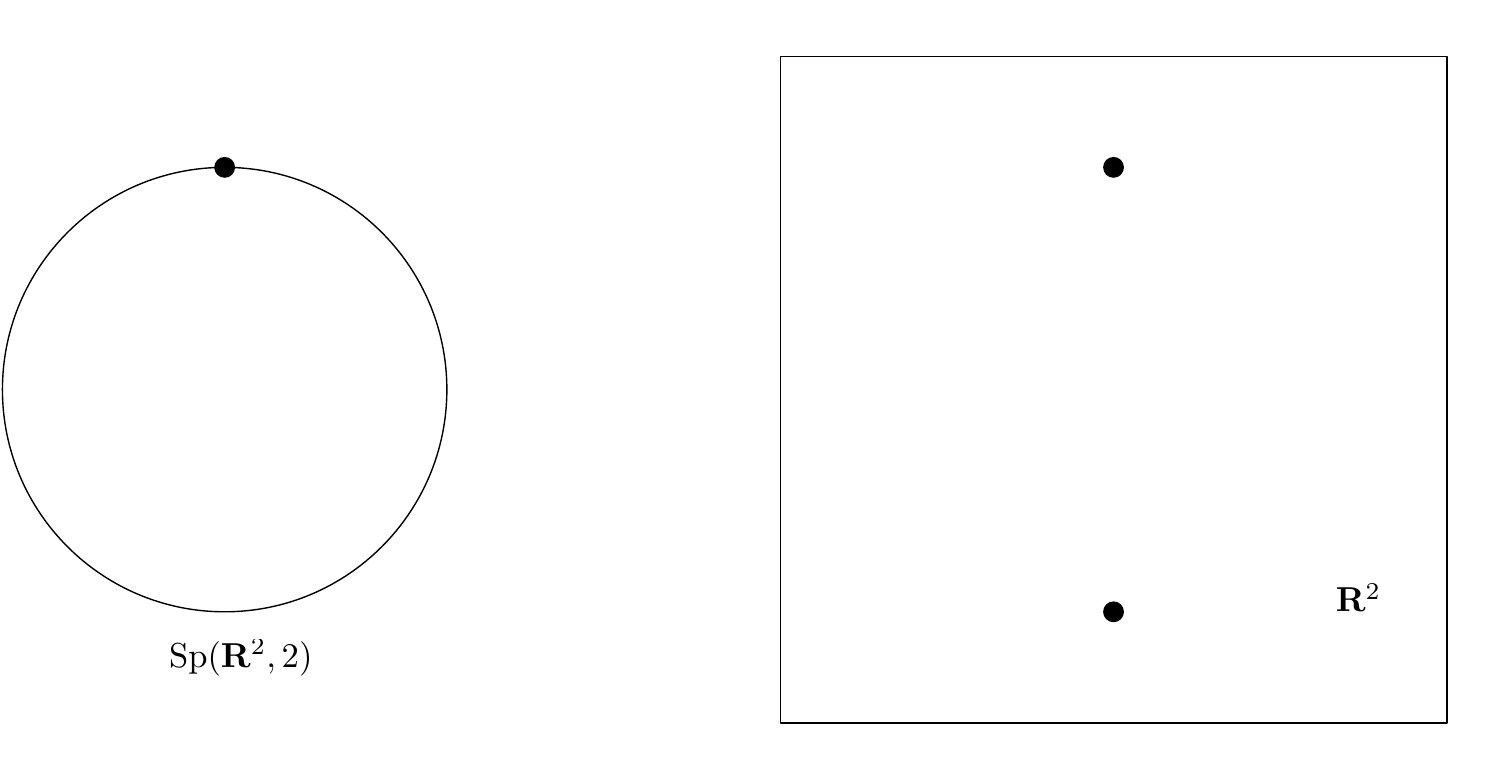}
\caption{\small An illustration of the embedding $\epicy_{2,2}\colon\Sp(\R^2,2)\longrightarrow \conf(\R^2,2)$.}
\label{fig:epicy-embedding}
\end{figure}

\begin{remark}
An analogous construction can be given using the little cubes operad\index{little cubes operad} $\CC_d(2)$ in the place of the sphere $S^{d-1}$. 
(For more details on little cubes operad see for example \cite{May1972} or consult Section \ref{subsub : little cube operad}.)
Indeed, let $d\geq 2$ be an integer or $d=\infty$, and let $m\geq 0$ be also an integer.
We define the space $\Ce(\R^d,2^m)$ of {\bf little cubes epicycles space\index{little cubes epicycles space}} and the corresponding $\Sy_{2^m}$-eqivariant map $\cepicy_{d,2^m}\colon\Ce(\R^d,2^m)\longrightarrow\CC_d(2^m)$ as follows.
	\begin{compactenum}[\rm \ (1)]
%	
	%---------------------
	\item If $m=0$, then we set $\Ce(\R^d,1):=\CC_d(1)$, and $\cepicy_{d,1}\colon \Ce(\R^d,1)\longrightarrow \CC_d(1)$ is the identity map.
	The group $\Sy_1=1$ acts on both spaces $\Ce(\R^d,1)$ and $\CC_d(1)$ trivially, and so $\cepicy_{d,1}$ is an $\Sy_1$-equivariant map.

	%---------------------
	\item If $m=1$, then we set $\Ce(\R^d,2):= (\Ce(\R^d,1)\times \Ce(\R^d,1))\times\CC_d(2)$, and the map 
	\[
	\cepicy_{d,2}\colon \Ce(\R^d,2)\longrightarrow \CC_d(2)
	\]
	is the composition map
	\[
	\xymatrix@1{
	(\Ce(\R^d,1)\times \Ce(\R^d,1))\times\CC_d(2) \ \ar[rrr]^-{ (\cepicy_{d,1}\times \cepicy_{d,1})\times\id}\ar[drrr] & &  & \ ( \CC_d(1)\times  \CC_d(1))\times \CC_d(2)\ar[d]^-{\mu} \\
	&&& \CC_d(2),
	}
	\]
	where $\mu$ denotes the structural map of the little cubes operad, as defined in Section \ref{subsub : little cube operad}. 
	The group $\Sy_2=(\Sy_1\times \Sy_1)\rtimes \Sym_2\lhook\joinrel\xrightarrow{~\iota_1~} \Sym_2$ coincides with the group $\Sy_{1,1;2}$ defined in Lemma \ref{lemma : map mu is equivariant} and acts on $\Ce(\R^d,2)$ as follows:
	\begin{eqnarray*}
	(h_1,h_2)\cdot(y_1,y_2,x) &:=& (h_1\cdot y_1, h_2\cdot y_2,x),\\	
	(h_1,h_2,\omega)\cdot(y_1,y_2,x) &:=& (h_2\cdot y_2, h_1\cdot y_1,-x),
	\end{eqnarray*}
 	where $(h_1,h_2)\in \Sy_{1}\times \Sy_{1}\subseteq  (\Sy_1\times \Sy_1)\rtimes \Sym_2$, $\omega$ is the generator of $\Sym_2\subseteq(\Sy_{1}\times \Sy_{1})\rtimes  \Sym_2$, and  $(y_1,y_2,x)\in (\Ce(\R^d,1)\times \Ce(\R^d,1))\times \CC_d(2)$.
	The assumed action of $\Sy_2=\Sy_{1,1;2}$ on $( \CC_d(1)\times  \CC_d(1))\times\CC_d(2)$ is described in Section \ref{subsub : little cube operad}.
	Finally the action of $\Sy_2$ on $\CC_d(2)$ is given via embedding $\iota_1\colon \Sy_2\longrightarrow \Sym_2$.
	With these actions both maps $ (\cepicy_{d,1}\times \cepicy_{d,1})\times \id$ and $\mu$ are $\Sy_2$-equivariant, and consequently the composition map $\cepicy_{d,2}$ is a $\Sy_2$-equivariant.
	It is important to notice that $\Sy_1=1$ and so $\Sy_2\cong\Z_2\cong\Sym_2$.

	%---------------------
	\item Let us assume that for $m=k$ we have defined the space 
	\[
	\Ce(\R^d,2^k)= (\Ce(\R^d,2^{k-1})\times \Ce(\R^d,2^{k-1}))\times\CC_d(2),
	\] 
	the embedding $\cepicy_{d,2^k}\colon \Ce(\R^d,2^k)\longrightarrow \CC_2(2^k)$, the embedding $\iota_k$ of the Sylow $2$-subgroup $\Sy_{2^k}$ into $\Sym_{2^k}$, and the action of the group $\Sy_{2^k}$ on $\Ce(\R^d,2^k)$ in such a way that $\cepicy_{d,2^k}$ is a $\Sy_{2^k}$-equivariant map.
	
	\medskip
	\noindent
	For $m=k+1$ we define the space 
	\[
		\Ce(\R^d,2^{k+1}):= (\Ce(\R^d,2^{k})\times \Ce(\R^d,2^{k}))\times \CC_d(2).
	\]
	The action of the group $\Sy_{2^{k+1}}$ on $\Ce(\R^d,2^{k+1})$ is given by
	\begin{eqnarray*}
	(h_1,h_2)\cdot(y_1,y_2,x) &:=& (h_1\cdot y_1, h_2\cdot y_2,x),\\	
	(h_1,h_2,\omega)\cdot(y_1,y_2,x) &:=& (h_2\cdot y_2, h_1\cdot y_1,-x),
	\end{eqnarray*}
 	where $(h_1,h_2)\in \Sy_{2^k}\times \Sy_{2^k}\subseteq  (\Sy_{2^k}\times \Sy_{2^k})\rtimes \Sym_2$, $\omega$ is the generator of $\Sym_2\subseteq (\Sy_{2^k}\times \Sy_{2^k})\rtimes \Sym_2$, and  $(y_1,y_2,x)\in  (\Ce(\R^d,2^{k})\times \Ce(\R^d,2^{k}))\times\CC_d(2)$.
	The map
	\[
	\cepicy_{d,2^{k+1}}\colon \Ce(\R^d,2^{k+1})\longrightarrow \CC_d(2^{k+1})
	\] 
	is defined to be the following composition
	\[
	\xymatrix@C=1.5em{
	 (\Ce(\R^d,2^{k})\times \Ce(\R^d,2^{k}))\times\CC_d(2)\ \ar[rrr]^-{ (\cepicy_{d,2^{k}}\times \cepicy_{d,2^{k}})\times\id}\ar[drrr]_{\cepicy_{d,2^{k+1}}} &  & & \ ( \CC_d(2^{k})\times  \CC_d(2^{k}))\times\CC_d(2)\ar[d]^-{\mu}\\
%	 &  & & \\
	 &  & &  \CC_d(2^{k+1}),
	 }
	\]
	where $\mu$ denotes the structural map of the little cubes operad.
	Under assumed actions, by direct inspection, we get that the two maps 
	$(\cepicy_{d,2^{k}}\times \cepicy_{d,2^{k}})\times \id$ and $\mu$
	are both $\Sy_{2^{k+1}}$-equivariant. 
	Consequently the composition map $\cepicy_{d,2^{k+1}}$ is also $\Sy_{2^{k+1}}$-equivariant.
	\end{compactenum}

\end{remark}
	
\begin{example}
\label{ex : 01}
	Let $m=2$. 
	Then the bijection $\beta\colon\Z_2^{\oplus 2}\longrightarrow [4]$ is given by
	\[
	(0,0)\longmapsto 1,\quad (0,1)\longmapsto 2,\quad (1,0)\longmapsto 3,\quad (1,1)\longmapsto 4.
	\]
	The group $\Sy_2=(\Sy_1\times \Sy_1)\rtimes\Z_2=\langle\varepsilon_1,\varepsilon_2\rangle\rtimes\langle\omega\rangle \cong (\Z_2\times\Z_2)\rtimes\Z_2$ embeds via $\iota_2$ into the symmetric group $\Sym_4$ by sending generators to the following permutations
	\[
	\varepsilon_1\longmapsto {1234 \choose 2134}, \qquad \varepsilon_2\longmapsto {1234 \choose 1243},\qquad
	\omega\longmapsto {1234 \choose 3412}.
	\]   
\end{example}

\medskip
Let $n\geq 1$ be an integer.
We consider the following vector subspace of $\R^n$:
\[
W_n:=\{(a_1,\ldots,a_n)\in\R^n : a_1+\cdots+a_n=0\}.
\] 
Then the subspace $\{(x_1,\ldots,x_n)\in(\R^d)^n : x_1+\cdots+x_n=0\}$ of $(\R^d)^n$ can be identified with the direct sum $W_n^{\oplus d}$.
The map $\epicy_{d,n}$ that we have defined has the following property.

\begin{proposition}
\label{prop : sum is zero}
	Let $d\geq 2$ be an integer or $d=\infty$, and let $m\geq 0$ be an integer.
	Then $\im (\epicy_{d,2^m})\subseteq	W_{2^m}^{\oplus d}$, that is, for every $(y_1,y_2,x)\in\Sp(\R^d,2^m)$ 
	\[
	\epicy_{d,2^m}^{1}(y_1,y_2,x)+\cdots+\epicy_{d,2^m}^{2^m}(y_1,y_2,x)=0.
	\]
\end{proposition}

\begin{proof}
We use induction on the integer $m\geq 0$.
For $m=0$ we have that 
\[
\im (\epicy_{d,1})=\{0\}=W_1^{\oplus d}.
\]
Assume that for $m=k>0$ we have $\im (\epicy_{d,2^k})\subseteq	W_{2^k}^{\oplus d}$.
Then for $m=k+1$ and $(y_1,y_2,x)\in\Sp(\R^d,2^{k+1})=(\Sp(\R^d,2^{k})\times \Sp(\R^d,2^{k}))\times S^{d-1}$, using the induction hypothesis, we get
\begin{align*}
\sum_{j=1}^{2^{k+1}} \epicy_{d,2^{k+1}}^j(y_1,y_2,x) &=	(x+\varepsilon\, \epicy_{d,2^k}^{1}(y_1))+\cdots+(x+\varepsilon\, \epicy_{d,2^k}^{2^k}(y_1))+  \\
&\quad \ (-x+\varepsilon\, \epicy_{d,2^k}^{1}(y_2))+\cdots+(-x+\varepsilon\, \epicy_{d,2^k}^{2^k}(y_2))\\
&= \varepsilon \big(\sum_{j=1}^{2^{k}} \epicy_{d,2^k}^{j}(y_1)\big) + \varepsilon \big(\sum_{j=1}^{2^{k}} \epicy_{d,2^k}^{j}(y_2)\big)\\
&=0.
\end{align*}
Consequently $\im (\epicy_{d,2^{k+1}})\subseteq	W_{2^{k+1}}^{\oplus d}$, and the induction is completed.
\end{proof}

\begin{figure}
\centering
\includegraphics[scale=0.63]{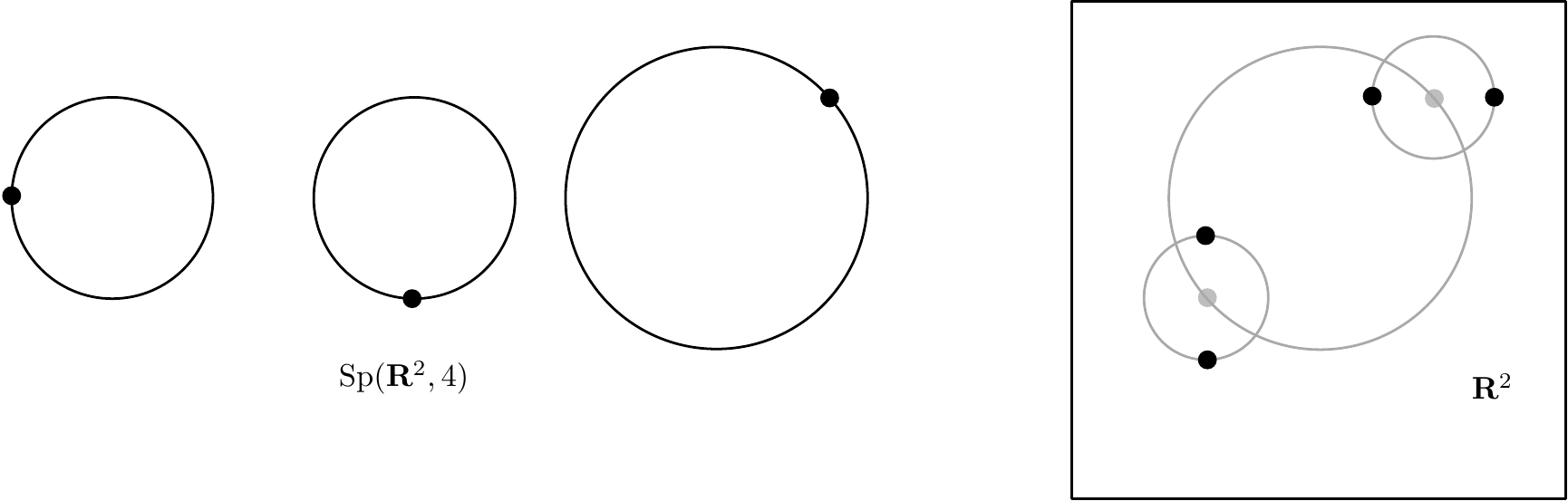}
\caption{\small An illustration of the embedding $\epicy_{2,4}\colon  \Sp(\R^2,4)\longrightarrow \conf(\R^2,4)$.}
\end{figure}

\medskip
Next we verify that the Ptolemaic epicycles embedding\index{Ptolemaic epicycles embedding} $\epicy_{d,2^m}$ is indeed an embedding, see {\cite[Lem.\,1.6]{Hung1990}}.

\begin{proposition}
	\label{prop : iota is embedding}
	Let $d\geq 2$ be an integer or $d=\infty$, and let $m\geq 0$ be an integer.
	For any fixed real number $0<\varepsilon<\tfrac13$ the map 
	\[
	\epicy_{d,2^m}\colon \Sp(\R^d,2^m)\longrightarrow \conf(\R^d,2^m)
	\]
	is an embedding.
\end{proposition}
 
\begin{proof}
	The continuity of the map $\epicy_{d,2^m}$ follows directly from Definition \ref{def : Epicycles}.
	Thus, we only need to show that $\epicy_{d,2^m}$ is an injective map.
	For that we use induction on integer $m\geq 0$.
	
	%\small
	\noindent
	The map $\epicy_{d,1}\colon \Sp(\R^d,1)\longrightarrow \conf(\R^d,1)$ is evidently injective, since $\Sp(\R^d,1)=\pt$. 
	Assume that for $m=k\geq 0$ the map $\epicy_{d,2^k}$ is injective.
	Now, for $m=k+1$, suppose that 
	\begin{equation}
		\label{eq : assumption-01}
		\epicy_{d,2^{k+1}}(y_1,y_2,x)=\epicy_{d,2^{k+1}}(y_1',y_2',x'),
	\end{equation} 
	where $(y_1,y_2,x),(y_1',y_2',x')\in\Sp(\R^d,2^{k+1})$. 
	Consequently, the sums of first $2^k$ coordinates must coincide 
	\[
	\sum_{j=1}^{2^k}\epicy_{d,2^{k+1}}^j(y_1,y_2,x)=\sum_{j=1}^{2^k}\epicy_{d,2^{k+1}}^j(y_1',y_2',x').
	\]
	From the definition of the map $\epicy_{d,2^{k+1}}$ it follows that
	\[
	2^kx+\varepsilon\Big( \sum_{j=1}^{2^k}\epicy_{d,2^{k}}^j(y_1)\Big)=2^kx'+\varepsilon\Big( \sum_{j=1}^{2^k}\epicy_{d,2^{k}}^j(y_1')\Big).
	\]
	From Proposition \ref{prop : sum is zero} we get
	\[
	\sum_{j=1}^{2^k}\epicy_{d,2^{k}}^j(y_1)=0 \qquad\text{and}\qquad \sum_{j=1}^{2^k}\epicy_{d,2^{k}}^j(y_1')=0,
	\]
	implying that $x=x'$.
	Furthermore, using Definition \ref{def : Epicycles} and \eqref{eq : assumption-01} we have that
	\[
	\epicy_{d,2^{k}}(y_1)=\epicy_{d,2^{k}}(y_1') \qquad\text{and}\qquad \epicy_{d,2^{k}}(y_2)=\epicy_{d,2^{k}}(y_2').
	\]
	Finally the induction hypothesis implies that $y_1=y_1'$ and $y_2=y_2'$ concluding the proof of the proposition.
\end{proof}

\medskip
In the case when $d=\infty$ the space $\Sp(\R^d,2^m)$ is a contractible space with a free $\Sy_{2^m}$-action, and therefore is a model for $\EEE \Sy_{2^m}$.
In particular, we can observe that
\[ 
\Sp(\R^{\infty} ,2^m)\cong\colim_{d\to\infty} \Sp(\R^d,2^m),
\]
where the colimit is defined via the inclusions $\R^d\longrightarrow\R^{d+1}$, $x\longmapsto (x,0)$, which induce the corresponding inclusion maps $\Sp(\R^d,2^m)\longrightarrow\Sp(\R^{d+1},2^m)$.
Furthermore, the induced $\Sy_{2^m}$-equivariant map $\Sp(\R^d,2^m)\longrightarrow \colim_{d\to\infty}\Sp(\R^d,2^m)$ is given by the inclusion maps $\R^d\longrightarrow\R^{\infty}$, $x\longmapsto (x,0,0,\ldots)$, and is denoted by
\begin{equation}
\label{eq : map kappa}
\kappa_{d,2^m}\colon \Sp(\R^d,2^m)\longrightarrow \Sp(\R^{\infty},2^m).	
\end{equation}
In summary, we have the following.

\begin{proposition}
	\label{prop : model for E}
	Let $m\geq 0$ be an integer.
	Then $\Sp(\R^{\infty} ,2^m)$ is a free contractible $\Sy_{2^m}$-CW complex.
\end{proposition}

\medskip
Before we continue towards the study of the equivariant cohomology of the space $\Sp(\R^{d},2^m)$, we specify an elementary abelian subgroup $\EE_m$ of the Sylow $2$-subgroup $\Sy_{2^m}$ of the symmetric group\index{symmetric group} $\Sym_{2^m}$.

\begin{definition}
	\label{def : el.ab.subgroup}
	Let $m\geq 0$ be an integer.
	The subgroup $\EE_m$ of the group $\Sy_{2^m}$ is defined inductively as follows.
	
	\begin{compactenum}[\rm \ (1)]
	
	%---------------------
	\item If $m=0$, then we set $\EE_0:=\Sy_1=1$.
	
	%---------------------
	\item If $m=1$, then we set $\EE_1:=\Sy_2\cong\Z_2$.	
	
	%---------------------
	\item Let us assume that for $m=k\geq 1$ we have defined the subgroup $\EE_k$ of $\Sy_{2^k}$.	
	%---------------------
	\item If $m=k+1$ and $\delta\colon \Sy_{2^k}\longrightarrow \Sy_{2^k}\times \Sy_{2^k}$ denotes the diagonal monomorphism given by $s\mapsto (s,s)$, then we set
	\[
	\EE_{k+1}:=\delta (\EE_k)\times \Z_2\cong \EE_k\times \Z_2\cong \Z_2^{\oplus k+1}\subseteq (\Sy_{2^k}\times \Sy_{2^k})\rtimes \Z_2=\Sy_{2^{k+1}}.
	\]
\end{compactenum}
\end{definition}	

\noindent
Furthermore, let $C_1,\ldots,C_m$ be cyclic groups isomorphic to $\Z_2$ with the property that 
\[
\EE_i=(C_1\times\cdots\times C_{i-1})\times C_i = \delta (\EE_{i-1})\times C_i.
\]
Here we make slight abuse of notation identifying $\EE_{i-1}$ with a subgroup of $\EE_{i}$.
Having this decomposition fixed we see that $\Sy_{2^m}= C_1\wr\cdots\wr C_m$.
In summary, we have inclusions of the groups
\begin{equation}
\label{eq : inclusion of groups}
\EE_m\subseteq \Sy_{2^m} \subseteq \Sym_{2^m}	
\end{equation}
where $\EE_m$ is the subgroup of $\Sym_{2^m}$ given by translations (seen as permutations) of $\Z_2^{\oplus m}$, the so called regular embedded subgroup \cite[Ex.\,III.2.7]{Adem2004}, and the inclusion $\Sy_{2^m} \subseteq \Sym_{2^m}$ is defined via the monomorphism $\iota_m$. 

\medskip
The cohomology of the elementary abelian group\index{elementary abelian group} $\EE_m$ with $\F_2$ coefficients is well known.
In particular, we fix the following presentation
\[
H^*(\EE_m;\F_2)=\F_2[y_1,\ldots,y_m],
\] 
where $\deg(y_k)=1$ for $1\leq k\leq m$, in such a way that the exact sequence of groups
\[
\xymatrix{
1 \ \ar[r]   & \ \EE_{m-1}\ar[r] & \ \EE_{m-1}\times C_m \ \ar[r]   &  \ C_m \ \ar[r]   & \  1
}
\]
induces the following sequence of algebras
\[
\xymatrix{
0 \ & \ \F_2[y_1,\ldots,y_{m-1}] \ \ar[l] & \ \F_2[y_1,\ldots,y_{m-1},y_m]\ \ar[l] & \ \F_2[y_m] \ \ar[l] & \ 0 \ ,\ar[l] 
}
\]
which is exact in each positive degree.

%%%%%%%%%%%%%%%%%%%%%%%%%%%%%%%%%%%%%%%%%%%%%%%%%%%%%%%%%%%%%%%%%%%%%%%%%%%%%%%%%%%%%
%%%%%%%%%%%%%%%%%%%%%%%%%%%%%%%%%%%%%%%%%%%%%%%%%%%%%%%%%%%%%%%%%%%%%%%%%%%%%%%%%%%%%
\section{The equivariant cohomology\index{equivariant cohomology} of $\Sp(\R^{d} ,2^m)$}
\label{sec : Equivariant cohomology of epicicles}
%%%%%%%%%%%%%%%%%%%%%%%%%%%%%%%%%%%%%%%%%%%%%%%%%%%%%%%%%%%%%%%%%%%%%%%%%%%%%%%%%%%%%
%%%%%%%%%%%%%%%%%%%%%%%%%%%%%%%%%%%%%%%%%%%%%%%%%%%%%%%%%%%%%%%%%%%%%%%%%%%%%%%%%%%%%

Let $d\geq 2$ be an integer or $d=\infty$, and let $m\geq 0$ be an integer.
In this section we study the equivariant cohomology of the space  $\Sp(\R^d,2^m)$  with respect to the already defined free action of the group $\Sy_{2^m}$, that is,
\begin{multline*}
H^*_{\Sy_{2^m}}(\Sp(\R^d,2^m);\F_2)\cong H^*(\Sp(\R^d,2^m)/\Sy_{2^m};\F_2) \\ \cong H^*(\EEE \Sy_{2^m}\times_{\Sy_{2^m}}\Sp(\R^d,2^m);\F_2) .	
\end{multline*}
In particular, if $d=\infty$, then according to Proposition \ref{prop : model for E} we have that
\[
H^*_{\Sy_{2^m}}(\Sp(\R^{\infty},2^m);\F_2)\cong H^*(\Sp(\R^{\infty},2^m)/\Sy_{2^m};\F_2)\cong H^*(\Sy_{2^m};\F_2).
\]
Since the space $\Sp(\R^d,2^m)$ is defined inductively, our computation utilizes this feature, and is an alternative to the calculation presented in \cite[Sec.\,2]{Hung1990}.
The methods we use are applicable in a more general setting.
%
%====
\subsection{Small values of $m$}
\label{sec:small values}
%====
 
Let us first consider the case when $m=0$.
In this case, according to Definition \ref{def : Epicycles}, we have that
\[
\Sp(\R^d,1)=\pt \qquad\text{and}\qquad \Sy_1=1.
\]
Consequently
\[
H^r_{\Sy_{1}}(\Sp(\R^d,1);\F_2) \cong  H^r(\Sp(\R^d,1)/\Sy_{1};\F_2)\cong H^r(\pt;\F_2)
\cong 
\begin{cases}
\F_2, & r=0,\\
0   , & r\neq 0.	
\end{cases}	
\]
We specify a particular additive basis\index{additive basis} for the cohomology $H^*(\Sp(\R^d,1)/
\Sy_{1};\F_2)$ by
\[
\basis(\R^d,1):=\{1\},
\]
where $1$ is the generator of $H^0(\Sp(\R^d,1)/\Sy_{1};\F_2)\cong\F_2$. 
In addition we partition the set $\basis(\R^d,1)$ into two subsets
\[
\basisa(\R^d,1):=\basis(\R^d,1)
\qquad\text{and}\qquad
\basisi(\R^d,1):=\basis(\R^d,1){\setminus}\basisa(\R^d,1).
\]
Thus, $\basisi(\R^d,1)=\emptyset$.

\medskip
Next, let $m=1$.
Then from Definition \ref{def : Epicycles} we get
\[
\Sp(\R^d,2)=S^{d-1}\qquad\text{and}\qquad \Sy_2=\Z_2.
\]
Therefore
\begin{align*}
H^r_{\Sy_{2}}(\Sp(\R^d,2);\F_2) &\cong  H^r(\Sp(\R^d,2)/\Z_2;\F_2)\cong H^r(\RP^{d-1};\F_2)	\\
& 
\cong
\begin{cases}
\F_2, & 0\leq r\leq d-1,\\
0   , & \text{otherwise}.	
\end{cases}	
\end{align*}
Again we specify an additive basis\index{additive basis} for the cohomology $H^*(\Sp(\R^d,2)/\Sy_{2};\F_2)$ as follows:
\[
\basis(\R^d,2):=\{1,e,e^2,\ldots,e^{d-1}\},
\]
where $e^i$ is the generator of $H^{i}(\RP^{d-1};\F_2)\cong\F_2$. 
The partition of the basis $\basis(\R^d,2)$ we use is defined by:
\[
\basisa(\R^d,2):=\{(x\otimes x)\otimes_{\Z_2}e^i : x\in \basisa(\R^d,1),\, 0\leq i\leq d-1\}=\{1,e,e^2,\ldots,e^{d-1}\},
\]
and 
\[
\basisi(\R^d,2):=\basis(\R^d,2){\setminus}\basisa(\R^d,2)=\emptyset.
\]

%====
\subsection{The case $m=2$}
%====

The first interesting case is when $m=2$.
Again from Definition \ref{def : Epicycles} we have that
\[
\Sp(\R^d,2^m)=\Sp(\R^d,4)= (\Sp(\R^d,2)\times \Sp(\R^d,2)) \times S^{d-1}=
(S^{d-1}\times S^{d-1})\times S^{d-1} ,
\]
and
\[
\Sy_{2^m}=\Sy_{4}= (\Z_2\times \Z_2)\rtimes \Z_2=\Z_2\wr\Z_2= \big(\langle \varepsilon_1\rangle\times\langle \varepsilon_2\rangle\big)\rtimes \langle\omega\rangle.
\]
(The group $\Sy_{4}$ is isomorphic to the dihedral group\index{dihedral group} $D_8$.)
The free action of the group $\Sy_{4}$ on $\Sp(\R^d,4)$ we introduced is given by
\begin{eqnarray*}
	\varepsilon_1\cdot(y_1,y_2,x) &=& (-y_1, y_2,x),\nonumber\\	
	\varepsilon_2\cdot(y_1,y_2,x) &=& (y_1, -y_2,x),\nonumber \\
	\omega\cdot(y_1,y_2,x) &=& ( y_2, y_1,-x),
\end{eqnarray*}
for $(y_1,y_2,x)\in S^{d-1}\times (S^{d-1}\times S^{d-1})$.
Thus
\begin{eqnarray*}
	\Sp(\R^d,4)/\Sy_{4} &=&  \big( (S^{d-1}\times S^{d-1})\times S^{d-1}\big)/\Sy_{4}\nonumber\\
	&\cong & \big( (S^{d-1}/\Z_2\times S^{d-1}/\Z_2)\times S^{d-1}\big)/\Z_2 \nonumber\\
	&\cong & \big((\RP^{d-1}\times \RP^{d-1})\times S^{d-1}\big)/\Z_2 = :  (\RP^{d-1}\times \RP^{d-1})\times_{\Z_2}S^{d-1}.
\end{eqnarray*}
The action of $\Z_2$ on $(\RP^{d-1}\times \RP^{d-1})\times S^{d-1}$ we have is given by 
\[
(u_1,u_2,x)\longmapsto (u_2,u_1,-x)
\]
for $(u_1,u_2,x)\in (\RP^{d-1}\times \RP^{d-1})\times S^{d-1}$.
Since this $\Z_2$-action is free then the corresponding quotient space $(\RP^{d-1}\times \RP^{d-1})\times_{\Z_2} S^{d-1}$ is the total space of the fiber bundle
\begin{equation}
\label{eq : fibration-01}
\xymatrix{
\RP^{d-1}\times \RP^{d-1}\ar[r] & (\RP^{d-1}\times \RP^{d-1})\times_{\Z_2} S^{d-1} \ar[r]^-{p} &  \RP^{d-1},
}
\end{equation}
where the map $p$ is induced by the ($\Z_2$-equivariant) projection on the second factor 
\[
(\RP^{d-1}\times \RP^{d-1})\times  S^{d-1}\longrightarrow S^{d-1}.
\]
Furthermore, the $\Sy_{4}$-equivariant inclusion $\kappa_{d,4}\colon \Sp(\R^d,4)\longrightarrow \Sp(\R^{\infty},4)$, introduced in \eqref{eq : map kappa}, induces the following morphism of fiber bundles
\begin{equation}
	\label{eq : fibration-02}
	\xymatrix{
 (\RP^{d-1}\times \RP^{d-1})\times_{\Z_2} S^{d-1} \ \ar[rr]^-{\kappa_{d,4}/\Sy_4}\ar[d] &  & \ (\RP^{\infty}\times \RP^{\infty})\times_{\Z_2}S^{\infty} \ar[d]\\
 \RP^{d-1} \ \ar[rr]& & \ \RP^{\infty}.
}
\end{equation}

\medskip
For every integer $d\geq 2$ or $d=\infty$ the fiber bundle \eqref{eq : fibration-01} induces a Serre spectral sequence\index{Serre spectral sequence} that  converges to $H^*((\RP^{d-1}\times \RP^{d-1})\times_{\Z_2}S^{d-1};\F_2)$.
The $E_2$-term of this spectral sequence is of the form
\begin{equation}
\label{eq : fibration-03}
E_{2}^{r,s}(d)=H^r(\RP^{d-1};\mathcal{H}^s(\RP^{d-1}\times \RP^{d-1};\F_2)).
\end{equation}
Here $\mathcal{H}^*(\RP^{d-1}\times \RP^{d-1};\F_2)$ denotes a local coefficient system determined by the action of the fundamental group of the base $\pi_1(\RP^{d-1})$ on the cohomology of the fiber $H^*(\RP^{d-1}\times \RP^{d-1};\F_2)$.
The cohomology ring\index{cohomology ring} of the fiber, via the K\"unneth formula \cite[Thm.\,VI.3.2]{Bredon2010}, can be presented in the following way
\begin{align*}
H^*(\RP^{d-1}\times \RP^{d-1};\F_2) &\cong  H^*(\RP^{d-1};\F_2)\otimes H^*(\RP^{d-1};\F_2)\\
&\cong \F_2[e_1]/\langle e_1^d\rangle \otimes  \F_2[e_2]/\langle e_2^d\rangle\\	
&\cong \F_2[e_1,e_2]/\langle e_1^d,e_2^d\rangle,
\end{align*}
where $\deg(e_1)=\deg(e_2)=1$.
Here $\langle e_1^d,e_2^d\rangle$ denotes the ideal in $\F_2[e_1,e_2]$ generated by the polynomials $e_1^d$ and $e_2^d$.
The fundamental group $\pi_1(\RP^{d-1})=\langle t\rangle$ is a cyclic group.
Indeed, $\pi_1(\RP^{1})\cong\Z$ and $\pi_1(\RP^{d-1})\cong\Z_2$ for $d\geq 3$.
The action of $\pi_1(\RP^{d-1})$ on the cohomology ring $H^*(\RP^{d-1}\times \RP^{d-1};\F_2)$ is given by $t\cdot e_1=e_2$.

\medskip
In the case when $d=\infty$ the spectral sequence \eqref{eq : fibration-03} becomes
\begin{equation}
\label{eq : fibration-04}
E_{2}^{r,s}(\infty)=H^r(\RP^{\infty};\mathcal{H}^s(\RP^{\infty}\times \RP^{\infty};\F_2))\cong H^r(\Z_2;H^s(\Z_2\times\Z_2;\F_2)).
\end{equation}
Now, the cohomology ring\index{cohomology ring} of the fiber, via the K\"unneth formula \cite[Thm.\,VI.3.2]{Bredon2010}, can be presented in the following way
\begin{align*}
H^*(\RP^{\infty}\times \RP^{\infty};\F_2) &\cong  H^*(\RP^{\infty};\F_2)\otimes H^*(\RP^{\infty};\F_2)	\\
&\cong  \F_2[e_1]\otimes \F_2[e_2]\cong \F_2[e_1,e_2],
\end{align*}
where $\deg(e_1)=\deg(e_2)=1$.
Notice an abuse of notation occurring in naming of the generators of the cohomology of the fibers in spectral sequences \eqref{eq : fibration-03} and \eqref{eq : fibration-04}.
The fundamental group $\pi_1(\RP^{\infty})=\langle t\rangle\cong\Z_2$ acts on the cohomology ring of the fiber $H^*(\RP^{\infty}\times \RP^{\infty};\F_2)$ by $t\cdot e_1=e_2$.

\begin{figure}[h]
\centering
\includegraphics[scale=0.7]{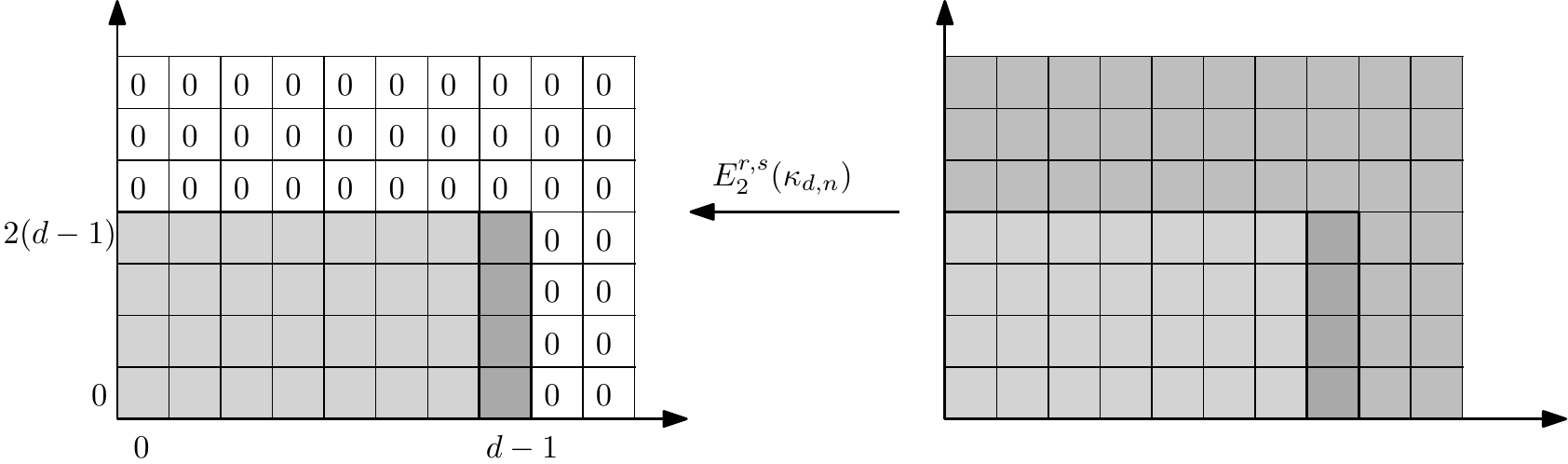}
\caption{\small The morphism $E_{2}^{r,s}(\kappa_{d,n})$.}
\label{fig : ss-04}
\end{figure}

\medskip
Next, we consider the morphism between the spectral sequences\index{morphism between spectral sequences} \eqref{eq : fibration-03} and \eqref{eq : fibration-04} induced by the morphism of fiber bundles \eqref{eq : fibration-02} on the level of $E_2$-terms:
\[
\xymatrix@1{
E_{2}^{r,s}(d)\ar@{=}[d] \ & & \ E_{2}^{r,s}(\infty)\ar[ll]_-{E_{2}^{r,s}(\kappa_{d,4}/\Sy_4)}\ar@{=}[d]\\
H^r(\RP^{d-1};\mathcal{H}^s(\RP^{d-1}\times \RP^{d-1};\F_2)) \ & &
\ H^r(\RP^{\infty};\mathcal{H}^s(\RP^{\infty}\times \RP^{\infty};\F_2))\ar[ll]_-{E_{2}^{r,s}(\kappa_{d,4}/\Sy_4)}.
}
\]
Then we have that
\begin{compactitem}[ \ ---]

\item $E_{2}^{r,s}(\kappa_{d,4}/\Sy_4)$ is an isomorphism for all $(r,s)\in\{0,\ldots,d-2\}\times\{0,\ldots,d-1\}$,

\item $E_{2}^{r,s}(\kappa_{d,4}/\Sy_4)$ is an epimorphism for all $(r,s)\in \{0,\ldots,d-2\}\times\{0,\ldots,2(d-1)\}$, 

\item $E_{2}^{r,s}(\kappa_{d,4}/\Sy_4)\neq 0$ is a monomorphism for all $(r,s)\in \{d-1\}\times\{1,\ldots, d-1\}$, 

\item $E_{2}^{r,s}(\kappa_{d,4}/\Sy_4)=0$ for every $(r,s)\notin \{ 0,\ldots,d-1\}\times\{0,\ldots, 2(d-1)\}$, because all $E_{2}^{r,s}(d)$ vanish.

\end{compactitem}
(For an illustration of the morphism $E_{2}^{r,s}(\kappa_{d,n})$ see Figure \ref{fig : ss-04}.)
Thus, if we prove that the spectral sequence \eqref{eq : fibration-04} collapses at the $E_2$-term the same will be true for the spectral sequence \eqref{eq : fibration-03}.
Indeed, in the next section we prove Theorem \ref{th : E_2 collapses} which guaranties that the spectral sequence \eqref{eq : fibration-04} collapses at the $E_2$-term. 
Consequently, the spectral sequence \eqref{eq : fibration-03} also collapses at the $E_2$-term because all the differentials emanating from positions $(d-1,s)\in\{d-1\}\times\Z$ are zero.

\medskip
 Now, the calculation of the cohomology of the real projective space with local coefficients $H^*(\RP^{d-1};\mathcal{M})$, presented in Section \ref{subsub : computation of local coefficients} of the appendix, in combination with Lemma \ref{lem : 02} from the next section, gives the complete description of the $E_2$-terms of the both spectral sequences.
\begin{compactitem}[ \ ---]

\item For the spectral sequence \eqref{eq : fibration-03} we have
\begin{eqnarray*}
\hspace{-10pt}E_{2}^{r,s}(d) &=& H^r(\RP^{d-1};\mathcal{H}^s(\RP^{d-1}\times \RP^{d-1};\F_2))\\
&=&
\begin{cases}
		H^0(\RP^{d-1};\F_2)\oplus \F_2^{\oplus q(s)} ,				 & r=0,  s\text{ even, }0\leq s\leq 2d-2, \\
		\F_2^{\oplus q(s)}, & r=0,  s\text{ odd, }1\leq s\leq 2d-3,\\
		H^{d-1}(\RP^{d-1};\F_2)\oplus \F_2^{\oplus q(s)} ,				 & r=d-1,  s\text{ even, }0\leq s\leq 2d-2,\\
		\F_2^{\oplus q(s)}, & r=d-1,  s\text{ odd, } 1\leq s\leq 2d-3, \\
		H^r(\RP^{d-1};\F_2), & 1\leq r\leq d-2,  s\text{ even, } 0\leq s\leq 2d-2,  \\
		0,							 & \text{otherwise},
\end{cases}
\end{eqnarray*}
where $q(s):=|\{(i,j)\in\Z\times\Z : 0\leq i<j\leq d-1, i+j=s\}|$.
\item For the spectral sequence \eqref{eq : fibration-04} we have
\begin{eqnarray*}
	E_{2}^{r,s}(\infty) &=& H^r(\B\Z_2;\mathcal{H}^s(\RP^{\infty}\times \RP^{\infty};\F_2))\\
	&\cong & H^r(\Z_2;H^s(\RP^{\infty}\times \RP^{\infty};\F_2)) \nonumber\\
	&=&
\begin{cases}
		H^0(\Z_2;\F_2)\oplus \F_2^{\oplus q(s)} ,		& r=0,  s\text{ is even, }s\geq 0,\\
		\F_2^{\oplus q(s)}, & r=0,  s \text{ is odd, } s\geq 1,\\
				H^r(\Z_2;\F_2), & r\geq 1,  s\text{ is even, }s\geq 0, \\
		0,							 & \text{otherwise}.
\end{cases}
\end{eqnarray*}
\end{compactitem}
In Figure \ref{fig : ss-01} the $E_2$-term of the Serre spectral sequence\index{Serre spectral sequence} \eqref{eq : fibration-03} associated with the fibration \eqref{eq : fibration-01} in the case $d=5$ is presented.

\begin{figure}[ht]
\centering
\includegraphics[scale=0.7]{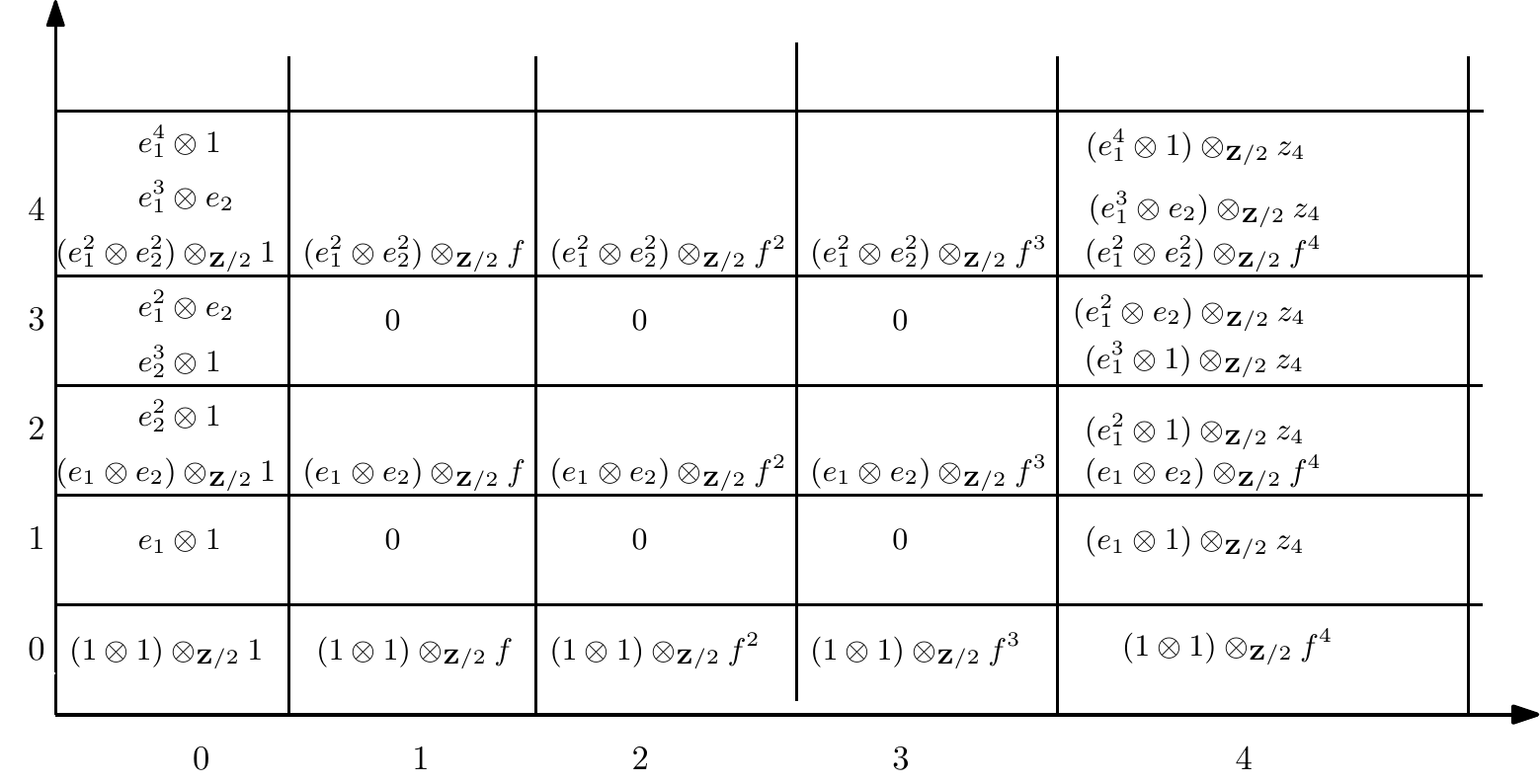}
\caption{\small $E_2=E_{\infty}$-term of the Serre spectral sequence\index{Serre spectral sequence} \eqref{eq : fibration-03} for the fibration \eqref{eq : fibration-01} when $d=5$.
In the picture,  for example, by $e_1\otimes 1\in E_2^{0,1}$ we denote the invariant element $e_1\otimes 1+1\otimes e_1$.
}
\label{fig : ss-01}
\end{figure}

\medskip
The additive basis\index{additive basis} $\basis(\R^d,4)$ for the cohomology $H^*(\Sp(\R^d,4)/\Sy_4;\F_2)$ is specified using already defined basis $\basis(\R^d,2)$ in the following way:
\begin{compactenum}[\rm \ (a)]
	\item $(e^i\otimes e^i)\otimes_{\Z_2} f^k=(e^i_1\otimes e^i_2)\otimes_{\Z_2} f^k\in \basis(\R^d,4)$ for $e^i\in \basis(\R^d,2)$, $0\leq k\leq d-1$,
	\item $(e^i\otimes e^j)\otimes_{\Z_2} 1=(e^i_1\otimes e^j_2)\otimes_{\Z_2} 1\in \basis(\R^d,4)$ for $e^i,e^j\in \basis(\R^d,2)$, $0\leq j<i\leq d-1$, 
	\item $(e^i\otimes e^j)\otimes_{\Z_2} z_{d-1}=(e^i_1\otimes e^j_2)\otimes_{\Z_2} z_{d-1}\in \basis(\R^d,4)$ when $d<\infty$ for $e^i,e^j\in \basis(\R^d,2)$, $0\leq j<i\leq d-1$.\newline
	Here $f\in H^1(\RP^{d-1};\F_2)$ denotes the multiplicative generator of the cohomology of the base space of the fibration \eqref{eq : fibration-01}.
	Moreover, in the definition of the basis $\basis(\R^d,4)$ we used two notations.
	One of them is analogues to the notation used in Section \ref{subsub : computation of local coefficients} for the generator $(e^i\otimes e^j)\otimes_{\Z_2} z_{d-1}$ of the cohomology group $H^{d-1}(\RP^{d-1};\mathcal{M}_{i,j})$ where the local coefficient system $\mathcal{M}_{i,j}\cong\F_2\oplus\F_2$ is given by cohomology classes $e^i\otimes e^j$ and $e^j\otimes e^i$.
\end{compactenum}
Now, a partition of the basis\index{additive basis} $\basis(\R^d,4)$ is defined by
\begin{multline*}
\basisa(\R^d,4):=\big\{(x\otimes x)\otimes_{\Z_2}f^k : x\in \basisa(\R^d,2),\, 0\leq k\leq d-1\big\}=\\
\big\{ (e^i\otimes e^i)\otimes_{\Z_2} f^k : e^i\in \basis(\R^d,2),\, 0\leq k\leq d-1\big\},
\end{multline*}
and in addition 
\begin{multline*}
\basisi(\R^d,4):=\basis(\R^d,4){\setminus}\basisa(\R^d,4)=\\
\big\{(e^i\otimes e^j)\otimes_{\Z_2} 1 : e^i,e^j\in \basis(\R^d,2),\ 0\leq j<i\leq d-1\big\} \cup \\
\big\{(e^i\otimes e^j)\otimes_{\Z_2} z_{d-1} : d<\infty,\ 0\leq j<i\leq d-1\big\}.
\end{multline*}
Furthermore, denote by 
\[
\AAA^*(\R^d,4):=\langle \basisa(\R^d,4)\rangle
\qquad\text{and}\qquad
\III^*(\R^d,4):=\langle \basisi(\R^d,4)\rangle.
\]
Then there is an additive decomposition of the cohomology of $\Sp(\R^d,4)/\Sy_4$ as follows:
\[
H^*(\Sp(\R^d,4)/\Sy_{4};\F_2)\cong \AAA^*(\R^d,4)\oplus \III^*(\R^d,4),
\]
where $\AAA^*(\R^d,4)$ turns out to be a subalgebra of $H^*(\Sp(\R^d,4)/\Sy_4;\F_2)$ while $\III^*(\R^d,4)$ is an ideal.
More precisely, there is an isomorphism of algebras
\[
\AAA^*(\R^d,4)\cong \langle  (1\otimes 1)\otimes_{\Z_2} f ,(e_1\otimes e_2)\otimes_{\Z_2} 1\rangle
\cong \F_2[V_{2,1},V_{2,2}]/\langle V_{2,1}^d, V_{2,2}^d\rangle,
\]
where $V_{2,1}=(1\otimes 1)\otimes_{\Z_2} f$ and $ V_{2,2}=(e\otimes e)\otimes_{\Z_2} 1=(e_1\otimes e_2)\otimes_{\Z_2} 1$.
Thus, we have proved that the cohomology $H^*(\Sp(\R^d,4)/\Sy_4;\F_2)$ can be decomposed into the sum of a subalgebra and an ideal as follows
\begin{equation}
	\label{eq : decomposition for n=4}
	H^*(\Sp(\R^d,4)/\Sy_4;\F_2)\cong  \F_2[V_{2,1},V_{2,2}]/\langle V_{2,1}^d, V_{2,2}^d\rangle \oplus \III^*(\R^d,4),
\end{equation}
where $\deg (V_{2,1})=2^{1-1}=1$ and $\deg (V_{2,2})=2^{2-1}=2$.
The part of the Serre spectral sequence\index{Serre spectral sequence} in the case $d=5$ induced by the subalgebra $\AAA^*(\R^5,4)$ is illustrated in Figure \ref{fig : ss-02}.

\begin{figure}[ht]
\centering
\includegraphics[scale=0.73]{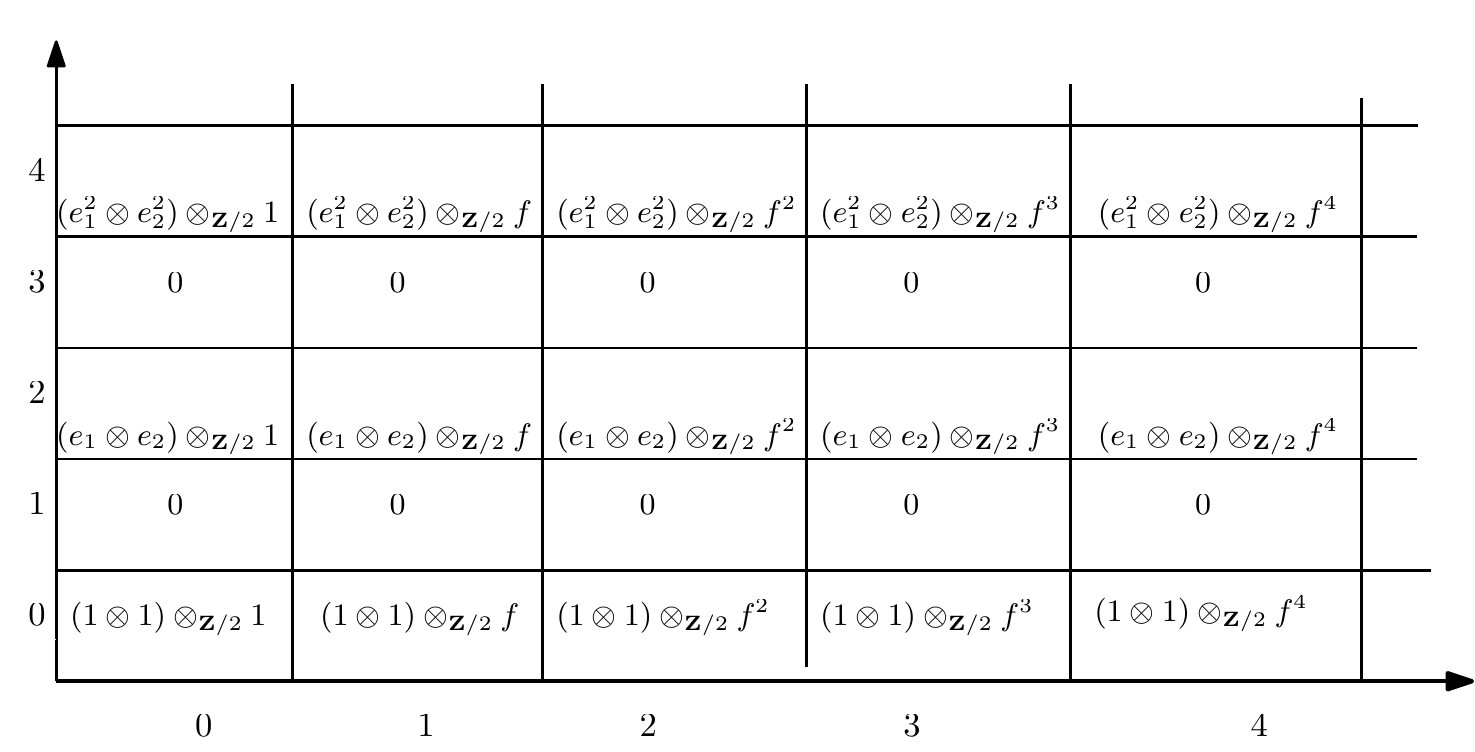}
\caption{\small The algebra $\AAA^*(\R^5,4)$.}
\label{fig : ss-02}
\end{figure}

%====
\subsection{Cohomology of $(X\times X)\times_{\Z_2}S^{d-1}$ and $(X\times X)\times_{\Z_2}\EEE\Z_2$}
\label{sub : cohomology of wreath product}
%====
Let $X$ be a CW-complex (not necessarily finite dimensional).
Consider an action of the group $\Z_2=\langle t\rangle$ on the product $X\times X$ given by $t\cdot (x_1,x_2):=(x_2,x_1)$.
Then the product spaces $(X\times X)\times S^{d-1}$ and $(X\times X)\times\EEE\Z_2$ are equipped with the diagonal $\Z_2$-action where the action on $\EEE\Z_2$ comes with the space definition and the action on $S^{d-1}$ is assumed to be antipodal. 
In this section, using the classical work of Minoru Nakaoka \cite{Nakaoka1961} and following the fundamental book of Alejandro Adem and James Milgram \cite[Sec.\,IV.1]{Adem2004},  we describe the cohomology with $\F_2$-coefficients of the following quotient spaces
\begin{align*}
(X\times X)\times_{\Z_2}S^{d-1} &:= \big((X\times X)\times S^{d-1} \big)/\Z_2,\\
	(X\times X)\times_{\Z_2}\EEE\Z_2 &:= \big((X\times X)\times \EEE\Z_2\big)/\Z_2.
\end{align*}
For this we will use Serre spectral sequences of the following fibrations.

\medskip
The spaces $(X\times X)\times_{\Z_2}S^{d-1}$ and $(X\times X)\times_{\Z_2}\EEE\Z_2$ are the total spaces of the following fiber bundles
\begin{equation}
	\label{eq : fib-01}
	\xymatrix{
X\times X \ \ar[r] & \ (X\times X)\times_{\Z_2}S^{d-1} \ \ar[r] &\ \RP^{d-1},
}
\end{equation}
where we use that $S^{d-1}/\Z_2\cong\RP^{d-1}$, and
\begin{equation}
	\label{eq : fib-02}
\xymatrix{
X\times X \ \ar[r] & \ (X\times X)\times_{\Z_2}\EEE\Z_2 \ \ar[r] & \ \RP^{\infty} \ \cong \ \B\Z_2.
}
\end{equation}
If for a model of $\EEE\Z_2$ we take $S^{\infty}$ with antipodal action the natural inclusion map $S^{d-1}\longrightarrow S^{\infty}$ induces the following map between corresponding quotient spaces
\begin{equation}
	\label{eq : fib-022}
\xymatrix{
(X\times X)\times_{\Z_2}S^{d-1} \ \ar[r]^-{i_d} & \ (X\times X)\times_{\Z_2}S^{\infty} \cong(X\times X)\times_{\Z_2}\EEE\Z_2.}
\end{equation}
This map induces a morphism of the fiber bundles \eqref{eq : fib-01} and \eqref{eq : fib-02}:
\begin{equation}
\label{eq : fib-03}
\xymatrix{
(X\times X)\times_{\Z_2}S^{d-1} \ \ar[rr]^-{i_d}\ar[d] & & \ (X\times X)\times_{\Z_2}S^{\infty} \cong(X\times X)\times_{\Z_2}\EEE\Z_2\ar[d]\\
%& & \\
 \RP^{d-1} \ \ar[rr]& & \ \RP^{\infty}\cong \B\Z_2.
}
\end{equation}

\medskip
The cohomology of the total spaces of both fibrations \eqref{eq : fib-01} and \eqref{eq : fib-02} can be computed using Serre spectral sequences. 
In the case of the fibration \eqref{eq : fib-01} the $E_2$-term of the associated Serre spectral sequence\index{Serre spectral sequence} is of the form
\begin{equation}
	\label{eq : fib-04}
	E_{2}^{r,s}(d)=H^r(\RP^{d-1};\mathcal{H}^s(X\times X;\F_2)).
\end{equation}
The local coefficient system $\mathcal{H}^*(X\times X;\F_2)$ is determined by the action of the fundamental group $\pi_1(\RP^{d-1})$ of the base space on the cohomology of the fiber $H^*(X\times X;\F_2)$.
Recall that $\pi_1(\RP^{d-1})$ is a cyclic group.
In particular, the K\"unneth formula \cite[Thm.\,VI.3.2]{Bredon2010} gives a presentation of the cohomology of the fiber
\[
H^*(X\times X;\F_2)\cong H^*(X;\F_2)\otimes H^*(X;\F_2),
\]
and the action of $\pi_1(\RP^{d-1})$ is given by the cyclic shift of factors in the tensor product.

\noindent
The Serre spectral sequence associated to the fibration \eqref{eq : fib-02} has the $E_2$-term
\begin{align}
	\label{eq : fib-05}
	E_{2}^{r,s}(\infty) &= H^r(\RP^{\infty};\mathcal{H}^s(X\times X;\F_2)) \cong H^r(\B\Z_2;\mathcal{H}^s(X\times X;\F_2))\\
	&\cong H^r(\Z_2;H^s(X\times X;\F_2)).	\nonumber
\end{align}
The fundamental group of the base $\pi_1(\B\Z_2)\cong\Z_2$ acts on the cohomology of the fiber and defines the local coefficient system $\mathcal{H}^s(X\times X;\F_2)$, or defines the $\Z_2$-module structure on $H^r(\Z_2;H^s(X\times X;\F_2))$.

\medskip
The morphism \eqref{eq : fib-03} between the fibrations \eqref{eq : fib-01} and \eqref{eq : fib-02} induces a morphism between the spectral sequences\index{morphism between spectral sequences} \eqref{eq : fib-04} and \eqref{eq : fib-05}:
\[
\xymatrix@1{
E_{2}^{r,s}(d)=H^r(\RP^{d-1};\mathcal{H}^s(X\times X;\F_2)) \ & 
\ E_{2}^{r,s}(\infty)=H^r(\RP^{\infty};\mathcal{H}^s(X\times X;\F_2))\ar[l]_-{E_2^{r,s}(i_d)}}.
\]
The homomorphisms $E_2^{r,s}(i_d)$ are isomorphisms whenever $E_{2}^{r,s}(d)\neq 0$ and $(r,s)\notin \{d-1\}\times\Z$, or $E_{2}^{r,s}(d)=E_{2}^{r,s}(\infty)=0$.
In the case when $(r,s)\in \{d-1\}\times\Z$ we have that
\[
E_{2}^{d-1,s}(d)\cong E_{2}^{d-1,s}(\infty)\oplus \F_2^{b(d,X)},
\]
where for $s$ odd
\[
b(d,X)=\sum_{0\leq i<j,\, i+j=s}\big(\mathrm{rank}(H^i(X;\F_2))\cdot \mathrm{rank}(H^j(X;\F_2))\big),
\]
and for $s$ even
\begin{multline*}
b(d,X)=\sum_{0\leq i<j,\, i+j=s}\big(\mathrm{rank}(H^i(X;\F_2))\cdot \mathrm{rank}(H^j(X;\F_2))\big)+\\
\frac12\big(\mathrm{rank}(H^{\tfrac{s}{2}}(X;\F_2))\big)^2-\mathrm{rank}(H^{\tfrac{s}{2}}(X;\F_2))
.
\end{multline*}
%TODO : Explain better
Consult also Section \ref{subsub : computation of local coefficients}.
In particular, all the homomorphisms $E_2^{r,s}(i_d)$ are monomorphism, and for $(r,s)\notin \{d-1\}\times\Z$ they are non-zero epimorphisms.

\medskip
To make any further progress in computation of the spectral sequences \eqref{eq : fib-04} and \eqref{eq : fib-05} we need to understand the action of the fundamental groups of the base spaces on the cohomology of the fibers.
For that we use the following lemma, which is a particular case of {\cite[Lem.\,IV.1.4]{Adem2004}}.

\begin{lemma}
%[{\cite[Lem.\,IV.1.4]{Adem2004}}]
\label{lem : 01}
	Let $V=\bigoplus_{n\geq 0}V_n$ be a graded $\F_2$ vector space, and let $\mathcal{B}:=\{v_i : i\in I\}$ be a basis of $V$ where the index set $I$ is equipped with a linear order.
	The tensor product vector space $V\otimes V$ as a $\F_2[\Z_2]$-module, where the action of $\Z_2$ is given by the cyclic shift, is a direct sum of free and trivial $\F_2[\Z_2]$-modules.
	The trivial modules are generated by the elements of the form $v_i\otimes v_i$ where $v_i$ is an element of the basis\index{additive basis} $\mathcal{B}$, while the free modules are generated by the elements of the form $v_i\otimes v_j$ where $i<j$ and $v_i,v_j$ belong to $\mathcal{B}$.
\end{lemma}
\begin{proof}
A basis of the vector space $V\otimes V$ is given by all vectors of the form $v_i\otimes v_j$ where $v_i,v_j\in\mathcal{B}$.
The $\Z_2=\langle t\rangle$-action on $V\otimes V$ preserves this basis.
Since $t\cdot (v_i\otimes v_j)=v_j\otimes v_i$ we have that each element $v_i\otimes v_i$ generates a copy of the trivial $\F_2[\Z_2]$-module and each element $v_i\otimes v_j$, $i<j$, generates a copy of the free $\F_2[\Z_2]$-module.
Thus, $V\otimes V$ is a direct sum of free and trivial $\F_2[\Z_2]$-modules.
\end{proof}

Since the homomorphisms $E_2^{r,s}(i_d)$ are isomorphisms for $(r,s)\in \{0,\ldots,d-2\}\times\Z$, monomorphism when $(r,s)\in \{d-1\}\times\Z$, and otherwise zero homomorphisms, we study first the spectral sequence \eqref{eq : fib-05}.
If we prove that $E_{2}^{r,s}(\infty)=E_{\infty}^{r,s}(\infty)$ it would imply that all its differentials vanish and consequently the same would hold for the spectral sequence \eqref{eq : fib-04} implying that $E_{2}^{r,s}(d)=E_{\infty}^{r,s}(d)$.
In the next step we describe the $E_2$-term of the spectral sequence \eqref{eq : fib-05}, see also \cite[Cor.\,IV.1.6]{Adem2004}.

\begin{lemma}
\label{lem : 02}
Let $\mathcal{B}:=\{v_i : i\in I\}$ be a basis\index{additive basis} of the $\F_2$ vector space $H^*(X;\F_2)$ where the index set $I$ is equipped with a linear order.
The $E_2$-term of the spectral sequence \eqref{eq : fib-05} can be presented as follows:
\begin{align*}
E_{2}^{r,s}(\infty) & = H^r(\B\Z_2;\mathcal{H}^s(X\times X;\F_2)) \cong H^r(\Z_2;H^s(X\times X;\F_2))\\	
&\cong \begin{cases}
		H^s(X\times X;\F_2)^{\Z_2}, & r=0,\\
		H^{\frac{s}{2}}(X;\F_2),				 & r>0,\,s\text{ \rm even},\\
		0,							 & \text{\rm otherwise}.
	\end{cases}
\end{align*}
Moreover, $E_{2}^{*,*}(\infty)$ as a $H^*(\Z_2;\F_2)$-module, ignoring the grading, decomposes into the direct sum 
\[
\bigoplus_{i\in I} H^*(\Z_2;\F_2) \oplus \bigoplus_{i<j\in I} \F_2
\]
where the action of $H^*(\Z_2;\F_2)$ on each summand of the first sum is given by the cup product, and on the each summand of the second sum is trivial. 
\end{lemma}

\begin{proof}
This is a direct consequence of Lemma \ref{lem : 01}, and the facts that: 
\begin{compactitem}[ \ ---]
\item $H^0(G;M)=M^G$ for any group $G$ and any $G$-module $M$, and
\item $H^i(G;F)=0$ for $i\geq 1$ when $F$ is a projective (free) $G$-module.
\end{compactitem}
\end{proof}

\medskip
Now, for the spectral sequence \eqref{eq : fib-04}, using the calculation presented in Section \ref{subsub : computation of local coefficients}, we get the following presentation of the corresponding $E_2$-term.

\begin{lemma} 
\label{lem : 022}
The the $E_2$-term of the spectral sequence \eqref{eq : fib-04} can be presented as follows
\begin{align*}
	E_{2}^{r,s}(d) &=H^r(\RP^{d-1};\mathcal{H}^s(X\times X;\F_2))\\
	&\cong \begin{cases}
		H^s(X\times X;\F_2)^{\Z_2}, & r\in\{0,d-1\},\\
		H^{\frac{s}2}(X;\F_2),				 & 1\leq r\leq d-2,\,s\text{ \rm even},\\
		0,							 & \text{\rm otherwise}.
	\end{cases}
\end{align*}
\end{lemma}

\medskip
The fact that the $E_2$-term of the spectral sequence \eqref{eq : fib-05} collapses dates back to work of 
Smith \cite{Smith1941}, Steenrod \cite{Steenrod1962} and 
Nakaoka \cite{Nakaoka1961}, 
for a result in more general setup consult more recent work of Leary~\cite[Thm.\,2.1]{Leary1997}.
The following theorem is a special case of {\cite[Thm.\,IV.1.7]{Adem2004}}.

\begin{theorem}
%[{\cite[Thm.\,IV.1.7]{Adem2004}}] 
	\label{th : E_2 collapses}
	Let $X$ be a CW-complex. 
	The Serre spectral sequence\index{Serre spectral sequence} of the fibration 
	\begin{equation}
	\label{eq : in theorem 01}
	\xymatrix{
		X\times X \ \ar[r] & \ (X\times X)\times_{\Z_2}\EEE\Z_2 \ \ar[r] & \ \B\Z_2.
	}
	\end{equation}
	collapses at the $E_2$-term, that is, $E_{2}^{r,s}(\infty)\cong E_{\infty}^{r,s}(\infty)$ for all $(r,s)\in\Z\times\Z$.
\end{theorem}
\begin{proof}
	Let $\mathcal{B}:=\{v_i : i\in I\}$ be a basis\index{additive basis} of the $\F_2$ vector space $H^*(X;\F_2)$ where the index set $I$ is equipped with a linear order.
	The $E_2$-term \eqref{eq : fib-05} of the Serre spectral sequence\index{Serre spectral sequence} of the fibration \eqref{eq : in theorem 01} is calculated in Lemma \ref{lem : 02}.
	Further on, the $\F_2[\Z_2]$-module structure on the cohomology of the fiber -- the coefficient system -- is described in Lemma \ref{lem : 01}.
	Since the differentials of these spectral sequences are $H^*(\Z_2;\F_2)$-module maps we concentrate on the generators of the $H^*(\Z_2;\F_2)$-module structure of the rows of the spectral sequences. 
	In this situation it means that we consider elements of the zero column and prove that they survive to the $E_{\infty}$-term.
	Consequently the proof of the theorem proceeds in two steps.
	
	\medskip
	{\bf (A)}
	Let $v_i$ and $v_j$ be two different cohomology classes from the basis $\mathcal{B}$. 
	In the first step we prove that all the elements (of the form $v_i\otimes v_j+ v_j\otimes v_i$) in $E_2^{0,s}$ associated to the invariants of free $\F_2[\Z_2]$-modules (generated by $v_i\otimes v_j$) in the decomposition of $H^*(X\times X;\F_2)$ survive to the $E_{\infty}$-term.
	Since $\EEE\Z_2$ is a contractible and free $\Z_2$-space we have that $(X\times X)\times \EEE\Z_2\simeq (X\times X)$ is a free $\Z_2$-space and the quotient map 
	\[
	\pi\colon (X\times X)\times \EEE\Z_2 \longrightarrow (X\times X)\times_{\Z_2} \EEE\Z_2
	\]
	is a covering map.
	Denote by $p\colon (X\times X)\times \EEE\Z_2\longrightarrow X\times X$ the projection. 
	Since it is a homotopy equivalence it induces an isomorphism in cohomology.
	Furthermore, there is a transfer homomorphism 
	\[
	\mathrm{tr}\colon H^*((X\times X)\times \EEE\Z_2;\F_2)\longrightarrow H^*((X\times X)\times_{\Z_2}\EEE\Z_2;\F_2)
	\]
	with a property that that composition
	\[
	(p^*)^{-1}\circ\pi^*\circ\mathrm{tr}\circ p^* \colon H^*(X\times X;\F_2) \longrightarrow H^*(X\times X;\F_2)
	\]
	is the map 
	\[
	v_i\otimes v_j \longmapsto v_i\otimes v_j+ v_j\otimes v_i =(1+t)\cdot (v_i\otimes v_j),
	\] 
	where $v_i,v_j\in H^*(X;\F_2)$, $t$ is a generator of $\Z_2$, and $1+t\in\Z[\Z_2]$.
	Thus, the image $\im ((p^*)^{-1}\circ\pi^*\circ\mathrm{tr}\circ p^*)$ is contained in $H^*(X\times X;\F_2)^{\Z_2}$ and each element is associated to an invariant element $v_i\otimes v_j+ v_j\otimes v_i$, where $v_i\neq v_j$,  of a free $\F_2[\Z_2]$-module in the decomposition of the cohomology $H^*(X\times X;\F_2)$.
	Since the composition $(p^*)^{-1}\circ\pi^*\circ\mathrm{tr}\circ p^*$ factors through the cohomology $H^*( (X\times X)\times_{\Z_2}\EEE\Z_2;\F_2)$
	all these elements survive to the $E_{\infty}$-term.
	
	\medskip
	{\bf (B)}
	Let $v$ be a cohomology classes of dimension $n$ from the basis $\mathcal{B}$.
	In the second step we prove that all the elements (of the form $v\otimes v$) in $E_2^{0,s}$ associated to the invariants of trivial $\F_2[\Z_2]$-modules (generated by $v\otimes v$) in the decomposition of $H^*(X\times X;\F_2)$ survive to the $E_{\infty}$-term.
	Using the bijective correspondence
%	%TODO : Add reference
	\begin{equation}
		\label{eq : correspondence}
		H^n(X;\F_2) \longleftrightarrow [X,K(\Z_2,n)]
	\end{equation}
	we can present the cohomology class $v$ as the image of the fundamental class $\iota_n\in H^*(K(\Z_2,n);\F_2)$ along the map $\nu\colon X\longrightarrow K(\Z_2,n)$ that is associated to $v$ via the correspondence \eqref{eq : correspondence}, that is $v=\nu^*(\iota_n)$.
	For more details about the correspondence \eqref{eq : correspondence} consult for example \cite[Thm.\,1, page 3]{MosherTangora1968}.
	The map $\nu$ induces the following $\Z_2$-equivariant map $\nu\times\nu\colon X\times X\longrightarrow K(\Z_2,n)\times K(\Z_2,n)$, and consequently a morphism of Borel construction fibrations:
	\[
	\xymatrix{
	(X\times X)\times_{\Z_2}\EEE\Z_2 \ \ar[rr]^-{(\nu\times\nu)\times_{\Z_2}\id}\ar[d]&  & \ (K(\Z_2,n)\times K(\Z_2,n))\times_{\Z_2}\EEE\Z_2\ar[d]\\
%	& & \\
	  \B\Z_2 \ \ar[rr] & & \ \B\Z_2.
	}
	\]
	This morphism of fibrations induces a morphism between associated Serre spectral sequences. 
	In particular, the map between $E_2^{0,2n}$ entries sends the class $\iota_n\otimes\iota_n$ to the class $v\otimes v$, see Figure \ref{fig : ss-033}.
	Consequently, if the class $\iota_n\otimes\iota_n$ survives to the $E_{\infty}$-term (all differentials evaluated at $\iota_n\otimes\iota_n$ are zero), then the class $v\otimes v$ also survives to the $E_{\infty}$-term.

\begin{figure}[h]
\centering
\includegraphics[scale=0.65]{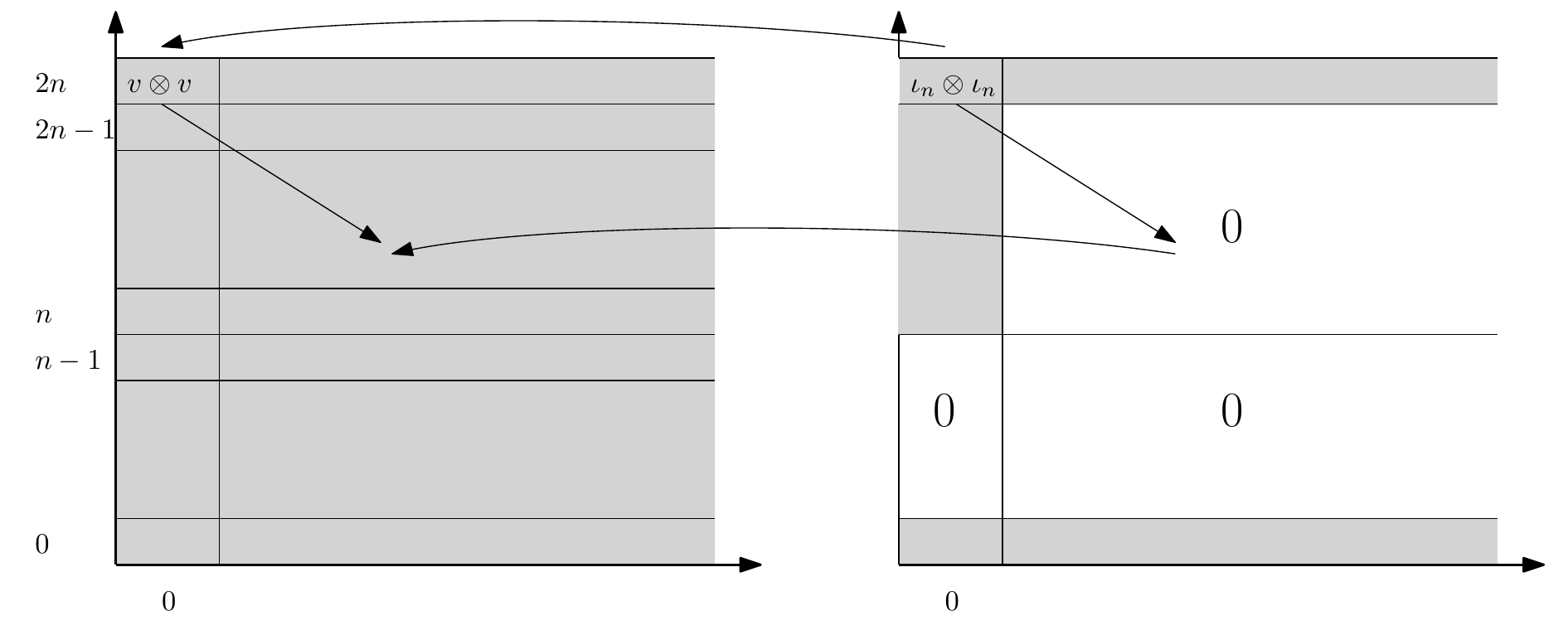}
\caption{\small The morphism between Serre spectral sequences induced by the map $\id\times_{\Z_2}(\nu\times\nu)$.}
\label{fig : ss-033}
\end{figure}
	
	Hence, we prove that the class $\iota_n\otimes\iota_n$ survives to the $E_{\infty}$-term in the Serre spectral sequences\index{Serre spectral sequence} associated to the Borel construction fibration
	\begin{equation}
		\label{eq : fibration -- 01}
	\xymatrix{
	K(\Z_2,n)\times K(\Z_2,n) \ \ar[r] & \ (K(\Z_2,n)\times K(\Z_2,n))\times_{\Z_2} \EEE\Z_2 \ \ar[r]^-{q} & \ \B\Z_2.
	}
	\end{equation}
	The $E_2$-term of this spectral sequence is of the form
	\begin{equation}
		\label{eq : fibration -- 02}
	E_2^{r,s}=H^r(\Z_2;H^s(K(\Z_2,n)\times K(\Z_2,n);\F_2)).
	\end{equation}
	Since the Eilenberg--Mac Lane space\index{Eilenberg--Mac Lane space} $K(\Z_2,n)$ is $(n-1)$-connected the K\"unneth formula \cite[Thm.\,VI.3.2]{Bredon2010} implies that
	\[
	H^s(K(\Z_2,n)\times K(\Z_2,n);\F_2)=
	\begin{cases}
		\F_2	,	&\text{for }s=0,\\
		0,		&\text{for }1\leq s\leq n-1,\\
		\text{free }\F_2[\Z_2]\text{-module},	& \text{for }n\leq s\leq 2n-1,\\
		\text{not relevant for our proof},		&\text{otherwise.}
	\end{cases}
	\]
	Here we used that for $n\leq s\leq 2n-1$ the following holds:
	\begin{multline*}
H^s(K(\Z_2,n)\times K(\Z_2,n);\F_2)=\\
\big(H^s(K(\Z_2,n);\F_2)\otimes H^0(K(\Z_2,n);\F_2)\big)\oplus \\ \big(H^0(K(\Z_2,n);\F_2)\otimes H^s(K(\Z_2,n);\F_2)\big).
	\end{multline*}
	Thus, a part of the $E_2$-term vanishes, meaning 
	\[
	E_2^{r,s}=H^r(\Z_2;H^s(K(\Z_2,n)\times K(\Z_2,n);\F_2))=0,
	\]
	for all $r\geq 1$ and $1\leq s\leq 2n-1$.
	Therefore the element  $\iota_n\otimes\iota_n\in E_2^{0,2n}$ survives to the the $E_{\infty}$-term if and only if $\partial_{2n+1}(\iota_n\otimes\iota_n)=0$.
	It suffices to prove that no non-zero differential lands in the zero row of the spectral sequence.
	
	\medskip
	The fibration \eqref{eq : fibration -- 01} has a section 
	\[
	\sigma\colon\B\Z_2\longrightarrow  (K(\Z_2,n)\times K(\Z_2,n)) \times_{\Z_2}\EEE\Z_2 
	\]
	induced by the $\Z_2$-map $\EEE\Z_2\longrightarrow (K(\Z_2,n)\times K(\Z_2,n))\times \EEE\Z_2$ given by $e\longmapsto (x_0,x_0,e)$ where $e\in\EEE\Z_2$ and $x_0\in K(\Z_2,n)$ is an arbitrary point that we fixed.
	Consequently, $q\circ\sigma=\id_{\B\Z_2}$.
	Passing to cohomology we get $\sigma^*\circ q^*=\id_{H^*(\B\Z_2;\F_2)}$ implying that 
	\[
	q^*\colon H^*(\B\Z_2;\F_2)\longrightarrow H^*((K(\Z_2,n)\times K(\Z_2,n))\times_{\Z_2}\EEE\Z_2;\F_2)
	\] 
	is a monomorphism.
	Hence, all the element of the zero row of the $E_2$-term of the spectral sequence \eqref{eq : fibration -- 02} survive to the $E_{\infty}$-term, implying that all the differentials lending in the zero row vanish.
	In particular, this means that $\partial_{2n+1}(\iota_n\otimes\iota_n)=0$, and we concluded the proof of the theorem.\
\end{proof}

\begin{corollary}
	\label{cor : E_2 collapses}
	Let $X$ be a CW-complex. 
	The Serre spectral sequence\index{Serre spectral sequence} of the fibration 
	\begin{equation*}
	\xymatrix{
		X\times X \ \ar[r] & \ (X\times X) \times_{\Z_2}S^{d-1} \ \ar[r] & \ \RP^{d-1}
	}
	\end{equation*}
	collapses at the $E_2$-term, that is, $E_{2}^{r,s}(d)\cong E_{\infty}^{r,s}(d)$ for all $(r,s)\in\Z\times\Z$.
\end{corollary}

\begin{proof}
We prove that all differentials of the spectral sequence $E_{*}^{*,*}(d)$ vanish.
In \eqref{eq : fib-022}, using $S^{\infty}$ as a model for $\EEE\Z_2$, we have defined the map 
\[
i_d\colon  (X\times X)\times_{\Z_2}S^{d-1}\longrightarrow (X\times X)\times_{\Z_2}\EEE S^{\infty} \cong  (X\times X)\times_{\Z_2}\EEE\Z_2.
\]
The map $i_d$ induces a morphism of the fiber bundles:
\[
\xymatrix{
(X\times X)\times_{\Z_2}S^{d-1} \ \ar[rr]^-{i_d}\ar[d] & & \ (X\times X)\times_{\Z_2}S^{\infty} \cong(X\times X)\times_{\Z_2}\EEE\Z_2\ar[d]\\
%& & \\
 \RP^{d-1} \ \ar[rr]& & \ \RP^{\infty}\cong \B\Z_2.
}
\]
that in turn gives a morphism between the corresponding Serre spectral sequences\index{morphism between spectral sequences}:
\[
\xymatrix@1{
E_{2}^{r,s}(d)=H^r(\RP^{d-1};\mathcal{H}^s(X\times X;\F_2)) \ & 
\ E_{2}^{r,s}(\infty)=H^r(\RP^{\infty};\mathcal{H}^s(X\times X;\F_2))\ar[l]_-{E_2^{r,s}(i_d)}}.
\]
Since by Theorem \ref{th : E_2 collapses} the spectral sequence $E_{2}^{r,s}(\infty)$ collapses at the $E_2$-term all the differential of this spectral sequence vanish. 
Consequently, all the elements in the image $\im\big(E_2^{r,s}(i_d)\big)\subseteq E_{2}^{r,s}(d)$ survive to the infinity term.
The only elements of $E_{2}^{r,s}(d)$ not contained in $\im\big(E_2^{r,s}(i_d)\big)$ belong to the $(d-1)$-column, correspond to $\F_2[\Z_2]$-free summands in the cohomology $H^*(X\times X;\F_2)$, and are of the form $(\text{something})\otimes_{\Z_2}z_{d-1}$, consult Section \ref{subsub : computation of local coefficients}.
Because all differentials emanating from the $(d-1)$-column of the spectral sequence $E_{*}^{*,*}(d)$ are zero and all differentials arriving at the $(d-1)$-column have to be zero we conclude that all differentials in $E_{*}^{*,*}(d)$ indeed vanish.
\end{proof}

\medskip
In summary, when an additive basis of the $\F_2$ vector space $H^*(X;\F_2)$ is given,
then we can describe a basis of the cohomology of the spaces $(X\times X)\times_{\Z_2}S^{d-1}$ and $ (X\times X)\times_{\Z_2}\EEE\Z_2$ as follows.
Keep in mind notation introduced in Section \ref{subsub : computation of local coefficients}.

\begin{theorem}
	\label{th : basis}
	Let $d\geq 2$ be integer or $d=\infty$, let $X$ be a CW-complex, and let $\mathcal{B}_X$ be an additive basis of $H^*(X;\F_2)$.
	Denote by $f^j$ for $0\leq j \leq d-1$ the additive generator of the group $H^j(\RP^{d-1};\F_2)\cong\F_2$, where $f\in H^1(\RP^{d-1};\F_2)$ is the multiplicative generator of the cohomology ring $H^*(\RP^{d-1};\F_2)$.
	Here we assume that $1=f^0\in H^0(\RP^{d-1};\F_2)$.
	An additive basis\index{additive basis}  $\mathcal{B}$ of the cohomology
	\[
	H^*((X\times X)\times_{\Z_2}S^{d-1};\F_2)
	\]
	can be given the following way:
	\begin{compactenum}[\rm \ (1)]
		\item If $v\in\mathcal{B}_X$, then $(v\otimes v)\otimes_{\Z_2}f^j \in \mathcal{B}$ for $0\leq j\leq d-1$ with $\deg ((v\otimes v)\otimes_{\Z_2}f^j)=2\deg(v)+\deg(f^j)=2\deg(v)+j$;
		\item If $u,v\in \mathcal{B}_X$ and $u\neq v$, then  $(u\otimes v)\otimes_{\Z_2} 1\in \mathcal{B}$ with $\deg((u\otimes v)\otimes_{\Z_2} 1)=\deg(u)+\deg(v)$; and
		\item If $u,v\in \mathcal{B}_X$ and $u\neq v$ with $d<\infty$, then we set $(u\otimes v)\otimes_{\Z_2}z_{d-1}\in \mathcal{B}$ where $\deg((u\otimes v)\otimes_{\Z_2}z_{d-1})=\deg(u)+\deg(v)+d-1$.
	\end{compactenum}	
\end{theorem}

\begin{proof}
	The spectral sequences \eqref{eq : fib-04} and \eqref{eq : fib-05}, depending whether $d<\infty$ or $d=\infty$, converge to the cohomology $H^*((X\times X)\times_{\Z_2}S^{d-1};\F_2)$.
	Since by Theorem \ref{th : E_2 collapses} and Corollary \ref{cor : E_2 collapses} both spectral sequences collapse at the $E_2$-term it suffices to find a basis of  $E_2$-terms.
	Thus, the proof is concluded by a direct application of Lemma \ref{lem : 02} and Lemma \ref{lem : 022}.
\end{proof}

\medskip
The calculation of the cohomology of  $(X\times X)\times_{\Z_2}S^{d-1}$ we presented was done with coefficients in the field $\F_2$.
The Universal Coefficient theorem\index{Universal Coefficient theorem} transcribes the arguments for cohomology into homology arguments, implying the following claim.%

\begin{theorem}
	\label{th : homology of wreath product}
	Let $d\geq 2$ be integer or $d=\infty$, let $X$ be a CW-complex, and let $\mathcal{B}_X^{'}$ be an additive basis of the homology $H_*(X;\F_2)$.
	Denote by $f_j$ for $0\leq j \leq d-1$ the generator of the group $H_j(\RP^{d-1};\F_2)\cong\F_2$.
	An additive basis\index{additive basis} $\mathcal{B}^{'}$ of the homology
	\[
	H_*((X\times X)\times_{\Z_2}S^{d-1};\F_2)
	\]
	can be given the following way:
\begin{compactenum}[\rm \ (1)]
\item If $v\in\mathcal{B}_X^{'}$, then $(v\otimes v)\otimes_{\Z_2}f_j \in \mathcal{B}^{'}$ for $1\leq j\leq d-1$ with $\deg ((v\otimes v)\otimes_{\Z_2}f_j)=2\deg(v)+\deg(f_j)=2\deg(v)+j$;
\item If $u,v\in \mathcal{B}_X^{'}$ and $u\neq v$, then  $(u\otimes v)\otimes_{\Z_2} 1\in \mathcal{B}$ with $\deg((u\otimes v)\otimes_{\Z_2} 1)=\deg(u)+\deg(v)$; and
\item If $u,v\in \mathcal{B}_X^{'}$ and $u\neq v$ with $d<\infty$, then we set $(u\otimes v)\otimes_{\Z_2}h_{d-1}\in \mathcal{B}^{'}$ where $\deg((u\otimes v)\otimes_{\Z_2}h_{d-1})=\deg(u)+\deg(v)+d-1$.
\end{compactenum}	
\end{theorem}

\medskip
A useful consequence of Theorem \ref{th : basis} is the following fact.

\begin{corollary}
	\label{cor : inj implies inj}
	Let $d\geq 2$ be integer or $d=\infty$, let $X$ and $Y$ be CW-complexes, and let $f\colon X\longrightarrow Y$ be a continuous map.
	If the induced map in cohomology $f^*\colon H^*(Y;\F_2)\longrightarrow H^*(X;\F_2)$ is injective, then the induced map
	\[
	((f\times f)\times_{\Z_2}\id)^*\colon H^*((Y\times Y)\times_{\Z_2}S^{d-1};\F_2)\longrightarrow H^*((X\times X)\times_{\Z_2}S^{d-1};\F_2),
	\]
	is also injective.
\end{corollary}
\begin{proof}
	The induced map $(f\times f)\times_{\Z_2}\id\colon (X\times X)\times_{\Z_2}S^{d-1} \longrightarrow (Y\times Y)\times_{\Z_2}S^{d-1}$ is covering the identity map in the following bundle morphism:
	\[
	\xymatrix{
	(X\times X)\times_{\Z_2}S^{d-1} \ \ar[rr]^-{(f\times f)\times_{\Z_2}\id}\ar[d]  & & \ (Y\times Y)\times_{\Z_2}S^{d-1}\ar[d]\\
%	 & & \\
	S^{d-1}/\Z_2 \ \ar[rr]^-{\id} & & \ S^{d-1}/\Z_2.
	}
	\]
	The bundle morphism induces the morphism between the corresponding Serre spectral sequences\index{morphism between spectral sequences} that corresponds to the homomorphism $((f\times f)\times_{\Z_2}\id)^*$.
	Since both spectral sequences collapse at $E_2$-term and there is no extension problem Theorem \ref{th : basis} in combination with the assumption that $f^*$ is injective yields the  injectivity of $((f\times f)\times_{\Z_2}\id)^*$.
	This concludes the proof.
\end{proof}

\medskip
The following dual version of the previous corollary also holds.  
\begin{corollary}
	\label{cor : suj implies sur}
	Let $d\geq 2$ be integer or $d=\infty$, let $X$ and $Y$ be CW-complexes, and let $f\colon X\longrightarrow Y$ be a continuous map.
	If the induced map in homology 
	\[
	f_*\colon H_*(X;\F_2)\longrightarrow H_*(Y;\F_2)
	\]
	is surjective, then the induced map
	\[
	((f\times f)\times_{\Z_2}\id)_*\colon H^*((X\times X)\times_{\Z_2}S^{d-1};\F_2)\longrightarrow H^*((Y\times Y)\times_{\Z_2}S^{d-1};\F_2),
	\]
	is also surjective.
\end{corollary}

%====
\subsection{The induction step}
%====
In this section we take $d\geq 2$ to be an integer or $d=\infty$.
Let us assume that for $m=k$ the cohomology 
\[
H^*(\Sp(\R^d,2^m)/\Sy_{2^m};\F_2)=H^*(\Sp(\R^d,2^k)/\Sy_{2^k};\F_2)
\] 
is determined by specifying a basis $\basis(\R^d,2^k)$ and its associated partition $\basis(\R^d,2^k)=\basisa(\R^d,2^k)\cup\basisi(\R^d,2^k)$ in such a way that 
\[
\AAA^*(\R^d,2^k):=\langle \basisa(\R^d,2^k)\rangle\cong  \F_2[V_{k,1},\ldots,V_{k,k}]/\langle V_{k,1}^d,\ldots, V_{k,k}^d\rangle
\]
is a subalgebra and $\III^*(\R^d,2^k):=\langle \basisi(\R^d,2^k)\rangle$ is an ideal of $H^*(\Sp(\R^d,2^k)/\Sy_{2^k};\F_2)$, and in addition  
\begin{align}
\label{eq : decomposition for n=2^k}	
H^*(\Sp(\R^d,2^k)/\Sy_{2^k};\F_2) &\cong \AAA^*(\R^d,2^k)\oplus \III^*(\R^d,2^k)\nonumber \\
&\cong \F_2[V_{k,1},\ldots,V_{k,k}]/\langle V_{k,1}^d,\ldots, V_{k,k}^d\rangle \oplus \III^*(\R^d,2^k),	
\end{align} 
where $\deg(V_{k,r})=2^{r-1}$ for $1\leq r\leq k$.

\medskip
Now, for $m=k+1$ we study the cohomology 
\[
 H^*(\Sp(\R^d,2^{k+1})/\Sy_{2^{k+1}};\F_2).
\]
According to Definition \ref{def : Epicycles} we have that 
\[
	\Sp(\R^d,2^{k+1})=(\Sp(\R^d,2^{k})\times \Sp(\R^d,2^{k}))\times S^{d-1} 
	\quad\text{and}\quad
	\Sy_{2^{k+1}}= (\Sy_{2^k}\times \Sy_{2^k})\rtimes \Z_2.
\]
Using the nature of the $\Sy_{2^{k+1}}$-action on the spaces $\Sp(\R^d,2^{k+1})$ we have that
\begin{align*}
\Sp(\R^d,2^{k+1})/\Sy_{2^{k+1}} &\cong	 \big((\Sp(\R^d,2^{k})\times \Sp(\R^d,2^{k}))\times S^{d-1}
\big)/\Sy_{2^{k+1}}\\
&\cong \big((\Sp(\R^d,2^{k})/\Sy_{2^{k}}\times \Sp(\R^d,2^{k})/\Sy_{2^{k}})\times S^{d-1} \big)/\Z_2\\
&=: \big(\Sp(\R^d,2^{k})/\Sy_{2^{k}}\times \Sp(\R^d,2^{k})/\Sy_{2^{k}} \big)\times_{\Z_2}S^{d-1}.
\end{align*}

\medskip
Since the additive basis\index{additive basis} $\basis(\R^d,2^k)$ for the cohomology 
$
H^*(\Sp(\R^d,2^k)/\Sy_{2^k};\F_2)
$ 
is already fixed we can now use Theorem \ref{th : basis} and get the basis $\basis(\R^d,2^{k+1})$ for the cohomology
\[
H^*(\Sp(\R^d,2^{k+1})/\Sy_{2^{k+1}};\F_2)
\]
as follows:
\begin{compactenum}[\rm \ (i)]
		\item If $v\in \basis(\R^d,2^k)$, then for all $0\leq j\leq d-1$
		\[
		 (v\otimes v)\otimes_{\Z_2}f^j \in\basis(\R^d,2^{k+1}),
		\]
		with $\deg ((v\otimes v)\otimes_{\Z_2}f^j)=2\deg(v)+\deg(f^j)=2\deg(v)+j$;
		\item If $u,v\in \basis(\R^d,2^k)$ and $u\neq v$, then  
		\[
		(u\otimes v)\otimes_{\Z_2} 1\in\basis(\R^d,2^{k+1}),
		\] 
		with $\deg((u\otimes v)\otimes_{\Z_2} 1)=\deg(u)+\deg(v)$; and
		\item If $u,v\in \basis(\R^d,2^k)$ and $u\neq v$ and $d<\infty$, then 
		\[
		(u\otimes v)\otimes_{\Z_2}z_{d-1}\in \basis(\R^d,2^{k+1}),
		\]
		with $\deg((u\otimes v)\otimes_{\Z_2}z_{d-1})=\deg(u)+\deg(v)+d-1$.
	\end{compactenum}
The partition of the basis\index{additive basis} $\basis(\R^d,2^{k+1})$ just introduced is given by
\[
 (v\otimes v)\otimes_{\Z_2}f^j \in\basisa(\R^d,2^{k+1}),
\]
for $v\in \basisa(\R^d,2^k)$ and $0\leq j\leq d-1$, that is,
\[
\basisa(\R^d,2^{k+1}):=\{ (v\otimes v)\otimes_{\Z_2}f^j	: v\in \basisa(\R^d,2^k), \ 0\leq j\leq d-1\}.
\]
Then we set $\basisi(\R^d,2^{k+1})=\basis(\R^d,2^{k+1}){\setminus}\basisa(\R^d,2^{k+1})$.
As before, we denote by 
\[
\AAA^*(\R^d,2^{k+1}):=\langle \basisa(\R^d,2^{k+1})\rangle
\qquad\text{and}\qquad 
\III^*(\R^d,2^{k+1}):=\langle \basisi(\R^d,2^{k+1})\rangle.
\]
Then we have the following additive decomposition of the cohomology
\[
H^*(\Sp(\R^d,2^{k+1})/\Sy_{2^{k+1}};\F_2) \cong \AAA^*(\R^d,2^{k+1})\oplus \III^*(\R^d,2^{k+1}).
\]

\begin{lemma}
\label{lem : computation of algebra}
	Let $d\geq 2$ be an integer or $d=\infty$.
	Then $\AAA^*(\R^d,2^{k+1})$ is a subalgebra and $ \III^*(\R^d,2^{k+1})$ is an ideal of the cohomology algebra\index{cohomology ring} $H^*(\Sp(\R^d,2^{k+1})/\Sy_{2^{k+1}};\F_2)$.
	Moreover, 
\begin{equation}
		\label{eq : iso--induction--step}
		\AAA^*(\R^d,2^{k+1})\cong \F_2[V_{k+1,1},\ldots,V_{k+1,k+1}]/\langle V_{k+1,1}^d,\ldots, V_{k+1,k+1}^d\rangle,
\end{equation}
and $\deg(V_{k+1,r})=2^{r-1}$ for $1\leq r\leq k+1$.	
\end{lemma}
\begin{proof}
	Since the description of the cohomology $H^*(\Sp(\R^d,2^{k+1})/\Sy_{2^{k+1}};\F_2)$ is derived from a spectral sequence with appropriate multiplication structure it follows directly that  $\AAA^*(\R^d,2^{k+1})$ is a subalgebra, and $ \III^*(\R^d,2^{k+1})$ is an ideal, consult Theorem \ref{th : basis}.
	It remains to establish the isomorphism \eqref{eq : iso--induction--step}.

Let us set 
	\begin{equation}
		\label{eq : definition of V's - 01}
		V_{k+1,1}:=(1\otimes 1)\otimes_{\Z_2}f,
	\end{equation}
	and for all $2\leq r\leq k+1$ let 
	\begin{equation}
		\label{eq : definition of V's - 02}
	V_{k+1,r}:=(V_{k,r-1} \otimes V_{k,r-1}) \otimes_{\Z_2}1.
	\end{equation}
	Then 
	\[
	\deg(V_{k+1,1})=2\deg(1)+\deg(f)=2\cdot 0 +1=1,
	\]
	and
	\[
	\deg(V_{k+1,r})=2\deg(V_{k,r-1})+\deg(1)=2\cdot 2^{r-2}+0=2^{r-1}
	\] 
	for $2\leq r\leq k+1$.
	Therefore, from the construction of the set $\basisa(\R^d,2^k)$ and the assumption about the structure of the subalgebra 
	\[
	\AAA^*(\R^d,2^k)\cong  \F_2[V_{k,1},\ldots,V_{k,k}]/\langle V_{k,1}^d,\ldots, V_{k,k}^d\rangle,
	\] 
	directly follows that the isomorphism \eqref{eq : iso--induction--step} holds.
\end{proof}
The calculations of this section establish the following theorem.

\begin{theorem}
\label{th: cohomology of Sp}
Let $d\geq 2$ be an integer or $d=\infty$, and let $m\geq 0$ be an integer.
Then\index{cohomology ring}
\begin{multline}
	\label{eq : cohomology of Sp}
	H^*(\Sp(\R^d,2^m)/\Sy_{2^m};\F_2)
	\cong  \\ \F_2[V_{m,1},\ldots,V_{m,m}]/\langle V_{m,1}^d,\ldots, V_{m,m}^d\rangle \oplus \III^*(\R^d,2^m),
\end{multline}
where $\III^*(\R^d,2^m)$ is an ideal, and $\deg(V_{m,r})=2^{r-1}$ for $1\leq r\leq m$.
In particular, for $d=\infty$ we have
\begin{equation}
	\label{eq : cohomology of Sp d=infty}
	H^*(\Sp(\R^{\infty},2^m)/\Sy_{2^m};\F_2)
	\cong   \F_2[V_{m,1},\ldots,V_{m,m}]  \oplus \III^*(\R^{\infty},2^m).
\end{equation}
\end{theorem}

%====
\subsection{The restriction homomorphisms\index{restriction homomorphism}}
\label{subsec : the restrictions}
%====
Let $m\geq 0$ be an integer.
Consider the sequence of inclusions 
\[
\xymatrix{
\EE_m \ \ar[r] & \ \Sy_{2^m} \ \ar[r]^{\iota_{2^m}} & \ \Sym_{2^m} \ \ar[r] & \ \OO(2^m)
}
\]
where the last inclusion is the embedding give via the permutation representation.
The corresponding sequence of maps between classifying spaces\index{classifying space} 
\[
\xymatrix{
\B\EE_m \ \ar[r] &\ \B\Sy_{2^m} \ \ar[r] & \ \B\Sym_{2^m} \ \ar[r] & \ \B\OO(2^m)
}
\]
induces the following sequence of restriction homomorphisms:
\[
\xymatrix@1{
H^*(\OO(2^m);\F_2) \ \ar[r]^-{\res^{\OO(2^m)}_{\Sym_{2^m}}}& 
\ H^*(\Sym_{2^m};\F_2) \ \ar[r]^-{\res^{\Sym_{2^m}}_{\Sy_{2^m}}}  & 
\ H^*(\Sy_{2^m};\F_2) \ \ar[r]^-{\res^{\Sy_{2^m}}_{\EE_m}} &  
\ H^*(\EE_m;\F_2).
}
\]
In this section we study various aspects of these restriction homomorphisms\index{restriction homomorphism}.

%------------------------
\subsubsection{}
%------------------------

For $d=\infty$ the isomorphism \eqref{eq : cohomology of Sp} gives the following decomposition of the cohomology of the group $\Sy_{2^m}$:
\begin{align}
\label{eq : cohomology of Sp - inf}
H^*(\Sy_{2^m};\F_2)&\cong H^*(\Sp(\R^{\infty},2^m)/\Sy_{2^m};\F_2)\\
	&\cong   \F_2[V_{m,1},\ldots,V_{m,m}]\oplus \III^*(\R^{\infty},2^m),\nonumber
\end{align}
where $\deg(V_{m,r})=2^{r-1}$ for $1\leq r\leq m$.

\medskip
Recall that in Definition \ref{def : el.ab.subgroup} we have specified the elementary abelian group $\EE_m\cong \Z_2^{\oplus m}$ as a subgroup of $\Sy_{2^m}$.
First we study the restriction map 
\[
\res^{\Sy_{2^m}}_{\EE_m}\colon H^*(\Sy_{2^m};\F_2)\longrightarrow H^*(\EE_m;\F_2).
\]
From the definition of the basis $\basis(\R^d,2^{m})$, its partition into subsetes $\basisa(\R^d,2^{m})$ and $\basisi(\R^d,2^{m})$, and the definition of the element $V_{m,1}$ follows that 
\[
V_{m,1}\cdot I^*(\R^{\infty},2^m)=0.
\]
In addition, if we recall how we introduced the subgroup $\EE_m$ of $\Sy_{2^m}$ (see Definition \ref{def : el.ab.subgroup}), and observe that multiplication by $\res^{\Sy_{2^m}}_{\EE_m}(V_{m,1})$ in $H^*(\EE_m;\F_2)$ is injective, we can conclude that
\begin{equation}
	\label{eq : restriction of ideal vanishes}
	\res^{\Sy_{2^m}}_{\EE_m}(\III^*(\R^{\infty},2^m))=0 
	\qquad\Longleftrightarrow\qquad 
	\III^*(\R^{\infty},2^m)\subseteq \ker (\res^{\Sy_{2^m}}_{\EE_m}).
\end{equation}

\medskip
Now, we shift our interest to the generators $V_{m,1},\ldots,V_{m,m}$ of the polynomial subalgebra in the decomposition \eqref{eq : cohomology of Sp - inf} and will identify its images under the restriction $\res^{\Sy_{2^m}}_{\EE_m}(V_{m,1}),\ldots,\res^{\Sy_{2^m}}_{\EE_m}(V_{m,m})$.
Even the definitions of the group $\Sy_{2^m}$ and its subgroup $\EE_m$ were inductive our approach to the description of the image of $\im (\res^{\Sy_{2^m}}_{\EE_m})$ is not be inductive.
For that we follow \cite[p.\,266]{Hung1990} and first note that according to Lemma \ref{lem : image of restriction} the restriction image is contained in the ring of invariants of the corresponding Weyl group\index{Weyl group} 
\begin{multline*}
\im \big( \res^{\Sy_{2^m}}_{\EE_m}\colon H^*(\Sy_{2^m};\F_2)\longrightarrow H^*(\EE_m;\F_2) \big)
\subseteq 
H^*(\EE_m;\F_2)^{W_{\Sy_{2^m}}(\EE_m)} \\ =\F_2[y_1,\ldots,y_m]^{W_{\Sy_{2^m}}(\EE_m)}.	
\end{multline*}

\medskip
From Lemma \ref{lem : Weyl group of E_m} we know that $W_{\Sy_{2^m}}(\EE_m)=\U_m(\F_2)$ is the Sylow $2$-subgroup of $\GL_m(\F_2)$ of all lower triangular matrices with $1$'s on the main diagonal.
Therefore all polynomials
\[
 v_{m,1}:= \res^{\Sy_{2^m}}_{\EE_m}(V_{m,1}), \ \ldots, \ v_{m,m}:=\res^{\Sy_{2^m}}_{\EE_m}(V_{m,m})
\]
are $\U_m(\F_2)$ invariant polynomials.
Each polynomial $v_{m,r}$ has $y_{m-r+1}$ as a factor by \eqref{eq : definition of V's - 01} and \eqref{eq : definition of V's - 02}, $\deg (v_{m,r})=2^{r-1}$, and $v_{m,r}$ is $\U_m(\F_2)$ invariant, consequently
\begin{equation}
	\label{eq : formula for v_{m,r}}
	v_{m,r}=\prod_{(\lambda_m,\ldots,\lambda_{m-r+2})\in\F_2^{r-1}}
\big( 
\lambda_m\,y_m+\cdots+\lambda_{m-r+2}\,y_{m-r+2}+y_{m-r+1}
\big),
\end{equation}
as stated in \cite[(2.14)]{Hung1982}.
In particular, the polynomials $v_{m,1},\ldots,v_{m,m}$ are algebraically independent.
In the notation of Theorem \ref{th : Mui ring of invariants of upper triangular matrices} we have that $v_{m,r}=h_r$, with identifications $m=n$ and $y_1=x_1,\ldots,y_m=x_m$.

%------------------------
\subsubsection{}
%------------------------

Let us now consider the restriction homomorphism\index{restriction homomorphism}
\[
\res^{\Sym_{2^m}}_{\EE_m}\colon H^*(\Sym_{2^m};\F_2)\longrightarrow H^*(\EE_m;\F_2).
\]
Like in the previous case using Lemma \ref{lem : image of restriction} we get that
\begin{multline*}
\im \big( \res^{\Sym_{2^m}}_{\EE_m}\colon H^*(\Sym_{2^m};\F_2)\longrightarrow H^*(\EE_m;\F_2) \big)
\subseteq
H^*(\EE_m;\F_2)^{W_{\Sym_{2^m}}(\EE_m)}\\ =\F_2[y_1,\ldots,y_m]^{W_{\Sym_{2^m}}(\EE_m)}.	
\end{multline*}
From Lemma \ref{lem : Weyl group of E_m - 2} we get that $W_{\Sym_{2^m}}(\EE_m)\cong\GL_m(\F_2)$, and consequently
\begin{multline*}
\im \big( \res^{\Sym_{2^m}}_{\EE_m}\colon H^*(\Sym_{2^m};\F_2)\longrightarrow H^*(\EE_m;\F_2) \big)
\subseteq
H^*(\EE_m;\F_2)^{W_{\Sym_{2^m}}(\EE_m)}\\ =\F_2[y_1,\ldots,y_m]^{\GL_m(\F_2)}.
\end{multline*}
Now the description of the ring of invariants given in Theorem \ref{th : Dickson ring of invariants} yields that
\begin{multline}
	\label{eq : about restriction - 01}
	\im \big( \res^{\Sym_{2^m}}_{\EE_m}\colon H^*(\Sym_{2^m};\F_2)\longrightarrow H^*(\EE_m;\F_2) \big)
\subseteq
H^*(\EE_m;\F_2)^{W_{\Sym_{2^m}}(\EE_m)}\\ =\F_2[y_1,\ldots,y_m]^{\GL_m(\F_2)}=\F_2[d_{m,0},\ldots,d_{m,m-1}],
\end{multline}
where $d_{m,0},\ldots,d_{m,m-1}$ are the Dickson invariants\index{Dickson invariants}.

\medskip
Next we recall that
\[
H^*(\B\OO(2^m);\F_2)=\F[w_1,\ldots,w_{2^m}],
\]
where $w_i$, for $1\leq i\leq 2^m$, denotes the $i$th Stiefel--Whitney class\index{Stiefel--Whitney classes} of the tautological vector bundle\index{vector bundle} $\gamma_{2^m}$ over $\B\OO(2^m)$, see \cite[Thm.\,7.1]{Milnor1974}. 
Let us introduce the following notation
\[
\xymatrix{
w_{2^m-2^r} \ \ar@{|->}[rr]^-{\res^{\OO(2^m)}_{\Sym_{2^m}}} & & \ w_{m,r} \ \ar@{|->}[rr]^-{\res^{\Sym_{2^m}}_{\Sy_{2^m}}} & & \ D_{m,r} \ \ar@{|->}[rr]^-{\res^{\Sy_{2^m}}_{\EE_m}} & & \ d_{m,r},
}
\]
where $0\leq r\leq m-1$.
From Theorem \ref{thm : SW classes of regular rep. over e.a.g.} we have that indeed the classes $d_{m,0},\ldots,d_{m,m-1}$ are Dickson invariants\index{Dickson invariants}, and furthermore  
	\[
	\res^{\OO(2^m)}_{\EE_m}(w_i)=
	\begin{cases}
		d_{m,r}=\res^{\Sym_{2^m}}_{\EE_m}(w_{m,r}), & i=2^m-2^r, \ 0\leq r\leq m-1, \\
		   1 ,   & i=0,\\
		   0 ,   & \text{otherwise}.   
	\end{cases}
	\]
Hence, from \eqref{eq : about restriction - 01} we conclude that 
\begin{multline*}
	\im \big( \res^{\Sym_{2^m}}_{\EE_m}\colon H^*(\Sym_{2^m};\F_2)\longrightarrow H^*(\EE_m;\F_2) \big)
=
H^*(\EE_m;\F_2)^{W_{\Sym_{2^m}}(\EE_m)} \\=\F_2[y_1,\ldots,y_m]^{\GL_m(\F_2)}=\F_2[d_{m,0},\ldots,d_{m,m-1}].
\end{multline*}
To summarise,  we have proved the following lemma.

\begin{lemma}
	Let $m\geq 0$ be an integer.
	Then
\begin{multline}
	\label{eq : about restriction - 02}
	\im \big( \res^{\Sym_{2^m}}_{\EE_m}\colon H^*(\Sym_{2^m};\F_2)\longrightarrow H^*(\EE_m;\F_2) \big)
=\F_2[y_1,\ldots,y_m]^{\GL_m(\F_2)} \\ =\F_2[d_{m,0},\ldots,d_{m,m-1}],
\end{multline}
where $d_{m,0},\ldots,d_{m,m-1}$ are the Dickson invariants.
Consequently, 
\begin{align}
	\label{eq : about restriction - 03}
	 H^*(\Sym_{2^m};\F_2) 
	 &\cong\im \big( \res^{\Sym_{2^m}}_{\EE_m}\colon H^*(\Sym_{2^m};\F_2)\longrightarrow H^*(\EE_m;\F_2) \big)\oplus  \ker( \res^{\Sym_{2^m}}_{\EE_m}) \\
	&\cong\F_2[d_{m,0},\ldots,d_{m,m-1}]\oplus \ker( \res^{\Sym_{2^m}}_{\EE_m})  \nonumber\\
	& \cong\F_2[w_{m,0},\ldots,w_{m,m-1}]\oplus \ker( \res^{\Sym_{2^m}}_{\EE_m}). \nonumber 
\end{align}
\end{lemma}

Let us reflect on the facts we obtained so far.
For $0\leq r\leq m-1$ we specified the following elements under the restriction maps:
\begin{equation}\label{eq : summary}
\xymatrix@C=1.85em{
H^*(\OO(2^m);\F_2)\ \ar[r]^-{\res^{\OO(2^m)}_{\Sym_{2^m}}}& 
\ H^*(\Sym_{2^m};\F_2)\ \ar[r]^-{\res^{\Sym_{2^m}}_{\Sy_{2^m}}}  &  
\ H^*(\Sy_{2^m};\F_2) \ \ar[r]^-{\res^{\Sy_{2^m}}_{\EE_m}} &  
\ H^*(\EE_m;\F_2)\\ 
w_{2^m-2^r} \ \ar@{|->}[r]^-{\res^{\OO(2^m)}_{\Sym_{2^m}}} & 
\ w_{m,r} \ \ar@{|->}[r]^-{\res^{\Sym_{2^m}}_{\Sy_{2^m}}} &  
\ D_{m,r} \ \ar@{|->}[r]^-{\res^{\Sy_{2^m}}_{\EE_m}}  & 
\ d_{m,r}\\
 &  &      V_{m,r+1} \ \ar@{|->}[r]^-{\res^{\Sy_{2^m}}_{\EE_m}}  & \ v_{m,r+1}.
}	
\end{equation}

%%%%%%%%%%%%%%%%%%%%%%%%%%%%%%%%%%%%%%%%%%%%%%%%%%%%%%%%%%%%%%%%%%%%%%%%%%
\begin{figure}[p]
   \rotatebox{90}{%
     \begin{minipage}{\textheight}%
\begin{equation}\label{eq : bigsummary-02}\hspace{-15mm}
\xymatrix{
H^*(\OO(2^m);\F_2)\ \ar[r]^-{\res^{\OO(2^m)}_{\Sym_{2^m}}}\ar@{<->}[d]^-{\cong}& 
\ H^*(\Sym_{2^m};\F_2) \ \ar[r]^-{\res^{\Sym_{2^m}}_{\Sy_{2^m}}}\ar@{<->}[d]^-{\cong}  & \ H^*(\Sy_{2^m};\F_2) \ \ar[r]^-{\res^{\Sy_{2^m}}_{\EE_m}}\ar@{<->}[d]^-{\cong} & \ H^*(\EE_m;\F_2)\ar@{<->}[d]^-{\cong}\\ 
\F_2[w_1,\dots,w_{2^m}]\ \ar[r]\ar@{->>}[ddddrrr] & \  \F_2[w_{m,0},\dots,w_{m,m-1}]\oplus \ker(\res^{\Sym_{2^m}}_{\EE_m})\ar[r]\ar@{->>}[ddddrr]  & \ \F_2[V_{m,1},\dots,V_{m,m}]\oplus I^*(\R^{\infty},2^m) \ \ar[r]\ar@{->>}[ddr] &\  \F_2[y_1,\dots,y_m]\\
& & & \\
& & & H^*(\EE_m;\F_2)^{\U_m(\F_2)}\cong\F_2[v_{m,1},\dots,v_{m,m}]\ar@{^{(}->}[uu]\\
& & & \\
& & &
 H^*(\EE_m;\F_2)^{\GL_m(\F_2)}\cong \F_2[d_{m,0},\ldots,d_{m,m-1}]\ar@{^{(}->}[uu]\\~
}
\end{equation}
     \end{minipage}%
  }%
  % No caption, then the figure number is not changed
\end{figure}
%%%%%%%%%%%%%%%%%%%%%%%%%%%%%%%%%%%%%%%%%%%%%%%%%%%%%%%%%%%%%%%%%%%%%%%%%%

Furthermore, we have proved the factorizations of the restriction homomorphism\index{restriction homomorphism} which are described by the diagram \eqref{eq : bigsummary-02} 
on the next page. 
In other words,
\begin{compactitem}[ \ ---]
\item the restriction homomorphism $\res^{\OO(2^m)}_{\EE_m}$ factors as follows:
\end{compactitem}
\[
\xymatrix@1{
H^*(\OO(2^m);\F_2) \ \ar@{<->}[d]^-{\cong} \ar@{->>}[r]& \ H^*(\EE_m;\F_2)^{\GL_m(\F_2)} \ \ar@{<->}[d]^-{\cong} \ \ar@{^{(}->}[r] & \ H^*(\EE_m;\F_2)\ar@{<->}[d]^-{\cong}\\ 
\F_2[w_1,\dots,w_{2^m}] \ \ar@{->>}[r] & \ \F_2[d_{m,0},\dots,d_{m,m-1}] \ \ \ar@{^{(}->}[r] & \ \F[y_1,\dots,y_m],
}
\]

\begin{compactitem}[ \ ---]
\item the restriction homomorphism\index{restriction homomorphism} $\res^{\Sym_{2^m}}_{\EE_m}$ factors as follows
\end{compactitem} 
\[
\xymatrix@C=0.81em{
H^*(\Sym_{2^m};\F_2)\ar@{<->}[d]^-{\cong} \ \ar@{->>}[r]& H^*(\EE_m;\F_2)^{\GL_m(\F_2)}\ar@{<->}[d]^-{\cong} \ \ \ar@{^{(}->}[r] &\ H^*(\EE_m;\F_2)\ar@{<->}[d]^-{\cong}  \\ 
\F_2[w_{m,0},\dots,w_{m,m-1}]\oplus\ker(\res^{\Sym_{2^m}}_{\EE_m}) \ \ar@{->>}[r] & \ \F_2[d_{m,0},\dots,d_{m,m-1}] \ \ \ar@{^{(}->}[r] & \ \F[y_1,\dots,y_m], 
}
\]

\begin{compactitem}[ \ ---]
\item the restriction homomorphism $\res^{\Sy_{2^m}}_{\EE_m}$ factors as follows:
\end{compactitem} 
\[
\xymatrix@1{
H^*(\Sy_{2^m};\F_2)\ar@{<->}[d]^-{\cong} \ \ar@{->>}[r]& \ H^*(\EE_m;\F_2)^{\U_m(\F_2)}\ar@{<->}[d]^-{\cong} \ \ \ar@{^{(}->}[r] & \ H^*(\EE_m;\F_2)\ar@{<->}[d]^-{\cong}\\ 
\F_2[V_{m,1},\dots,V_{m,m}]\oplus I^*(\R^{\infty},2^m) \ \ar@{->>}[r] & \ \F_2[v_{m,1},\dots,v_{m,m}] \ \ \ar@{^{(}->}[r] & \ \F[y_1,\dots,y_m].
}
\]

\medskip
Now from Proposition \ref{prop : Dickson invariants recursive formula} we get a connection between the Dickson invariant\index{Dickson invariants} polynomials $d_{m,0},\ldots,d_{m,m-1}$ and the $\U_m(\F_2)$-invariant polynomials $v_{m,1},\ldots,v_{m,m}$.
For $0\leq r \leq m-1$ and $d_{m-1,-1}=0$ holds:
\begin{equation}
	\label{eq : recurrence relations d and v}
	d_{m,r}=(\chi_m d_{m-1,r})\,v_{m,m} + (\chi_m d_{m-1,r-1})^2.
\end{equation}
Here $\chi_m\in\GL_m(\F_2)$ is the variable change given by the matrix
\[
\begin{pmatrix}
	0 & 0 & \cdots & 0 & 1\\
	0 & 0 & \cdots & 1 & 0\\
	  &   & \cdots &   &  \\
	0 & 1 & \cdots & 0 & 0 \\
	1 & 0 & \cdots & 0 & 0
\end{pmatrix} .
\]

%------------------------
\subsubsection{}
%------------------------

In the final part of this section we prove the following two lemmas.
For the next lemma see also \cite[Lem.\,3.14]{Hung1990}.

\begin{lemma}
	\label{lem : image of res on w_{m,0}}
	Let $m\geq 0$ be an integer.
	Then
	\begin{equation}
		\label{eq : from lemma 3.14}
		\res^{\Sym_{2^m}}_{\Sy_{2^m}}\big(\langle w_{m,0}\rangle\big)\subseteq \F_2[V_{m,1},\ldots,V_{m,m}],
	\end{equation}
	where $\langle w_{m,0}\rangle$ denotes the principal ideal generated by the class $w_{m,0}$ in $H^*(\Sym_{2^m};\F_2)$.
\end{lemma}
\begin{proof}
	Recall that in \eqref{eq : cohomology of Sp - inf} we concluded that
\[
H^*(\Sy_{2^m};\F_2) \cong \F_2[V_{m,1},\ldots,V_{m,m}]\oplus \III^*(\R^{\infty},2^m).
\]
In order to prove \eqref{eq : from lemma 3.14} it suffices to show that
\begin{equation}
	\label{eq : relation - 10}
	D_{m,0}\cdot  \III^*(\R^{\infty},2^m)=0,
	\qquad\text{and}\qquad
	D_{m,0}= V_{m,1}\cdots V_{m,m}.
\end{equation}
Indeed, if the equalities \eqref{eq : relation - 10} hold, and because $\res^{\Sym_{2^m}}_{\Sy_{2^m}}(w_{m,0})= D_{m,0}$, we have that
\begin{multline*}
\res^{\Sym_{2^m}}_{\Sy_{2^m}}\big(\langle w_{m,0}\rangle\big)=\res^{\Sym_{2^m}}_{\Sy_{2^m}}\big( w_{m,0} \cdot H^*(\Sym_{2^m};\F_2)\big) \\ \subseteq D_{m,0}\cdot H^*(\Sy_{2^m};\F_2) = 
D_{m,0}\cdot \big(\F_2[V_{m,1},\ldots,V_{m,m}]\oplus \III^*(\R^{\infty},2^m)\big)\\
=D_{m,0}\cdot \F_2[V_{m,1},\ldots,V_{m,m}]\subseteq \F_2[V_{m,1},\ldots,V_{m,m}].
\end{multline*}
Thus, in order to finish the proof of the lemma it remains to verify equalities \eqref{eq : relation - 10}.
 
\medskip
First we verify that $D_{m,0}\cdot  \III^*(\R^{\infty},2^m)=0$.
For that we use a classical result of Quillen \cite{Quillen1971} about detection of group cohomology, see Section \ref{sec : detection}.
In particular, according to Theorem \ref{thm : symmetric group  p^n detected} we have that the cohomology $H^*(\Sy_{2^m};\F_2)$ of the Sylow $2$-subgroup $\Sy_{2^m}$ modulo $\F_2$ is detected by the subgroups $\EE_m$ and $\Sy_{2^{m-1}}\times \Sy_{2^{m-1}}$.
From \eqref{eq : restriction of ideal vanishes} we have that $\res^{\Sy_{2^m}}_{\EE_m}(\III^*(\R^{\infty},2^m))=0$.
Consequently, if we prove that  $\res^{\Sy_{2^m}}_{\Sy_{2^{m-1}}\times \Sy_{2^{m-1}}}(D_{m,0})=0$ it would follow that $D_{m,0}\cdot  \III^*(\R^{\infty},2^m)=0$.
To see that this restriction of $D_{m,0}$ vanishes we first recall that $D_{m,0}$ is a $(2^m-1)$-Stiefel--Whitney class\index{Stiefel--Whitney classes} of the vector bundle $\eta_{2^m}$:
\[
\xymatrix{
\R^{2^m}\ \ar[r] &\  \EEE\Sy_{2^m}\times_{\Sy_{2^m}}\R^{2^m} \ \ar[r] & \ \B \Sy_{2^m},
}
\]
see Section \ref{subsub : Dickson invariants as characteristic classes}.
The vector bundle\index{vector bundle} $\eta_{2^m}$ can be decomposed into a Whitney sum of two vector bundles where one of them is a trivial line bundle. 
Consequently, the $2^m$-Stiefel--Whitney class of the bundle vanishes. 
The trivial line subbundle is determined by the trivial $\Sy_{2^m}$ subrepresentation $\{(x_1,\ldots,x_{2^m})\in \R^{2^m} : x_1=\cdots =x_{2^m}\}$ of $\R^{2^m}$.
Using the naturality property of  Stiefel--Whitney classes \cite[Ax.\,2,\,p.\,35]{Milnor1974} we have that $\res^{\Sy_{2^m}}_{\Sy_{2^{m-1}}\times \Sy_{2^{m-1}}}(D_{m,0})$ is the $(2^m-1)$-Stiefel--Whitney class of the pull-back vector bundle:
\[
\xymatrix{
\EEE (\Sy_{2^{m-1}}\times \Sy_{2^{m-1}})\times_{(\Sy_{2^{m-1}}\times \Sy_{2^{m-1}})}\R^{2^m}\ \ar[rr]\ar[d]_{\omega_{2^m}} & &\  \EEE \Sy_{2^m}\times_{\Sy_{2^m}}\R^{2^m}\ar[d]_{\eta_{2^m}} \\
%& &\\
\B(\Sy_{2^{m-1}}\times \Sy_{2^{m-1}})\ \ar[rr] & & \ \B \Sy_{2^m}.
}
\]
The pull-back vector bundle\index{vector bundle} $\omega_{2^m}$ can be  decomposed into a Whitney sum of two vector bundles where one of them is a two dimensional trivial vector bundle.
This trivial vector subbundle  is determined by the trivial $\Sy_{2^{m-1}}\times \Sy_{2^{m-1}}$ subrepresentation $\{(x_1,\ldots,x_{2^m})\in \R^{2^m} : x_1=\cdots =x_{2^{m-1}},\, x_{2^{m-1}+1}=\cdots =x_{2^{m}}\}$ of $\R^{2^m}$.
Hence, $(2^m-1)$-Stiefel--Whitney class\index{Stiefel--Whitney classes} of this bundle $w_{2^m-1}(\omega_{2^m})=\res^{\Sy_{2^m}}_{\Sy_{2^{m-1}}\times \Sy_{2^{m-1}}}(D_{m,0})$ has to vanish.
This completes the proof of the first equality $D_{m,0}\cdot  \III^*(\R^{\infty},2^m)=0$ in \eqref{eq : relation - 10}.

\medskip
Next we prove that $D_{m,0}= V_{m,1}\cdots V_{m,m}$. 
Once again we use the fact that $H^*(\Sy_{2^m};\F_2)$ is detected by the subgroups $\EE_m$ and $\Sy_{2^{m-1}}\times \Sy_{2^{m-1}}$.
Since we showed that $\res^{\Sy_{2^m}}_{\Sy_{2^{m-1}}\times \Sy_{2^{m-1}}}(D_{m,0})=0$ it suffices to show that
\[
 \res^{\Sy_{2^m}}_{\EE_m} (D_{m,0}) = \res^{\Sy_{2^m}}_{\EE_m}(V_{m,1}\cdots V_{m,m}) 
 \qquad\Longleftrightarrow\qquad
 d_{m,0}=v_{m,1}\cdots v_{m,m}.
\]
Indeed, the equality $d_{m,0}=v_{m,1}\cdots v_{m,m}$ can be established by direct computation using the induction on $m$ in combination with relations \eqref{eq : formula for v_{m,r}} and \eqref{eq : recurrence relations d and v}, and observation that
\[
\chi_m (v_{m-1,1}\cdots v_{m-1,m-1})=v_{m,1}\cdots v_{m,m-1}.
\]
Hence, we showed that $D_{m,0}= V_{m,1}\cdots V_{m,m}$ and consequently verification of the second equality in \eqref{eq : relation - 10}.
This completes the proof of the lemma.
\end{proof}

In the final lemma of this section we describe the kernel of the restriction homomorphism $\res^{\Sym_{2^m}}_{\Sym_{2^{m-1}}\times \Sym_{2^{m-1}}}$\index{restriction homomorphism}.
 
 \begin{lemma}
 \label{lem : kernel of restriction - 2}
	Let $m\geq 0$ be an integer.
	Then
	\begin{equation}
		\label{eq : kernel of restriction - 2}
		\ker \big( \res^{\Sym_{2^m}}_{\Sym_{2^{m-1}}\times \Sym_{2^{m-1}}} \big)=\langle w_{m,0}\rangle\subseteq H^*(\Sym_{2^m};\F_2) .
	\end{equation}
\end{lemma}
\begin{proof}
For the proof of the lemma we use again	a classical result of Quillen  on the detection of group cohomology, see Section \ref{sec : detection}.
From Theorem \ref{thm : symmetric group  p^n detected} we have that the cohomology $H^*(\Sym_{2^m};\F_2)$ of the symmetric group\index{symmetric group} $\Sym_{2^m}$ modulo $\F_2$ is detected by the subgroups $\EE_m$ and $\Sym_{2^{m-1}}\times \Sym_{2^{m-1}}$.
This means that the homomorphism
\[
\xymatrix{
H^*(\Sym_{2^m};\F_2) \ \ar[rrrr]^-{ \res^{\Sym_{2^m}}_{\EE_m}  \times \res^{\Sym_{2^m}}_{\Sym_{2^{m-1}}\times \Sym_{2^{m-1}}} } & & & & \ H^*(\EE_m;\F_2) \times H^*(\Sym_{2^{m-1}}\times \Sym_{2^{m-1}};\F_2)
}
\]
is a monomorphism. 
Thus, 
\[
0\neq x\in \ker \big( \res^{\Sym_{2^m}}_{\Sym_{2^{m-1}}\times \Sym_{2^{m-1}}} \big)
\quad\Longrightarrow\quad
\res^{\Sym_{2^m}}_{\EE_m} (x)\neq 0.
\]
Further on, using the decomposition \eqref{eq : about restriction - 03} we get the implication
\[
0\neq x\in \ker \big( \res^{\Sym_{2^m}}_{\Sym_{2^{m-1}}\times \Sym_{2^{m-1}}} \big)
\quad\Longrightarrow\quad
x\in \langle w_{m,0},\ldots, w_{m,m-1}\rangle.
\]

\medskip
Like in the proof of the previous lemma we consider the vector bundle\index{vector bundle} $\xi_{2^m}$ and its pull-back $\theta_{2^m}$ introduced by the following pull-back diagram:
\[
\xymatrix{
\EEE (\Sym_{2^{m-1}}\times \Sym_{2^{m-1}})\times_{(\Sym_{2^{m-1}}\times \Sym_{2^{m-1}})}\R^{2^m} \ \ar[rr]\ar[d]_{\theta_{2^m}} & & \ \EEE \Sym_{2^m}\times_{\Sym_{2^m}}\R^{2^m}\ar[d]_{\xi_{2^m}} \\
%& & \\
\B(\Sym_{2^{m-1}}\times \Sym_{2^{m-1}}) \ \ar[rr] & & \ \B \Sym_{2^m}.
}
\]
As we know the classes $w_{m,0},\ldots, w_{m,m-1}$ are the Stiefel--Whitney classes\index{Stiefel--Whitney classes} of the vector bundle $\xi_{2^m}$ in dimensions $2^m-2^0,\ldots, 2^m-2^{m-1}$, respectively.
The pull-back vector bundle $\theta_{2^m}$ can be decomposed into a Whitney sum of two vector bundles where one of them is a two dimensional trivial vector bundle.
The trivial vector subbundle  is determined by the trivial $\Sym_{2^{m-1}}\times \Sym_{2^{m-1}}$ subrepresentation $\{(x_1,\ldots,x_{2^m})\in \R^{2^m} : x_1=\cdots =x_{2^{m-1}},\, x_{2^{m-1}+1}=\cdots =x_{2^{m}}\}$ of $\R^{2^m}$.
Hence, $(2^m-1)$-Stiefel--Whitney class of this bundle $\res^{\Sym_{2^m}}_{\Sym_{2^{m-1}}\times \Sym_{2^{m-1}}}(w_{m,0})$ has to vanish, or equivalently
\[
w_{m,0}\in \ker \big( \res^{\Sym_{2^m}}_{\Sym_{2^{m-1}}\times \Sym_{2^{m-1}}} \big).
\]

\medskip
On the other hand the pull-back vector bundle $\theta_{2^m}$ is isomorphic to the vector bundle $\xi_{2^{m-1}}\times \xi_{2^{m-1}}$.
Therefore, 
\[
 \res^{\Sym_{2^m}}_{\Sym_{2^{m-1}}\times \Sym_{2^{m-1}}} \big(w (\xi_{2^m})\big)= w(\theta_{2^m}) = w (\xi_{2^{m-1}}\times \xi_{2^{m-1}}) =w (\xi_{2^{m-1}}) \times w (\xi_{2^{m-1}}),
\]
and consequently for $2\leq r\leq m-1$ we have
\begin{align*}
\res^{\Sym_{2^m}}_{\Sym_{2^{m-1}}\times \Sym_{2^{m-1}}} (w_{m,r}) &=	\res^{\Sym_{2^m}}_{\Sym_{2^{m-1}}\times \Sym_{2^{m-1}}} (w_{2^m-2^r}(\xi_{2^m})) \\
&= w_{2^m-2^r}(\xi_{2^{m-1}}\times \xi_{2^{m-1}})\\
&=\sum_{i=0}^{2^m-2^r} w_{i}(\xi_{2^{m-1}})\times  w_{2^m-2^r-i}(\xi_{2^{m-1}})\\
&=\sum_{i=0}^{2^m-2^r} w_{i}(\xi_{2^{m-1}})\otimes  w_{2^m-2^r-i}(\xi_{2^{m-1}})\\
&\in\bigoplus_{i=0}^{2^m-2^r} H^i(\Sym_{2^{m-1}};\F_2)\otimes  H^{2^m-2^r-i}(\Sym_{2^{m-1}};\F_2).
\end{align*}
Here we silently use the Eilenberg--Zilber isomorphism\index{Eilenberg--Zilber isomorphism} \cite[Th.\,VI.3.2]{Bredon2010}.
In particular, we can isolate a concrete (direct) summand in the decomposition as follows:
\begin{multline*}
\res^{\Sym_{2^m}}_{\Sym_{2^{m-1}}\times \Sym_{2^{m-1}}} (w_{m,r})=
\sum_{i=0}^{2^m-2^r} w_{i}(\xi_{2^{m-1}})\otimes  w_{2^m-2^r-i}(\xi_{2^{m-1}})= \\
w_{2^{m-1}-2^{r-1}}(\xi_{2^{m-1}})\otimes  w_{2^{m-1}-2^{r-1}}(\xi_{2^{m-1}}) + \\
\sum_{i\neq 2^{m-1}-2^{r-1}}  w_{i}(\xi_{2^{m-1}})\otimes  w_{2^m-2^r-i}(\xi_{2^{m-1}}).
\end{multline*}
Now we use the fact that the Stiefel--Whitney classes\index{Stiefel--Whitney classes} 
\[
w_{2^{m-1}-2^{m-2}}(\xi_{2^{m-1}}),\dots, w_{2^{m-1}-2^{0}}(\xi_{2^{m-1}})
\]
are algebraically independent.
Indeed, they restrict to the corresponding Dickson invariants\index{Dickson invariants} $d_{m-1,0},\dots, d_{m-1,m-2}$ for which we know to be algebraically independent.
Consequently, the restricted homomorphism
\[
\res^{\Sym_{2^m}}_{\Sym_{2^{m-1}}\times \Sym_{2^{m-1}}} |_{\langle w_{m,1},\ldots, w_{m,m-1}\rangle }
\]
has to be a monomorphism.
Hence, 
$\ker ( \res^{\Sym_{2^m}}_{\Sym_{2^{m-1}}\times \Sym_{2^{m-1}}} ) = \langle w_{m,0}\rangle$,
and the proof of the lemma is complete.
\end{proof}

%%%%%%%%%%%%%%%%%%%%%%%%%%%%%%%%%%%%%%%%%%%%%%%%%%%%%%%%%%%%%%%%%%%%%%%%%%%%%%%%%%%%%
%%%%%%%%%%%%%%%%%%%%%%%%%%%%%%%%%%%%%%%%%%%%%%%%%%%%%%%%%%%%%%%%%%%%%%%%%%%%%%%%%%%%%
\section{{\Hung}'s Injectivity Theorem}
\label{sec : injectivity}
%%%%%%%%%%%%%%%%%%%%%%%%%%%%%%%%%%%%%%%%%%%%%%%%%%%%%%%%%%%%%%%%%%%%%%%%%%%%%%%%%%%%%
%%%%%%%%%%%%%%%%%%%%%%%%%%%%%%%%%%%%%%%%%%%%%%%%%%%%%%%%%%%%%%%%%%%%%%%%%%%%%%%%%%%%%
Let $d\geq 2$ be an integer or $d=\infty$, and let $m\geq 0$ be an integer.
Consider the composition map $\rho_{d,2^m}:=\id/\Sym_{2^m}\circ\epicy_{d,2^m}/\Sy_{2^m}$ between the quotient spaces
\begin{equation}
	\label{eq : rho map}
	\xymatrix@1{
 \Sp(\R^d,2^m)/\Sy_{2^m}\ \ \ar[rr]^-{\epicy_{d,2^m}/\Sy_{2^m}}&  & \ \  \conf(\R^d,2^m)/\Sy_{2^m}\ \ \ar[r]^-{\id/\Sym_{2^m}} &\ \    \conf(\R^d,2^m)/\Sym_{2^m},
}
\end{equation}
where the first map is induced by the $\Sy_{2^m}$-equivariant map $\epicy_{d,2^m}\colon\Sp(\R^d,2^m)\longrightarrow \conf(\R^d,2^m)$, and the second map is induced by the identity.

\medskip
The central objective of this section is to present a new and complete proof of the following claim, but first it is necessary to explain in detail several critical gaps in the published proof of this result, \cite[Thm.\,3.1]{Hung1982}.

\begin{theorem}
	\label{th : injection }
	Let $d\geq 2$ be an integer or $d=\infty$, and let $m\geq 0$ be an integer.
	Then the homomorphism
	\begin{equation}
	\label{injectivity-01}
		\rho_{d,2^m}^*\colon H^*(\conf(\R^d,2^m)/\Sym_{2^m};\F_2)\longrightarrow H^*(\Sp(\R^d,2^m)/\Sy_{2^m};\F_2)
	\end{equation}
	is a monomorphism.
\end{theorem}

\begin{remark}
	\label{remark : inj - 01}
	The homomorphism $\rho_{d,2^m}^*$ decomposes into the composition 
	\[
	(\epicy_{d,2^m}/\Sy_{2^m})^*\circ(\id/\Sym_{2^m})^*.
	\]
	Since the map $\id/\Sym_{2^m}\colon  \conf(\R^d,2^m)/\Sy_{2^m}\longrightarrow \conf(\R^d,2^m)/\Sym_{2^m}$ is a covering map then the composition homomorphism
	\[
	\xymatrix{
	H^*(\conf(\R^d,2^m)/\Sym_{2^m};\F_2) \ \ar[rr]^-{(\id/\Sym_{2^m})^*}\ar[drr]_-{[\Sym_{2^m}:\Sy_{2^m}]\cdot~} & & \ H^*(\conf(\R^d,2^m)/\Sy_{2^m};\F_2)\ar[d]^-{\mathrm{tr}}\\ 
	& &  H^*(\conf(\R^d,2^m)/\Sym_{2^m};\F_2)
	}
	\]
	is the multiplication with the index $[\Sym_{2^m}:\Sy_{2^m}]$.
	Here $\mathrm{tr}$ denotes the classical transfer homomorphism\index{transfer homomorphism}, consult for example \cite[Sec.\,3.G]{Hatcher2002}.
	Since $\Sy_{2^m}$ is a Sylow $2$-subgroup the index $[\Sym_{2^m}:\Sy_{2^m}]$ has to be odd.
	Hence the composition $\mathrm{tr}\circ (\id/\Sym_{2^m})^*$ is an isomorphism implying that $(\id/\Sym_{2^m})^*$ is a monomorphism.
	This means that in order to prove Theorem \ref{th : injection } it suffices to show that the homomorphism 
	\begin{equation}
	\label{injectivity-02}
	\xymatrix{
	(\epicy_{d,2^m}/\Sy_{2^m})^*\colon H^*(\conf(\R^d,2^m)/\Sy_{2^m};\F_2) \ \ar[r]& \ H^*(\Sp(\R^d,2^m)/\Sy_{2^m};\F_2)
	}
	\end{equation}
	is a monomorphis.
\end{remark}

\begin{remark}
	\label{remark : inj - 02}
	In the case when $d=\infty$ the homomorphism $\rho_{\infty,2^m}^*$ becomes the restriction homomorphism\index{restriction homomorphism} $\res^{\Sym_{2^m}}_{\Sy_{2^m}}$.
	Since we are working in the field $\F_2$ and $\Sy_{2^m}$ is a Sylow $2$-subgroup the restriction map $\res^{\Sym_{2^m}}_{\Sy_{2^m}}$ is injective, see for example \cite[Prop.\,III.9.5(ii) and Thm.\,III.10.3]{Brown1994}. 
	Thus, Theorem \ref{th : injection }  holds for $d=\infty$.
\end{remark}

%====
\subsection{Critical points in {\Hung}'s proof of his Injectivity Theorem}
\label{subsec : injectivity}
%====

In order to simplify the comparison with the work of {\Hung} we begin with a dictionary that translates between our notation and the notation used in \cite{Hung1990}. 

\medskip
{\footnotesize
\begin{center}
  \begin{tabular}{  l  l  l }
    \hline
    Paper \cite{Hung1990}  & This paper &  \\ \hline\noalign{\smallskip}
    $\Sym_{2^m}$ & $\Sym_{2^m}$ &  the symmetric group\index{symmetric group} on the set $\Z_2^{\oplus m}$\\ \noalign{\smallskip}
    $\Sym_{2^m,2}$ & $\Sy_{2^m}$ & the Sylow $2$-subgroup\index{Sylow $2$-subgroup} of $\Sym_{2^m}$ that contains $\EE_m$\\ \noalign{\smallskip}
    $E^m$ & $\EE_m$ &  $\Z_2^{\oplus m}$ regularly embedded elementary abelian group\index{elementary abelian group} in $\Sym_{2^m}$\\ \noalign{\smallskip}
    $F(X,n)$ & $\conf(X,n)$ & the ordered configuration space\index{ordered configuration space} of $n$ distinct points in $X$ \\ \noalign{\smallskip}
    $\widetilde{M}(d,m)$ & $\Sp(\R^d,2^m)$ & the ordered Ptolemaic epicycles space\index{Ptolemaic epicycles space} $(S^{d-1})^{2^m-1}$ \\ \noalign{\smallskip}
    $M(d,m)$ & $\Sp(\R^d,2^m)/\Sy_{2^m}$ & the unordered Ptolemaic epicycles space\\ \noalign{\smallskip}
    $\widetilde{i}(d,m)$ & $\epicy_{d,2^m}$ & the map from Def.\,\ref{def : Epicycles} with $d\geq 1$ integer or $d=\infty$ \\ \noalign{\smallskip}
    $i(d,m)$ & $\rho_{d,2^m}$ & the map introduced in \eqref{eq : rho map} with $d\geq 1$ integer or $d=\infty$\\ \noalign{\smallskip}
         $i(M,d)$ & $\kappa_{d,m}/\Sy_{2^m}$ & the map $\Sp(\R^d,m)/\Sy_{2^m}\longrightarrow \Sp(\R^{\infty},m)/\Sy_{2^m}$  \\
    \noalign{\smallskip}
    $i(F,d)$ & & the map $F(\R^d,2^m)\longrightarrow F(\R^{\infty},2^m)$  \\ \noalign{\smallskip}
   $W_{m,r}$ & $w_{m,r}$ & restriction of the Stiefel--Whitney class\index{Stiefel--Whitney classes} $w_{2^m-2^r}$ to $H^*(\Sym_{2^m})$\\ \noalign{\smallskip}
   $\overline{Q}_{m,r}$ & $D_{m,r}$ & restriction of the Stiefel--Whitney class $w_{2^m-2^r}$ to $H^*(\Sy_{2^m})$ \\ \noalign{\smallskip}
   $Q_{m,r}$ & $d_{m,r}$ & restriction of the Stiefel--Whitney class $w_{2^m-2^r}$ to $H^*(\EE_m)$,  \\  \noalign{\smallskip}
    & & or the Dickson invariant \\ \noalign{\smallskip}
    $\overline{V}_{m,r}$ & $V_{m,r}$ & elements of $H^*(\Sy_{2^m})$ defined in Thm. \ref{th: cohomology of Sp}   \\ \noalign{\smallskip}
    $V_{m,r}$ & $v_{m,r}$ & elements of $H^*(\EE_m)$ given by restriction $\res^{\Sy_{2^m}}_{\EE_m}(V_{m,r})$  \\
    \noalign{\smallskip}\hline
  \end{tabular}
\end{center}
}
 
\bigskip
The statement of Theorem \ref{th : injection } in \cite[Thm.\,3.1]{Hung1990} is written as follows; the coefficient field $\F_2$ is always to be assumed.
\begin{quote}
{\small
	{\sc 3.1.\,Theorem  }{\em $i^*(q,n)\colon H^*(F(\R^q,2^n)/\Sym_{2^n})\longrightarrow H^*(M(q,n))$ is a monomorphism for $q\geq 1$, $n\geq 0$.}}
\end{quote}

\medskip
The proof of \cite[Thm.\,3.1]{Hung1990} presented in {\Hung}'s paper is by induction on $n$.
It starts on the page 269 and ends on page 271.
This proof relies on \cite[Prop.\,3.5]{Hung1990}.
The claim of  \cite[Prop.\,3.5]{Hung1990} is proven on page 275 and relies on \cite[Lem.\,3.14]{Hung1990} and \cite[Lem.\,3.19]{Hung1990}.

\medskip
Now we outline the proof given by {\Hung} and exhibit two critical points that we have identified. 
The proof is by induction on $n$.
For $n=0$ the statement is easy to verify since $F(\R^q,2^n)=\R^q$ and $\widetilde{M}(q,n)=\pt$.
Let us assume that $i^*(q,n-1)$ is injective.
Before we make the next step in the proof we define the maps $\mu_{m,n}$ introduced in \cite[(3.2)]{Hung1990}, and the map $\varphi_{n-1}$ defined in \cite[(2.3)]{Hung1990}.

\medskip
For integers $m\geq 1$ and $n\geq 1$ consider the map
\[
\xymatrix{
\mu_{m,n}\colon F(\R^q,m)/\Sym_m\times F(\R^q,n)/\Sym_n \ \ar[r]& \ F(\R^q,m+n)/\Sym_{m+n},
}
\]
given by
\[
\xymatrix{
[(x_1,\ldots,x_m)]\times [(y_1,\ldots,y_n)] \ \ar@{|->}[r] & \ [(x_1,\ldots,x_m,y_1+z,\ldots,y_n+z)].
}
\]
Here for
\[
R_1=\max_{1\leq k\leq m}\Big\|x_k-\frac{1}{m}\sum_{i=1}^mx_i\Big\|
\qquad\text{and}\qquad
R_2=\max_{1\leq k\leq n}\Big\|y_k-\frac{1}{n}\sum_{j=1}^ny_j\Big\|
\]
we define
\[
z=\frac{1}{m}\sum_{i=1}^mx_i-\frac{1}{n}\sum_{j=1}^ny_j-(R_1+R_2+1,0,\ldots,0).
\]
Next for any $n\geq1$ we introduce the following map
\[
\xymatrix{
\varphi_{n-1}\colon M(q,n-1)\times M(q,n-1)\ar[r] & M(q,n), &  (x,y)\ar@{|->}[r] & [(x,y,*)].
}
\]

\medskip
Let us now consider the following diagram that commutes up to a homotopy
\[
\xymatrix{
F(\R^q,2^n)/\Sym_{2^n} \ & & \ \ar[ll]_-{\mu:=\mu_{2^{n-1},2^{n-1}}}\big( F(\R^q,2^{n-1})/\Sym_{2^{n-1}} \big)^2\\
M(q,n)\ar[u]^-{i(q,n)} \ & &\  \ar[ll]_-{\varphi:=\varphi_{n-1}}M(q,n-1)^2\ar[u]^-{i(q,n-1)^2}.
}
\]
This diagram induces the following commutative diagram in cohomology where {\Hung} claimed that \underline{each row is exact}\footnote{The first critical point that is explained in Claim \ref{claim-01}.}: 
\begin{equation}{
\label{Hung-01}
\xymatrix@C=1.5em{
0\ \ar[r] & \ \ker(\mu^*) \ \ar[r]\ar[d]_{i^*(q,n)|_{\ker(\mu^*)}} & \ H^*(F(\R^q,2^n)/\Sym_{2^n})\ar[r]^-{\mu^*}\ar[d]_{i^*(q,n)} & \ H^*((F(\R^q,2^{n-1})/\Sym_{2^{n-1}})^2)\ar[r]\ar[d]_{i^*(q,n-1)^2} & \ 0\\
0\ \ar[r] & \ \ker(\varphi^*) \ \ar[r] & \ H^*(M(q,n))\ar[r]^-{\varphi^*}  & \ H^*(M(q,n-1)^2)\ar[r] & \ 0 
}}
\end{equation}
From induction hypothesis we have that $i^*(q,n-1)^2=i^*(q,n-1)\otimes i^*(q,n-1)$ is a monomorphism. 
Then from $5$-lemma in order to conclude the induction, and consequently prove \cite[Thm.\,3.1]{Hung1990}, it suffices to prove that
\[
i^*(q,n)|_{\ker(\mu^*)}\colon \ker(\mu^*)\longrightarrow\ker(\varphi^*)
\]
is a monomorphism.

\medskip
At this point we already crossed path with the first critical point in the proof (indicated by a footnote).
The following claim explains the nature of the problem that appears in the diagram \eqref{Hung-01}.

\begin{claim}
\label{claim-01}
For any integer $n\geq 2$, the map 
\[
\varphi^*\colon  H^*(M(q,n))\longrightarrow H^*(M(q,n-1)^2)
\]
in \eqref{Hung-01} is not surjective.
\end{claim}
\begin{proof}
In the proof of the claim we use the notation of {\Hung} and keep in mind that 
\[
M(q,n)=\widetilde{M}(q,n)/\Sym_{2^n,2}=\Sp(\R^q,2^n)/\Sy_{2^n}.
\]
The cohomology of this space is described in Section \ref{sec : Equivariant cohomology of epicicles}.
For a reader convenience we repeat some of the arguments already presented.

\medskip
From Definition \ref{def : Epicycles} we have that
\[
M(q,n)=\widetilde{M}(q,n)/\Sym_{2^n,2}\cong (M(q,n-1)\times M(q,n-1))\times_{\Z_2} S^{q-1}.
\]
Since the action of $\Z_2$ on the sphere $S^{q-1}$ is free the projection on the last coordinate induces the following fiber bundle 
\begin{multline*}
M(q,n-1)\times M(q,n-1) \longrightarrow  \\(M(q,n-1)\times M(q,n-1))\times_{\Z_2} S^{q-1} \longrightarrow 
 \RP^{q-1}, 
\end{multline*}
where the map $\varphi=\varphi_{n-1}$ is the fiber embedding. 

\medskip
The Serre spectral sequence\index{Serre spectral sequence} associated to this fibration has the $E_2$-term given by
\[
E_2^{r,s}=H^r(\RP^{q-1};\mathcal{H}^s(M(q,n-1)\times M(q,n-1))).
\]
As we have seen in Corollary \ref{cor : E_2 collapses} this spectral sequence collapses at the $E_2$-term, that is $E_2^{r,s}\cong E_{\infty}^{r,s}$.
In particular, for an arbitrary integer $k\geq 0$ this means that the map $\varphi^*$ factors as follows:
\begin{multline*}
H^k(M(q,n))\cong \bigoplus_{r+s=k} E^{r,s}_{2} \longrightarrow \\
 E^{0,k}_2 \cong H^k(M(q,n-1)\times M(q,n-1))^{\pi_1(\RP^{q-1})}\\ \longrightarrow H^k(M(q,n-1)\times M(q,n-1)).
\end{multline*}
Here the first map is the projection and the second map is the inclusion.
It is important to recall that $\pi_1(\RP^{q-1})$ acts on $H^s(M(q,n-1)\times M(q,n-1))$ by interchanging the factors in the product. 
Thus, while the first map -- the projection -- is surjective, the second map is not surjective in all positive dimensions where $H^k(M(q,n-1)\times M(q,n-1))\neq 0$.
\end{proof}

\medskip
Thus already at this point the proof of \cite[Thm.\,3.1]{Hung1990} has the first problem.
Nevertheless, we continue to outline next steps of the proof that now concentrates on proving that
\[
i^*(q,n)|_{\ker(\mu^*)}\colon \ker(\mu^*)\longrightarrow\ker(\varphi^*)
\]
is a monomorphism.
The complexity of the proof suggests that we first explain the strategy that was used by {\Hung} and then study particular details.
Consider the commutative diagram \eqref{Hung-02} on the next page, 
which is an enrichment of the diagram \eqref{Hung-01} that we have already considered.
\medskip  
The proof of the injectivity of the map $i^*(q,n)|_{\ker(\mu^*)}$ presented by {\Hung} consists of several steps that we now list:

\medskip
\begin{compactenum}[\rm \ (A)]

\item Description of $\ker\big(\res_{\Sym^2_{2^{n-1}}}^{\Sym_{2^n}}\big)$ in terms of the dual Nakamura elements.

\item Description of $\ker(\mu^*)$ via the surjectivity of the map 
\[
i^*(F,q)|=i^*(F,q)|_{\ker\big(\res_{\Sym^2_{2^{n-1}}}^{\Sym_{2^n}}\big)}\colon\ker\big(\res_{\Sym^2_{2^{n-1}}}^{\Sym_{2^n}}\big) \longrightarrow  \ker(\mu^*).
\]

\item Description of the image 
\[
(i^*(M,q)\circ i^*(\infty,n))\big(\ker\big(\res_{\Sym^2_{2^{n-1}}}^{\Sym_{2^n}}\big)\big)\subseteq\ker(\varphi^*).
\]

\item A proof that  
\[
\ker(\mu^*)\cong(i^*(M,q)\circ\, i^*(\infty,n))\big(\ker\big(\res_{\Sym^2_{2^{n-1}}}^{\Sym_{2^n}}\big)\big)\]
as $\F_2$-vector spaces.
\end{compactenum}

\begin{figure}[p]
   \rotatebox{90}{%
     \begin{minipage}{\textheight}%
\begin{equation}
\label{Hung-02} 
\xymatrix{
 & & & & & & &\\
 & &  &H^*(\Sym_{2^n})\ar@{=}[d] & &H^*(\Sym_{2^{n-1}}^2)\ar@{=}[d] & &\\ 
0 \ \ar[r] & \ \ker\big(\res_{\Sym^2_{2^{n-1}}}^{\Sym_{2^n}}\big)\ \ar[rr]\ar[dd]_{i^*(F,q)\mid}\ar[dddr]|(.65)\hole^-{i^*(\infty,n)\mid} &  &\ H^*(F(\R^{\infty},2^n)/\Sym_{2^n}) \ \ar[rr]^-{\res_{\Sym^2_{2^{n-1}}}^{\Sym_{2^n}}}\ar[dd]_{i^*(F,q)}\ar[dddr]|(.65)\hole^-{i^*(\infty,n)} & & \ H^*((F(\R^{\infty},2^{n-1})/\Sym_{2^{n-1}})^2)\ar[dd]_{i^*(F,q)^2} & &\\
 & & & & & & &\\
0\ \ar[r] & \ \ker(\mu^*) \ \ar[rr]\ar[dd]_{i^*(q,n)|_{\ker(\mu^*)}} & & \ H^*(F(\R^q,2^n)/\Sym_{2^n}) \ \ar[rr]^-{\mu^*}\ar[dd]_{i^*(q,n)} &  & \ H^*((F(\R^q,2^{n-1})/\Sym_{2^{n-1}})^2)\ \ar[r]\ar[dd]_{i^*(q,n-1)^2} & \ 0 &\\
 & & i^*(\infty,n)(\ker(\res))\ar[dl]^-{i^*(M,q)\mid} & & \ H^*(M(\infty,n))\ar[dl]^-{i^*(M,q)} & &\\
0\ \ar[r] &\  \ker(\varphi^*) \ \ar[rr] & &\ H^*(M(q,n)) \ \ar[rr]^-{\varphi^*}  & &\  H^*(M(q,n-1)^2) \ \ar[r] &\ 0 &
}
\end{equation}

     \end{minipage}%
  }
\end{figure}
These claims along with commutativity of the diagram \eqref{Hung-02} imply that the map
\[
i^*(q,n)|_{\ker(\mu^*)}\colon \ker(\mu^*)\longrightarrow (i^*(M,q)\circ i^*(\infty,n))\big(\ker\big(\res_{\Sym^2_{2^{n-1}}}^{\Sym_{2^n}}\big)\big)
\] 
is an isomorphism, and consequently  
\[
i^*(q,n)|_{\ker(\mu^*)}\colon \ker(\mu^*)\longrightarrow\ker(\varphi^*)
\]
is a monomorphism.
Thus, establishing claims we listed above would complete the proof of \cite[Thm.\,3.1]{Hung1990}.
Now we discuss these steps separately exhibiting the second critical point of the proof.

\medskip
{\bf (A)} and {\bf (B)}
For these steps multiple results of May \cite{May1976LNM533-01}, Nakaoka \cite{Nakaoka1961}, Nakamura \cite{Nakamura1963}, and {\Hung} \cite{Hung1982} are recalled and will be used in the proof.
For the reader's convenience we collect the relevant facts as presented in \cite[Sec.\,3]{Hung1990}.

\medskip
First, the homology $H_*(F(\R^q,\infty)/\Sym_{\infty})$ can be identified with a Hopf subalgebra of the homology $H_*(F(\R^{\infty},\infty)/\Sym_{\infty})=H_*(\Sym_{\infty})$, \cite[Sec.\,5]{May1972}.
Furthermore, $H_*(F(\R^q,\infty)/\Sym_{\infty})$ is equipped with multiplicity  in such a way that
\[
_m H_*(F(\R^q,\infty)/\Sym_{\infty}) = H_*(F(\R^q,m)/\Sym_{m},F(\R^q,m-1)/\Sym_{m-1}),
\]
consult \cite[Sec.\,2]{Hung1982}.
In general, an algebra $A$ equipped with multiplicity has a decomposition $A=\bigoplus_{n\geq 0}\,{_n}A$, and we define its filtration by multiplicities with $A(m):=\bigoplus_{0\leq n\leq m}\,{_n}A$.
For more detailed definitions see \cite[p.\,96]{Nakamura1963}.
In our concrete situation we have that
\[
H_*(F(\R^q,\infty)/\Sym_{\infty})(m)=H_*(F(\R^q,m)/\Sym_{m}).
\]

\medskip
Let us further on, for integers $k_0,\ldots,k_{n-1}\geq 0$, denote by 
\[
N_{k_0,\ldots,k_{n-1}}\in H_*(F(\R^q,\infty)/\Sym_{\infty})
\] 
the so called Nakamura element of multiplicity $2^n$, as introduced in \cite[Sec.\,2]{Hung1982}.
Now we can quote the following theorem from \cite[Thm.\,3.4]{Hung1990}.
\begin{quote}
{\small
	{\sc 3.4.\,Theorem }(Nakamura \cite{Nakamura1963}, May \cite{May1976LNM533-01}, Hu\'ynh M\'ui \cite{Mui1980})
	{\em 
	\begin{compactenum}[\rm (i)]
		\item Let $q>0$ and 
		\[
		J^+(q)=\Big\{ K=(k_0,\ldots,k_{n-1}) : n\geq 1, k_0\geq 1, k_1,\ldots,k_{n-1}\geq 0, \sum_{i=0}^{n-1}k_i\leq q-1\Big\}.
		\]
		Then
		\[
		H_*(F(\R^q,\infty)/\Sym_{\infty})=\F_2[N_K : K\in J^+(q)]
		\]
		as algebras with multiplicities.
		So we have for every $0\leq n\leq \infty$ that
		\[
		H_*(F(\R^q,m)/\Sym_{m})=\F_2[N_K : K\in J^+(q)](m).
		\]
		In other words $H_*(F(\R^q,m)/\Sym_{m})$ has the $\F_2$-basis consisting of all monomials in $\F_2[N_K : K\in J^+(q)]$ of multiplicities $\leq m$.
		This is called Nakamura basis.
		
		\item The homomorphism 
		\[
			i_*(F,q)\colon H_*(F(\R^q,\infty)/\Sym_{\infty})\longrightarrow H_*(F(\R^{\infty},\infty)/\Sym_{\infty})
		\]
		induced by the canonical embeddings $F(\R^q,m)\subset F(\R^{\infty},m)$, $0\leq m<\infty$, is an injection.
		It sends $N_K$ to the element denoted by the same notation $N_K$ for $K\in J^+(q)$.
	\end{compactenum}
	}}
\end{quote}

\medskip
Now a part of the commutative diagram \eqref{Hung-02} is considered:
\[
\xymatrix{
H^*(F(\R^{\infty},2^n)/\Sym_{2^n}) \ \ar[rr]^-{\res_{\Sym^2_{2^{n-1}}}^{\Sym_{2^n}}}\ar[d]_{i^*(F,q)}  & & \ H^*((F(\R^{\infty},2^{n-1})/\Sym_{2^{n-1}})^2)\ar[d]_{i^*(F,q)^2}\\
H^*(F(\R^{q},2^n)/\Sym_{2^n}) \ \ar[rr]^-{\mu^*} & & \ H^*((F(\R^{q},2^{n-1})/\Sym_{2^{n-1}})^2),
}
\]
where the maps $i^*(F,q)$ are induced by the inclusion $\R^q\longrightarrow \R^{\infty}$. 
Observe that
\[
H^*(F(\R^{\infty},2^n)/\Sym_{2^n})\cong H^*(\Sym_{2^n})
\]
and
\[
H^*(F(\R^{\infty},2^{n-1})/\Sym_{2^{n-1}})^2\cong H^*(\Sym_{2^{n-1}}\times \Sym_{2^{n-1}}).
\]
It is now claimed that based of \cite[Thm.\,3.4]{Hung1990}, which we quoted in full, and the definition of the algebra structure on $H_*(F(\R^q,\infty)/\Sym_{\infty})$, see \cite[Sec.\,2]{Hung1987}, the following equality holds: 
\[
	\ker\big( \res_{\Sym^2_{2^{n-1}}}^{\Sym_{2^n}}\big)   =  \spann\{ N_{k_0,\ldots,k_{n-1}}^* : k_0\geq 1\},
\]
where $N_{k_0,\ldots,k_{n-1}}^*$ is the dual of the Nakamura element $N_{k_0,\ldots,k_{n-1}}$.
Furthermore
\begin{align*}
	\ker(\mu^*) & =  \spann\{ N_{k_0,\ldots,k_{n-1}}^* : (k_0,\ldots,k_{n-1})\in J^+(q)\}\\
			    & =  i^*(F,q)\big( \spann\{ N_{k_0,\ldots,k_{n-1}}^* : k_0\geq 1\}\big)\\
			    & = 	i^*(F,q)\big( \ker\big( \res_{\Sym^2_{2^{n-1}}}^{\Sym_{2^n}}\big) \big).
\end{align*}
In particular, the map $i^*(F,q)|_{ \ker\big( \res_{\Sym^2_{2^{n-1}}}^{\Sym_{2^n}}\big)}$ is surjective. 
With this parts (A) and (B) of {\Hung}'s injectivity proof are concluded.

\medskip
{\bf (C)} 
In \cite[Prop.\,3.5]{Hung1990} the image $(i^*(M,q)\circ i^*(\infty,n))\big(\ker\big(\res_{\Sym^2_{2^{n-1}}}^{\Sym_{2^n}}\big)\big)$ was described. 
We present this proposition with the paragraph that precedes it as in the original.

\medskip
\begin{quote}
{\small
	On the other, let $\rho_{2^n}\colon\Sym_{2^n}\longrightarrow\OO(2^n)$ denote the natural representation of the symmetric group $\Sym_{2^n}$ in the orthogonal group $\OO(2^n)$. As it is well known:
	\[
	H^*(\OO(2^n))=\Z_2[W_1,\dots,W_{2^n}],
	\]
	where $W_i$ denotes the $i$th universal Stiefel--Whitney class\index{Stiefel--Whitney classes} (of dimension $i$). 
	We define $(2^n-2^s)$th Stiefel--Whitney class of $\rho_{2^n}$ by putting
	\[
	W_{n,s}=\rho_{2^n}^*W_{2^n-2^s}, \qquad 0\leq s<n.
	\]
	Further, we set
	\[
	\overline{Q}_{n,s}=\mathrm{Res}(\Sym_{2^n,2},\Sym_{2^n})(W_{n,s})\in H^*(\Sym_{2^n,2}),  \qquad 0\leq s<n.
	\]	
	\medskip
	
	{\sc 3.5.\,Proposition}. {\em  Let
	\[
	\begin{array}{l}
		i(M,q) \colon M(q,n)\longrightarrow M(\infty,n),\\
		i(\infty,n)\colon  M(\infty,n)\longrightarrow F(\R^{\infty},2^n)/\Sym_{2^n}
	\end{array}
	\]
	be well-known embeddings. 
	Then we have
	\[
	(i^*(M,q)\circ i^*(\infty,n))\big(\ker\big(\res_{\Sym^2_{2^{n-1}}}^{\Sym_{2^n}}\big)\big)=
	\overline{Q}_{n,0}\,\F_2[\overline{Q}_{n,0},\ldots,\overline{Q}_{n,n-1}]/I(\overline{Q},q).
	\]
	Here $I(\overline{Q},q)$ denotes the ideal of $\overline{Q}_{n,0}\,\F_2[\overline{Q}_{n,0},\ldots,\overline{Q}_{n,n-1}]$ generated by monomials of degree $q$.}
	}
\end{quote}

\medskip
\noindent
The proof given by {\Hung} proceeds as follows. 
According to Lemma \ref{lem : kernel of restriction - 2}, the decomposition \eqref{eq : about restriction - 03}, and the fact that $H^*(\Sym_{2^n};\F_2)$ is detected by the subgroups $E^n$ and $\Sym_{2^{n-1}}^2$  we have that
\[
\ker\big(\res_{\Sym^2_{2^{n-1}}}^{\Sym_{2^n}}\big)=\langle W_{n,0}\rangle = W_{n,0}\,\F_2[ W_{n,0},\ldots,  W_{n,n-1}].
\] 
Consequently, using the notation dictionary $i^*(\infty,n)=\res^{\Sym_{2^n}}_{\Sy_{2^n}}$, we get
\begin{align*}
i^*(\infty,n)\big(\ker\big(\res_{\Sym^2_{2^{n-1}}}^{\Sym_{2^n}}\big)\big) &=
i^*(\infty,n)\big(\langle W_{n,0}\rangle\big)\\
&=i^*(\infty,n)\big(W_{n,0}\,\F_2[ W_{n,0},\ldots,  W_{n,n-1}]\big)	\\
&=\overline{Q}_{n,0}\,\F_2[\overline{Q}_{n,0},\ldots,\overline{Q}_{n,n-1}].
\end{align*}
On the other hand from Lemma \ref{lem : image of res on w_{m,0}} we have that 
\[
i^*(\infty,n)\big(\langle W_{n,0}\rangle\big)\subseteq\F_2[\overline{V}_{n,1},\ldots, \overline{V}_{n,n}].
\]
Thus,
\[
i^*(\infty,n)\big(\ker\big(\res_{\Sym^2_{2^{n-1}}}^{\Sym_{2^n}}\big)\big)=
\overline{Q}_{n,0}\,\F_2[\overline{Q}_{n,0},\ldots,\overline{Q}_{n,n-1}]\subseteq
\F_2[\overline{V}_{n,1},\ldots, \overline{V}_{n,n}].
\]

\medskip
Now from Theorem \ref{th: cohomology of Sp} we have that 
\[
i^*(M,q)(\F_2[\overline{V}_{n,1},\ldots, \overline{V}_{n,n}])\cong\F_2[\overline{V}_{n,1},\ldots, \overline{V}_{n,n}]/\langle \overline{V}_{n,1}^q,\ldots, \overline{V}_{n,n}^q \rangle .
\]
Therefore,
\[
(i^*(M,q)\circ i^*(\infty,n))\big(\ker\big(\res_{\Sym^2_{2^{n-1}}}^{\Sym_{2^n}}\big)\big)\cong
	\overline{Q}_{n,0}\,\F_2[\overline{Q}_{n,0},\ldots,\overline{Q}_{n,n-1}]/I(\overline{Q},q)
\]
where ideal $I(\overline{Q},q)$ is given by
\begin{equation}
\label{eq : the ideal - 01}
I(\overline{Q},q) = \overline{Q}_{n,0}\,\F_2[\overline{Q}_{n,0},\ldots,\overline{Q}_{n,n-1}]\cap  \langle \overline{V}_{n,1}^q,\ldots, \overline{V}_{n,n}^q \rangle.	
\end{equation}

\medskip
In order to complete the step (C) in the proof of \cite[Prop.\,3.5]{Hung1990} it remains to show that the ideal $I(\overline{Q},q)$ is the ideal generated by the monomials of degree $q$.
The necessary argument for this was given in \cite[Lem.\,3.19]{Hung1990} by using the restriction homomorphism\index{restriction homomorphism} $\res^{\Sy_{2^n}}_{\EE_n}$ and considering the corresponding claim in $H^*(\EE_n)$.
We give the original formulation without introducing new variables.

\medskip
\begin{quote}
{\small
	{\sc 3.19.\,Lemma}. {\em  Let
	\[
	\mathrm{pr}\colon \F_2[V_{n,1},\ldots,V_{n,n}]
	\longrightarrow
	 \F_2[V_{n,1},\ldots,V_{n,n}]/
	 \langle V_{n,1}^q,\ldots, V_{n,n}^q \rangle
	\]
	be the projection.
	Then, for the subring $\F_2[Q_{n,0},\ldots,Q_{n,n-1}]$ of $\F_2[V_{n,1},\ldots,V_{n,n}]$, we have
	\[
	\mathrm{pr}\big(\F_2[Q_{n,0},\ldots,Q_{n,n-1}]\big)=
	\F_2[Q_{n,0},\ldots,Q_{n,n-1}]/I(Q,q).
	\]
	Here $I(Q,q)$ denotes the ideal of $\F_2[Q_{n,0},\ldots,Q_{n,n-1}]$ generated by monomials of degree $q$.\footnote{The second critical point that is explained in Claim \ref{claim-02}.}
	}}
\end{quote}

\medskip
\noindent
It will turn out that the ideal $I(Q,q)$ is {\em not} the ideal of $\F_2[Q_{n,0},\ldots,Q_{n,n-1}]$ generated by monomials of degree $q$, implying that the claim of \cite[Lem.\,3.19]{Hung1990} does not stand.

\medskip
First we equivalently transform the statement of the lemma.
Recall that the classes $Q_{n,0},\ldots,Q_{n,n-1}$ are Dickson invariants\index{Dickson invariants} and therefore $\GL_n(\F_2)$-invariants, while
$V_{n,1},\ldots,V_{n,n}$ are $\U_n(\F_2)$-invariants.
Furthermore, from \eqref{eq : recurrence relations d and v} we have that 
\begin{equation}
	\label{eq : recurrence relations d and v - 2} 
	Q_{n,r}=(\chi_n Q_{n-1,r})\,V_{n,n} + (\chi_n Q_{n-1,r-1})^2
\end{equation}
where $\chi_n\in\GL_n(\F_2)$ can be interpreted as the variable change in $\F_2[y_1,\ldots, y_n]$ given by  $y_i\longmapsto y_{n-i+1}$ for $1\leq i\leq n$.
Applying $\chi_n$ to the equality \eqref{eq : recurrence relations d and v - 2} yields  
\[
 \chi_n Q_{n,r}=(\chi_n^2 Q_{n-1,r})\,\chi_n V_{n,n} + (\chi_n^2 Q_{n-1,r-1})^2.
\]
Since $\chi_n^2=\id$ and $Q_{n,r}$ is a $\GL_n(\F_2)$-invariant we get the following recurrent relation
\[ 
	  Q_{n,r}= Q_{n-1,r}\,(\chi_n V_{n,n}) + (Q_{n-1,r-1})^2.
\]
For simplicity the following notation was introduced $V_r:=\chi_n V_{n,r}$ for all $1\leq r\leq n$.
Then the projection map
\[
\mathrm{pr}\colon \F_2[V_{n,1},\ldots,V_{n,n}]\longrightarrow \F_2[V_{n,1},\ldots,V_{n,n}]/
	 \langle V_{n,1}^q,\ldots, V_{n,n}^q \rangle
\]
can written by
\[
\mathrm{pr}\colon\F_2[V_{1},\ldots,V_{n}]\longrightarrow  \F_2[V_{1},\ldots,V_{n}]/
	 \langle V_{1}^q,\ldots, V_{n}^q \rangle .
\]
Thus \cite[Lem.\,3.19]{Hung1990}, equivalently, states that
	\[
	\mathrm{pr}\big(\F_2[Q_{n,0},\ldots,Q_{n,n-1}]\big)=
	\F_2[Q_{n,0},\ldots,Q_{n,n-1}]/I(Q,q),
	\]
where $I(Q,q)$ is the ideal of $\F_2[Q_{n,0},\ldots,Q_{n,n-1}]$ generated by monomials of degree $q$, and 
\begin{equation}
	\label{eq : recurrence relations d and v - 3} 
	  Q_{n,r}= Q_{n-1,r}\,V_{n} + (Q_{n-1,r-1})^2,
\end{equation}
where we assume that $Q_{k,k}=1$ and $Q_{k,-1}=0$ for any integer $k\geq 1$.
In particular, from the second equality in \eqref{eq : relation - 10} we get that
\begin{equation}
	\label{eq : recurrence relations d and v - 4} 
	Q_{n,0}=V_1\cdots V_n.
\end{equation}
Now we explain the problem that occurs in the description of the ideal $I(Q,q)$.

\begin{claim}
\label{claim-02}
	The ideal $I(Q,q)$, defined in \eqref{eq : the ideal - 01}, is not in general the ideal of the ring $\F_2[Q_{n,0},\ldots,Q_{n,n-1}]$ generated by the monomials of degree $q$.
	For example, this fails for $n=2$ and $q=3$ or $q=4$. 
\end{claim} 
\begin{proof}
The proof is given by exhibiting several counterexamples in the case when $n=2$.
From equalities \eqref{eq : recurrence relations d and v - 3} and \eqref{eq : recurrence relations d and v - 4} we get that 
\[
Q_{2,0}=V_1 V_2
\qquad\qquad\text{and}\qquad\qquad
Q_{2,1}=V_2+ V_1^2.
\]
Now we discuss different values of $q$. 
\begin{compactenum}[\rm \ (1)]
%----
\item Let $q=3$.
Then by a direct computation in the ring $ \F_2[V_{1},V_{2}]$ we have that monomials
\[
Q_{2,0}^3=V_1^3 V_2^3,\quad Q_{2,0}^2Q_{2,1}=V_1^2 V_2^3+V_1^4 V_2^2,\quad Q_{2,0}Q_{2,1}^2=V_1 V_2^3+V_1^5 V_2 
\]
belong to the ideal $I(Q,3)$.
Furthermore, since $Q_{2,0}^2=V_1^2V_2^2\notin I(Q,3)$ and 
\[
Q_{2,1}^3=V_2^3+V_2^2V_1^2+V_2V_1^4+V_1^6\notin I(Q,3)
\] 
we obtain that the monomial
\[
Q_{2,0}^2+Q_{2,1}^3 = V_2^3+V_2V_1^4+V_1^6 \ \in  \ I(Q,3).
\]
Thus, $Q_{2,1}^3 \notin I(Q,3)$ and $Q_{2,0}^2+Q_{2,1}^3 \in  I(Q,3)$ giving us the first counterexample to the description of the ideal $I(Q,q)$.

%----
\item Let $q=4$. 
Then working in  the ring $ \F_2[V_{1},V_{2}]$ we get that the monomials
\begin{align*}
Q_{2,0}^4=V_1^4 V_2^4, &\qquad 	Q_{2,0}^3Q_{2,1}=V_1^3 V_2^4+V_1^5 V_2^3,\\
Q_{2,0}^2Q_{2,1}^2=V_1^2 V_2^4+V_1^6 V_2^2,&\qquad Q_{2,1}^4=V_1^8+V_2^4
\end{align*}
are in the ideal $I(Q,4)$,while
\[
Q_{2,0}Q_{2,1}^3=V_1 V_2^4+V_1^3 V_2^3+V_1^5 V_2^2+V_1^7 V_2 \ \notin \ I(Q,4).
\]
Thus we obtained yet another evidence that the description of the ideal $I(Q,q)$ is incorrect. 
\end{compactenum}
This concludes the proof of the claim and opens a question of correct description of the ideal $I(Q,q)$.
\end{proof}

\medskip
We explained an essential gap in the proof of the step (C): ``description of the image $(i^*(M,q)\circ i^*(\infty,n))\big(\ker\big(\res_{\Sym^2_{2^{n-1}}}^{\Sym_{2^n}}\big)\big)$.'' 
This automatically invalidates the proof of the next step (D).
Thus based on two gaps explained in Claim \ref{claim-01} and Claim \ref{claim-02} we have shown that the proof 
of \cite[Thm.\,3.1]{Hung1990} is incorrect.
Moreover, we do not see how the approach taken by {\Hung} can be easily repaired.

\begin{remark}
\label{rem: gap}
The presented gaps also invalidate results of \cite[Sec.\,4]{Hung1990}.
In particular, counterexamples given in the proof of Claim \ref{claim-02} are also counter examples of the equality \cite[(4.7)]{Hung1990} that we copy as in the original:
\begin{quote}
{\small
	(4.7)\qquad\qquad\qquad $H^*(F(\R^q,2^n)/\Sym_{2^n})=i^*(F,q)R\oplus i^*(F,q)(\langle W_{n,0}\rangle)$.
	}
\end{quote}
\end{remark}

%====
\subsection{Proof of the Injectivity Theorem}
\label{subsec : proof of injectivity}
%====

In this section we prove that the map
\[
\rho_{d,2^m}^*\colon H^*(\conf(\R^d,2^m)/\Sym_{2^m};\F_2)\longrightarrow H^*(\Sp(\R^d,2^m)/\Sy_{2^m};\F_2)
\]
is a monomorphism for all $d\geq 2$ and all $m\geq 0$.
For the case when $d=\infty$, in Remark \ref{remark : inj - 02}, we already explained why $\rho_{\infty,2^m}^*$ is a monomorphism.

\medskip
The proof we present is by induction on $m$.
For $m=0$ we have that 
\[
\Sp(\R^d,2^0)/\Sy_{2^0}=\Sp(\R^d,2^0)=\{\pt\}
\qquad\text{and}\qquad
\conf(\R^d,2^0)/\Sym_{2^0}=\conf(\R^d,2^0)=\R^d,
\]
where $\epicy_{d,2^0}\colon \Sp(\R^d,2^0)\longrightarrow\conf(\R^d,2^0)$ is given my $\pt\longmapsto 0\in\R^d$.
Thus, $\rho_{d,2^0}^*$ is obviously an isomorphism and consequently a monomorphism.
Let $m\geq 1$, and let us assume that 
\[
\rho_{d,2^{m-1}}^*\colon H^*(\conf(\R^d,2^{m-1})/\Sym_{2^{m-1}};\F_2)\longrightarrow H^*(\Sp(\R^d,2^{m-1})/\Sy_{2^{m-1}};\F_2)
\]
is a monomorphism.
From Corollary \ref{cor : inj implies inj} we have that the map:
\begin{multline*}
(\rho_{d,2^{m-1}} \times \rho_{d,2^{m-1}})\times_{\Z_2} \id\colon \\
\qquad \qquad (\Sp(\R^d,2^{m-1})/\Sy_{2^{m-1}}\times \Sp(\R^d,2^{m-1})/\Sy_{2^{m-1}})\times_{\Z_2} S^{d-1}
\longrightarrow \hfill \\
(\conf(\R^d,2^{m-1})/\Sym_{2^{m-1}}\times \conf(\R^d,2^{m-1})/\Sym_{2^{m-1}})\times_{\Z_2} S^{d-1}	
\end{multline*}
induces a monomorphism $((\rho_{d,2^{m-1}} \times \rho_{d,2^{m-1}})\times_{\Z_2} \id)^*$ in cohomology.
Since from Definition \ref{def : Epicycles} we know that 
\[
\Sp(\R^d,2^{m})/\Sy_{2^{m}}=(\Sp(\R^d,2^{m-1})/\Sy_{2^{m-1}}\times \Sp(\R^d,2^{m-1})/\Sy_{2^{m-1}})\times_{\Z_2} S^{d-1}
\]
we have obtained the monomorphism $((\rho_{d,2^{m-1}} \times \rho_{d,2^{m-1}})\times_{\Z_2} \id)^*$ in cohomology:
\begin{multline*}
H^*((\conf(\R^d,2^{m-1})/\Sym_{2^{m-1}}\times \conf(\R^d,2^{m-1})/\Sym_{2^{m-1}})\times_{\Z_2} S^{d-1};\F_2)
\longrightarrow  \\
H^*(\Sp(\R^d,2^{m})/\Sy_{2^{m}};\F_2).	
\end{multline*}

\medskip
Next, consider the following diagram of spaces
\begin{equation}
\label{diagram up to heq} 	
\xymatrix@1{
\Sp(\R^d,2^{m})/\Sy_{2^{m}} \ \ar[rr]^-{\rho_{d,2^m}}\ar[dd]_{(\rho_{d,2^{m-1}})^2 \times_{\Z_2} \id} & &\ \conf(\R^d,2^m)/\Sym_{2^m}    \\
& & \\
(\conf(\R^d,2^{m-1})/\Sym_{2^{m-1}})^2\times_{\Z_2} S^{d-1} & &  \\
& & \\
(\conf(\R^d,2^{m-1})/\Sym_{2^{m-1}})^2\times_{\Z_2} \conf(\R^d,2)\ar[uu]^{(\id)^2\times_{\Z_2}r} & &  \\
& & \\
(\CC_d(2^{m-1})/\Sym_{2^{m-1}})^2 \times_{\Z_2}\CC_d(2)\ar[uu]^{(\ev_{d,2^{m-1}}/\Sym_{2^{m-1}})^2\times_{\Z_2} \ev_{d,2}} \ \ar[rr]^-{\mu_{d,2^{m}}} & & \ \CC_d(2^{m})/\Sym_{2^{m}}\ar[uuuuuu]^{\ev_{d,2^m}/\Sym_{2^{m}}}
}
\end{equation}
which ``commutes up to a homotopy.'' 
Here 
\begin{compactitem}[\ ---]
\item $r\colon \conf(\R^d,2)\longrightarrow S^{d-1}$ is the deformation retraction $r(x_1,x_2):=\frac{x_1-x_2}{\|x_1-x_2\|}$,	
\item $\CC_d$ is the little $d$-cubes operad\index{little cubes operad}, consult Definition \ref{def : little cube operad},
\item $\ev_{d,n}\colon\CC_d(n)\longrightarrow \conf(\R^d,n)$ is the evaluation map introduced in \eqref{eq : definition of ev map} which is an $\Sym_n$-equivariant homotopy equivalence, as stated in Lemma \ref{lemma : little cube -- configuration space}, and
\item $\mu_{d,2^{m}}$ is induced by the structural map of the little $d$-cubes operad:
\[
\mu\colon (\CC_d(2^{m-1}) \times \CC_d(2^{m-1}))\times\CC_d(2)\longrightarrow  \CC_d(2^{m}).
\]
(See Definition \ref{def : little cube operad}.)
\end{compactitem}
What we mean by ``commutes up to a homotopy'' here is that the diagram \eqref{diagram up to heq}, after substituting the maps $(\id)^2\times_{\Z_2}r$ and $(\ev_{d,2^{m-1}}/\Sym_{2^{m-1}})^2\times_{\Z_2} \ev_{d,2}$ with its homotopy inverses, becomes commutative up to a homotopy.
Since we know that the maps $(\id)^2\times_{\Z_2}r$, $(\ev_{d,2^{m-1}}/\Sym_{2^{m-1}})^2\times_{\Z_2} \ev_{d,2}$ and $\ev_{d,2^m}/\Sym_{2^m}$ are homotopy equivalences, and $((\rho_{d,2^{m-1}} \times \rho_{d,2^{m-1}})\times_{\Z_2} \id)^*$ is an injection by induction hypothesis, we obtained the following claim. 
\begin{lemma}
	\label{lem : if 1-1 then 1-1}
	If the homomorphism 
\begin{multline*}
(\mu_{d,2^{m}})^*\colon H^*(\CC_d(2^{m})/\Sym_{2^{m}};\F_2)\longrightarrow \\
H^*((\CC_d(2^{m-1})/\Sym_{2^{m-1}} \times \CC_d(2^{m-1})/\Sym_{2^{m-1}})\times_{\Z_2}\CC_d(2);\F_2)	
\end{multline*}
is a monomorphism, then the homomorphism
\[
\rho_{d,2^m}^*\colon H^*(\conf(\R^d,2^m)/\Sym_{2^m};\F_2)\longrightarrow H^*(\Sp(\R^d,2^m)/\Sy_{2^m};\F_2)
\]
is also a monomorphism. 
\end{lemma}

\medskip
Hence, the induction step, in the proof of Theorem \ref{th : injection }, would follow from the proof of injectivity of the homomorphism  $(\mu_{d,2^{m}})^*$. 
To conclude the induction step, and consequently complete the proof of Theorem \ref{th : injection }, we show the following dual theorem.

\begin{theorem}
\label{th : surjection }
Let $d\geq 2$ and $m\geq 1$ be integers.
The homomorphism
\begin{multline}
\label{surjectivity-02}
(\mu_{d,2^{m}})_*\colon 
H_*((\CC_d(2^{m-1})/\Sym_{2^{m-1}} \times \CC_d(2^{m-1})/\Sym_{2^{m-1}})\times_{\Z_2}\CC_d(2);\F_2) \\
\longrightarrow
H_*(\CC_d(2^{m})/\Sym_{2^{m}};\F_2)
\end{multline}
	is an epimorphism.
\end{theorem}

For the proof of the theorem we use results of many authors that we first review in generality we need, and 
then in Section \ref{subsub : Poof of Theorem surj} we give the proof of Theorem \ref{th : surjection }. 
This will finalize the proof of Theorem \ref{th : injection }.

%----------------------------------
\subsubsection{Prerequisites}
\label{subsub : Prerequisites}
%----------------------------------

Let $X$ be a path-connected space in $\Ctoppt$\index{category of compactly generated weak
Hausdorff spaces with non-degenerate base points}.
The free $\CC_d$-space\index{$\CC_d$-space} generated by $X$ is defined as the quotient space
\[
\CC_d(X):=\Big(\coprod_{m\geq 0}\CC_d(m)\times_{\Sym_m}X^m\Big)/_{\approx},
\]
where for $(\vec{c}_1,\ldots, \vec{c}_m)\in \CC_d(m)$ and $(x_1,\ldots,x_{m-1},x_m)\in X^m$, with $x_m=\pt$ the base point, we define
\[
((\vec{c}_1,\ldots, \vec{c}_{m-1}, \vec{c}_m), (x_1,\ldots,x_{m-1},x_m))\approx
((\vec{c}_1,\ldots, \vec{c}_{m-1}), (x_1,\ldots,x_{m-1})).
\]
(For more details consult \cite[Cons.\,2.4]{May1972} or Section \ref{subsub : free C_d space}.)
The space $\CC_d(X)$ is equipped with the natural filtration\index{filtration} given by the number of cubes, that is for $k\geq 0$ we define
\begin{multline*}
\TF_k\CC_d(X):=\im\Big(\coprod_{0\leq m\leq k}\CC_d(m)\times_{\Sym_m}X^m\longrightarrow \\
\coprod_{m\geq 0}\CC_d(m)\times_{\Sym_m}X^m\longrightarrow
\Big(\coprod_{m\geq 0}\CC_d(m)\times_{\Sym_m}X^m\Big)/_{\approx}\Big),	
\end{multline*}
where the first map is the obvious inclusion and the second map is the identification map.
Thus, we obtained the filtration of $\CC_d(X)$:
\[
\emptyset=\TF_{-1}\CC_d(X) \subseteq \TF_0\CC_d(X)\subseteq \TF_1\CC_d(X)\subseteq\dots\subseteq \TF_{k-1}\CC_d(X)\subseteq \TF_{k}\CC_d(X) \subseteq \cdots ,
\]
where each pair of spaces $(\TF_{k}\CC_d(X),\TF_{k-1}\CC_d(X))$  is an NDR-pair; see \cite[Prop.\,2.6]{May1972}. 
Further on we denote successive quotients by
\[
\TD_k\CC_d(X):=\TF_k\CC_d(X)/\TF_{k-1}\CC_d(X).
\]
Recall that we have set $\TF_{-1}\CC_d(X)=\emptyset$.
Next, we introduce the real $k$-dimensional vector bundle\index{vector bundle} $\xi_{d,k}$ over the quotient space $\CC_d(k)/\Sym_k$  by
\begin{equation}\label{eq : definition of xi_d,k}
\xymatrix{
\R^k \ \ar[r] &\ \CC_d(k)\times_{\Sym_k}\R^k \ \ar[r] & \ \CC_d(k)/\Sym_k,
}	
\end{equation}
where $\R^k$ is assumed to be the real $\Sym_k$-representation with the action given by permutation of the coordinates.
The facts we are going to use --- in the case when the space $X$ is a sphere --- are collected in the following theorem, see \cite[Thm.\,A-C]{Bodigheimer1989}, \cite{Cohen1978-2} and \cite[Thm.\,2.6]{Cohen1979}.

\begin{theorem}
	\label{th : facts - 01}
	Let $L\geq 1$ and $N\geq 1$ be integers.
	\begin{compactenum}[\rm \  (1)]
	\item The space $\TD_k\CC_d(S^L)$ is homeomorphic to the Thom space\index{Thom space} of the vector bundle $\xi_{d,k}^{\oplus L}$, that is $\TD_k\CC_d(S^L)\approx \Th(\xi_{d,k}^{\oplus L})$.
	Consequently, for every $i\geq 0$ there is the Thom isomorphism
	\begin{multline*}
		\widetilde{H}_{i+Lk}(\TD_k\CC_d(S^L);\F_2)\cong \widetilde{H}_{i+Lk}(\Th(\xi_{d,k}^{\oplus L}))\cong\\ H_i(\CC_d(k)/\Sym_k;\F_2)\cong H_i(\conf(\R^d,k)/\Sym_k;\F_2).
	\end{multline*}
	\item For $N$ large enough there is a homotopy equivalence
	\[
	\Sigma^N \big(\TF_k\CC_d(S^L)\big)\simeq \Sigma^N\big(\TD_k\CC_d(S^L)\vee \TF_{k-1}\CC_d(S^L)  \big).
	\]
	
	\end{compactenum}
\end{theorem}

\noindent
Here $\Sigma(X)$ denotes the suspension of the spaces $X$.

\medskip
The Approximation theorem\index{Approximation theorem} of May, Theorem \ref{th : approximation}, applied to the sphere $S^L$, where $L\geq 1$ is an integer, yields the weak homotopy equivalence 
\[
	\alpha_d\colon\CC_d(S^L)\longrightarrow\Omega^d\Sigma^dS^L.
\]
In particulart, $\alpha_d$ induces the isomorphism in homology with $\F_2$ coefficients:
\[
(\alpha_d)_*\colon H_*(\CC_d(S^L);\F_2)\longrightarrow H_*(\Omega^d\Sigma^dS^L;\F_2).
\] 
(The results in \cite[Sec.\,3]{Cohen1976LNM533} describe further properties of the homology isomorphism $(\alpha_d)_*$.)  
The homology of the iterated loop space\index{iterated loop space} $\Omega^d\Sigma^dS^L=\Omega^d S^{d+L}$ with $\F_2$ coefficients was described as a  Pontryagin ring\index{Pontryagin ring} by Araki and Kudo in their seminal paper \cite[Thm.\,7.1]{Araki1956} using, what we call now, Araki--Kudo--Dyer--Lashof homology operations\index{Araki--Kudo--Dyer--Lashof homology operations}; see Section \ref{sub : Araki--Kudo--Dyer--Lashof operations} for more details.

\begin{theorem}
	\label{th : homology of iterated loop space}
	Let $d\geq 1$ and $L\geq 1$ be integers.
	The homology $H_*(\Omega^d\Sigma^dS^L;\F_2)$, as a Pontryagin ring\index{Pontryagin ring}, is a polynomial algebra generated by a generator $u_L$ and all $Q_{i_1}Q_{i_2}\cdots Q_{i_s}u_L$ where $1\leq i_1\leq i_2\leq\cdots\leq i_s\leq d-1$.
	Furthermore, $\deg(u_L)=L$ and 
	\[
	\deg(Q_{i_1}Q_{i_2}\cdots Q_{i_s}u_L)=i_1+2i_2+4i_3+\cdots +2^{s-1}i_s+2^sL.
	\]
	We use notation
	\begin{multline*}
		H_*(\Omega^d\Sigma^dS^L;\F_2)\cong \\
\F_2\big[\{u_L\}\cup\{Q_{i_1}Q_{i_2}\cdots Q_{i_s}u_L  \, : \, s\geq 1, \, 1\leq i_1\leq i_2\leq\cdots\leq i_s\leq d-1\}\big].
	\end{multline*}
\end{theorem}

\begin{remark}
	\label{rem : geometric meaning}
	The generators of the homology $H_*(\Omega^d\Sigma^dS^L;\F_2)$ have a concrete geometric description.
	Consider the map $e_d\colon S^L\longrightarrow \Omega^d\Sigma^dS^L$ that corresponds to the identity map $\id\colon\Sigma^dS^L\longrightarrow\Sigma^dS^L$ along the adjunction $[A,\Omega^dB]_{\pt}\longleftrightarrow [\Sigma^dA,B]_{\pt}$.
	Here $[\cdot,\cdot]_{\pt}$ denotes the set of all homotopy classes of pointed maps  between the pointed spaces.
	After taking $d$-fold suspension, or the smash product with the sphere $S^d$, we get the following commutative triangle
	\[
	\xymatrix{
	\Sigma^dS^L =S^d \wedge S^L \ \ar[rr]^-{\id \wedge e_d}\ar[ddrr]_{\id} & &\ \Sigma^d\Omega^d\Sigma^dS^L= S^d\wedge (\Omega^d\Sigma^dS^L)\ar[dd]^{ev_d}\\
	 & & \\
	 & & \Sigma^dS^L = S^d\wedge S^L,
	}
	\]
	where $ev_d\colon \Sigma^d\Omega^dX\longrightarrow X$ is the evaluation map. 
	Thus, the induced map in homology 
	\[
	(e_d)_*\colon H_*(S^L;\F_2)\longrightarrow H_*(\Omega^d\Sigma^dS^L;\F_2)
	\]
	is a monomorphism. 
	The generator $u_L$, the so called fundamental class of $\Omega^d\Sigma^dS^L$, is the image of the generator of $H_L(S^L;\F_2)$ along $(e_d)_*$.
	Now the generators $Q_{i_1}Q_{i_2}\cdots Q_{i_s}u_L$ for $1\leq i_1\leq i_2\leq\cdots\leq i_s\leq d-1$ are images of $u_L$ under the sequence $Q_{i_1}Q_{i_2}\cdots Q_{i_s}$ of  Araki--Kudo--Dyer--Lashof homology operations\index{Araki--Kudo--Dyer--Lashof homology operations}.
	Thus we have complete description of the homology $H_*(\Omega^d\Sigma^dS^L;\F_2)$.
\end{remark}
 
\medskip
Next we define a filtration\index{filtration} of the polynomial algebra 
\begin{multline*}
\F_2\big[\{u_L\}\cup\{Q_{i_1}Q_{i_2}\cdots Q_{i_s}u_L  \, : \, s\geq 1, \, 1\leq i_1\leq i_2\leq\cdots\leq i_s\leq d-1\}\big]
\cong \\
H_*(\Omega^d\Sigma^dS^L;\F_2)	
\end{multline*}
by defining the weight function $\omega$ on its monomials as follows:
\[
\omega(u_L)=1,\qquad\qquad
\omega(Q_iu)=2\omega(u),\qquad\qquad
\omega(v_1\cdot v_2)=\omega(v_1)+\omega(v_2),
\]
where $u$ is an algebra generator, and $v_1,v_2$ are monomials is algebra generators.
In particular, the weight of algebra generators are alway powers of two, that is
\[
\omega (Q_{i_1}Q_{i_2}\cdots Q_{i_s}u_L)= 2\omega (Q_{i_2}Q_{i_3}\cdots Q_{i_s}u_L)=\cdots=2^s.
\]
Now, for $k\geq 0$, we set $\AF_kH_*(\Omega^d\Sigma^dS^L;\F_2)$ to be the vector subspace of the polynomial ring $H_*(\Omega^d\Sigma^dS^L;\F_2)$ generated by all monomials with weight at most $k$.
In this way we obtained a filtration of the polynomial ring $H_*(\Omega^d\Sigma^dS^L;\F_2)$ by vector spaces:
\begin{multline*}
0=\AF_0H_*(\Omega^d\Sigma^dS^L;\F_2)\subseteq\AF_1H_*(\Omega^d\Sigma^dS^L;\F_2)\subseteq\\ \cdots\subseteq\AF_{k-1}H_*(\Omega^d\Sigma^dS^L;\F_2)\subseteq\AF_{k}H_*(\Omega^d\Sigma^dS^L;\F_2)\subseteq\cdots ,
\end{multline*}
the so called {\bf weight filtration}.
Furthermore we denote the sequence of quotients by
\[
\AD_kH_*(\Omega^d\Sigma^dS^L;\F_2):=\AF_kH_*(\Omega^d\Sigma^dS^L;\F_2)/\AF_{k-1}H_*(\Omega^d\Sigma^dS^L;\F_2).
\]

\begin{example}
In order to illustrate the notion of the weight filtration in the following table we list all the generators of $H_*(\Omega^d\Sigma^dS^L;\F_2)$ with weight at most $4$.

\begin{center}
  \begin{tabular}{  c  c   l  l}
  \hline
   monomial   & weight  & degree   \\ \hline\noalign{\smallskip}
    $u_L$     & $1$     & $L$ &\\ \arrayrulecolor{gray}\hline\noalign{\smallskip} 
    $u_L^2$   & $2$     & $2L$ &\\ \hline\noalign{\smallskip} 
    $Q_iu_L$  & $2$     & $i+2L$&  $1\leq i\leq d-1$ \\ \hline\noalign{\smallskip}
    $u_L^3$   & $3$     & $3L$ &\\ \hline\noalign{\smallskip}
   $u_LQ_iu_L$& $3$     & $i+3L$&  $1\leq i\leq d-1$\\ \hline\noalign{\smallskip}
   $u_L^4$    & $4$     & $4L$ &\\ \hline\noalign{\smallskip}
 $u_L^2Q_iu_L$& $4$     & $i+4L$&  $1\leq i\leq d-1$\\ \hline\noalign{\smallskip} 
$Q_iu_LQ_ju_L$& $4$     & $i+j+4L$&  $1\leq i,j\leq d-1$\\ \hline\noalign{\smallskip}
  $Q_iQ_ju_L$ & $4$     & $i+2j+4L$&  $1\leq i\leq j \leq d-1$ \\ \arrayrulecolor{black}\hline\noalign{\smallskip}
\end{tabular}
\end{center} 
Thus, for example:
\begin{align*}
	\AF_1H_*(\Omega^d\Sigma^dS^L;\F_2)&=\text{the }\F_2\text{-vector space with a basis }\{u_L\},\\
	\AF_2H_*(\Omega^d\Sigma^dS^L;\F_2)&=\text{the }\F_2\text{-vector space with a basis }\{u_L,u_L^2,Q_iu_L\},\\
	\AF_3H_*(\Omega^d\Sigma^dS^L;\F_2)&=\text{the }\F_2\text{-vector space with a basis } \\ 
	&\quad \ \{u_L,u_L^2,Q_iu_L,u_L^3,u_LQ_iu_L\}.
\end{align*}
Consequently,
\begin{align*}
\AD_1H_*(\Omega^d\Sigma^dS^L;\F_2)&=\text{the }\F_2\text{-vector space with a basis }\{u_L\},\\
\AD_2H_*(\Omega^d\Sigma^dS^L;\F_2)&=\text{the }\F_2\text{-vector space with a basis }\{u_L^2,Q_iu_L\},\\
\AD_3H_*(\Omega^d\Sigma^dS^L;\F_2)&=\text{the }\F_2\text{-vector space with a basis }\{u_L^3,u_LQ_iu_L\},\\
\AD_4H_*(\Omega^d\Sigma^dS^L;\F_2)&=\text{the }\F_2\text{-vector space with a basis } \\ 
&\quad \ \{u_L^4,u_L^2Q_iu_L,Q_iu_LQ_ju_L,Q_iQ_ju_L\}.
\end{align*}
For more details calculation of this type see \cite[Sec.\,5]{Bodigheimer1989}.
\end{example}

\medskip
The central property of the filtration\index{filtration} we just defined is given in the following theorem, see \cite[Cor.\,3.3]{Cohen1976LNM533}.

\begin{theorem}
	\label{th : isomorphism of filtrations}
	Let $d\geq 2$ and $L\geq 1$ be integers.
	For every $k\geq 0$ there is an isomorphism of graded vector spaces
	\[
	\alpha_{d,k}\colon  H_*(\TF_k\CC_d(S^L);\F_2) \longrightarrow \AF_kH_*(\Omega^d\Sigma^dS^L;\F_2) .
	\]
	Moreover, for every $k\geq 0$ there is an isomorphism of vector spaces
	\[
	\overline{\alpha}_{d,k}\colon \widetilde{H}_*(\TD_k\CC_d(S^L);\F_2)  \longrightarrow  \AD_kH_*(\Omega^d\Sigma^dS^L;\F_2),
	\]
	such that the following diagram commutes
	\[
	\xymatrix{
	 H_*(\TF_k\CC_d(S^L);\F_2)\  \ar[rr]^-{\alpha_{d,k}}\ar[d] & & \ \AF_kH_*(\Omega^d\Sigma^dS^L;\F_2)\ar[d]\\
	 \widetilde{H}_*(\TD_k\CC_d(S^L);\F_2) \ \ar[rr]^-{\overline{\alpha}_{d,k}} & &\ \AD_kH_*(\Omega^d\Sigma^dS^L;\F_2).
	}
	\]
	Here the left vertical map is induced by the quotient map of topological spaces
	\[
	\TF_k\CC_d(S^L)\longrightarrow \TD_k\CC_d(S^L)=\TF_k\CC_d(S^L)/\TF_{k-1}\CC_d(S^L),
	\]
	while the right vertical map is the quotient map of vector spaces
	\begin{multline*}
\AF_kH_*(\Omega^d\Sigma^dS^L;\F_2)\longrightarrow \AD_kH_*(\Omega^d\Sigma^dS^L;\F_2) \\ =\AF_kH_*(\Omega^d\Sigma^dS^L;\F_2)/\AF_{k-1}H_*(\Omega^d\Sigma^dS^L;\F_2).
	\end{multline*}
\end{theorem}

\noindent
The results presented in \cite[Sec.\,3 and Sec.4]{Cohen1976LNM533} imply that the isomorphisms $\alpha_{d,k}$ and $\overline{\alpha}_{d,k}$ are induced by the isomorphism $(\alpha_d)_*$.

\medskip
Now, directly from Theorem \ref{th : facts - 01} and Theorem \ref{th : isomorphism of filtrations}, we get a description of homology of the unordered configuration space\index{unordered configuration space} with $\F_2$ coefficients, see \cite[Sec.\,4.4]{Bodigheimer1989}.  

\begin{corollary}
	\label{cor : homology of configuration space}
	Let $d\geq 2$, $k\geq 1$ and $L\geq 1$ be integers. 
	There is an isomorphism of graded vector space
	\[
	H_{*-kL}(\conf(\R^d,k)/\Sym_k;\F_2)\cong \AD_kH_*(\Omega^d\Sigma^dS^L;\F_2).
	\]
\end{corollary}

Alternatively, we can describe the homology of the unordered configurations space $\conf(\R^d,k)/\Sym_k$ with $\F_2$ coefficients in the following way.
 
\begin{corollary}
\label{cor : homology of configuration space concrete}
Let $d\geq 2$, $k\geq 1$, $L\geq 1$ and $0\leq i\leq (d-1)(k-1)$ be integers. 
The homology of the unordered configurations space
\[
H_i(\conf(\R^d,k)/\Sym_k;\F_2)
\]
is isomorphic to the $\F_2$ vector space spanned by all monomials of degree $i+kL$ of the polynomial ring
\[
\F_2\big[\{u_L\}\cup\{Q_{i_1}Q_{i_2}\cdots Q_{i_s}u_L \, : \, s\geq 1, \, 1\leq i_1\leq i_2\leq\cdots\leq i_s\leq d-1\}\big]
\]
whose weights are exactly $k$.
\end{corollary}	

In order to illustrate the previous result we give a few simple examples of evaluation of the homology of the unordered configuration space.

\begin{example}
From Corollary \ref{cor : homology of configuration space concrete} we have that the homology $H_*(\conf(\R^d,k)/\Sym_k;\F_2)$ of the unordered configuration space is

\medskip
\begin{compactenum}[\rm \ (1)]

\item for the case $d=2$ and $k=2$ spanned in 
\begin{center}
\begin{tabular}{l l} \noalign{\smallskip}
	dimension $0$ by the monomial: &\quad 	$u_L^2$, \\ \noalign{\smallskip}
	dimension $1$ by the monomial:  &\quad 	$Q_1u_L$,\\ \noalign{\smallskip}
\end{tabular}
\end{center}
(Thus, indeed $H_*(\conf(\R^2,2)/\Sym_2;\F_2)\cong H_*(\RP^1;\F_2)$ as expected.)\smallskip
 
\item for the case $d=2$ and $k=3$ spanned in 
\begin{center}
\begin{tabular}{l l} \noalign{\smallskip}
	dimension $0$ by the monomial: &\quad 	$u_L^3$, \\ \noalign{\smallskip}
	dimension $1$ by the monomial:  &\quad 	$u_L^3\,(Q_1u_L)$,\\ \noalign{\smallskip}
\end{tabular}
\end{center} \smallskip

\item for the case $d=2$ and $k=4$ spanned in 
\begin{center}
\begin{tabular}{l l} \noalign{\smallskip}
	dimension $0$ by the monomial: &\quad 	$u_L^4$, \\ \noalign{\smallskip}
	dimension $1$ by the monomial:  &\quad 	$u_L^2\,(Q_1u_L)$,\\ \noalign{\smallskip}
	dimension $2$ by the monomial:  &\quad 	$(Q_1u_L)\,(Q_1u_L)$,\\ \noalign{\smallskip}
	dimension $3$ by the monomial:  &\quad 	$Q_1Q_1u_L$,\\ \noalign{\smallskip}
\end{tabular}
\end{center} \smallskip

\item for the case $d=3$ and $k=2$ spanned in 
\begin{center}
\begin{tabular}{l l} \noalign{\smallskip}
	dimension $0$ by the monomial: &\quad 	$u_L^2$, \\ \noalign{\smallskip}
	dimension $1$ by the monomial:  &\quad 	$Q_1u_L$,\\ \noalign{\smallskip}
	dimension $1$ by the monomial:  &\quad 	$Q_2u_L$,\\ \noalign{\smallskip}
\end{tabular}
\end{center}
(Again, we see that $H_*(\conf(\R^3,2)/\Sym_2;\F_2)\cong H_*(\RP^2;\F_2)$ as expected.)\smallskip

\item for the case $d=3$ and $k=4$ spanned in 
\begin{center}
\begin{tabular}{l l} \noalign{\smallskip}
	dimension $0$ by the monomial: &\quad 	$u_L^4$, \\ \noalign{\smallskip}
	dimension $1$ by the monomial:  &\quad 	$u_L^2\,(Q_1u_L)$,\\ \noalign{\smallskip}
	dimension $2$ by the monomials:  &\quad $u_L^2\,(Q_2u_L)$ and	$(Q_1u_L)\,(Q_1u_L)$,\\ \noalign{\smallskip}
	dimension $3$ by the monomials:  &\quad $(Q_1u_L)\,(Q_2u_L)$ and	$Q_1Q_1u_L$,\\ \noalign{\smallskip}
	dimension $4$ by the monomial:  &\quad $(Q_2u_L)\,(Q_2u_L)$,\\ \noalign{\smallskip}
	dimension $5$ by the monomial:  &\quad $Q_1Q_2u_L$,\\ \noalign{\smallskip}
	dimension $6$ by the monomial:  &\quad $Q_2Q_2u_L$,\\ \noalign{\smallskip}
\end{tabular}
\end{center} 
Thus, we have that
\[
H_i(\conf(\R^3,4)/\Sym_4;\F_2)\cong
\begin{cases}
	\F_2, & i=0,1,4,5,6,\\
	\F_2^{\oplus 2}, &i=3,4,\\
	0, &\text{otherwise}.
\end{cases}
\]
\end{compactenum}

\end{example}

\medskip
With this result we collected all necessary ingredients that we need for the proof of Theorem \ref{th : surjection }.

%----------------------------------
\subsubsection{Proof of the dual Epimorphism Theorem}
\label{subsub : Poof of Theorem surj}
%----------------------------------

Let the integers $d\geq 2$ and $m\geq 1$ be fixed.
Now we prove that the following homomorphism, induced by a structural map of the little cubes operad $\CC_d$,  
\begin{multline*}
(\mu_{d,2^{m}})_*\colon 
H_*((\CC_d(2^{m-1})/\Sym_{2^{m-1}} \times \CC_d(2^{m-1})/\Sym_{2^{m-1}})\times_{\Z_2}\CC_d(2);\F_2) 
\longrightarrow \\
H_*(\CC_d(2^{m})/\Sym_{2^{m}};\F_2)	
\end{multline*}
is an epimorphism.

\medskip
Let $L\geq 1$ be an arbitrary integer, and consider the free $\CC_d$-space\index{$\CC_d$-space} $\CC_d(S^L)$.
From Theorem \ref{th : homology of iterated loop space}, using the the weak homotopy equivalence 
\[
\alpha_d\colon\CC_d(S^L)\longrightarrow\Omega^d\Sigma^dS^L,
\] 
we get the following isomorphisms
% The first isomorphism is a multiplicative isomorphism
\begin{multline*}
H_*(\CC_d(S^L);\F_2)\cong H_*(\Omega^d\Sigma^dS^L;\F_2) \cong \bigoplus_{k\geq 0} \AD_kH_*(\Omega^d\Sigma^dS^L;\F_2)\cong \\ \F_2\big[\{u_L\}\cup\{Q_{i_1}Q_{i_2}\cdots Q_{i_s}u_L  \, : \, s\geq 1, \, 1\leq i_1\leq i_2\leq\cdots\leq i_s\leq d-1\}\big].
\end{multline*}
Since $\CC_d(S^L)$ is a $\CC_d$-space\index{$\CC_d$-space}, as explained in Lemma \ref{lemma : equivalent definition of action of an operad}, there exists a map
\[
\xymatrix@1{
\Theta_2\colon  (\CC_d(S^L)\times \CC_d(S^L))\times \CC_d(2) \ar[r]& \CC_d(S^L).}
\]
(In the following we will abuse notation and all maps induced by $\Theta_2$ will be denoted in the same way.)
On the level of space filtration, for integers $k_1\geq 0$ and $k_2\geq 0$, we have that $\Theta_2$ induces a map
\[
\xymatrix@1{
\Theta_2\colon  (\TF_{k_1}\CC_d(S^L)\times \TF_{k_2}\CC_d(S^L)) \times \CC_d(2)\ar[r]& \TF_{k_1+k_2}\CC_d(S^L).
}
\]
Thus, for any integer $k\geq 0$ we have an induced map on quotient spaces 
\[
\xymatrix@1{
\Theta_2 \colon  (\TD_{k}\CC_d(S^L)\times \TD_{k}\CC_d(S^L))\times \CC_d(2)\ar[r]&  \TD_{2k}\CC_d(S^L).
}
\]

The symmetric group\index{symmetric group} $\Sym_2\cong\Z_2$ acts naturally on: the space $\CC_d(2)$ of pairs of little $d$-cubes  by interchanging the cubes, and on the product $\TD_{k}\CC_d(S^L)\times \TD_{k}\CC_d(S^L)$ by interchanging the factors.
If the trivial action on the space $\TD_{2k}\CC_d(S^L)$ is assumed, the equivariance property of the structural map of the little cubes operad\index{little cubes operad} implies that the last $\Theta_2$ map is an $\Sym_2$-equivariant map.
Consequently, it induces  maps such that the following diagram commutes
\[
\xymatrix@1{
	 (\CC_d(S^L)\times \CC_d(S^L)) \times_{\Z_2} \CC_d(2) \ \ar[r]^-{\Theta_2}  & \ \CC_d(S^L)\\
%	 & &\\
	 (\TF_{k}\CC_d(S^L)\times \TF_{k}\CC_d(S^L)) \times_{\Z_2}\CC_d(2) \ar[r]^-{\Theta_2}\ar[u]\ar[d] \ & \ \TF_{2k}\CC_d(S^L)\ar[u]\ar[d]\\
%	 & &\\
	     (\TD_{k}\CC_d(S^L)\times \TD_{k}\CC_d(S^L))\times_{\Z_2}\CC_2(2) \ar[r]^-{\Theta_2} \ & \ \TD_{2k}\CC_d(S^L).
}
\]

\noindent
Here the vertical maps are either inclusions or quotient maps. 
Passing to homology we get the following commutative diagram:

\begin{equation}
\label{diagram : homology}
\xymatrix@1{
	 H_*((\CC_d(S^L)\times \CC_d(S^L)) \times_{\Z_2} \CC_d(2);\F_2)\ar[r]^-{(\Theta_2)_*} \ & \ H_*(\CC_d(S^L);\F_2)\\
%	 & &\\
	 H_*((\TF_{k}\CC_d(S^L)\times \TF_{k}\CC_d(S^L)) \times_{\Z_2}\CC_d(2);\F_2) \ar[r]^-{(\Theta_2)_*}\ar[u]\ar[d]  \ &\  H_*(\TF_{2k}\CC_d(S^L);\F_2)\ar[u]\ar[d]\\
%	 & &\\
	    H_*( (\TD_{k}\CC_d(S^L)\times \TD_{k}\CC_d(S^L))\times_{\Z_2}\CC_2(2);\F_2) \ar[r]^-{(\Theta_2)_*} \ & \ H_*(\TD_{2k}\CC_d(S^L);\F_2).
}	
\end{equation}

\noindent
In the case when $k=2^{m-1}$ the homomorphism  
 \begin{multline} 
 \label{surjectivity-03}
  (\Theta_2)_*\colon H_*( (\TD_{k}\CC_d(S^L)\times \TD_{k}\CC_d(S^L))\times_{\Z_2}\CC_2(2);\F_2)\longrightarrow \\
  H_*(\TD_{2k}\CC_d(S^L);\F_2)	
 \end{multline}
  coincides with the map $(\mu_{d,2^{m}})_*$ from \eqref{surjectivity-02} --- after appropriate dimension shift by $-2kL$.
 Therefore in order to complete the proof of Theorem \ref{th : surjection } it suffices to prove that the homomorphism $ (\Theta_2)_*$, from \eqref{surjectivity-03}, is surjective.
 
 \medskip
  The map $H_*(\TF_{k}\CC_d(S^L);\F_2)\longrightarrow H_*(\TD_{k}\CC_d(S^L);\F_2)$ is surjective for every $k$.
  Therefore, using Corollary \ref{cor : suj implies sur} and commutativity of the diagram \eqref{diagram : homology}, we should first consider the homomorphism  
\begin{multline}
\label{surjectivity-04}
(\Theta_2)_*\colon H_*( (\TF_{k}\CC_d(S^L)\times \TF_{k}\CC_d(S^L))\times_{\Z_2}\CC_2(2);\F_2) \longrightarrow \\
 H_*(\TF_{2k}\CC_d(S^L);\F_2).	
\end{multline}

\medskip
For $d\geq 2$, let
 \[
 \mathcal{I}_d:=\{ (i_1,\dots,i_s) : s\geq 1, \, 1\leq i_1\leq \cdots\leq i_s\leq d-1 \}
 \]
be the set of all admissible iterations.
Then, every element $I\in \mathcal{I}_d$ defines the iterated Araki--Kudo--Dyer--Lashof homology operation\index{Araki--Kudo--Dyer--Lashof homology operations}
 \[
 Q_I:= Q_{i_1}Q_{i_2}\cdots Q_{i_s}.
 \]
Hence, as we have seen in Theorem \ref{th : isomorphism of filtrations}, the homology $H_*(\TF_{k}\CC_d(S^L);\F_2)$ can be identified with the vector space spanned by all monomials in, variables from the set $\{ Q_Iu_L : I\in \mathcal{I}_d\}\cup\{u_L\}$ with weight at most $k$. 
Furthermore, from definition of Araki--Kudo--Dyer--Lashof homology operations\index{Araki--Kudo--Dyer--Lashof homology operations} given in Section \ref{sub : Araki--Kudo--Dyer--Lashof operations}, it follows that for every $Q_Iu_L$, $I\in \mathcal{I}_d$, of weight exactly $2^m$, there exists $i\geq 1$ and an element $Q_Ju_L$ or $u_L$, $J\in \mathcal{I}_d$, of weight exactly $2^{m-1}$ such that
 \begin{equation}
 	\label{hitting algebra gen}
 	Q_Iu_L =  (\Theta_2)_* ((Q_Ju_L\otimes Q_Ju_L)\otimes_{\Z_2}f_i)  .
 \end{equation}
It is important to notice that this {\bf does not} mean that the homomorphism $(\Theta_2)_*$ from \eqref{surjectivity-04} is surjective.

\medskip
Now we specialize to the case $k=2^{m-1}$ and analyze the map \eqref{surjectivity-03}:  
\begin{multline*}
(\Theta_2)_*\colon
H_*( (\TD_{2^{m-1}}\CC_d(S^L)\times \TD_{2^{m-1}}\CC_d(S^L))\times_{\Z_2}\CC_d(2);\F_2)
\longrightarrow H_*(\TD_{2^{m}}\CC_d(S^L);\F_2)	
\end{multline*}
and prove that it is surjective.
The vector space $H_*(\TD_{2^{m}}\CC_d(S^L);\F_2)$ is a quotient of   $H_*(\TF_{2^{m}}\CC_d(S^L);\F_2)$.
It can be identified with the vector space spanned by all monomials of weight exactly $2^m$ in variables $\{u_L\}\cup\{ Q_Iu_L : I\in \mathcal{I}_d\}$.
From \eqref{hitting algebra gen} we that all elements of the form $Q_Iu_L$ of weight $2^m$ (variables) are in the image of $(\Theta_2)_*$.
It remains to show that all other monomials are also in $\im ( (\Theta_2)_*)$.

\medskip
Let $Q_{I_1}u_L\cdots Q_{I_t}u_L$ with $t\geq 2$, $I_1,\dots I_t\in \mathcal{I}_d\cup\{0\}$ and $Q_0u_L:=u_L$, be a typical element of the basis of $H_*(\TD_{2^{m}}\CC_d(S^L);\F_2)$ which is not a variable.
Then, 
\[
2^m = \omega(Q_{I_1}u_L\cdots Q_{I_t}u_L)= \omega(Q_{I_1}u_L)+\dots +  \omega(Q_{I_t}u_L) =2^{b_1}+\cdots+2^{b_t},
\] 
where 
\[
\omega (Q_{I_1}u_L)=2^{b_1}, \ \omega (Q_{I_2}u_L)=2^{b_2}, \ \dots, \ \omega (Q_{I_t}u_L)=2^{b_t}.
\]
Without loss of generality, we can assume that there exists an integer $1\leq \ell\leq t-1$ such that
\[
\omega (Q_{I_1}u_L\cdots Q_{I_{\ell}}u_L)=2^{m-1}
\qquad\text{and}\qquad
\omega (Q_{I_{\ell+1}}u_L\cdots Q_{I_{t}}u_L)=2^{m-1}.
\]
Consequently,
\[
Q_{I_1}u_L\cdots Q_{I_{\ell}}u_L \ \in \ H_*(\TD_{2^{m-1}}\CC_d(S^L);\F_2)
\]
and
\[
Q_{I_{\ell+1}}u_L\cdots Q_{I_{t}}u_L \ \in \ H_*(\TD_{2^{m-1}}\CC_d(S^L);\F_2).
\]
The product structure on the homology of $\CC_d$-space $\CC_d(S^L)$ is introduced by an arbitrary element of $(\vec{c}_1,\vec{c}_2)\in\CC_d(2)$ via the map 
\[
\Theta_2(\vec{c}_1,\vec{c}_2)\colon \CC_d(S^L)\times \CC_d(S^L)\longrightarrow \CC_d(S^L),
\] 
see \cite[Lem.\,1.9(i)]{May1972}.
The invariance of the action of the little cubes operad\index{little cubes operad} with respected to the actions of related symmetric groups implies that the product map $\Theta_2(\vec{c}_1,\vec{c}_2)$ factors thought the induced map 
\[\Theta_2\colon  (\CC_d(S^L)\times \CC_d(S^L))\times_{\Z_2} \CC_d(2) \longrightarrow \CC_d(S^L).\]
Hence, on the level of homology with $\F_2$ coefficients we have that
\begin{align*}
Q_{I_1}u_L\dots Q_{I_t}u_L &=(Q_{I_1}u_L\cdots Q_{I_{\ell}}u_L)\cdot (Q_{I_{\ell+1}}u_L\cdots Q_{I_{t}}u_L)\\
&=(\Theta_2)_*\big((Q_{I_1}u_L\cdots Q_{I_{\ell}}u_L)\otimes (Q_{I_{\ell+1}}u_L\cdots Q_{I_{t}}u_L)\otimes_{\Z_2} h_0\big),	
\end{align*}
where $h_0$ is the generator of $H_{0}(\RP^{d-1};\mathcal{M})$, as explained in Section \ref{subsub : homology computation of local coefficients}.
Thus, we have proved that all additive generators of $H_*(\TD_{2^{m}}\CC_d(S^L);\F_2)$ are in the image of $(\Theta_2)_*$ implying that the map $(\Theta_2)_*$ in \eqref{surjectivity-03} is surjective.
This concludes the proof of Theorem \ref{th : surjection } and consequently we have proved Theorem \ref{th : injection }.

%====
\subsection{An unexpected corollary}
\label{subsec : an unexpected corollary}
%====

In this section we take a small detour from the objectives of this book and present an unexpected corollary which is an artefact of the proof of Theorem \ref{th : surjection }.
This result shares the spirit with the the classical result of Michael Atiyah \cite{Atiyah1966} about the complex representation ring of the symmetric group  $R(\Sym_k)\cong K_{\Sym_k}^*(\pt)$ and the modern breakthrough of Chad Giusti, Paolo Salvatore \& Dev Sinha \cite[Thm.\,1.2]{Giusti2012} in describing the cohomology of the symmetric groups collected disjointly together $H^*(\coprod_{k\geq 1} \B\Sym_k;\F_2)$.

%====
\subsubsection{Motivation}
%====

In 1966 Atiyah studied the following direct sum of abelian groups
\[
R_*:=\bigoplus_{k\geq 0}\hom (R(\Sym_k),\Z)\cong \bigoplus_{k\geq 0}\hom (K_{\Sym_k}^*(\pt),\Z)
\] 
which was additionally equipped with the structure of a commutative ring via the product induced by the inclusion maps $\Sym_r\times\Sym_s\longrightarrow\Sym_{r+s}$. 
For details and original presentation  see \cite[p.\,169]{Atiyah1966}.
He showed \cite[Cor.\,1.3]{Atiyah1966} that the ring $R_*$ is isomorphic to the polynomial ring
\[
\Z[\sigma^1,\dots,\sigma^k,\dots]
\]
on generators $\sigma^k\colon R(\Sym_k)\longrightarrow\Z$, for $k\geq 1$, given by the dimension of the fixed $\Sym_k$-submodule. 

\medskip
The study of the homology and cohomology of symmetric groups has a rich and exciting history with many spectacular breakthroughs and crucial applications in different areas of mathematics. 
The major contributions go back to the extraordinary work of Nakaoka \cite{Nakaoka1960} \cite{Nakaoka1961} \cite{Nakaoka1962}, Daniel Quillen \cite{Quillen1971} \cite{Quillen1971AC}, Quillen \& Boris Venkov \cite{Quillen1972}, Hu\`{y}nh M\`{u}i \cite{Mui1975},  Benjamin Mann \cite{Mann1978}, Adem, John Maginnis \& Milgram, \cite{AdemMaginnisMilgram1990}, Adem \& Milgram \cite{Adem2004}, all the way towards the modern results by Mark Feshbach \cite{Feshbach2002} and Giusti, Salvatore \& Sinha \cite{Giusti2012}.

\medskip
A detailed understanding of the cohomology of a particular symmetric group is still a challenging problem in general.
On the other hand, Giusti, Salvatore \& Sinha \cite{Giusti2012}, considering all symmetric groups together, like Atiyah did, gave a compact description of the cohomology $H^*\big(\coprod_{k\geq 1} \B\Sym_k;\F_2\big)$ in the language of Hopf rings where the operations are naturally introduced.

\medskip
In more details, a Hopf ring is a five-tuple $(V,\odot,\cdot,\Delta,S)$ of the vectors space $V$, two multiplications $\odot$ and $\cdot$, one comultiplication $\Delta$, and an antipode $S$ such that:
\begin{compactenum}[\rm \ (1)]
\item $(V,\odot,\Delta,S)$ is a Hopf algebra, 	
\item $(V,\cdot,\Delta)$ is a bialgebra, and
\item $u \cdot (v\odot w)=\sum_{\Delta u=\sum u'\otimes u''}(u'\cdot v)\odot (u''\cdot w)$ for all $u,v,w\in V$.
\end{compactenum}
Next, the cohomology $H^*\big(\coprod_{k\geq 1} \B\Sym_k;\F_2\big)$ can be equipped with a structure of a Hopf ring where:
\begin{compactitem}[\ --- ]
\item the first product $\odot$ is the so call transfer product, for a definition consult \cite[Sec.\,3]{Giusti2012},
\item the second product $\cdot$ is the cup product, extended to be zero on the classes coming from different disjoint components,
\item the coproduct $\Delta$ on is the dual to the standard Pontryagin product on the homology of $H_*\big(\coprod_{k\geq 1} \B\Sym_k;\F_2\big)$, that is the coproduct induced by the natural inclusions of the symmetric groups $\Sym_r\times\Sym_s\longrightarrow\Sym_{r+s}$, and finally 
\item the anipode $S$ is just the identity map.
\end{compactitem}
Now the main result of Giusti, Salvatore \& Sinha  \cite[Thm.\,1.2]{Giusti2012} is as follows.

\begin{theorem}
\label{th : Giusti, Salvatore, Sinha}	
The cohomology $H^*\big(\coprod_{k\geq 1} \B\Sym_k;\F_2\big)$, as a Hopf ring, is generated by the unit classes on each component and classes $u_{\ell,n}\in H^*(\B\Sym_{n2^{\ell}};\F_2)$.
The coproduct is generated by the formula 
\[
\Delta u_{\ell,n}=\sum_{r+s=n}u_{\ell,r}\otimes u_{\ell,s},
\]
while the relations between the transfer products on generators are given by
\[
u_{\ell,r}\odot u_{\ell,s}={r+s \choose r}\, u_{\ell,r+s}.
\]
The second (cup-)product $\cdot$ between the generators from different components is zero, and there are no further  relations between products of the generators.  
In addition, the antipode is the identity.
\end{theorem}
  
%====
\subsubsection{Corollary}
%====

Motivated by the work of Atiyah and Giusti, Salvatore \& Sinha, for a fixed integer $d\geq 2$ or $d=\infty$, we  describe the $\F_2$ homology of the space of all finite subsets of $\R^d$ with additional base point, that is
$
H_*\big(\coprod_{k\geq 0} \conf(\R^d,k)/\Sym_k;\F_2\big)
$.
Here, $\conf(\R^d,0)$ and $\conf(\R^d,0)/\Sym_0$ stand for base points.

\medskip
In addition to the additive structure on the homology, which can be already read from Corollary \ref{cor : homology of configuration space concrete}, we can also identify the ring structure with the respect to a naturally induced multiplication.
For a clear definition of the product structure on this homology we use the little $d$-cubes operad isomorphic model and set 
\[
\TT_d:= H_*\Big(\coprod_{k\geq 0} \CC_d(k)/\Sym_k;\F_2\Big) \cong H_*\Big(\coprod_{k\geq 0} \conf(\R^d,k)/\Sym_k;\F_2\Big).
\]
Recall that by Definition \ref{def : operad} the space $\CC_d(0)$ is just a point and therefore coincides with the $\conf(\R^d,0)$. 
Now, the product structure on $\TT_d$, we care about, is induced by the structural maps of the little $d$-cubes operad 
\begin{equation}
\label{eq : product map inducing}	
\mu_{\CC_d}(\vec{c}_1,\vec{c}_2)\colon \CC_d(r)\times \CC_d(s)\longrightarrow \CC_d(r+s),
\end{equation}
where $(\vec{c}_1,\vec{c}_2)\in\CC_d(2)$ is assumed to be an arbitrary but fixed pair of little $d$-cubes.

\begin{theorem}
\label{th : homology of finite subsets as a ring}
Let $d\geq 2$ be an integer or $d=\infty$. 
Then $\TT_d$ is isomorphic to the polynomial ring 
\[
\mathcal{R}_d:=\F_2\big[\{u_0\}\cup\{\Q_{i_1}\Q_{i_2}\cdots \Q_{i_s}u_0 \, : \, s\geq 1, \, 1\leq i_1\leq i_2\leq\cdots\leq i_s\leq d-1\}\big]
\]
on generators 
\[
u_0\in H_0(\CC_d(1)/\Sym_1;\F_2)
\]
and
\[
\Q_{i_1}\Q_{i_2}\cdots\Q_{i_s}u_0\in H_{i}(\CC_d(2^s)/\Sym_{2^s};\F_2)
\]
where $s\geq 1$, $1\leq i_1\leq i_2\leq\cdots\leq i_s\leq d-1$,  and $i=i_1+2i_2+2^2i_3+\dots+2^{s-1}i_s$. 
The unit of the ring $\mathcal{R}_d$ is the generator of the $0$-homology of the base point, that is $1\in H_0(\CC_d(0)/\Sym_0;\F_2)$.

\end{theorem}

%-----------
\begin{proof}
The ring $\TT_d$ has two gradings, one with respect to the number of cubes (points) and the second one with respect to the homology degree:
\begin{multline*}
\TT_d= H_*\Big(\coprod_{k\geq 0} \CC_d(k)/\Sym_k;\F_2\Big) \cong \bigoplus_{k\geq 0} H_*( \CC_d(k)/\Sym_k;\F_2)\cong \\	
\bigoplus_{k\geq 0} \bigoplus_{i\geq 0}H_i( \CC_d(k)/\Sym_k;\F_2).
\end{multline*}
Indeed the product, induced by the map \eqref{eq : product map inducing}, is the composition homomorphism:
\[
\xymatrix@1{
H_i( \CC_d(r)/\Sym_r;\F_2) \otimes H_j( \CC_d(s)/\Sym_k;\F_2)\ar[r]^-{\times}\ar[rdd] \ &\  H_{i+j}( \CC_d(r)/\Sym_r\times \CC_d(s)/\Sym_k;\F_2)\ar[dd]^{(\mu_{\CC_d}(\vec{c}_1,\vec{c}_2))_*}\\
& & \\
& H_{i+j}( \CC_d(r+s)/\Sym_{r+s};\F_2).
}
\]
The horizontal map is the homology cross product which is an isomorphism in our situation; consult for example \cite[Thm.\,VI.1.6]{Bredon2010}.

\medskip
In order to simplify notation for $d\geq 2$ we set
\[
\mathcal{I}_d:=\{ (i_1,\dots,i_s) : s\geq 1, \, 1\leq i_1\leq \cdots\leq i_s\leq d-1 \}.
\] 
Let $I=(i_1,\dots,i_s)\in \mathcal{I}_d$, then we set the length of $I$ to be $|I|:=s$.
Like in the previous section, we denote the iterated Araki--Kudo--Dyer--Lashof homology operation $Q_{i_1}Q_{i_2}\cdots Q_{i_s}$ applied on $u_L$ by $Q_Iu_L$ where $I=(i_1,\dots,i_s)\in\mathcal{I}_d$. 
Furthermore, we write $\Q_Iu_0:= \Q_{i_1}\Q_{i_2}\cdots \Q_{i_s}u_0$ for a typical generator of the polynomial ring $\mathcal{R}$ which differs from the generator $u_0$.
 Recall, that $\deg(u_L)=L$, and  
\[
\deg(Q_Iu_l)=\deg(Q_{i_1}Q_{i_2}\cdots Q_{i_s}u_L)=i_1+2i_2+4i_3+\cdots +2^{s-1}i_s+2^sL.
\]
On the other hand, we have set that  
\[
\deg(\Q_Iu_0)=\deg(\Q_{i_1}\Q_{i_2}\cdots \Q_{i_s}u_0)=i_1+2i_2+4i_3+\cdots +2^{s-1}i_s,
\]
with $\deg(u_0)=0$.	
In other words, if $L$ could be zero the rings $\TT_d$ and $\mathcal{R}_d$ would coincide precisely. 

\medskip
From Corollary \ref{cor : homology of configuration space concrete} we know that the homology $H_i(\CC_d(k)/\Sym_k;\F_2)$, for $0\leq i\leq (d-1)(k-1) $, is the $\F_2$ vector space spanned by all monomials of degree $i+kL$ of the polynomial ring $\F_2[\{u_L\}\cup\{Q_Iu_L : I\in \mathcal{I}_d\}]$ whose weights are exactly $k$.
It is important to keep in mind that the base elements of this vector space are indexed by, but are not equal to, the iterated Araki--Kudo--Dyer--Lashof homology operations applied to the class $u_L$.
Furthermore, recall that for $I=(i_1,\dots,i_s)\in \mathcal{I}_d$ the weight of $Q_Iu_L$ is $2^s$, while the weight of $u_L$ is $1$.
Then we the ring $\TT_d$ can be decomposed into a direct sum of vector spaces as follows:
{
\begin{multline*}
\TT_d= H_*\Big(\coprod_{k\geq 0} \CC_d(k)/\Sym_k;\F_2\Big) \\ \cong   \bigoplus_{k\geq 0} H_*( \CC_d(k)/\Sym_k;\F_2)  \cong
\bigoplus_{k\geq 0}\bigoplus_{i\geq 0} H_i( \CC_d(k)/\Sym_k;\F_2)\\  \cong
\bigoplus_{k\geq 0}\bigoplus_{i\geq 0}  
\bigg\langle  {\small (Q_{I_1}u_L)\cdots ( Q_{I_t}u_L) \,u_L^a  : 
\begin{array}{l}
\omega((Q_{I_1}u_L)\cdots ( Q_{I_t}u_L) \,u_L^a)=k\\
\deg((Q_{I_1}u_L)\cdots ( Q_{I_t}u_L) \,u_L^a)=i+kL
\end{array}}
\bigg\rangle \qquad \quad \ \\ \cong
\bigoplus_{k\geq 0}\bigoplus_{i\geq 0}  
\bigg\langle  {\small (Q_{I_1}u_L)\cdots ( Q_{I_t}u_L) \,u_L^a  : 
\begin{array}{l}
2^{s_1}+\cdots+2^{s_t}+a=k\\
\iota (I_1,\dots,I_t)+(2^{s_1}+\cdots+2^{s_t}+a)L=i+kL
\end{array}}
\bigg\rangle \\   \cong
\bigoplus_{k\geq 0}\bigoplus_{i\geq 0}  
\bigg\langle  {\small (Q_{I_1}u_L)\cdots ( Q_{I_t}u_L) \,u_L^a  : 
\begin{array}{l}
2^{s_1}+\cdots+2^{s_t}+a=k\\
\iota (I_1,\dots,I_t)=i
\end{array}}
\bigg\rangle
  . 
\end{multline*}}
Here we took $I_1=(i_{11},i_{12},\dots,i_{1s_1}),\dots, I_t=(i_{t1},i_{t2},\dots,i_{ts_t})$ and denoted by $\iota (I_1,\dots,I_t)$ the sum $(i_{11}+\dots+2^{s_1-1}i_{1s_1})+\dots+(i_{t1}+\dots+2^{s_t-1}i_{ts_t})$, that is 
\[
\iota (I_1,\dots,I_t):=(i_{11}+\dots+2^{s_1-1}i_{1s_1})+\dots+(i_{t1}+\dots+2^{s_t-1}i_{ts_t}).
\]

\medskip
On the other hand, in a similar way as above, the polynomial ring $\mathcal{R}_d$ can be decomposed into a direct sum of vector subspaces as follows:
\begin{multline*}
\mathcal{R}_d:=\F_2[\{u_0\}\cup\{\Q_Iu_0 : I\in \mathcal{I}_d\}]	\cong\\
\bigoplus_{k\geq 0}\bigoplus_{i\geq 0}  
\bigg\langle  {\small (\Q_{I_1}u_0)\cdots ( \Q_{I_t}u_0) \,u_0^a \ : 
\begin{array}{l}
2^{s_1}+\cdots+2^{s_t}+a=k\\
\iota (I_1,\dots,I_t)=i
\end{array}}
\bigg\rangle
  .
\end{multline*}

\medskip
Thus, the bijective correspondence $Q_Iu_L \longleftrightarrow\Q_Iu_0$, for $I\in \mathcal{I}_d$, and $u_L\longleftrightarrow u_0$ induces bijection between the (monomial) basis of $\TT_d$ and $\mathcal{R}_d$, which further on induces an isomorphism of $\TT_d$ and $\mathcal{R}_d$ considered as $\F_2$ vector spaces. 

\medskip
It remains to show that the induced isomorphism on the underlying vector space structures of $\TT_d$ and $\mathcal{R}_d$ also respects the corresponding product structures. 
This fact follows from the commutativity of the following diagram:
\begin{equation}
\label{diagram for multiplication}
{\small 
\xymatrix{
%--1
H_i(\CC_d(r)/\Sym_r)\otimes H_j(\CC_d(s)/\Sym_s)\ar[d]^{=}\ & \ \ar[l]_-{\cong} \widetilde{H}_{i+rL}(\TD_r\CC_d(S^L))\otimes \widetilde{H}_{j+sL}(\TD_s\CC_d(S^L))\ar[d]^{\cong}  \\
%--2
H_i(\CC_d(r)/\Sym_r)\otimes H_j(\CC_d(s)/\Sym_s) \ar[d]^{\cong} \ & \ \ar[l]_-{\cong}  \widetilde{H}_{i+rL}(\Th(\xi_{d,r}^{\oplus L}))\otimes \widetilde{H}_{j+sL}(\Th(\xi_{d,s}^{\oplus L}))\ar[d]^{\cong} \\
%--3
H_{i+j}(\CC_d(r)/\Sym_r\times \CC_d(s)/\Sym_s)\ar[d]^-{=} \ & \ \ar[l]_-{\cong} \widetilde{H}_{i+j+(r+s)L}(\Th(\xi_{d,r}^{\oplus L}) \wedge \Th(\xi_{d,s}^{\oplus L}))\ar[d]^{\cong}\\
%--4
H_{i+j}(\CC_d(r)/\Sym_r\times \CC_d(s)/\Sym_s)\ar[d]^{(\mu_{\CC_d}(\vec{c}_1,\vec{c}_2))_*} \ & \ \ar[l]_-{\cong} \widetilde{H}_{i+j+(r+s)L}(\Th(\xi_{d,r}^{\oplus L} \times\xi_{d,s}^{\oplus L}))\ar[d]^{(\mu_{\CC_d}(\vec{c}_1,\vec{c}_2))_*}\\
%--5
H_{i+j}(\CC_d(r+s)/\Sym_{r+s})\ar[d]^{=} \ & \ \ar[l]_-{\cong}  \widetilde{H}_{i+j+(r+s)L}(\Th(\xi_{d,r+s}^{\oplus L}))\ar[d]^{\cong}\\
%--6
H_{i+j}(\CC_d(r+s)/\Sym_{r+s}) \ & \ \ar[l]_-{\cong}  \widetilde{H}_{i+j+(r+s)L}(\TD_{r+s}\CC_d(S^L))
}
}
\end{equation}
where the coefficients in all homologies are assumed to be in the field $\F_2$.
(Recall, here $\Th(\xi)$ stands for the Thom space of the vector bundle $\xi$.)
Thus, we have proved that $\TT_d$ and $\mathcal{R}_d$ are isomorphic as rings. 
\end{proof}

\begin{remark}
It is important to observe that alternatively we could see the disjoint union $\coprod_{k\geq 0} \conf(\R^d,k)/\Sym_k$ as the free $\CC_d$-space $\CC_d(S^d)$ generated by the sphere $S^0$.
Then the homology $H_*(\coprod_{k\geq 0} \conf(\R^d,k)/\Sym_k;\F_p)$, for $p$ a prime, was described already in 1976 by Cohen \cite[Thm.\,3.1]{Cohen1976LNM533} using the framework of an appropriate class of Hopf algebras.	
\end{remark}

%%%%%%%%%%%%%%%%%%%%%%%%%%%%%%%%%%%%%%%%%%%%%%%%%%%%%%%%%%%%%%%%%%%%%%%%%%%%%%%%%%%%%
%%%%%%%%%%%%%%%%%%%%%%%%%%%%%%%%%%%%%%%%%%%%%%%%%%%%%%%%%%%%%%%%%%%%%%%%%%%%%%%%%%%%%
\section{On highly regular embeddings -- revised}
\label{sec : regular embeddings}
%%%%%%%%%%%%%%%%%%%%%%%%%%%%%%%%%%%%%%%%%%%%%%%%%%%%%%%%%%%%%%%%%%%%%%%%%%%%%%%%%%%%%
%%%%%%%%%%%%%%%%%%%%%%%%%%%%%%%%%%%%%%%%%%%%%%%%%%%%%%%%%%%%%%%%%%%%%%%%%%%%%%%%%%%%%

The study of highly regular embeddings has a long history, which starts in 1957 with the work of Borsuk~\cite{Borsuk1957} on the existence $k$-regular embeddings.
Boltjanski{\u\i}, Ry{\v{s}}kov \& {\v{S}}a{\v{s}}kin \cite{Boltjanski1963} in 1963  gave the first lower bounds for the existence of $2k$-regular embeddings.
In the 1970's and 1980's the accumulated knowledge about the topology of the unordered configuration space allowed for further progress, which was made by Cohen \& Handel~\cite{Cohen1978}, Chisholm~\cite{Chisholm1978}, Handel~\cite{Handel1979}, and Handel \& Segal~\cite{HandelSegal1980}.
In the 1990's the existence of $k$-regular embeddings was related to the notion of $k$-neighbourly submanifolds; this was considered by Vassiliev~\cite{Vassiliev1998} and Handel~\cite{Handel1996}.
The result of Chisholm \cite[Theorem 2]{Chisholm1978} is the strongest result from that time.

\medskip
In the first decade of 21st century the notion of $\ell$-skew embeddings was introduced and studied by Ghomi \& Tabachnikov~\cite{Ghomi2008}.
Furthermore, combining the notions of $k$-regular embeddings\index{$k$-regular embedding} and of $\ell$-skew embeddings\index{$\ell$-skew embedding}, in 2006 Stojenovi\'c \cite{Stojanovic2006} defined the notion of $k$-regular-$\ell$-skew embeddings\index{$k$-regular-$\ell$-skew embedding} and gave the first bounds for their existence.

\medskip
The second decade of the new century brought the first non-elementary constructions of $k$-regular embedding. 
In 2019, Jaros\l aw Buczy\'{n}ski,  Tadeusz Januszkiewicz, Joachim Jelisiejew \& Mateusz Micha\l ek \cite{BuczynskiJanuszkiewiczJelisiejewMateusz2019}, using advanced methods of algebraic geometry, constructed $k$-regular embeddings of finite dimensional real and complex vector spaces.

\medskip
In this section we revise the results of Blagojevi\'c, L\"uck and Ziegler presented in the paper \cite{Blagojevic2016-01}.
The proofs of \cite[Thm.\,2.1, Thm.\,3.1, Thm.\,4.1]{Blagojevic2016-01} relied essentially on a result of {\Hung} \cite[(4.7)]{Hung1990} that turned out to be incorrect, see Remark \ref{rem: gap}.
Specifically, \cite[Lem.\,2.15]{Blagojevic2016-01} is not true when $d$ is not a power of $2$.

%----------------------------------
\subsection{$k$-regular embeddings}
\label{sub : k-regular maps}
%----------------------------------

In this section we correct \cite[Sec.\,2]{Blagojevic2016-01}.
In the following all topological spaces are Hausdorff spaces and all maps are assumed to be continuous.

\begin{definition}\label{def : k-regular map}
  Let $k\geq 1$ and $N\geq 1$ be integers, and let $X$ be a topological space.
  A continuous map $f\colon X\longrightarrow \R^N$ is a {\bf $k$-regular embedding\index{$k$-regular embedding}} if for every $(x_1,\ldots,x_k)\in\conf(X,k)$ the vectors $f(x_1),\ldots,f(x_k)$ are linearly independent.
  In particular, $0\notin\im(f)$.
\end{definition}

For $N\geq 1$ and $1\leq k\leq N$ integers, the Stiefel manifold\index{Stiefel manifold} $V_k(\R^N)$ of all ordered $k$-frames is a (topological) subspace of the product $(\R^N)^k$ given by
\[
V_k(\R^N)=\{(y_1,\ldots,y_k)\in (\R^N)^k : y_1,\ldots,y_k\text{ are linearly independent}\}.
\]
The symmetric group\index{symmetric group} $\Sym_k$ acts freely on the Stiefel manifold by permuting the vectors in the frame.
Now from Definition \ref{def : k-regular map} we get a necessary condition for the existence of a $k$-regular embedding phrased in term of the existence of an equivariant map.

\begin{lemma}
\label{lem:Existence_of_a_map}
 If there exists a $k$-regular embedding\index{$k$-regular embedding} $X\longrightarrow \R^N$, then there exists an $\Sym_k$-equivariant map 
 \begin{equation}
 \label{eq : map from conf to stiefel}
  \conf(X,k)\longrightarrow V_k(\R^N).	
 \end{equation}
\end{lemma}

In order to study the existence of an $\Sym_k$-equivariant map \eqref{eq : map from conf to stiefel} we use the criterion of Cohen and Handel \cite[Prop.\,2.1]{Cohen1978}, see also \cite[Lem.\,2.11]{Blagojevic2016-01}.
Let $\R^k$ be endowed with an $\Sym_k$-action given by coordinate permutation.  
Then the subspace $W_k=\{(a_1,\ldots, a_k)\in\R^k : a_1+\dots+a_k=0\}$ is an $\Sym_k$-invariant subspace.
Consider the following vector bundles\index{vector bundle}
\begin{equation}\label{R-vector-bundles}
\xymatrix@R-2pc{
\xi_{X,k} \colon & \R^k\ar[r] & \conf(X,k)\times_{\Sym_k}\R^k\ar[r] \ & \ \conf(X,k)/\Sym_k,
  \\
 \zeta_{X,k} \colon & W_k \ar[r] & \conf(X,k)\times_{\Sym_k}W_k  \ar[r] \  &\ \conf(X,k)/\Sym_k,\\
  \tau_{X,k} \colon & \R   \ar[r] &\conf(X,k)/\Sym_k\times\R     \ar[r] \ &\ \conf(X,k)/\Sym_k,
}
\end{equation}
where $\tau_{X,k}$ is a trivial line bundle.
An obvious decomposition holds: 
\begin{equation}
\label{eq : decomposition of vb}
\xi_{X,k}\cong\zeta_{X,k}\oplus \tau_{X,k}.	
\end{equation}
The bundle $\xi_{d,k}$ introduced in \eqref{eq : definition of xi_d,k} is the pull-back vector bundle of the vector bundle $\xi_{\R^d,k}$ along the homotopy equivalence $\CC_d(k)/\Sym_k\longrightarrow\conf(\R^d,k)/\Sym_k$ which is induced by the evaluation at centers of cubes map $\ev_{d,n}\colon\CC_d(n)\longrightarrow\conf(\R^d,n)$ from Lemma \ref{lemma : little cube -- configuration space}.

\begin{lemma}
  \label{lemma:equivalence}
  An $\Sym_k$-equivariant map $\conf(X,k)\longrightarrow V_k(\R^N)$ exists if and only if the $k$-dimensional vector bundle $\xi_{X,k}$ admits an $(N-k)$-dimensional inverse.
\end{lemma}

\smallskip
As a direct consequence of Lemmas \ref{lem:Existence_of_a_map} and \ref{lemma:equivalence} we get a criterion for the non-existence of a $k$-regular embedding in term of the dual Stiefel--Whitney class of the vector bundle $\xi_{X,k}$.
We recall only a particular case of \cite[Lem.\,2.12\,(2)]{Blagojevic2016-01} that we use here.

\begin{lemma} 
\label{lemma:criterion}
Let $d\geq 1$ and $k\geq 1$ be integers.   
If the dual Stiefel--Whitney class\index{Stiefel--Whitney classes}
\[
\overline{w}_{N-k+1}(\xi_{\R^d,k})\neq 0
\] 
does not vanish, then there cannot be any $k$-regular embedding $\R^d\longrightarrow \R^{N}$.
\end{lemma}

The criterion for the non-existence of a $k$-regular embedding $X\longrightarrow \R^{N}$ given in the previous lemma motivates the study of Stiefel--Whitney classes of the vector bundle $\xi_{X,k}$.
In \cite[Thm.\,2.13]{Blagojevic2016-01}, based on the work of {\Hung} \cite[(4.7)]{Hung1990}, the following theorem about the dual Stiefel--Whitney classes of  $\xi_{\R^d,k}$ was proved.
\begin{quote}
{\small
	{\bf Theorem 2.13.}
	Let $k,d\geq 1$ be integers.
	Then the  dual Stiefel--Whitney class\index{Stiefel--Whitney classes}
	\[
	\overline{w}_{(d-1)(k-\alpha(k))}(\xi_{\R^d,k})
	\] 
	does not vanish.
	}
\end{quote}
The previous theorem and the criterion given in Lemma \ref{lemma:criterion} implied \cite[Thm.\,2.1]{Blagojevic2016-01}:
\begin{quote}
{\small
	{\bf Theorem 2.1.}
	Let $k,d\geq 1$ be integers.
	There is no $k$-regular embedding\index{$k$-regular embedding} $\R^d\longrightarrow\R^N$ for
	\[
	N\leq d(k-\alpha(k))+\alpha(k)-1,
	\] 
	where $\alpha(k)$ denotes the number of ones in the dyadic presentation of $k$.
	}
\end{quote}

\noindent
The proof of the key theorem \cite[Thm.\,2.13]{Blagojevic2016-01} was done in two steps.
First, the special case when $k$ is a power of $2$ was established \cite[Lem.\,2.15]{Blagojevic2016-01}, and then the general case for arbitrary integer $k\geq2$ was derived \cite[Lem.\,2.17]{Blagojevic2016-01}. 
In the following we quote both lemmas as in they were written in the original.

\begin{quote}
	{\small
	{\bf Lemma 2.15.}
	Let $d\geq 1$ be an integer, and $k=2^m$ for some $m\geq 1$.  Then the dual
	Stiefel--Whitney class $\overline{w}_{(d-1)(k-1)}(\xi_{\R^d,k})$ does not vanish.
	
	\smallskip
	\noindent
	{\bf Lemma 2.17.}
	Let $d,k\geq 1$ be integers.
	Then the dual Stiefel--Whitney class $\overline{w}_{(d-1)(k-\alpha(k))}(\xi_{\R^d,k})$ does not vanish.
	}
\end{quote}

\noindent
The proof of \cite[Lem.\,2.15]{Blagojevic2016-01} contains a gap at the place where \cite[(4.7)]{Hung1990} was applied and is not true for $d\geq1$ not a power of $2$. 
More precisely, in the proof stands:

\begin{quote}
{\small
The following decomposition of $H^*( \conf(\R^d,k)/\Sym_k)$ was proved by {\Hung} in~[20,\,(4.7),\,page~279]:
 \[
 \quad (5) \quad   H^*( \conf(\R^d,k)/\Sym_k)
    \cong
    \alpha_{d,k}^*\big(\ker(\res^{\Sym_k}_{E_m})\oplus \F_2[q_{m,m-1},\ldots,q_{m,1}]\big)
    \oplus  
    \alpha_{d,k}^*(\langle q_{m,0}\rangle). 
\]
  In particular, this implies that
  \[
   \alpha_{d,k}^*\big(\ker(\res^{\Sym_k}_{E_m})\oplus \F_2[q_{m,m-1},\ldots,q_{m,1}]\big)
   \cap
   \alpha_{d,k}^*(\langle q_{m,0}\rangle)
   =\{0\}.
  \]}
\end{quote}

\noindent
Consequently, \cite[Lem.\,2.15]{Blagojevic2016-01}, \cite[Lem.\,2.17]{Blagojevic2016-01}, \cite[Thm.\,2.13]{Blagojevic2016-01} and \cite[Thm.\,2.1]{Blagojevic2016-01} do not hold.
Furthermore, a technical result stated in \cite[Cor.\,2.16]{Blagojevic2016-01} is incorrect.
In the following we correct these gaps.

\medskip
For the corrections we will combine different results we established so far.
Many of them can be summarized in the commutative diagram \eqref{big diagram 02}. 
The cohomologies appearing in the diagram \eqref{big diagram 02} are assumed to be with the field $\F_2$ coefficients. 

\medskip
First, we show that \cite[Lem.\,2.15]{Blagojevic2016-01}, \cite[Lem.\,2.17]{Blagojevic2016-01} and consequently \cite[Thm.\,2.13]{Blagojevic2016-01} hold, as stated, in the case when the dimension $d$ is a power of~$2$. 
For that we use the following result from \cite[Thm.\,3.1]{Blagojevic2016-02}:
For integers $d\geq 2$ and $k\geq 2$ 
\begin{equation}\label{eq : hight}
\hght(H^{*}( \conf(\R^d,k)/\Sym_k;\F_2))\leq \min\{ 2^t: 2^t\geq d \},	
\end{equation}
where $\hght( \conf(\R^d,k)/\Sym_k;\F_2))$ is the height of the algebra $H^{*}( \conf(\R^d,k)/\Sym_k;\F_2)$.
This result, obtained as an application of Kahn--Priddy~\cite[Thm.\,pp.\,103]{KahnPriddy1978}, is the best possible as explained in \cite[Sec.\,3]{Blagojevic2016-02}.
It is worth mentioning that a similar upper bound for the hight of the cohomology ring $H^{*}( \conf(\R^d,k)/\Sym_k;\F_p)$, where $p$ is now an odd prime, was given in \cite[Thm.\,3.2]{Blagojevic2016-02}.

\medskip
Recall, for an algebra $A$ over a field $\F$,  the {\em height of an element}\index{height of an element} $a\in A{\setminus}A^*$ is a natural number or infinity $\hght(a):=\min \{n\in\N : a^n=0\}$.
Here $A^*$ denotes the group of all invertible elements of the algebra $A$.
The {\em height of the algebra}\index{height of an algebra} $A$ is defined to be $\hght(A):=\max \{\hght(a):a\in A{\setminus}A^*\}$.

%%%%%%%%%%%%%%%%%%%%%%%%%%%%%%%%%%%%%%%%%%%%%%%%%%%%%%%%%%%%%%
\begin{figure}[p]
   \rotatebox{90}{%
     \begin{minipage}{\textheight}%
{\small 
\begin{equation}
\label{big diagram 02}\hspace{-10mm}
\xymatrix@C=0.45em{
\vspace{-10pt}H^*(\Sym_{2^m})\cong  H^*(F(\R^{\infty},2^m)/\Sym_{2^m})\cong\F_2[w_{m,0},\dots,w_{m,m-1}]\oplus\ker(\res^{\Sym_k}_{\EE_m}) \ar[rr]^-{\rho_{\infty,2^m}^*}\ar[ddd] \ & & \ 
H^*(\Sy_{2^m})\cong H^*(\Sp(\R^{\infty},2^m)/\Sy_{2^m})\cong \F_2[V_{m,1},\dots,V_{m,m}]\oplus I^*(\R^{\infty},k)\ar[ddd]^{(\kappa_{d,2^m}/\Sy_{2^m})^*} & & \\
 & & & &\\
 & & & &\\
 H^*(F(\R^{d},2^m)/\Sym_{2^m})\ar[rr]^-{\rho_{d,2^m}^*} \ & & \ H^*(\Sp(\R^{d},2^m)/\Sy_{2^m})\cong \F_2[V_{m,1},\dots ,V_{m,m}]/\langle V_{m,1}^d,\dots,V_{m,m}^d \rangle \oplus I^*(\R^{d},k)\ar[ddd]^{\text{projection}} & &\\
 & & & &\\
 & & & &\\
 & &  \F_2[V_{m,1},\dots,V_{m,m}]/\langle V_{m,1}^d,\dots,V_{m,m}^d \rangle\ar[ddd]^{\text{isomorphism induced by }\chi_m} & &\\
  & & & &\\
 & & & &\\
  & &  \F_2[V_{1},\dots,V_{m}]/\langle V_{1}^d,\dots,V_{m}^d \rangle  & &
}	
\end{equation} 
}
\end{minipage}%
  }%
% No caption, then the figure number is not changed
\end{figure}
%%%%%%%%%%%%%%%%%%%%%%%%%%%%%%%%%%%%%%%%%%%%%%%%%%%%%%%%%%%%%%
\medskip
We prove the following lemma which can be extracted from work of Chisholm \cite{Chisholm1978}.
For alternative proof see also \cite[Prop.\,5.2]{Crabb2012}.

\begin{lemma}
	\label{lemma : corection power of 2}
	Let $d=2^a$ and $k=2^m$ for some integers $a\geq 1$ and $m\geq 1$.  
	Then the dual Stiefel--Whitney class\index{Stiefel--Whitney classes} 
	\[
	\overline{w}_{(d-1)(k-1)}(\xi_{\R^d,k})=\overline{w}_{(2^a-1)(2^m-1)}(\xi_{\R^d,2^m})
	\]
	does not vanish.
\end{lemma}

\begin{proof}
Let  us fix $d=2^a$ and $k=2^m$ for integers $a\geq 1$ and $m\geq 1$.
For simplicity denote by $w:=w(\xi_{\R^d,2^m})$ the total Stiefel--Whitney class of the vector bundle\index{vector bundle} $\xi_{\R^d,2^m}$ and by $\overline{w}:=\overline{w}(\xi_{\R^d,2^m})$ the associated total dual Stiefel--Whitney, that is $w\cdot\overline{w}=1$.
From the decomposition \eqref{eq : decomposition of vb} we get that $w_i:=w_i(\xi_{\R^d,2^m})=w_i(\zeta_{\R^d,2^m})=0$ for all $i\geq 2^m$, and so $w=1+w_1+\cdots+w_{2^m-1}\in H^{*}( \conf(\R^d,2^m)/\Sym_{2^m};\F_2)$.
From the inequality \eqref{eq : hight} we have that 
\[
\hght(H^{*}( \conf(\R^d,2^m)/\Sym_{2^m};\F_2))\leq \min\{ 2^t: 2^t\geq d \}=2^a=d.
\]
Thus, we compute
\[
w^{d}=w^{2^a}=(1+w_1+\dots+w_{2^m-1})^{2^a}=1+w_1^{2^a}+\dots+w_{2^m-1}^{2^a}=1,
\]
and so $w^{d}=w\cdot w^{d-1}=1$.
Therefore,
\begin{align*}
\overline{w} 
& = w^{d-1}=w^{2^a-1}=w^{1+2^1+2^2+\dots+2^{a-1}}=(1+w_1+\dots+w_{2^m-1})^{1+2^1+2^2+\dots+2^{a-1}}	\\
& = \prod_{i=0}^{a-1}\big( 1+w_1^{2^i}+\dots+w_{2^m-1}^{2^i} \big)= w_{2^m-1}^{\sum_{i=0}^{a-1}2^i}+ \big(\text{\small terms of degree}< (d-1)(2^m-1)\big)\\
& =  w_{2^m-1}^{d-1}+ \big(\text{\small terms of degree}< \deg(w_{2^m-1}^{d-1})=(d-1)(2^m-1)\big).
\end{align*}
In summary,
\begin{multline}
	\label{eq : dual class for d power of 2}
	\overline{w} = w_{2^m-1}^{d-1}+ \big(\text{\small terms of degree}< \deg(w_{2^m-1}^{d-1})=(d-1)(2^m-1)\big)\\ \in H^*( \conf(\R^d,2^m)/\Sym_{2^m};\F_2).
\end{multline}
The fact that $w_{2^m-1}^{d-1}\neq 0$ follows from general obstruction computation in \cite[Cor.\,4.5]{Blagojevic2014}.
Nevertheless, we present a different argument that follows from the results of Sections \ref{sec : Equivariant cohomology of epicicles} and \ref{sec : injectivity}.
First we apply the map $\rho_{d,2^m}^*$ to the relation \eqref{eq : dual class for d power of 2} and get
\begin{multline}
	\label{eq : dual class for d power of 2  - 02}
\rho_{d,2^m}^*(\overline{w})=
\rho_{d,2^m}^*(w_{2^m-1})^{d-1}+ \big(\text{\small terms of degree}< \deg(w_{2^m-1}^{d-1})\big) \\ \in H^*(\Sp(\R^d,2^m)/\Sy_{2^m};\F_2).
\end{multline}
Recall that according to Theorem \ref{th: cohomology of Sp} we have that 
\[
	H^*(\Sp(\R^d,2^m)/\Sy_{2^m};\F_2)
	\cong   \F_2[V_{m,1},\dots,V_{m,m}]/\langle V_{m,1}^d,\ldots, V_{m,m}^d\rangle \oplus \III^*(\R^d,2^m).
\]
Furthermore, from  the proof of Lemma \ref{lem : image of res on w_{m,0}}, the commutative diagram \eqref{big diagram 02} and the relation \eqref{eq : relation - 10} we have that
\begin{multline*}
\rho_{d,2^m}^*(w_{2^m-1})=V_{m,1}\cdots V_{m,m}\in 
\F_2[V_{m,1},\dots,V_{m,m}]/\langle V_{m,1}^d,\ldots, V_{m,m}^d\rangle \\
\subseteq
H^*(\Sp(\R^d,2^m)/\Sy_{2^m};\F_2).
\end{multline*}
Combining these facts we get  
\begin{multline*}
\rho_{d,2^m}^*(\overline{w})=
(V_{m,1}\cdots V_{m,m})^{d-1}+ \big(\text{\small terms of degree}< \deg(w_{2^m-1}^{d-1})\big) \\ \in H^*(\Sp(\R^d,2^m)/\Sy_{2^m};\F_2),
\end{multline*}
and furthermore  
\[
\rho_{d,2^m}^*(\overline{w}_{(d-1)(2^m-1)})=\rho_{d,2^m}^*(w_{2^m-1})^{d-1}=(V_{m,1}\cdots V_{m,m})^{d-1}\neq 0.
\]
Consequently, $\overline{w}_{(d-1)(2^m-1)}\neq 0$, and the proof of the lemma is concluded. 
\end{proof}

\medskip
The extension of the previous lemma (for dimension $d=2^a$) to the case of where $k\geq 1$ is an arbitrary integer (still for dimension $d=2^a$) can be done as in the proof of \cite[Lem.\,2.17]{Blagojevic2016-01}.
For the sake of completeness we give a detailed proof using the presentation from \cite[Proof of Lem.\,2.17]{Blagojevic2016-01}.
Here $\alpha(k)$ denotes the number of $1$s in the binary presentation of the integer $k\geq 1$. 

\begin{lemma}
	\label{lemma : corection not a power of 2}
	Let $d=2^a$ for some integer $a\geq 1$, and let $k\geq 1$ be an integer.  
	Then the dual Stiefel--Whitney class\index{Stiefel--Whitney classes} 
	\[
	\overline{w}_{(d-1)(k-\alpha(k))}(\xi_{\R^d,k})=\overline{w}_{(2^a-1)(k-\alpha(k))}(\xi_{\R^d,k})
	\] 
	does not vanish.
\end{lemma}
\begin{proof}
 Let $r:=\alpha(k)$ be the number of $1$s in the binary presentation of the integer $k\geq 1$, and let $k=2^{b_1}+\dots+2^{b_r}$ where $0\leq b_1<b_2<\dots<b_r$.  
 Consider a morphism between vector bundles\index{vector bundle} $\prod_{i=1}^{r}\xi_{\R^d,2^{b_i}}$ and $\xi_{\R^d,k}$ where the following commutative square is a pullback diagram: 
 \[
\xymatrix{\prod_{i=1}^{r}\xi_{\R^d,2^{b_i}} \ar[rr]^-{\Theta}  \ar[d] \
& & \ \xi_{\R^d,k}   \ar[d] 
\\
%& & \\
\prod_{i=1}^{r} \conf(\R^d,2^{b_i})/\Sym_{2^{b_i}}  \ar[rr]^-{\theta}  \ & & \
 \conf(\R^d,k)/{\Sym_k}.
}
\]
The map $\theta$ is induced, up to an equivariant homotopy, from a restriction of the little cubes operad\index{little cubes operad} structural map
\[
(\CC_d(2^{b_1})\times\dots\times \CC_d(2^{b_r}))\times \CC_d(r)\longrightarrow\CC_d(2^{b_1}+\dots+2^{b_r}).
\]
Alternatively, fix embeddings $e_i \colon \R^d \longrightarrow \R^d$ for $1\leq i\leq r$ such that their images are pairwise disjoint open $d$-balls. 
They induces embeddings $ \conf(\R^d,\ell) \longrightarrow  \conf(\R^d,\ell)$ denoted by the same letter $e_i$ for all natural numbers $\ell$. 
Thus, the map $\theta$ is induced by the map $\prod_{i=1}^{r} \conf(\R^d,2^{b_i}) \longrightarrow  \conf(\R^d,k)$ defined by
\begin{multline*}
\big((x_{1,1},\dots,x_{1,2^{b_1}}),\dots , (x_{r,1},\dots,x_{r,2^{b_r}}) \big)  \longmapsto \\ e_1(x_{1,1},\dots,x_{1,2^{b_1}}) \times \cdots \times e_r(x_{r,1},\dots,x_{r,2^{b_r}}).
\end{multline*}
The map $\Theta$ that covers $\theta$ is given by
\begin{multline*}
\big((x_{1,1},\dots,x_{1,2^{b_1}};v_1),\dots , (x_{r,1},\dots,x_{r,2^{b_r}};v_r) \big)  \longmapsto \\
\big( e_1(x_{1,1},\dots,x_{1,2^{b_1}}) \times \cdots \times e_r(x_{r,1},\dots,x_{r,2^{b_r}}),v_1\times \cdots\times  v_r \big).	
\end{multline*}
Consequently, the pullback vector bundle\index{vector bundle} is the product vector bundle $\theta^*\xi_{\R^d,k} \cong  \prod_{i=1}^{r}\xi_{\R^d,2^{b_i}}$.
Now the naturality property of the Stiefel--Whitney classes~\cite[Ax.\,2, p.\,37]{Milnor1974}\index{Stiefel--Whitney classes} implies that in cohomology we get
\[
\theta^*(\overline{w}_{(d-1)(k-r)}(\xi_{\R^d,k}))=\overline{w}_{(d-1)(k-r)}\Big(\prod_{i=1}^{r}\xi_{\R^d,2^{b_i}}\Big).
\]
The product formula for Stiefel--Whitney classes \cite[Pr.\,4-A, p.\,54]{Milnor1974} implies that
\[
 \overline{w}\Big(\prod_{i=1}^{r}\xi_{\R^d,2^{b_i}}\Big)=\overline{w}(\xi_{\R^d,2^{b_1}})\times\cdots\times\overline{w}(\xi_{\R^d,2^{b_r}}).
\]
Now we compute
\begin{align*}
 \theta^*(\overline{w}_{(d-1)(k-r)}(\xi_{\R^d,k})) &=\overline{w}_{(d-1)(k-r)}\Big(\prod_{i=1}^{r}\xi_{\R^d,2^{b_i}}\Big)\\
 &=\sum_{s_1+\cdots +s_r=(d-1)(k-r)}\overline{w}_{s_1}(\xi_{\R^d,2^{b_1}})\times\cdots\times\overline{w}_{s_r}(\xi_{\R^d,2^{b_r}}).	
\end{align*}
From the K\"unneth formula\index{K\"unneth formula} \cite[Thm.\,VI.3.2]{Bredon2010} we have that each term of the previous sum $\overline{w}_{s_1}(\xi_{\R^d,2^{b_1}})\times\cdots\times\overline{w}_{s_r}(\xi_{\R^d,2^{b_r}})$ belongs to a different direct summand in the following direct decomposition of the $(d-1)(k-r)$th cohomology
\begin{multline*}
H^{(d-1)(k-r)}\Big(\prod_{i=1}^{r} \conf(\R^d,2^{b_i})/\Sym_{2^{b_i}};\F_2\Big)\cong\\
\bigoplus_{s_1+\cdots +s_r=(d-1)(k-r)}H^{s_1}( \conf(\R^d,2^{b_1})/\Sym_{2^{b_1}};\F_2)\otimes\cdots\otimes H^{s_r}( \conf(\R^d,2^{b_r})/\Sym_{2^{b_r}};\F_2).
\end{multline*}
Thus, we have the equivalence
\begin{multline*}
\overline{w}_{(d-1)(k-r)}\Big(\prod_{i=1}^{r}\xi_{\R^d,2^{b_i}}\Big) \neq 0 \quad \Longleftrightarrow \\
\overline{w}_{s_1}(\xi_{\R^d,2^{b_1}})\times\cdots\times\overline{w}_{s_r}(\xi_{\R^d,2^{b_r}})\neq 0
\; \text{\small for some} \; s_1+\cdots +s_r=(d-1)(k-r).
\end{multline*}
Now, since $d$ is a power of $2$, Lemma~\ref{lemma : corection power of 2} implies that 
$\overline{w}_{(d-1)(2^{b_i}-1)}(\xi_{\R^d,2^{b_i}})\neq 0$ for all $1\leq i\leq r$, and so
\[
 \overline{w}_{(d-1)(2^{b_1}-1)}(\xi_{\R^d,2^{b_1}})\times\cdots\times\overline{w}_{(d-1)(2^{b_r}-1)}(\xi_{\R^d,2^{b_r}})\neq 0.
\]
Hence, $\theta^*(\overline{w}_{(d-1)(k-r)}(\xi_{\R^d,k}))\neq 0$, and furthermore 
$\overline{w}_{(d-1)(k-\alpha(k))}(\xi_{\R^d,k})\neq 0$.
\end{proof}

\medskip
Next we consider the case when $d$, the dimension of the Euclidean space $\R^d$, is not a power of~$2$.
Like in the previous situation this is done in two separate steps depending whether $k$ is power of $2$ or not.
Now we give two corrections of \cite[Lem.\,2.15]{Blagojevic2016-01}.
In the first one, next lemma, we do not take into account the dyadic presentation of the dimensions $d$.

\begin{lemma}
\label{lem : correction of 2.15}
Let $d\geq 3$ be an integer that is not a power of $2$, and let $k=2^m$ for some integer $m\geq 1$.  
Then the dual Stiefel--Whitney class\index{Stiefel--Whitney classes} 
\[
\overline{w}_{(d-1)k/2}(\xi_{\R^d,k})=\overline{w}_{(d-1)2^{m-1}}(\xi_{\R^d,2^m})
\] 
does not vanish.
\end{lemma}
\begin{proof}
Let us fix an integer $d\geq3$ which is not a power of $2$.
Then there exists an integer $a\geq 2$ such that $2^{a-1}+1\leq d\leq 2^a-1$.
Furthermore, let $k=2^m$ where $m\geq 1$ is an integer.
Then $d-1$ can be presented as $d-1=2^{a_1}+\dots+2^{a_q}$ where $0\leq a_1<\dots<a_q=a-1$.
As before, for simplicity we again denote by $w=w(\xi_{\R^d,2^m})$ the total Stiefel--Whitney class of the vector bundle\index{vector bundle} $\xi_{\R^d,2^m}$ and by $\overline{w}=\overline{w}(\xi_{\R^d,2^m})$ the associated total dual Stiefel--Whitney.
Hence, $w\cdot\overline{w}=1$.
Now the inequality \eqref{eq : hight} implies that 
\[
\hght(H^{*}( \conf(\R^d,2^m)/\Sym_{2^m};\F_2))\leq \min\{ 2^t: 2^t\geq d \}=2^a.
\]
Consequently,
$w^{2^a}=(1+w_1+\dots+w_{2^m-1})^{2^a}=1+w_1^{2^a}+\dots+w_{2^m-1}^{2^a}=1$,
and so 
\begin{align}
\label{eq : dual class for d not a power of 2}
\overline{w} 
&= w^{2^a-1} \nonumber \\
&=(1+w_1+\dots+w_{2^m-1})^{1+2^1+2^2+\dots+2^{a-1}}  \nonumber \\
&= \prod_{i=0}^{a-1}\big( 1+w_1^{2^i}+\dots+w_{2^m-1}^{2^i} \big).	
\end{align}
Now we apply the monomorphism
\[
 \rho_{d,2^m}^*\colon H^*(\conf(\R^d,2^m)/\Sym_{2^m};\F_2)\longrightarrow H^*(\Sp(\R^d,2^m)/\Sy_{2^m};\F_2)
\] 
from Theorem \ref{th : injection } to the equality \eqref{eq : dual class for d not a power of 2} and use the decomposition of the cohomology
\[
 H^*(\Sp(\R^d,2^m)/\Sy_{2^m};\F_2)
 \cong
 \F_2[V_{m,1},\dots,V_{m,m}]/\langle V_{m,1}^d,\ldots, V_{m,m}^d\rangle \oplus \III^*(\R^d,2^m), \] 
given in Theorem \ref{th: cohomology of Sp}, to get that
\begin{multline*}
\rho_{d,2^m}^*(\overline{w})
=  \prod_{i=0}^{a-1}\big( 1+D_{m,m-1}^{2^i}+\dots+D_{m,0}^{2^i} \big)+R \\
\ \in \ \F_2[V_{m,1},\dots,V_{m,m}]/\langle V_{m,1}^d,\ldots, V_{m,m}^d\rangle \oplus \III^*(\R^d,2^m)
,
\end{multline*}
where
\begin{compactitem}[ \---]
\item $D_{m,r}=(\kappa_{d,2^m}/\Sy_{2^m})^*(D_{m,r})=\rho_{d,2^m}^*(w_{2^m-2^r})$, for $0\leq r\leq m-1$, with the obvious abuse of notation, see \eqref{eq : summary}, and 
\item $R\in \III^*(\R^d,2^m)$.
\end{compactitem}
Let us denote by $\pi$ the following composition of the maps from the diagram \eqref{big diagram 02}:
\[
\xymatrix@C=0.5em{
\F_2[V_{m,1},\dots ,V_{m,m}]/\langle V_{m,1}^d,\dots,V_{m,m}^d \rangle \oplus I^*(\R^{d},2^m)\ar[d]^-{\text{projection}}\\
\F_2[V_{m,1},\dots ,V_{m,m}]/\langle V_{m,1}^d,\dots,V_{m,m}^d \rangle\ar[d]^-{\chi_m} \\
\F_2[V_{1},\dots ,V_{m}]/\langle V_{1}^d,\dots,V_{m}^d \rangle, \\
}
\]
where the first map is the projection on a direct summand and the second map is induced by the change of variables $\chi_m\in\GL_m(\F_2)$.
Next we apply $\pi$ on $\rho_{d,2^m}^*(\overline{w})$ and get:
\[
\pi(\rho_{d,2^m}^*(\overline{w}))=\prod_{i=0}^{a-1}\big( 1+\pi(D_{m,m-1})^{2^i}+\dots+\pi(D_{m,0})^{2^i} \big).
\]
Since the change of the variables $\chi_m$ transforms $\U_m(\F_2)$-invariants into $\mathrm{U}_m(\F_2)$-invariants and Dickson polynomials\index{Dickson invariants}, $\GL_m(\F_2)$-invariants, can be presented in terms of $\mathrm{U}_m(\F_2)$-invariants, as explained in \eqref{relation - recurrence Dickson - 03} and \eqref{relation - recurrence Dickson - 04}, we have

{\small
\begin{align*} 
\pi(\rho_{d,2^m}^*(\overline{w}))&=\prod_{i=0}^{a-1}\big( 1+\pi(D_{m,m-1})^{2^i}+\dots+\pi(D_{m,0})^{2^i} \big)	\\
&=\prod_{i=0}^{a-1}
\Big(
1+
\big(V_1^{2^{m-1}}+V_2^{2^{m-2}}+\cdots+V_m^{2^{0}} \big)^{2^i}+\dots+\\
&\quad\big(\sum_{1\leq j_1<\cdots< j_r\leq m}
(V_1\cdots V_{j_1-1})^{2^r}\, (V_{j_1+1}\cdots V_{j_2-1})^{2^{r-1}}\cdots  (V_{j_r+1}\cdots V_{m})^{2^{0}}\big)^{2^i} \\ &\qquad\qquad\qquad\qquad\qquad\qquad\qquad\qquad +\dots+\\
&\quad\big(V_1\cdots V_m\big)^{2^i}
\Big)\\
&=\prod_{i=0}^{a-1}
\Big(
1+
\big(V_1^{2^{i+m-1}}+V_2^{2^{i+m-2}}+\cdots+V_m^{2^{i}}\big) +\dots+\\
&\quad\big(\sum_{1\leq j_1<\cdots< j_r\leq m}
(V_1\cdots V_{j_1-1})^{2^{i+r}}\, (V_{j_1+1}\cdots V_{j_2-1})^{2^{i+r-1}}\cdots  (V_{j_r+1}\cdots V_{m})^{2^{i}}\big) \\ &\qquad\qquad\qquad\qquad\qquad\qquad\qquad\qquad +\dots+\\
&\quad\big(V_1\cdots V_m\big)^{2^i}
\Big).
\end{align*}}

\noindent
Since $d-1=2^{a_1}+\dots+2^{a_q}$ and $0\leq a_1<\dots<a_q=a-1$, by choosing terms $V_m^{2^{a_{\ell}+0}}$ from factors in the product indexed by $i=a_1,\dots,a_q$ we get that
\[
\pi(\rho_{d,2^m}^*(\overline{w}_{(d-1)2^{m-1}}))=V_m^{d-1}+S \ \in \ \F_2[V_{1},\dots ,V_{m}]/\langle V_{1}^d,\dots,V_{m}^d \rangle,
\]
where $S$ denotes a sum of some monomials in $V_1,\dots,V_m$ of degree $(d-1)2^{m-1}=(d-1)k/2$ which are all different from $V_m^{d-1}$.
Consequently, $\pi(\rho_{d,k}^*(\overline{w}_{(d-1)k/2}))\neq 0$ and $\overline{w}_{(d-1)k/2}\neq 0$.
\end{proof}

\begin{remark}
In general, without analyzing the dyadic presentation of $d-1$ in more detail, we cannot give a better result.
For example, in the case $k=2^2=4$ and $d=3$ we have 
\begin{align*}
\pi(\rho_{3,4}^*(\overline{w}))&=(1+\pi(D_{2,1})+\pi(D_{2,0}))(1+\pi(D_{2,1})^2+\pi(D_{2,0})^2)\\
&=(1+V_1^2+V_2+V_1V_2)(1+V_1^4+V_2^2+V_1^2V_2^2)\\
&=(1+V_1^2+V_2+V_1V_2)(1+V_2^2+V_1^2V_2^2)\\
&=1+ (V_1^2+V_2)+ V_1V_2+V_2^2\in\F_2[V_1,V_2]/\langle V_1^3, V_2^3\rangle.
\end{align*}
Hence, $\pi(\rho_{3,4}^*(\overline{w}_4))=V_2^2\neq 0$ and $\pi(\rho_{3,4}^*(\overline{w}_i))=0$ for all $i>k(d-1)/2=4$. 
\end{remark}

\begin{remark}
On the other hand a better result can be obtained, as the following example will show.
Let $k=2^2=4$ and $d=6$.
We compute
\begin{align*}
\pi(\rho_{6,4}^*(\overline{w}))&=(1+\pi(D_{2,1})+\pi(D_{2,0}))(1+\pi(D_{2,1})^2+\pi(D_{2,0})^2)\\
&\quad \ (1+\pi(D_{2,1})^4+\pi(D_{2,0})^4)\\
&=(1+V_1^2+V_2+V_1V_2)(1+V_1^4+V_2^2+V_1^2V_2^2)\\
&\quad \ (1+V_1^8+V_2^4+V_1^4V_2^4)\\
&=1+(V_1^2+V_2)+V_1V_2+(V_1^4+V_2^2)+(V_1^4V_2+V_2^3)+\\
&\quad \ (V_1V_2^3+V_1^5V_2)+(V_1^2V_2^3+V_1^4V_2^2+V_2^4)+V_1^3V_2^3 + (V_1^2V_2^4+V_2^5)+ \\
&\quad \ V_1V_2^5\in\F_2[V_1,V_2]/\langle V_1^6, V_2^6\rangle.
\end{align*}
Thus, $\pi(\rho_{3,4}^*(\overline{w}_{10}))=V_1^2V_2^4+V_2^5\neq 0$, as Lemma \ref{lem : correction of 2.15} predicts, but actually more is true $\pi(\rho_{3,4}^*(\overline{w}_{11}))= V_1V_2^5\neq 0$.
A natural question arises: Can Lemma \ref{lem : correction of 2.15} be improved?
\end{remark}

\begin{remark}
The non-vanishing of the dual Stiefel--Whitney class\index{Stiefel--Whitney classes} $\overline{w}_{(d-1)k/2}(\xi_{\R^d,k})$ in the case when $d\geq 3$ is not a power of $2$ and  $k=2^m$ which was established in Lemma \ref{lem : correction of 2.15} can be also obtain as follows. 
Consider the embedding
\[
\underbrace{\conf(\R^d,2)\times\dots\times \conf(\R^d,2)}_{2^{m-1}}\longrightarrow \conf(\R^d,2^m)
\]
induced by the fixed embeddings $e_i \colon \R^d \longrightarrow \R^d$ for $1\leq i\leq 2^{m-1}$ such that their images are pairwise disjoint open $d$-balls. 
It induces a continuous map
\[
\underbrace{\conf(\R^d,2)/\Z_2\times\dots\times \conf(\R^d,2)/\Z_2}_{2^{m-1}}\longrightarrow \conf(\R^d,2^m)/\Sym_{2^m}.
\]
Then the pull-back of the vector bundle\index{vector bundle} $\xi_{\R^d,k}$ is isomorphic to the product vector bundle $\prod_{i=1}^{2^{m-1}} \xi_{\R^d,2}$.
Now, the embedding $S^{d-1}\longrightarrow \conf(\R^d,2)$ given by $x\longmapsto (x,-x)$ induces a continuous map $\RP^{d-1}\longrightarrow \conf(\R^d,2)/\Z_2$ such that the pullback bundle of the vector bundle $\xi_{\R^d,2}$ is isomorphic to the Whitney sum of the tautological line vector bundle $\gamma_{1}^{d-1}$ and a trivial line bundle.
Thus, the pull-back of the vector bundle $\xi_{\R^d,k}$ along the composition 
\[
\underbrace{\RP^{d-1}\times\dots\times \RP^{d-1}}_{2^{m-1}}\longrightarrow 
\underbrace{\conf(\R^d,2)/\Z_2\times\dots\times \conf(\R^d,2)/\Z_2}_{2^{m-1}}\longrightarrow \conf(\R^d,2^m)/\Sym_{2^m}
\]
is the product vector bundle $\gamma:=\prod_{i=1}^{2^{m-1}}(\gamma_{1}^{d-1}\oplus\tau_1^{d-1})$.
Here $\tau_1^{d-1}$ denotes the trivial line vector bundle over the real projective space $\RP^{d-1}$. 
Now we compute 
\[
\overline{w}(\gamma)=\overline{w}\big(\prod_{i=1}^{2^{m-1}}(\gamma_{1}^{d-1}\oplus\tau_1^{d-1})\big)=\bigtimes_{i=1}^{2^{m-1}}\overline{w}(\gamma_{1}^{d-1}\oplus\tau_1^{d-1})=\bigtimes_{i=1}^{2^{m-1}}\overline{w}(\gamma_{1}^{d-1}).
\]
Here ``$\bigtimes$'' denotes the cross product in cohomology.
Since $w(\gamma_{1}^{d-1})=1+u$, where $u\in H^*(\RP^{d-1};\F_2)=\F_2[u]/\langle u^d\rangle$, we have that
$\overline{w}(\gamma_{1}^{d-1})=w(\gamma_{1}^{d-1})^{-1}=1+u+\dots+u^{d-1}$.
Therefore,
\[
\overline{w}_{(d-1)2^{m-1}}(\gamma)=\overline{w}_{(d-1)k/2}(\gamma)= \underbrace{u\times\dots\times u}_{2^{m-1}}\neq 0.
\]
Having this simpler argument in mind it is natural to ask why we chose to present the more involved proof of the same fact.
The reason lies in the method of the proof of Lemma \ref{lem : correction of 2.15} which is used once again for the proof of the next lemma.
\end{remark}

Now, using a similar line of arguments as in the proof of Lemma \ref{lem : correction of 2.15} and taking into account the dyadic presentation of $d-1$ we get the following particular result. 

\begin{lemma}
\label{lem : correction of 2.15-better}
Let $d\geq 6$ be an even integer such that for some $a\geq 3$ holds $2^{a-1}+1\leq d\leq 2^a-1$, and let $k=2^m$ for some integer $m\geq 1$.  
Then the dual Stiefel--Whitney class\index{Stiefel--Whitney classes} 
\[
\overline{w}_{dk/2-1}(\xi_{\R^d,k})=\overline{w}_{d2^{m-1}-1}(\xi_{\R^d,2^m})
\] 
does not vanish.
\end{lemma}
\begin{proof}
Let us assume that $d-1=2^{a_1}+\dots+2^{a_q}$ where $0= a_1<\dots<a_q=a-1$.
Following the steps of the proof of Lemma \ref{lem : correction of 2.15} we reach again the equality:
	{\small
\begin{align} \label{eq:longeq}
\pi(\rho_{d,2^m}^*(\overline{w}))&=\prod_{i=0}^{a-1}\big( 1+\pi(D_{m,m-1})^{2^i}+\dots+\pi(D_{m,0})^{2^i} \big)\\
&=\prod_{i=0}^{a-1}
\Big(
1+
\big(V_1^{2^{m-1}}+V_2^{2^{m-2}}+\cdots+V_m^{2^{0}} \big)^{2^i}+\dots+\nonumber\\
&\quad\big(\sum_{1\leq j_1<\cdots< j_r\leq m}
(V_1\cdots V_{j_1-1})^{2^r}\, (V_{j_1+1}\cdots V_{j_2-1})^{2^{r-1}}\cdots  (V_{j_r+1}\cdots V_{m})^{2^{0}}\big)^{2^i}\nonumber \\ &\qquad\qquad\qquad\qquad\qquad\qquad\qquad\qquad +\dots+ \nonumber\\
&\quad\big(V_1\cdots V_m\big)^{2^i}
\Big)\nonumber\\
&=\prod_{i=0}^{a-1}
\Big(
1+
\big(V_1^{2^{i+m-1}}+V_2^{2^{i+m-2}}+\cdots+V_m^{2^{i}}\big) +\dots+\nonumber\\
&\quad\big(\sum_{1\leq j_1<\cdots< j_r\leq m}
(V_1\cdots V_{j_1-1})^{2^{i+r}}\, (V_{j_1+1}\cdots V_{j_2-1})^{2^{i+r-1}}\cdots  (V_{j_r+1}\cdots V_{m})^{2^{i}}\big)\nonumber \\ &\qquad\qquad\qquad\qquad\qquad\qquad\qquad\qquad +\dots+\nonumber\\
&\quad\big(V_1\cdots V_m\big)^{2^i}
\Big).\nonumber
\end{align}}
\noindent
Only now we show that 
\begin{multline*}
\pi(\rho_{d,2^m}^*(\overline{w}_{d2^{m-1}-1}))=\\
V_1\cdots V_{m-1}V_m^{d-1}+S \ \in \ \F_2[V_{1},\dots ,V_{m}]/\langle V_{1}^d,\dots,V_{m}^d \rangle,
\end{multline*}
where $S$ is a sum of monomials in $V_1,\dots,V_m$ of degree $d2^{m-1}-1$ which are  different from $V_1\cdots V_{m-1}V_m^{d-1}$.
Hence, $\pi(\rho_{d,k}^*(\overline{w}_{d2^{m-1}-1}))\neq 0$, and consequently $\overline{w}_{d2^{m-1}-1}\neq 0$.

\medskip
Indeed, observe that in every monomial of the $i$th factor of the product \eqref{eq:longeq} which has the variable $V_m$ with positive exponent this exponent is always the same and equal to $2^i$, $0\leq i \leq a-1$. 
In particular, in the $i$th factor each monomial with the variable $V_m$ is of the form $p_i(V_1,\dots, V_{m-1})^{2^i}V_m^{2^i}$ where $p_i(V_1,\dots, V_{m-1})$ is a monomial in variables $V_1,\dots, V_{m-1}$.
Now, when multiplying out the product \eqref{eq:longeq} the monomial of the form $p(V_1,\dots, V_{m-1})V_m^{d-1}$ can appear in the final result if and only if we take non-zero
\begin{compactitem}[ \ ---]
\item monomials $p_i(V_1,\dots, V_{m-1})^{2^i}V_m^{2^i}$ from $i$th factors where  $i\in \{a_1,\dots,a_q\}$, and
\item monomials $p_j'(V_1,\dots, V_{m-1})$ from  $j$th factors for  $j\in\{0,\dots, a-1\}{\setminus}\{a_1,\dots,a_q\}$.
\end{compactitem}
Thus, we have  
\begin{multline*}
p(V_1,\dots, V_{m-1})V_m^{d-1} = \prod_{i\in \{a_1,\dots,a_q\}}p_i(V_1,\dots, V_{m-1})^{2^i}V_m^{2^i} \cdot \\ \prod_{j\in\{0,\dots, a-1\}{\setminus}\{a_1,\dots,a_q\}}p_j'(V_1,\dots, V_{m-1}).
\end{multline*}
Observe that, if $p_i(V_1,\dots, V_{m-1})\neq 1$ for some $i\in \{a_1,\dots,a_q\}$, then there exists $1\leq t \leq m-1$ such that $V_t\mid p_i(V_1,\dots, V_{m-1})$. 
Hence, $V_t^{2^i}\mid p_i(V_1,\dots, V_{m-1})^{2^i}$.

\medskip
Now we want to understand in how many ways we can obtain the monomial $V_1\dots V_{m-1}V_m^{d-1}$ when we multiply out the product \eqref{eq:longeq}.
This means that we need to find all possible $p_i$s and $p_j'$'s such that
\begin{multline*}
V_1\cdots V_{m-1}V_m^{d-1} = \prod_{i\in \{a_1,\dots,a_q\}}p_i(V_1,\dots, V_{m-1})^{2^i}V_m^{2^i} \cdot \\ \prod_{j\in\{0,\dots, a-1\}{\setminus}\{a_1,\dots,a_q\}}p_j'(V_1,\dots, V_{m-1}).
\end{multline*}
From the previous observation and the fact that $0=a_1<a_2<\dots<a_q=a-1$ we conclude that $p_{a_2}(V_1,\dots, V_{m-1})=\dots=p_{a_q}(V_1,\dots, V_{m-1})=1$.
Thus, the previous equality becomes
\begin{multline*}
V_1\cdots V_{m-1}V_m^{d-1} = p_{a_1}(V_1,\dots, V_{m-1})^{2^{a_1}}V_m^{d-1} \cdot \\ \prod_{j\in\{0,\dots, a-1\}{\setminus}\{a_1,\dots,a_q\}}p_j'(V_1,\dots, V_{m-1}).
\end{multline*}
Taking additionally into account that $a_1=0$ we have that
\begin{multline*}
V_1\cdots V_{m-1}V_m^{d-1} = p_{0}(V_1,\dots, V_{m-1})V_m^{d-1} \cdot \\ \prod_{j\in\{1,\dots, a-1\}{\setminus}\{a_2,\dots,a_q\}}p_j'(V_1,\dots, V_{m-1}).
\end{multline*}
Therefore, the monomial $V_1\cdots V_{m-1}V_m^{d-1}$ can be obtained only in the case when we choose 
\begin{align*}
p_{0}(V_1,\dots, V_{m-1}) &=V_1\cdots V_{m-1},	\\
p_{i}(V_1,\dots, V_{m-1})&=1,\\
p_{j}'(V_1,\dots, V_{m-1})&=1,
\end{align*}
for all $i\in \{a_2,\dots,a_q\}$ and all $j\in \{1,\dots, a-1\}{\setminus}\{a_2,\dots,a_q\}$.

\medskip
Indeed, if we assume that $p_{a_1}(V_1,\dots, V_{m-1})=p_{0}(V_1,\dots, V_{m-1})\neq 1$, then obviously we must have  $p_{0}(V_1,\dots, V_{m-1})=V_1\cdots V_{m-1}$.
This is possible by taking monomial $V_1\dots V_{m-1}V_m$ in the $a_0$th factor of the product \eqref{eq:longeq}.
On the other hand, if we assume that $p_{a_1}(V_1,\dots, V_{m-1})=p_{0}(V_1,\dots, V_{m-1})=1$ then we should have that 
\[
\prod_{j\in\{1,\dots, a-1\}{\setminus}\{a_2,\dots,a_q\}}p_j'(V_1,\dots, V_{m-1})=V_1\cdots V_{m-1}.
\]
This is not possible since for every  $p_j'(V_1,\dots, V_{m-1})\neq 1$ in the previous product there is $1\leq t \leq m-1$ such that $V_t\mid p_j'(V_1,\dots, V_{m-1})$.
Looking closer at the typical monomials in the product \eqref{eq:longeq} we see that more is true, actually $V_t^{2^j}\mid p_j'(V_1,\dots, V_{m-1})$ and so 
\[
V_t^{2^j}\mid \prod_{j\in\{1,\dots, a-1\}{\setminus}\{a_2,\dots,a_q\}}p_j'(V_1,\dots, V_{m-1})=V_1\cdots V_{m-1}.
\]
Since $j\in \{1,\dots, a-1\}{\setminus}\{a_2,\dots,a_q\}$ we have that $2^j\geq 2$ and so $V_t^{2}\mid V_1\cdots V_{m-1}$; contradiction.

\medskip
Hence, we showed that 
\begin{multline*}
\pi(\rho_{d,2^m}^*(\overline{w}_{d2^{m-1}-1}))=
V_1\cdots V_{m-1}V_m^{d-1}+S \ \in \ \F_2[V_{1},\dots ,V_{m}]/\langle V_{1}^d,\dots,V_{m}^d \rangle,
\end{multline*}
where $S$ is a sum of monomials in $V_1,\dots,V_m$ of degree $d2^{m-1}-1$ which are  different from $V_1\cdots V_{m-1}V_m^{d-1}$.
Thus, $\overline{w}_{d2^{m-1}-1}\neq 0$.
\end{proof}

\medskip
Now we extend the previous two lemmas to the case of an arbitrary integer $k\geq 1$ and consequently correct \cite[Lem.\,2.17]{Blagojevic2016-01}.
For an integer $k\geq 1$ let $\epsilon(k)$ be the reminder of $k$ modulo $2$, that is $\epsilon(k)=1$ for $k$ odd, and $\epsilon(k)=0$ when $k$ is even.
\begin{lemma}
~
\label{lem : correction of 2.17}
\begin{compactenum}[\rm (1)]
\item Let $d\geq 3$ be an integer which is not a power of $2$, and let $k\geq 1$ be an integer.
	Then the dual Stiefel--Whitney class\index{Stiefel--Whitney classes} 
	\[
	\overline{w}_{(d-1)(k-\epsilon(k))/2}(\xi_{\R^d,k})
	\]
	does not vanish.
\item Let $d\geq 6$ be an even integer which is not a power of $2$, and let $k\geq 1$ be an integer.  
Then the dual Stiefel--Whitney class 
\[
\overline{w}_{d(k-\epsilon(k))/2-\alpha(k)+\epsilon(k)}(\xi_{\R^d,k}) 
\] 
does not vanish.
\end{compactenum}

\end{lemma}

\begin{proof}
Let $r:=\alpha(k)$ and $k=2^{b_1}+\dots+2^{b_r}$ where $0\leq b_1<b_2<\dots<b_r$.  
As in the proof of Lemma \ref{lemma : corection not a power of 2} we consider a morphism between vector bundles\index{vector bundle} $\prod_{i=1}^{r}\xi_{\R^d,2^{b_i}}$ and $\xi_{\R^d,k}$ where the following commutative square is a pullback diagram: 
 \[
\xymatrix{\prod_{i=1}^{r}\xi_{\R^d,2^{b_i}} \ar[rr]^-{\Theta}  \ar[d] \
& & \ \xi_{\R^d,k}   \ar[d] 
\\
%& & \\
\prod_{i=1}^{r} \conf(\R^d,2^{b_i})/\Sym_{2^{b_i}}  \ar[rr]^-{\theta} \  & & \ 
 \conf(\R^d,k)/{\Sym_k}.
}
\]
The map $\theta$ is induced, up to an equivariant homotopy, from a restriction of the little cubes operad\index{little cubes operad} structural map
\[
(\CC_d(2^{b_1})\times\dots\times \CC_d(2^{b_r}))\times \CC_d(r)\longrightarrow\CC_d(2^{b_1}+\dots+2^{b_r}),
\] 
as explained in the proof of Lemma \ref{lemma : corection not a power of 2}.
The naturality  of the Stiefel--Whitney classes~\cite[Ax.\,2, p.\,37]{Milnor1974}\index{Stiefel--Whitney classes} implies that 
\[
\theta^*(\overline{w}_{N}(\xi_{\R^d,k}))=\overline{w}_{N}\Big(\prod_{i=1}^{r}\xi_{\R^d,2^{b_i}}\Big),
\]
for any integer $N\geq 0$.
Further on, the product formula \cite[Pr.\,4-A, p.\,54]{Milnor1974} implies that
\[
 \overline{w}\Big(\prod_{i=1}^{r}\xi_{\R^d,2^{b_i}}\Big)=\overline{w}(\xi_{\R^d,2^{b_1}})\times\cdots\times\overline{w}(\xi_{\R^d,2^{b_r}}).
\]
Consequently,
\begin{align}\label{sum-10}
 \theta^*(\overline{w}_{N}(\xi_{\R^d,k}))  &=\overline{w}_{N}\Big(\prod_{i=1}^{r}\xi_{\R^d,2^{b_i}}\Big)\nonumber \\
&=\sum_{s_1+\cdots +s_r=N}\overline{w}_{s_1}(\xi_{\R^d,2^{b_1}})\times\cdots\times\overline{w}_{s_r}(\xi_{\R^d,2^{b_r}}).	
\end{align}
According to the K\"unneth formula\index{K\"unneth formula} \cite[Thm.\,VI.3.2]{Bredon2010}  each term, the cross product, $\overline{w}_{s_1}(\xi_{\R^d,2^{b_1}})\times\cdots\times\overline{w}_{s_r}(\xi_{\R^d,2^{b_r}})$ in the previous sum belongs to a different direct summand of the  cohomology
\begin{multline*}
H^{N}\Big(\prod_{i=1}^{r} \conf(\R^d,2^{b_i})/\Sym_{2^{b_i}};\F_2\Big)\cong\\
\bigoplus_{s_1+\cdots +s_r=N}H^{s_1}( \conf(\R^d,2^{b_1})/\Sym_{2^{b_1}};\F_2)\otimes\cdots\otimes H^{s_r}( \conf(\R^d,2^{b_r})/\Sym_{2^{b_r}};\F_2).
\end{multline*}
Hence the following equivalence holds
\begin{multline*}
\overline{w}_{N}\Big(\prod_{i=1}^{r}\xi_{\R^d,2^{b_i}}\Big) \neq 0 \Longleftrightarrow\\
\overline{w}_{s_1}(\xi_{\R^d,2^{b_1}})\times\cdots\times\overline{w}_{s_r}(\xi_{\R^d,2^{b_r}})\neq 0
\; \text{for some} \; s_1+\cdots +s_r=N.
\end{multline*}
To isolate a non-zero summand in \eqref{sum-10} we use either Lemma~\ref{lem : correction of 2.15} or Lemma~\ref{lem : correction of 2.15-better}.

\medskip 
Let $d\geq 3$ be an integer which is not a power of $2$, and let $N=(d-1)(k-\epsilon(k))/2$.
The Lemma~\ref{lem : correction of 2.15} states that the dual Stiefel--Whitney class\index{Stiefel--Whitney classes} $\overline{w}_{(d-1)2^{b_i-1}}(\xi_{\R^d,2^{b_i}})$ does not vanish when $b_i\geq 1$.
First consider the case when $k$ is even.
Hence, $\epsilon(k)=0$, and since $k=2^{b_1}+\dots+2^{b_r}$ we have that $1\leq b_1<b_2<\dots<b_r$.
Now, since  $N=(d-1)2^{b_1-1}+\dots +(d-1)2^{b_r-1}$ the following summand in \eqref{sum-10} does not vanish
\[
 \overline{w}_{(d-1)2^{b_1-1}}(\xi_{\R^d,2^{b_1}})\times\cdots\times\overline{w}_{(d-1)2^{b_r-1}}(\xi_{\R^d,2^{b_r}})\neq 0.
\]
In the case when $\epsilon(k)=1$,  $k=2^{b_1}+\dots+2^{b_r}$ and $0=b_1<b_2<\dots<b_r$, the summand in \eqref{sum-10} that does not vanish is 
\[
 \overline{w}_{0}(\xi_{\R^d,2^{b_1}})\times \overline{w}_{(d-1)2^{b_2-1}}(\xi_{\R^d,2^{b_2}}) \times\cdots\times\overline{w}_{(d-1)2^{b_r-1}}(\xi_{\R^d,2^{b_r}})\neq 0.
\]
Consequently, $\theta^*(\overline{w}_{(d-1)(k-\epsilon(k))/2}(\xi_{\R^d,k}))\neq 0$, and so 
$\overline{w}_{(d-1)(k-\epsilon(k))/2}(\xi_{\R^d,k})\neq 0$.

\medskip 
Let now $d\geq 6$ be an even integer which is not a power of $2$, and let $N=d(k-\epsilon(k))/2-\alpha(k)+\epsilon(k)$.
The Lemma~\ref{lem : correction of 2.15-better} states that the dual Stiefel--Whitney class $\overline{w}_{d2^{b_i-1}-1}(\xi_{\R^d,2^{b_i}})$ does not vanish when $b_i\geq 1$.
Again we first consider the case when $k$ is even, that is $\epsilon(k)=0$.
Since $k=2^{b_1}+\dots+2^{b_r}$ we have that $1\leq b_1<b_2<\dots<b_r$.
Now, since  $N=(d2^{b_1-1}-1)+\dots +(d2^{b_r-1}-1)$ the following summand in \eqref{sum-10} does not vanish
\[
 \overline{w}_{d2^{b_1-1}-1}(\xi_{\R^d,2^{b_1}})\times\cdots\times\overline{w}_{d2^{b_r-1}-1}(\xi_{\R^d,2^{b_r}})\neq 0.
\]
In the case of odd $k$ we have that $\epsilon(k)=1$,  $k=2^{b_1}+\dots+2^{b_r}$ and $0=b_1<b_2<\dots<b_r$.
The summand in \eqref{sum-10} that does not vanish is 
\[
 \overline{w}_{0}(\xi_{\R^d,2^{b_1}})\times \overline{w}_{d2^{b_2-1}-1}(\xi_{\R^d,2^{b_2}}) \times\cdots\times\overline{w}_{d2^{b_r-1}-1}(\xi_{\R^d,2^{b_r}})\neq 0.
\]
Consequently, $\theta^*(\overline{w}_{d(k-\epsilon(k))/2-\alpha(k)+\epsilon(k)}(\xi_{\R^d,k}))\neq 0$, and so the dual Stiefel--Whiney class 
$\overline{w}_{d(k-\epsilon(k))/2-\alpha(k)+\epsilon(k)}(\xi_{\R^d,k})$ does not vanish.
\end{proof}

\medskip
Summarizing the results obtained in Lemma \ref{lemma : corection power of 2}, Lemma \ref{lemma : corection not a power of 2}, Lemma  \ref{lem : correction of 2.15}, Lemma  \ref{lem : correction of 2.15-better} and Lemma \ref{lem : correction of 2.17} we get the following theorem which is a correction of \cite[Thm.\,2.13]{Blagojevic2016-01}.

\begin{theorem}
	\label{th : Correctionog T 2.13}
	Let $k\geq 1$ and $d\geq 1$ be integers.
	\begin{compactenum}[ \ \rm (1)]
	\item If $d$ is a power of $2$, then the dual Stiefel--Whitney class 
	\[\overline{w}_{(d-1)(k-\alpha(k))}(\xi_{\R^d,k})\]
	does not vanish.
	\item If $d$ is not a power of $2$, then the dual Stiefel--Whitney class\index{Stiefel--Whitney classes} 
	\[\overline{w}_{(d-1)(k-\epsilon(k))/2}(\xi_{\R^d,k})\] 
	does not vanish. 
	\item If $d$ is an even integer which is not a power of $2$, then the dual Stiefel--Whitney class\index{Stiefel--Whitney classes}
	\[\overline{w}_{d(k-\epsilon(k))/2-\alpha(k)+\epsilon(k)}(\xi_{\R^d,k}) 
	\] 
does not vanish.	
	\end{compactenum}
\end{theorem}

\medskip
Now, we use Theorem \ref{th : Correctionog T 2.13} and the criterion given in Lemma \ref{lemma:criterion} to correct the result stated in \cite[Thm.\,2.1]{Blagojevic2016-01}.
In this way we completed corrections of the invalid results in \cite[Sec.\,2]{Blagojevic2016-01}.
\begin{theorem}
	\label{th : Correctionog T 2.1}
	Let $k\geq 1$ and $d\geq 1$ be integers.
	Denote by $\alpha(k)$  the number of $1$s in the dyadic presentation of $k$, and by $\epsilon(k)$ the remainder of $k$ modulo $2$.
	\begin{compactenum}[ \ \rm (1)]
	
	\item\label{th : Correctionog T 2.1-1} If $d$ is a power of $2$, then there is no $k$-regular embedding\index{$k$-regular embedding} $\R^d\longrightarrow\R^N$ for
	\[
	N\leq d(k-\alpha(k))+\alpha(k)-1.
	\] 
	
	\item\label{th : Correctionog T 2.1-2} If $d$ is not a power of $2$, then there is no $k$-regular embedding $\R^d\longrightarrow\R^N$ for
	\[
	N\leq  \frac12(d-1)(k-\epsilon(k))+k-1.
	\] 
	\item\label{th : Correctionog T 2.1-3} If $d$ is an even integer which is not a power of $2$, then is no $k$-regular embedding $\R^d\longrightarrow\R^N$ for
	\[
	N\leq  \frac12d(k-\epsilon(k))+k-\alpha(k)+\epsilon(k) -1.
	\] 

	\end{compactenum}

\end{theorem}

%===========================
\subsection{$\ell$-skew embeddings}
\label{sub : l-skew maps}
%===========================

In this section we revise \cite[Sec.\,3]{Blagojevic2016-01} and correct the related results \cite[Thm.\,3.1, Thm.\,3.7]{Blagojevic2016-01}.
In order for our presentation to be complete we recall basic definitions and necessary facts. 

\medskip
The affine subspaces $L_1,\ldots,L_{\ell}$ of the Euclidean space $\R^N$ are
{\bf affinely independent} if 
\[
\dimaff(L_1\cup\dots\cup L_{\ell})=(\dimaff L_1+1)+\cdots+(\dimaff L_{\ell}+1)-1.
\]  
In particular,  any two lines in $\R^3$ are skew if and only if they are affinely independent.
 
\medskip
Let $M$ be a real smooth $d$-dimensional manifold.
Then $TM$ denotes he tangent bundle of $M$, and $T_yM$ stands for the tangent space of $M$ at the point $y\in M$.
For a smooth map $f \colon M\longrightarrow \R^N$ we denote by $df\colon TM\longrightarrow T\R^N$  the differential map between associated tangent vector bundles induced by $f$.
Further on, let $\iota \colon T\R^N \longrightarrow \R^N$ denotes the map that sends a tangent vector $v \in T_x\R^N$ at the point $x \in \R^N$ to the point $x+v$.
Here the standard identification $T_x\R^N = \R^N$ is assumed.

\begin{definition}\label{def:skew_map}
  Let $\ell\geq 1$ be an integer, and let $M$ be a real smooth $d$-dimensional manifold.  
  A smooth embedding
  $f \colon M\longrightarrow \R^N$ is an {\bf $\ell$-skew embedding}\index{$\ell$-skew embedding} if for every point $(y_1,\ldots,y_{\ell})\in \conf(M,\ell)$
  the affine subspaces
  \[
  (\iota\circ df_{y_1})(T_{y_1}M),\ldots ,(\iota\circ df_{y_{\ell}})(T_{y_{\ell}}M)
  \]
  of $\R^N$ are affinely independent.
\end{definition}

\medskip
Now, like in the case of $k$-regular embeddings, a criterion for non-existence of $\ell$-skew embedding can be derived in terms of Stiefel--Whitney class\index{Stiefel--Whitney classes} of appropriate vector bundle over the configuration space.
We recall \cite[Lem.\,3.6]{Blagojevic2016-01}.

\begin{lemma} \label{lem:dual:Whitney_skew}
  Let $d\geq 1$ and $\ell\geq 1$ be integers.  
  If the dual Stiefel--Whitney class 
  \[
  \overline{w}_{N - (d+1)\ell +2}(\xi_{\R^d,\ell}^{\oplus (d+1)})
  \]
  does not vanish, then there is no  $\ell$-skew embedding $\R^d\longrightarrow \R^N$.
\end{lemma}

\medskip
Motivated by the criterion in Lemma \ref{lem:dual:Whitney_skew}, based on the work of {\Hung} \cite[(4.7)]{Hung1990}, a relevant study of Stiefel--Whitney classes of the vector bundle\index{vector bundle} $\xi_{\R^d,\ell}^{\oplus (d+1)}$ was given in \cite[Thm.\,3.7]{Blagojevic2016-01}.
\begin{quote}
{\small
	{\bf Theorem 3.7.}
	Let $d,\ell\geq 1$ be integers.
	Then the  dual Stiefel--Whitney class 
	\[
	\overline{w}_{(2^{\gamma(d)}-d-1)(\ell-\alpha(\ell))}\bigl(\xi_{\R^d,\ell}^{\oplus (d+1)}\bigr)
	\] 
	does not vanish. 
	}
\end{quote}
Here $\gamma(d)=\lfloor\log_2d\rfloor+1$ for $d\geq 1$.
The result of the previous theorem in combination with \cite[Lem.\,3.6]{Blagojevic2016-01} directly implied the following result \cite[Thm.\,3.1]{Blagojevic2016-01}.
\begin{quote}
{\small
	{\bf Theorem 3.1.}
	 Let $\ell,d\geq 2$ be integers.  
	 There is no $\ell$-skew embedding\index{$\ell$-skew embedding} $\R^d\longrightarrow \R^{N}$ for 
	 \[
	 N\leq 2^{\gamma(d)}(\ell-\alpha(\ell))+(d+1)\alpha(\ell)-2,
	 \] 
	 where $\alpha(\ell)$ denotes the number of ones in the dyadic presentation of $l$ and
  $\gamma(d)=\lfloor\log_2d\rfloor+1$.
	}
\end{quote}

\medskip
The proof of \cite[Thm.\,3.7]{Blagojevic2016-01} was based on some of the results presented in \cite[Sec.\,2]{Blagojevic2016-01}.
In particular, in \cite[Sec.\,3.3.3]{Blagojevic2016-01} an incorrect result was used \cite[Cor.\,2.16]{Blagojevic2016-01}.
Consequently, \cite[Thm.\,3.1]{Blagojevic2016-01} does not stand.
In the following we give correct versions, first of \cite[Thm.\,3.7]{Blagojevic2016-01}, and then of \cite[Thm.\,3.1]{Blagojevic2016-01}.

\begin{theorem}
\label{thm : Correction of T.3.7}
Let $d\geq 1$ and $\ell\geq 2$ be integers.
\begin{compactenum}[ \ \rm (1)]
%--1--
\item If $d=2$, and if $\ell\geq 2$ is an integer, then 
\[
\overline{w}_{\ell-\alpha(\ell)}\bigl(\xi_{\R^d,\ell}^{\oplus (d+1)}\bigr)=\overline{w}_{\ell-\alpha(\ell)}\bigl(\xi_{\R^2,\ell}^{\oplus 3}\bigr)\neq 0.
\] 
%--2--
\item If $d\geq 1$ is an integer, and if $\ell=2$, then 
\[
\overline{w}_{2^{\gamma(d)}-d-1}\bigl(\xi_{\R^d,\ell}^{\oplus (d+1)}\bigr)=\overline{w}_{2^{\gamma(d)}-d-1}\bigl(\xi_{\R^d,2}^{\oplus (d+1)}\bigr)\neq 0.
\]  
%--3--
\item If $d\geq 2$ is a power of $2$, and if $\ell\geq 2$ is an integer, then 
\[
\overline{w}_{(d-1)(\ell-\alpha(\ell))}\bigl(\xi_{\R^d,\ell}^{\oplus (d+1)}\bigr)\neq 0.
\]
%--4--
\item If $d+1\geq 2$ is a power of $2$, and if $\ell\geq 2$ is an integer, then  
\[
\overline{w}\bigl(\xi_{\R^d,\ell}^{\oplus (d+1)}\bigr)=w\bigl(\xi_{\R^d,\ell}^{\oplus (d+1)}\bigr)=1.
\]
%%--5--
%\item If $d\geq 5$ is an integer which is not a power of $2$, and in addition $d+1$ is not a power $2$, and if $\ell\geq 2$ is a power of $2$, then
%\[
%\overline{w}_{(2^{\gamma(d)}-d-1)\ell/2}\bigl(\xi_{\R^d,\ell}^{\oplus (d+1)}\bigr)\neq 0.
%\]
%--6--
\item If $d\geq 5$ is an integer which is not a power of $2$, and in addition $d+1$ is not a power of $2$, and if $\ell\geq 3$ is an integer, then
\[
\overline{w}_{(2^{\gamma(d)}-d-1)(\ell-\epsilon(\ell) )/2}\bigl(\xi_{\R^d,\ell}^{\oplus (d+1)}\bigr)\neq 0.
\]
%--7--
\item If $d\geq 5$ is an integer which is not a power of $2$, $2^{\gamma(d)}-d-1=2^{a_1}z$ where $a_1\geq 0$ is an integer and $z\geq 1$ is an odd integer, and $d+1$ is not a power of $2$, and if $\ell\geq 3$ is an integer, then
\[
\overline{w}_{(2^{\gamma(d)}-d-1+2^{a_1})(\ell-\epsilon(\ell) )/2-2^{a_1}\alpha(\ell)}\bigl(\xi_{\R^d,\ell}^{\oplus (d+1)}\bigr)\neq 0.
\]
\end{compactenum}
\end{theorem}

\begin{proof}
We prove the theorem by discussing all cases separately.

\medskip
{\bf (1)} 
Let $d=2$, and let $\ell\geq 2$ be an integer. 
Then from \cite[Thm.\,1]{CohenMachowaldMilgram1978} we get that the bundle $\xi_{\R^2,\ell}^{\oplus 2}$ is a trivial bundle.
Consequently, $\overline{w}\bigl(\xi_{\R^2,\ell}^{\oplus 3}\bigr)=\overline{w}\bigl(\xi_{\R^2,\ell}\bigr)$. 
Since $d=2$ is power of $2$ we can use Lemma \ref{lemma : corection not a power of 2} to get that
\[
\overline{w}_{\ell-\alpha(\ell)}\bigl(\xi_{\R^d,\ell}^{\oplus (d+1)}\bigr)=
\overline{w}_{\ell-\alpha(\ell)}\bigl(\xi_{\R^2,\ell}^{\oplus 3}\bigr)=
\overline{w}_{\ell-\alpha(\ell)}\bigl(\xi_{\R^2,\ell}\bigr)\neq 0.
\]

\medskip
{\bf (2)}
Let $d\geq 2$ be an integer and let $\ell=2$. 
In this case the base space of the vector bundle $\xi_{\R^2,\ell}$ is the unordered configuration space $\conf(\R^d,2)/\Sym_2$.
The $\Sym_2$-equivariant map $\epicy_{d,2}\colon S^{d-1}\longrightarrow \conf(\R^d,2)$ given by $x\longmapsto (x,-x)$ is an $\Sym_2$-equivariant homotopy equivalence.
Consequently, $\epicy_{d,2}$ induces homotopy equivalence $\rho_{d,2}\colon S^{d-1}/ \Sym_2 \longrightarrow \conf(\R^d,2)/\Sym_2$.
Recall that $S^{d-1}/ \Sym_2\cong\RP^{d-1}$.
It can be directly checked that the vector bundle $\xi_{\R^d,2}$ over $\conf(\R^d,2)/\Sym_2$ pulls back to the vector bundle\index{vector bundle} isomorphic to the Whitney sum of the tautological line bundle and trivial line bundle over the projective space $\RP^{d-1}$.
Hence, if we denote the cohomology of the projective space by 
\[
H^*(\RP^{d-1},\F_2)\cong H^*(\conf(\R^d,2)/\Sym_2,\F_2)=\F_2[w_1]/\langle w_1^d\rangle,
\] 
where $\deg(w_1)=1$, we have that 
\[
w(\xi_{\R^d,2})=1+w_1.
\]
Now, the total Stiefel--Whitney class\index{Stiefel--Whitney classes} of the vector bundle $\xi_{\R^d,\ell}^{\oplus (d+1)}$ is
\[
w\bigl(\xi_{\R^d,\ell}^{\oplus (d+1)}\bigr)=
w\bigl(\xi_{\R^d,2}^{\oplus (d+1)}\bigr)=
(1+w_1)^{d+1}.
\]
Now notice that $2^{\gamma(d)}$ is the minimal power of $2$ that is greater that $d$.
Therefore, 
\begin{align*}
w\bigl(\xi_{\R^d,2}^{\oplus (d+1)}\bigr)(1+w_1)^{2^{\gamma(d)}-d-1}&=
(1+w_1)^{d+1}(1+w_1)^{2^{\gamma(d)}-d-1}\\
&=(1+w_1)^{2^{\gamma(d)}}\\
&=1.	
\end{align*}
Thus, we have that 
\[
\overline{w}\bigl(\xi_{\R^d,\ell}^{\oplus (d+1)}\bigr)=
\overline{w}\bigl(\xi_{\R^d,2}^{\oplus (d+1)}\bigr)=
(1+w_1)^{2^{\gamma(d)}-d-1}=
1+\dots+w_1^{2^{\gamma(d)}-d-1}.
\]
Sine $2^{\gamma(d)}-d-1<d$ we have that $w_1^{2^{\gamma(d)}-d-1}\neq 0$, and so $\overline{w}_{2^{\gamma(d)}-d-1}\bigl(\xi_{\R^d,2}^{\oplus (d+1)}\bigr)\neq 0$.

\medskip
{\bf (3)} 
Let $d=2^a$ for $a\geq 1$ an integer, and let $\ell\geq 2$ be an integer.
From the decomposition of vector bundles \eqref{eq : decomposition of vb} we have that
\[
w(\xi_{\R^d,\ell})=w(\zeta_{\R^d,\ell})=1+w_1+\dots+w_{\ell-1},
\]
where $w_i:=w_i(\xi_{\R^d,\ell})=w_i(\zeta_{\R^d,\ell})$.
Consequently,
\begin{align*}
w\bigl(\xi_{\R^d,\ell}^{\oplus (d+1)}\bigr)
&=(1+w_1+\dots+w_{\ell-1})^{d+1}\\
&=(1+w_1+\dots+w_{\ell-1})^{2^a}(1+w_1+\dots+w_{\ell-1})	\\
&=(1+w_1^{2^a}+\dots+w_{\ell-1}^{2^a})(1+w_1+\dots+w_{\ell-1})\\
&=1+w_1+\dots+w_{\ell-1}\\
&=w(\xi_{\R^d,\ell}).
\end{align*}
Here we used that from \eqref{eq : hight} we know that 
\[
\hght((H^{*}( \conf(\R^d,\ell)/\Sym_{\ell};\F_2))\leq \min\{ 2^t: 2^t\geq d \}=2^a=d.
\]
Thus, with the additional help of Lemma \ref{lemma : corection not a power of 2}, we get that
\[
\overline{w}_{(d-1)(\ell-\alpha(\ell))}\bigl(\xi_{\R^d,\ell}^{\oplus (d+1)}\bigr)=
\overline{w}_{(d-1)(\ell-\alpha(\ell))}(\xi_{\R^d,\ell})\neq 0.
\]

\medskip
{\bf (4)} 
Let $d=2^a-1$ for $a\geq 1$ an integer, and let $\ell\geq 2$ be an integer.
In the footsteps of the proof of the previous case we calculate:
\begin{align*}
w\bigl(\xi_{\R^d,\ell}^{\oplus (d+1)}\bigr)&=
(1+w_1+\dots+w_{\ell-1})^{d+1}\\
&=(1+w_1+\dots+w_{\ell-1})^{2^a}\\
&=1+w_1^{2^a}+\dots+w_{\ell-1}^{2^a}\\
&=1.
\end{align*}
In this case we used the fact that 
\[
\hght(H^{*}( \conf(\R^d,\ell)/\Sym_{\ell};\F_2))\leq \min\{ 2^t: 2^t\geq d \}=2^a=d+1.
\]
Consequently, $\overline{w}\bigl(\xi_{\R^d,\ell}^{\oplus (d+1)}\bigr)=1$.

\medskip
{\bf (5)}
This case is analyzed in two separate steps depending whether $\ell$ is a power of $2$ or not. 

\medskip\noindent
{\bf (5A)} 
Let $d\geq 5$ be an integer such that $2^{a-1}+1\leq d\leq 2^a-2$ where $a\geq 3$ is an integer.
We first consider the case when $\ell=2^m$ for $m\geq 2$ an integer.
The decomposition of vector bundles \eqref{eq : decomposition of vb} implies that
\[
w(\xi_{\R^d,2^m})=w(\zeta_{\R^d,2^m})=1+w_1+\dots+w_{2^m-1},
\]
and consequently,
\[
w\bigl(\xi_{\R^d,2^m}^{\oplus (d+1)}\bigr)=(1+w_1+\dots+w_{2^m-1})^{d+1}.
\]
From \eqref{eq : hight} we have that 
\[
\hght(H^{*}( \conf(\R^d,2^m)/\Sym_{2^m};\F_2))\leq \min\{ 2^t: 2^t\geq d \}=2^a.
\]
Therefore, 
\begin{align*}
w\bigl(\xi_{\R^d,2^m}^{\oplus (d+1)}\bigr)(1+w_1+\dots+w_{2^m-1})^{2^a-d-1}&=
(1+w_1+\dots+w_{2^m-1})^{2^a}\\
&=1+w_1^{2^a}+\dots+w_{2^m-1}^{2^a}\\
&=1,	
\end{align*}
and consequently,
\[
\overline{w}\bigl(\xi_{\R^d,2^m}^{\oplus (d+1)}\bigr)=(1+w_1+\dots+w_{2^m-1})^{2^a-d-1}.
\]
Now, let $2^a-d-1=2^{a_1}+\dots+2^{a_q}$ where $0\leq a_1<\dots<a_q\leq a-2$.
Then
\begin{align}
	\label{eq : dual class for d not a power of 2 -- skew}
	\overline{w}\bigl(\xi_{\R^d,2^m}^{\oplus (d+1)}\bigr)&=(1+w_1+\dots+w_{2^m-1})^{2^{a_1}+\dots+2^{a_q}}\nonumber \\
&= \prod_{i=1}^{q}\big( 1+w_1^{2^{a_i}}+\dots+w_{2^m-1}^{2^{a_i}} \big).
\end{align}

\medskip
Following the calculation in the proof of Lemma \ref{lem : correction of 2.15} we apply the monomorphism
\[
 \rho_{d,2^m}^*\colon H^*(\conf(\R^d,2^m)/\Sym_{2^m};\F_2)\longrightarrow H^*(\Sp(\R^d,2^m)/\Sy_{2^m};\F_2)
 \] 
 from Theorem \ref{th : injection } to the equality \eqref{eq : dual class for d not a power of 2 -- skew}, and use the decomposition of the cohomology
 \[
 H^*(\Sp(\R^d,2^m)/\Sy_{2^m};\F_2)
 \cong
 \F_2[V_{m,1},\dots,V_{m,m}]/\langle V_{m,1}^d,\ldots, V_{m,m}^d\rangle \oplus \III^*(\R^d,2^m),
 \] 
 given in Theorem \ref{th: cohomology of Sp}, to get that
\begin{align}
	\label{eq : dual class for d not a power of 2 -- skew - 02}
\rho_{d,2^m}^*(\overline{w}\bigl(\xi_{\R^d,2^m}^{\oplus (d+1)}\bigr))
&=  \prod_{i=1}^{q}\big( 1+D_{m,m-1}^{2^{a_i}}+\dots+D_{m,0}^{2^{a_i}} \big)+R \\
&\ \in \  H^*(\Sp(\R^d,2^m)/\Sy_{2^m};\F_2), \nonumber
\end{align}
where
\begin{compactitem}[ \ ---]
\item $D_{m,r}=(\kappa_{d,2^m}/\Sy_{2^m})^*(D_{m,r})=\rho_{d,2^m}^*(w_{2^m-2^r})$, for $0\leq r\leq m-1$, with the obvious abuse of notation, see \eqref{eq : summary}, and 
\item $R\in \III^*(\R^d,2^m)$.
\end{compactitem}

\medskip
Let $\pi$ be the following composition of the maps from the diagram \eqref{big diagram 02}:
\[
\xymatrix@C=0.5em{
\F_2[V_{m,1},\dots ,V_{m,m}]/\langle V_{m,1}^d,\dots,V_{m,m}^d \rangle \oplus I^*(\R^{d},2^m)\ar[d]^-{\text{projection}}\\
\F_2[V_{m,1},\dots ,V_{m,m}]/\langle V_{m,1}^d,\dots,V_{m,m}^d \rangle\ar[d]^-{\chi_m} \\
\F_2[V_{1},\dots ,V_{m}]/\langle V_{1}^d,\dots,V_{m}^d \rangle. \\
}
\]
The first map is the projection on a direct summand and the second map is induced by the change of variables $\chi_m\in\GL_m(\F_2)$.
We apply $\pi$ to \eqref{eq : dual class for d not a power of 2 -- skew - 02} and have
\[
\pi\big(\rho_{d,2^m}^*\big(\overline{w}\bigl(\xi_{\R^d,2^m}^{\oplus (d+1)}\bigr)\big)\big)=\prod_{i=1}^{q}\big( 1+\pi(D_{m,m-1})^{2^i}+\dots+\pi(D_{m,0})^{2^i} \big).
\]
The change of the variables $\chi_m$ transforms $\U_m(\F_2)$-invariants into $\mathrm{U}_m(\F_2)$-invariants and Dickson polynomials\index{Dickson invariants}, $\GL_m(\F_2)$-invariants, can be presented in terms of $\mathrm{U}_m(\F_2)$-invariants, as explained in \eqref{relation - recurrence Dickson - 03} and \eqref{relation - recurrence Dickson - 04}.
Hence,  $p:=\pi\big(\rho_{d,2^m}^*\big(\overline{w}\bigl(\xi_{\R^d,2^m}^{\oplus (d+1)}\bigr)\big)\big)$ can be computed further as follows:
\begin{align*}  
\quad p&=\prod_{i=1}^{q}\big( 1+\pi(D_{m,m-1})^{2^{a_i}}+\dots+\pi(D_{m,0})^{2^{a_i}} \big)	\nonumber\\
&=\prod_{i=1}^{q}
\Big(
1+
\big(V_1^{2^{m-1}}+V_2^{2^{m-2}}+\cdots+V_m^{2^{0}} \big)^{2^{a_i}}+\dots+\nonumber\\
&\quad\big(\sum_{1\leq j_1<\cdots< j_r\leq m}
(V_1\cdots V_{j_1-1})^{2^r}\, (V_{j_1+1}\cdots V_{j_2-1})^{2^{r-1}}\cdots  (V_{j_r+1}\cdots V_{m})^{2^{0}}\big)^{2^{a_i}}\nonumber \\ &\qquad\qquad\qquad\qquad\qquad\qquad\qquad\qquad +\dots+ \nonumber\\
&\quad\big(V_1\cdots V_m\big)^{2^{a_i}}
\Big)\nonumber\\
&=\prod_{i=1}^{q}
\Big(
1+
\big(V_1^{2^{a_i+m-1}}+V_2^{2^{a_i+m-2}}+\cdots+V_m^{2^{a_i}}\big) +\dots+\\
&\quad\big(\sum_{1\leq j_1<\cdots< j_r\leq m}
(V_1\cdots V_{j_1-1})^{2^{a_i+r}}\, (V_{j_1+1}\cdots V_{j_2-1})^{2^{a_i+r-1}}\cdots  (V_{j_r+1}\cdots V_{m})^{2^{a_i}}\big)\nonumber \\ &\qquad\qquad\qquad\qquad\qquad\qquad\qquad\qquad +\dots+\nonumber\\
&\quad\big(V_1\cdots V_m\big)^{2^{a_i}}
\Big).\nonumber
\end{align*} 
Since $2^a-d-1=2^{a_1}+\dots+2^{a_q}\leq d-1$ by choosing terms $V_m^{2^{a_{i}+0}}$ from each factors in the product indexed by $i=1,\dots,q$ we get that
\[
p=
%\pi\big(\rho_{d,2^m}^*\big(\overline{w}\bigl(\xi_{\R^d,2^m}^{\oplus (d+1)}\bigr)\big)\big)=
V_m^{2^a-d-1}+S \ \in \ \F_2[V_{1},\dots ,V_{m}]/\langle V_{1}^d,\dots,V_{m}^d \rangle,
\]
where $S$ is a sum of some monomials in $V_1,\dots,V_m$ of degree $(2^a-d-1)2^{m-1}=(2^a-d-1)\ell/2$ which are all different from $V_m^{d-1}$.
Hence, $p\neq 0$ and consequently $\overline{w}_{(2^a-d-1)\ell/2}\bigl(\xi_{\R^d,2^m}^{\oplus (d+1)}\bigr)=\overline{w}_{(2^a-d-1)(\ell-\epsilon(\ell) )/2}\bigl(\xi_{\R^d,2^m}^{\oplus (d+1)}\bigr)\neq 0$.

\medskip\noindent
{\bf (5B)}
Let $d\geq 5$ be an integer that is not a power of $2$, and furthermore $d+1$ is not a power of $2$.
Now we consider the case when $\ell\geq 3$ is not a power of $2$.
Set $r:=\alpha(\ell)\geq 2$ and $\ell=2^{b_1}+\dots+2^{b_r}$ where $0\leq b_1<b_2<\dots<b_r$.  
As in the proofs of Lemma \ref{lemma : corection not a power of 2} and Lemma  \ref{lem : correction of 2.17} we consider a morphism between vector bundles\index{vector bundle} $\prod_{i=1}^{r}\xi_{\R^d,2^{b_i}}$ and $\xi_{\R^d,\ell}$ where the following commutative square is a pullback diagram: 
 \[
\xymatrix{\prod_{i=1}^{r}\xi_{\R^d,2^{b_i}} \ar[rr]^-{\Theta}  \ar[d] \
& & \ \xi_{\R^d,\ell}   \ar[d] 
\\
%& & \\
\prod_{i=1}^{r} \conf(\R^d,2^{b_i})/\Sym_{2^{b_i}}  \ar[rr]^-{\theta} \  & &
\ \conf(\R^d,\ell)/{\Sym_\ell}.
}
\]
That is $\theta^*\xi_{\R^d,\ell}\cong \prod_{i=1}^{r}\xi_{\R^d,2^{b_i}}$.
The naturality  of the Stiefel--Whitney classes~\cite[Ax.\,2, p.\,37]{Milnor1974}\index{Stiefel--Whitney classes} gives the equality 
\[
\theta^*(\overline{w}_{(2^{\gamma(d)}-d-1)(\ell-\epsilon(\ell))/2}(\xi_{\R^d,\ell}^{\oplus (d+1)}))=\overline{w}_{(2^{\gamma(d)}-d-1)(\ell-\epsilon(\ell))/2}\Big(\prod_{i=1}^{r}\xi_{\R^d,2^{b_i}}^{\oplus (d+1)}\Big).
\]
The product formula \cite[Pr.\,4-A, p.\,54]{Milnor1974} implies that
\[
 \overline{w}\Big(\prod_{i=1}^{r}\xi_{\R^d,2^{b_i}}^{\oplus (d+1)}\Big)=\overline{w}\big(\xi_{\R^d,2^{b_1}}^{\oplus (d+1)}\big)\times\cdots\times\overline{w}\big(\xi_{\R^d,2^{b_r}}^{\oplus (d+1)}\big).
\]
Thus,
\begin{align} \label{sum-100}
& \theta^*(\overline{w}_{(2^{\gamma(d)}-d-1)(\ell-\epsilon(\ell))/2}(\xi_{\R^d,\ell}^{\oplus (d+1)}))\\& \qquad\qquad=\overline{w}_{(2^{\gamma(d)}-d-1)(\ell-\epsilon(\ell))/2}\Big(\prod_{i=1}^{r}\xi_{\R^d,2^{b_i}}^{\oplus (d+1)}\Big)\nonumber \\
&\qquad\qquad =\sum_{s_1+\cdots +s_r=(2^{\gamma(d)}-d-1)(\ell-\epsilon(\ell))/2}\overline{w}_{s_1}(\xi_{\R^d,2^{b_1}}^{\oplus (d+1)})\times\cdots\times\overline{w}_{s_r}(\xi_{\R^d,2^{b_r}}^{\oplus (d+1)}).	\nonumber
\end{align}
The K\"unneth formula\index{K\"unneth formula} \cite[Thm.\,VI.3.2]{Bredon2010} implies that each term 
\[
\overline{w}_{s_1}(\xi_{\R^d,2^{b_1}}^{\oplus (d+1)})\times\cdots\times\overline{w}_{s_r}(\xi_{\R^d,2^{b_r}}^{\oplus (d+1)})
\] 
in the previous sum belongs to a different direct summand of the  cohomology
{\small \begin{multline*}
H^{(2^{\gamma(d)}-d-1)(\ell-\epsilon(\ell))/2}\Big(\prod_{i=1}^{r} \conf(\R^d,2^{b_i})/\Sym_{2^{b_i}};\F_2\Big)\cong\\
\bigoplus_{s_1+\cdots +s_r=(2^{\gamma(d)}-d-1)(k-\epsilon(k))/2}H^{s_1}( \conf(\R^d,2^{b_1})/\Sym_{2^{b_1}};\F_2)\otimes\cdots\otimes H^{s_r}( \conf(\R^d,2^{b_r})/\Sym_{2^{b_r}};\F_2).
\end{multline*}}\noindent
Therefore, the following equivalence holds
\begin{multline*}
\overline{w}_{(2^{\gamma(d)}-d-1)(\ell-\epsilon(\ell))/2}\Big(\prod_{i=1}^{r}\xi_{\R^d,2^{b_i}}^{\oplus (d+1)}\Big) \neq 0 \quad \Longleftrightarrow \\
  \overline{w}_{s_1}(\xi_{\R^d,2^{b_1}}^{\oplus (d+1)})\times\cdots\times\overline{w}_{s_r}(\xi_{\R^d,2^{b_r}}^{\oplus (d+1)})\neq 0
\; \\ \text{for some} \; s_1+\cdots +s_r=(2^{\gamma(d)}-d-1)(\ell-\epsilon(\ell))/2.
\end{multline*}
To isolate a non-zero summand in \eqref{sum-100} we use the previous case of this theorem which states that $\overline{w}_{(2^{\gamma(d)}-d-1)2^{b_i-1}}(\xi_{\R^d,2^{b_i}}^{\oplus (d+1)})\neq 0$ when $b_i\geq 1$.

\medskip
We discuss two separate cases.
Let us assume that $\ell$ be even, or in other words $\epsilon(\ell)=0$.
Since $\ell=2^{b_1}+\dots+2^{b_r}$ it follows that $1\leq b_1<b_2<\dots<b_r$.
Hence, the following summand in \eqref{sum-100} does not vanish
\[
 \overline{w}_{(2^{\gamma(d)}-d-1)2^{b_1-1}}(\xi_{\R^d,2^{b_1}}^{\oplus (d+1)})\times\cdots\times\overline{w}_{(2^{\gamma(d)}-d-1)2^{b_r-1}}(\xi_{\R^d,2^{b_r}}^{\oplus (d+1)})\neq 0.
\]
When $\epsilon(\ell)=1$ we have that $0=b_1<b_2<\dots<b_r$, and the summand in \eqref{sum-100} which does not vanish is 
\[
 \overline{w}_{0}(\xi_{\R^d,2^{b_1}}^{\oplus (d+1)})\times \overline{w}_{(2^{\gamma(d)}-d-1)2^{b_2-1}}(\xi_{\R^d,2^{b_2}}^{\oplus (d+1)}) \times\cdots\times\overline{w}_{(2^{\gamma(d)}-d-1)2^{b_r-1}}(\xi_{\R^d,2^{b_r}}^{\oplus (d+1)})\neq 0.
\]
In summary, $\theta^*(\overline{w}_{(2^{\gamma(d)}-d-1)(\ell-\epsilon(\ell))/2}(\xi_{\R^d,\ell}^{\oplus (d+1)}))\neq 0$ and therefore 
\[
\overline{w}_{(2^{\gamma(d)}-d-1)(\ell-\epsilon(\ell))/2}(\xi_{\R^d,\ell}^{\oplus (d+1)})\neq 0.
\]

\medskip
{\bf (6)}
In this case we follow the footsteps of the proof of the previous claim. 
For completeness reasons we discuss all steps of the proof. 
Again we distinguish case when  $\ell$ is a power of $2$ from the case when  $\ell$ is not a power of $2$

\medskip\noindent
{\bf (6A)} 
Let $d\geq 5$ be an integer such that $2^{a-1}+1\leq d\leq 2^a-2$ where $\gamma(d)=a\geq 3$ is an integer.
Take $\ell=2^m$ for $m\geq 2$ an integer.

Using the decomposition of vector bundles \eqref{eq : decomposition of vb} we get that
\[
w\bigl(\xi_{\R^d,2^m}^{\oplus (d+1)}\bigr)=w\bigl(\zeta_{\R^d,2^m}^{\oplus (d+1)}\bigr)=(1+w_1+\dots+w_{2^m-1})^{d+1}.
\]
From  \eqref{eq : hight} the height of the algebra $H^{*}( \conf(\R^d,2^m)/\Sym_{2^m};\F_2)$ is known:
\[
\hght(H^{*}( \conf(\R^d,2^m)/\Sym_{2^m};\F_2))\leq \min\{ 2^t: 2^t\geq d \}=2^a.
\]
Therefore, 
\[
\overline{w}\bigl(\xi_{\R^d,2^m}^{\oplus (d+1)}\bigr)=(1+w_1+\dots+w_{2^m-1})^{2^a-d-1}.
\]
Let $2^a-d-1=2^{a_1}+\dots+2^{a_q}$ for $0\leq a_1<\dots<a_q\leq a-2$. 
Then
\begin{align}
	\label{eq : dual class for d not a power of 2 -- skew-6}
	\overline{w}\bigl(\xi_{\R^d,2^m}^{\oplus (d+1)}\bigr)&=(1+w_1+\dots+w_{2^m-1})^{2^{a_1}+\dots+2^{a_q}}\nonumber \\
&= \prod_{i=1}^{q}\big( 1+w_1^{2^{a_i}}+\dots+w_{2^m-1}^{2^{a_i}} \big).
\end{align}

\medskip
We apply the monomorphism
\[
 \rho_{d,2^m}^*\colon H^*(\conf(\R^d,2^m)/\Sym_{2^m};\F_2)\longrightarrow H^*(\Sp(\R^d,2^m)/\Sy_{2^m};\F_2)
 \] 
 from Theorem \ref{th : injection } to the equality \eqref{eq : dual class for d not a power of 2 -- skew-6}.
 Using  the decomposition of the cohomology
 \[
 H^*(\Sp(\R^d,2^m)/\Sy_{2^m};\F_2)
 \cong
 \F_2[V_{m,1},\dots,V_{m,m}]/\langle V_{m,1}^d,\ldots, V_{m,m}^d\rangle \oplus \III^*(\R^d,2^m),
 \] 
given in Theorem \ref{th: cohomology of Sp}, we have
\begin{align}
	\label{eq : dual class for d not a power of 2 -- skew - 02-6}
\rho_{d,2^m}^*(\overline{w}\bigl(\xi_{\R^d,2^m}^{\oplus (d+1)}\bigr))
&=  \prod_{i=1}^{q}\big( 1+D_{m,m-1}^{2^{a_i}}+\dots+D_{m,0}^{2^{a_i}} \big)+R 
\end{align}
where
\begin{compactitem}[ \ ---]
\item $D_{m,r}=(\kappa_{d,2^m}/\Sy_{2^m})^*(D_{m,r})=\rho_{d,2^m}^*(w_{2^m-2^r})$, for $0\leq r\leq m-1$, with the obvious abuse of notation, see \eqref{eq : summary}, and 
\item $R\in \III^*(\R^d,2^m)$.
\end{compactitem}

\medskip
Furthermore, let $\pi$ denote the following composition of the maps from the diagram \eqref{big diagram 02}:
\[
\xymatrix@C=0.5em{
\F_2[V_{m,1},\dots ,V_{m,m}]/\langle V_{m,1}^d,\dots,V_{m,m}^d \rangle \oplus I^*(\R^{d},2^m)\ar[d]^-{\text{projection}}\\
\F_2[V_{m,1},\dots ,V_{m,m}]/\langle V_{m,1}^d,\dots,V_{m,m}^d \rangle\ar[d]^-{\chi_m} \\
\F_2[V_{1},\dots ,V_{m}]/\langle V_{1}^d,\dots,V_{m}^d \rangle, \\
}
\]
The first map is the projection on a direct summand and the second map is induced by the change of variables $\chi_m\in\GL_m(\F_2)$.
Applying $\pi$ to \eqref{eq : dual class for d not a power of 2 -- skew - 02-6} we get
\[
\pi\big(\rho_{d,2^m}^*\big(\overline{w}\bigl(\xi_{\R^d,2^m}^{\oplus (d+1)}\bigr)\big)\big)=\prod_{i=1}^{q}\big( 1+\pi(D_{m,m-1})^{2^i}+\dots+\pi(D_{m,0})^{2^i} \big).
\]
Recall that the change of the variables $\chi_m$ transforms $\U_m(\F_2)$-invariants into $\mathrm{U}_m(\F_2)$-invariants and Dickson polynomials\index{Dickson invariants}, $\GL_m(\F_2)$-invariants, can be presented in terms of $\mathrm{U}_m(\F_2)$-invariants, as explained in \eqref{relation - recurrence Dickson - 03} and \eqref{relation - recurrence Dickson - 04}.
Now,  the element $p:=\pi\big(\rho_{d,2^m}^*\big(\overline{w}\bigl(\xi_{\R^d,2^m}^{\oplus (d+1)}\bigr)\big)\big)$ can be expressed as follows:
\begin{align}  \label{eq:longeq-6}
\quad p&=\prod_{i=1}^{q}\big( 1+\pi(D_{m,m-1})^{2^{a_i}}+\dots+\pi(D_{m,0})^{2^{a_i}} \big)	\nonumber\\
&=\prod_{i=1}^{q}
\Big(
1+
\big(V_1^{2^{m-1}}+V_2^{2^{m-2}}+\cdots+V_m^{2^{0}} \big)^{2^{a_i}}+\dots+\nonumber\\
&\quad\big(\sum_{1\leq j_1<\cdots< j_r\leq m}
(V_1\cdots V_{j_1-1})^{2^r}\, (V_{j_1+1}\cdots V_{j_2-1})^{2^{r-1}}\cdots  (V_{j_r+1}\cdots V_{m})^{2^{0}}\big)^{2^{a_i}}\nonumber \\ &\qquad\qquad\qquad\qquad\qquad\qquad\qquad\qquad +\dots+ \nonumber\\
&\quad\big(V_1\cdots V_m\big)^{2^{a_i}}
\Big)\nonumber\\
&=\prod_{i=1}^{q}
\Big(
1+
\big(V_1^{2^{a_i+m-1}}+V_2^{2^{a_i+m-2}}+\cdots+V_m^{2^{a_i}}\big) +\dots+\\
&\quad\big(\sum_{1\leq j_1<\cdots< j_r\leq m}
(V_1\cdots V_{j_1-1})^{2^{a_i+r}}\, (V_{j_1+1}\cdots V_{j_2-1})^{2^{a_i+r-1}}\cdots  (V_{j_r+1}\cdots V_{m})^{2^{a_i}}\big)\nonumber \\ &\qquad\qquad\qquad\qquad\qquad\qquad\qquad\qquad +\dots+\nonumber\\
&\quad\big(V_1\cdots V_m\big)^{2^{a_i}}
\Big).\nonumber
\end{align} 

\medskip
Now we want to show that 
\begin{multline*}
\pi\big(\rho_{d,2^m}^*\big(\overline{w}_{(2^{\gamma(d)}-d-1+2^{a_1})2^{m-1}-2^{a_1}}\bigl(\xi_{\R^d,2^m}^{\oplus (d+1)}\bigr)\big)\big)=\\
(V_1\cdots V_{m-1})^{2^{a_1}}V_m^{2^{\gamma(d)}-d-1}+S \ \in \ \F_2[V_{1},\dots ,V_{m}]/\langle V_{1}^d,\dots,V_{m}^d \rangle.
\end{multline*}
Here $S$ is a sum of monomials in $V_1,\dots,V_m$ of degree $(2^{\gamma(d)}-d-1+2^{a_1})2^{m-1}-2^{a_1}$ which are  different from the monomial $(V_1\cdots V_{m-1})^{2^{a_1}}V_m^{2^{\gamma(d)}-d-1}+S $.
Hence, $\pi\big(\rho_{d,2^m}^*\big(\overline{w}_{(2^{\gamma(d)}-d-1+2^{a_1})2^{m-1}-2^{a_1}}\bigl(\xi_{\R^d,2^m}^{\oplus (d+1)}\bigr)\big)\big)\neq 0$, and consequently 
\[
\overline{w}_{(2^{\gamma(d)}-d-1+2^{a_1})2^{m-1}-2^{a_1}}\bigl(\xi_{\R^d,2^m}^{\oplus (d+1)}\bigr)\neq 0.
\]

\medskip
Indeed, observe that in every monomial of the $i$th factor of the product \eqref{eq:longeq-6} which has the variable $V_m$, with a positive exponent, the variable $V_m$ has always the same  exponent equal to $2^{a_i}$, $1\leq i \leq q$. 
In particular, in the $i$th factor each monomial with the variable $V_m$ is of the form $p_i(V_1,\dots, V_{m-1})^{2^{a_i}}V_m^{2^{a_i}}$ where $p_i(V_1,\dots, V_{m-1})$ is a monomial in variables $V_1,\dots, V_{m-1}$.
Now, when multiplying out the product \eqref{eq:longeq-6}, since $2^{\gamma(d)}-d-1=2^{a_1}+\dots+2^{a_q}$,  the monomial of the form $p(V_1,\dots, V_{m-1}) V_m^{2^{\gamma(d)}-d-1}$ can appear in the final result if and only if we take from each factor a non-zero monomial of the form $p_i(V_1,\dots, V_{m-1})^{2^{a_i}}V_m^{2^{a_i}}$.
Thus, we have  
\[
p(V_1,\dots, V_{m-1})V_m^{2^{\gamma(d)}-d-1} = \prod_{i= 1}^qp_i(V_1,\dots, V_{m-1})^{2^{a_i}}V_m^{2^{a_i}} 
\]
Observe that, if $p_i(V_1,\dots, V_{m-1})\neq 1$ for some $1\leq i\leq q$, then there exists $1\leq t \leq q$ such that $V_t\mid p_i(V_1,\dots, V_{m-1})$. 
Hence, $V_t^{2^{a_i}}\mid p_i(V_1,\dots, V_{m-1})^{2^{a_i}}$.

\medskip
Now we want to count in how many ways we can obtain the monomial 
\[
(V_1\cdots V_{m-1})^{2^{a_1}}V_m^{2^{\gamma(d)}-d-1}
\] 
when we multiply out the product \eqref{eq:longeq-6}.
This means that we need to find all possible $p_i$s and $p_j'$'s such that
\[
(V_1\cdots V_{m-1})^{2^{a_1}}V_m^{2^{\gamma(d)}-d-1} = \prod_{i= 1}^q p_i(V_1,\dots, V_{m-1})^{2^{a_i}}V_m^{2^{a_i}} 
\]
From the previous observation and the fact that $0\leq a_1<a_2<\dots<a_q\leq a-2$ we conclude that $p_{a_2}(V_1,\dots, V_{m-1})=\dots=p_{a_q}(V_1,\dots, V_{m-1})=1$.
Thus, the previous equality becomes
\[
(V_1\cdots V_{m-1})^{2^{a_1}}V_m^{2^{\gamma(d)}-d-1} = p_1(V_1,\dots, V_{m-1})^{2^{a_i}}V_m^{2^{\gamma(d)}-d-1}.
\]
Therefore, the monomial $(V_1\cdots V_{m-1})^{2^{a_1}}V_m^{2^{\gamma(d)}-d-1}$ can be obtained for
\[
p_{a_1}(V_1,\dots, V_{m-1})=V_1\cdots V_{m-1},
\]
and
\[
p_{a_2}(V_1,\dots, V_{m-1})=\dots=p_{a_q}(V_1,\dots, V_{m-1})=1.
\]

\medskip
Hence, we completed the proof of the non-vanishing of the dual class: 
\[
\overline{w}_{(2^{\gamma(d)}-d-1+2^{a_1})2^{m-1}-2^{a_1}}\bigl(\xi_{\R^d,2^m}^{\oplus (d+1)}\bigr)\neq 0.
\]

\medskip\noindent
{\bf (6B)}
Let $d\geq 5$ be an integer that is not a power of $2$, and furthermore $d+1$ is not a power of $2$.
Consider the case when $\ell\geq 3$ is not a power of $2$.
Set $r:=\alpha(\ell)\geq 2$ and $\ell=2^{b_1}+\dots+2^{b_r}$ where $0\leq b_1<b_2<\dots<b_r$.  
As many times before we consider a morphism between vector bundles\index{vector bundle} $\prod_{i=1}^{r}\xi_{\R^d,2^{b_i}}$ and $\xi_{\R^d,\ell}$ where the following commutative square is a pullback diagram: 
 \[
\xymatrix{\prod_{i=1}^{r}\xi_{\R^d,2^{b_i}} \ar[rr]^-{\Theta}  \ar[d] \
& & \ \xi_{\R^d,\ell}   \ar[d] 
\\
%& & \\
\prod_{i=1}^{r} \conf(\R^d,2^{b_i})/\Sym_{2^{b_i}}  \ar[rr]^-{\theta}  \ & & \
 \conf(\R^d,\ell)/{\Sym_\ell}.
}
\]
In particular,  $\theta^*\xi_{\R^d,\ell}\cong \prod_{i=1}^{r}\xi_{\R^d,2^{b_i}}$.
The naturality  of the Stiefel--Whitney classes~\cite[Ax.\,2, p.\,37]{Milnor1974}\index{Stiefel--Whitney classes} gives the equality 
\[
\theta^*(\overline{w}(\xi_{\R^d,\ell}^{\oplus (d+1)}))=\overline{w}\Big(\prod_{i=1}^{r}\xi_{\R^d,2^{b_i}}^{\oplus (d+1)}\Big),
\]
while the product formula \cite[Pr.\,4-A, p.\,54]{Milnor1974} implies that
\[
 \overline{w}\Big(\prod_{i=1}^{r}\xi_{\R^d,2^{b_i}}^{\oplus (d+1)}\Big)=\overline{w}\big(\xi_{\R^d,2^{b_1}}^{\oplus (d+1)}\big)\times\cdots\times\overline{w}\big(\xi_{\R^d,2^{b_r}}^{\oplus (d+1)}\big).
\]
Thus, for every integer $N\geq 0$
\begin{align} \label{sum-100-6}
\theta^*(\overline{w}_{N}(\xi_{\R^d,\ell}^{\oplus (d+1)}))& =\overline{w}_{N}\Big(\prod_{i=1}^{r}\xi_{\R^d,2^{b_i}}^{\oplus (d+1)}\Big)\nonumber \\
&=\sum_{s_1+\cdots +s_r=N}\overline{w}_{s_1}(\xi_{\R^d,2^{b_1}}^{\oplus (d+1)})\times\cdots\times\overline{w}_{s_r}(\xi_{\R^d,2^{b_r}}^{\oplus (d+1)}).	 
\end{align}
The K\"unneth formula\index{K\"unneth formula} \cite[Thm.\,VI.3.2]{Bredon2010} implies that each term 
\[
\overline{w}_{s_1}(\xi_{\R^d,2^{b_1}}^{\oplus (d+1)})\times\cdots\times\overline{w}_{s_r}(\xi_{\R^d,2^{b_r}}^{\oplus (d+1)})
\] 
in the previous sum belongs to a different direct summand of the  cohomology
\begin{multline*}
H^{N}\big(\prod_{i=1}^{r} \conf(\R^d,2^{b_i})/\Sym_{2^{b_i}};\F_2\big)\cong\\
\bigoplus_{s_1+\cdots +s_r=N}H^{s_1}( \conf(\R^d,2^{b_1})/\Sym_{2^{b_1}};\F_2)\otimes\cdots\otimes H^{s_r}( \conf(\R^d,2^{b_r})/\Sym_{2^{b_r}};\F_2).
\end{multline*}
Therefore, the following equivalence holds
\begin{multline*}
\overline{w}_{N}\Big(\prod_{i=1}^{r}\xi_{\R^d,2^{b_i}}^{\oplus (d+1)}\Big) \neq 0  \Longleftrightarrow 
  \overline{w}_{s_1}(\xi_{\R^d,2^{b_1}}^{\oplus (d+1)})\times\cdots\times\overline{w}_{s_r}(\xi_{\R^d,2^{b_r}}^{\oplus (d+1)})\neq 0
\; \\ \text{for some} \; s_1+\cdots +s_r=N.
\end{multline*}
To isolate a non-zero summand in \eqref{sum-100-6} for 
\[
N=(2^{\gamma(d)}-d-1+2^{a_1})(\ell-\epsilon(\ell) )/2-2^{a_1}\alpha(\ell)
\]
we use the previous case of this theorem which states that for $b_i\geq 1$:
\[
\overline{w}_{(2^{\gamma(d)}-d-1+2^{a_1})2^{b_i-1}-2^{a_1}}(\xi_{\R^d,2^{b_i}}^{\oplus (d+1)})\neq 0.
\]

\medskip
We discuss two separate cases.
Let $\ell$ be even, or $\epsilon(\ell)=0$.
Since $\ell=2^{b_1}+\dots+2^{b_r}$ it follows that $1\leq b_1<b_2<\dots<b_r$.
Thus, the following summand in \eqref{sum-100-6} does not vanish
\[
 \overline{w}_{(2^{\gamma(d)}-d-1+2^{a_1})2^{b_1-1}-2^{a_1}}(\xi_{\R^d,2^{b_1}}^{\oplus (d+1)})\times\cdots\times\overline{w}_{(2^{\gamma(d)}-d-1+2^{a_1})2^{b_r-1}-2^{a_1}}(\xi_{\R^d,2^{b_r}}^{\oplus (d+1)})\neq 0.
\]
When $\epsilon(\ell)=1$ we have that $0=b_1<b_2<\dots<b_r$, and the summand in \eqref{sum-100-6} which does not vanish is 
\begin{multline*}
 \overline{w}_{0}(\xi_{\R^d,2^{b_1}}^{\oplus (d+1)})\times  \overline{w}_{(2^{\gamma(d)}-d-1+2^{a_1})2^{b_2-1}-2^{a_1}}(\xi_{\R^d,2^{b_1}}^{\oplus (d+1)})\times\cdots\times\\
 \overline{w}_{(2^{\gamma(d)}-d-1+2^{a_1})2^{b_r-1}-2^{a_1}}(\xi_{\R^d,2^{b_r}}^{\oplus (d+1)})\neq 0.	
\end{multline*}
In summary,   
\[
\overline{w}_{(2^{\gamma(d)}-d-1+2^{a_1})(\ell-\epsilon(\ell) )/2-2^{a_1}\alpha(\ell)}\bigl(\xi_{\R^d,\ell}^{\oplus (d+1)}\bigr)\neq 0.
\]

\end{proof}

\medskip
Now, we use Theorem \ref{thm : Correction of T.3.7} and the criterion from Lemma \ref{lem:dual:Whitney_skew} to correct the result stated in \cite[Thm.\,3.1]{Blagojevic2016-01}.
In this way we completed corrections of the invalid claims in \cite[Sec.\,3]{Blagojevic2016-01}.

\newpage
\begin{theorem}
	\label{th : Correctionog T 3.1}
 Let $\ell\geq 1$ and $d\geq 2$ be integers.
	\begin{compactenum}[ \ \rm (1)]
	%--1--
	\item\label{th : Correctionog T 3.1-1} If $d=2$, and if $\ell\geq 2$ is an integer, then there is no $\ell$-skew embedding\index{$\ell$-skew embedding} $\R^2\longrightarrow\R^N$ for
	\[
	N\leq  4\ell-\alpha(\ell)-2.
	\] 
	%--2--
	\item\label{th : Correctionog T 3.1-2} If $d\geq 1$ is an integer, and if $\ell=2$, then there is no $2$-skew embedding $\R^d\longrightarrow\R^N$ for
	\[
	N\leq  2^{\gamma(d)}+d-1.
	\] 
	%--3--
	\item\label{th : Correctionog T 3.1-3} If $d\geq 2$ is a power of $2$, and if $\ell\geq 2$ is an integer, then there is no $\ell$-skew embedding $\R^d\longrightarrow\R^N$ for
	\[
	N\leq  2d\ell-(d-1)\alpha(\ell)-2.
	\] 
	%--4--
	\item\label{th : Correctionog T 3.1-4}  If $d+1\geq 2$ is a power of~$2$, and if $\ell\geq 2$ is an integer, then this method does not produce any non-trivial result about the existence of $\ell$-skew embeddings $\R^d\longrightarrow\R^N$.
	%--6--
	\item\label{th : Correctionog T 3.1-6} If $d\geq 5$ is an integer which is not a power of $2$, and in addition $d+1$ is not a power $2$, and if $\ell\geq 3$ is an integer, then there is no $\ell$-skew embedding $\R^d\longrightarrow\R^N$ for
	\[
	N\leq  \frac{1}{2}\big(2^{\gamma(d)}-d-1\big)(\ell-\epsilon(\ell)) +(d+1)\ell-2.
	\] 
	%--7--
	\item\label{th : Correctionog T 3.1-7} If $d\geq 5$ is an integer which is not a power of $2$, $2^{\gamma(d)}-d-1=2^{a_1}z$ where $a_1\geq 0$ is an integer and $z\geq 1$ is an odd integer, and $d+1$ is not a power of $2$, and if $\ell\geq 3$ is an integer, then there is no $\ell$-skew embedding\index{$\ell$-skew embedding} $\R^d\longrightarrow\R^N$ for
	\[
	N\leq  \frac{1}{2}\big(2^{\gamma(d)}-d-1+2^{a_1}\big)(\ell-\epsilon(\ell) )-2^{a_1}\alpha(\ell)+(d+1)\ell-2.
	\] 

\end{compactenum}
\end{theorem}

%
%===========================
\subsection{$k$-regular-$\ell$-skew embeddings}
\label{sub : k-regular-l-skew maps}
%===========================
In this section we revise \cite[Sec.\,4]{Blagojevic2016-01} and correct the related results \cite[Thm.\,4.1, Thm.\,4.8]{Blagojevic2016-01}.
First we recall some basic notions on $k$-regular-$\ell$-skew embeddings. 

\medskip
\begin{definition}\label{def:regular_skew_map}
Let $k\geq1$ and $\ell\geq 1$ be an integer, and let $M$ be a real smooth $d$-dimensional manifold.
A smooth embedding $f\colon M\longrightarrow \R^N$ is {\bf $k$-regular-$\ell$-skew embedding}\index{$k$-regular-$\ell$-skew embedding} if for every $(x_1,\ldots,x_k,y_1,\ldots,y_{\ell})$ in $\conf(M,k+\ell)$ the affine subspaces
  \[
  \{f(x_1)\},\ldots,\{f(x_k)\},(\iota\circ df_{y_1})(T_{y_1}M),\ldots ,(\iota\circ
  df_{y_{\ell}})(T_{y_{\ell}}M)
  \]
of $\R^N$ are affinely independent.
\end{definition}	

\medskip

Now, like in the case of $k$-regular embeddings and $\ell$-skew embeddings  a criterion for non-existence of $k$-regular-$\ell$-skew embedding can be derived in terms of Stiefel--Whitney class of appropriate vector bundle over the relevant configuration space.
We state a consequence of \cite[Lem.\,4.6]{Blagojevic2016-01} and \cite[Lem.\,4.7]{Blagojevic2016-01} used for the proof of \cite[Thm.\,4.1]{Blagojevic2016-01}.

\begin{lemma} \label{lem:dual:Whitney_regular_skew}
  Let $d\geq 1$, $k\geq 1$ and $\ell\geq 1$ be integers.  
  If the dual Stiefel--Whitney class\index{Stiefel--Whitney classes} 
  \[
  \overline{w}_{N - (d+1)\ell-k +2}(\xi_{\R^d,k}\times\xi_{\R^d,\ell}^{\oplus (d+1)})
  \] 
  does not vanish, then there is no $k$-regular-$\ell$-skew embedding $\R^d\longrightarrow \R^N$.
\end{lemma} 

\medskip
Like in the case of $k$-regular embeddings and $\ell$-skew embeddings, now based on invalid results in \cite[Thm.\,2.13]{Blagojevic2016-01} and \cite[Thm.\,3.7]{Blagojevic2016-01} the following theorem was proved \cite[Thm.\,4.8]{Blagojevic2016-01}:
\begin{quote}
{\small
	{\bf Theorem 4.8.}
	 Let $\ell,k,d\geq 2$ be integers.
	The dual Stiefel--Whitney class\index{Stiefel--Whitney classes}
	\[
	\overline{w}_{(d-1)(k-\alpha(k))+(2^{\gamma(d)}-d-1)(\ell-\alpha(\ell))}(\xi_{\R^d,k}\times \xi_{\R^d,\ell}^{\oplus (d+1)})
	\]
	does not vanish.
	}
\end{quote}
The result of the previous theorem in combination with \cite[Lem.\,4.6]{Blagojevic2016-01} and \cite[Lem.\,4.4]{Blagojevic2016-01} directly implied the following result \cite[Thm.\,4.1]{Blagojevic2016-01}. 
\begin{quote}
{\small
	{\bf Theorem 4.1.}
	 et $\ell,d\geq 2$ be integers.
	There is no $k$-regular-$\ell$-skew embedding\index{$k$-regular-$\ell$-skew embedding} $\R^d\longrightarrow \R^N$ for
	\[
	N\leq (d-1)(k-\alpha(k))+(2^{\gamma(d)}-d-1)(\ell-\alpha(\ell))+(d+1)\ell+k-2,
	\]
	where $\alpha(c)$ denotes the number of ones in the dyadic presentation of $c$, and $\gamma(d):=\lfloor\log_2d\rfloor+1$.
	}
\end{quote}

\medskip
The proof of \cite[Thm.\,4.8]{Blagojevic2016-01} was based on incorrect results \cite[Cor.\,2.13]{Blagojevic2016-01} and \cite[Cor.\,3.7]{Blagojevic2016-01}.
Since we corrected these results  we can now give correct versions, first of \cite[Thm.\,4.8]{Blagojevic2016-01}, and then of \cite[Thm.\,4.1]{Blagojevic2016-01}.

\begin{theorem}
\label{thm : Correction of T.4.8}
Let $d\geq 1$, $k\geq1$, and $\ell\geq 1$ be integers.
\begin{compactenum}[ \ \rm (1)]
%--1--
\item
If $d=2$, and if $k\geq1$ and  $\ell\geq 1$ are integers, then 
\[
 \overline{w}_{k-\alpha(k)+\ell-\alpha(\ell)}\big(\xi_{\R^d,k}\times\xi_{\R^d,\ell}^{\oplus (d+1)}\big)\neq 0.
\] 
%--2--
\item If $d\geq 2$ is a power of $2$, $\ell=2$, and if $k\geq 1$ is an integer, then 
\[
 \overline{w}_{(d-1)(k-\alpha(k)+1)}\big(\xi_{\R^d,k}\times\xi_{\R^d,\ell}^{\oplus (d+1)}\big)\neq 0.
\]
%--3--
\item If $d\geq 3$ is not a power of $2$, $\ell=2$, and if $k\geq 1$ is an integer, then 
\[
 \overline{w}_{(d-1)(k-\epsilon(k))/2+2^{\gamma(d)}-d-1}\big(\xi_{\R^d,k}\times\xi_{\R^d,\ell}^{\oplus (d+1)}\big)\neq 0.
\]
%--4--
\item  If $d\geq 2$ is a power of $2$, and if $k\geq 1$ and $\ell\geq1$ are integers, then  
\[
 \overline{w}_{(d-1)(k-\alpha(k)+\ell-\alpha(\ell) )}\big(\xi_{\R^d,k}\times\xi_{\R^d,\ell}^{\oplus (d+1)}\big)\neq 0.
\]
%--5--
\item  If $d+1\geq 2$ is a power of $2$, and if $k\geq 1$ and $\ell\geq1$ are integers, then  
\[
 \overline{w}_{(d-1)(k-\epsilon(k))/2}\big(\xi_{\R^d,k}\times\xi_{\R^d,\ell}^{\oplus (d+1)}\big)\neq 0.
\]
%--6--
\item  If $d\geq 5$ is an integer which is not a power of $2$, and in addition $d+1$ is a not power of $2$, and if $k\geq 1$ and $\ell\geq1$ are integers, then  
\[
\overline{w}_{(d-1)(k-\epsilon(k))/2+(2^{\gamma(d)}-d-1)(\ell-\epsilon(\ell))/2}\big(\xi_{\R^d,k}\times\xi_{\R^d,\ell}^{\oplus (d+1)}\big)\neq 0.
\]

%--7--
\item If $d\geq 6$ be an even integer which is not a power of $2$, $2^{\gamma(d)}-d-1=2^{a_1}z$ where $a_1\geq 0$ is an integer and $z\geq 1$ is an odd integer, and if $k\geq 1$ and $\ell\geq 3$ are integers, then
\[
\overline{w}_{
d(k-\epsilon(k))/2-\alpha(k)+\epsilon(k)+
(2^{\gamma(d)}-d-1+2^{a_1})(\ell-\epsilon(\ell) )/2-2^{a_1}\alpha(\ell)
}\big(\xi_{\R^d,k}\times\xi_{\R^d,\ell}^{\oplus (d+1)}\big)\neq 0.
\]

\end{compactenum}
\end{theorem}
\begin{proof}
In order to compute the dual Stiefel--Whitney class\index{Stiefel--Whitney classes} of the product vector bundle\index{vector bundle} $\xi_{\R^d,k}\times \xi_{\R^d,\ell}^{\oplus (d+1)}$ we use the product formula~\cite[Problem~4-A, page~54]{Milnor1974}:  
\[
 \overline{w}\big(\xi_{\R^d,k}\times\xi_{\R^d,\ell}^{\oplus (d+1)}\big)= \overline{w}(\xi_{\R^d,k})\times\overline{w}\big(\xi_{\R^d,\ell}^{\oplus (d+1)}\big),
\]
where ``$\times$'' on the right hand side denotes the cross product in cohomology.
In particular, for a fixed integer $r\geq 0$ we have that
\begin{multline*}
\overline{w}_{r}\big(\xi_{\R^d,k}\times \xi_{\R^d,\ell}^{\oplus (d+1)}\big)=
\sum_{i+j=r}
\overline{w}_{i}(\xi_{\R^d,k})\times \overline{w}_{j}\big(\xi_{\R^d,\ell}^{\oplus (d+1)}\big) \\
\ \in \
H^r(\conf(\R^d,k)/\Sym_k\times\conf(\R^d,\ell)/\Sym_{\ell};\F_2 )
.	
\end{multline*}	
The K\"unneth formula\index{K\"unneth formula} \cite[Thm.\,VI.3.2]{Bredon2010} implies that each of the terms $\overline{w}_{i}(\xi_{\R^d,k})\times \overline{w}_{j}\big(\xi_{\R^d,\ell}^{\oplus (d+1)}\big)$ in the previous sum belongs to a different direct summand of the  cohomology
\begin{multline*}
H^r(\conf(\R^d,k)/\Sym_k\times\conf(\R^d,\ell)/\Sym_{\ell};\F_2 )\cong \\
\bigoplus_{i+j=r}H^i(\conf(\R^d,k)/\Sym_k ;\F_2 )\otimes H^j(\conf(\R^d,\ell)/\Sym_{\ell};\F_2 ).
\end{multline*}	
Therefore, the following equivalence holds
\[
\overline{w}_{r}\big(\xi_{\R^d,k}\times \xi_{\R^d,\ell}^{\oplus (d+1)}\big)  \neq 0 \Longleftrightarrow 
\overline{w}_{i}(\xi_{\R^d,k})\times \overline{w}_{j}\big(\xi_{\R^d,\ell}^{\oplus (d+1)}\big) \neq 0
\; \text{for some} \; i+j=r.
\]
Now, using Theorem \ref{th : Correctionog T 2.13} and Theorem \ref{thm : Correction of T.3.7} to prove all cases of the theorem.

\medskip
{\bf (1)} 
Let $d=2$, and let $k\geq1$ and  $\ell\geq 1$ be integers.
Then
\[
\overline{w}_{k-\alpha(k)}(\xi_{\R^d,k})\neq 0 \qquad\text{and}\qquad \overline{w}_{\ell-\alpha(\ell)}\big(\xi_{\R^d,\ell}^{\oplus (d+1)}\big)\neq 0,
\]
and consequently $\overline{w}_{k-\alpha(k)+\ell-\alpha(\ell)}\big(\xi_{\R^d,k}\times\xi_{\R^d,\ell}^{\oplus (d+1)}\big)\neq 0$. 

\medskip
{\bf (2)} 
Let $d\geq 2$ be a power of $2$, $\ell=2$, and let $k\geq 1$ be an integer.
Then  
\[
\overline{w}_{(d-1)(k-\alpha(k))}(\xi_{\R^d,k})\neq 0 \qquad\text{and}\qquad \overline{w}_{d-1}\big(\xi_{\R^d,\ell}^{\oplus (d+1)}\big)\neq 0,
\] 
and consequently $\overline{w}_{(d-1)(k-\alpha(k)+1)}\big(\xi_{\R^d,k}\times\xi_{\R^d,\ell}^{\oplus (d+1)}\big)\neq 0$.

\medskip
{\bf (3)} 
Let $d\geq 2$ be not a power of $2$, $\ell=2$, and let $k\geq 1$ be an integer.
Then 
\[
\overline{w}_{(d-1)(k-\epsilon(k))/2}(\xi_{\R^d,k})\neq 0\qquad\text{and}\qquad\overline{w}_{2^{\gamma(d)}-d-1}\big(\xi_{\R^d,\ell}^{\oplus (d+1)}\big)\neq 0,
\] 
and consequently $\overline{w}_{(d-1)(k-\epsilon(k))/2+2^{\gamma(d)}-d-1}\big(\xi_{\R^d,k}\times\xi_{\R^d,\ell}^{\oplus (d+1)}\big)\neq 0$.

\medskip
{\bf (4)}
Let $d\geq 2$ be not a power of $2$, and let $k\geq 1$ and $\ell\geq1$ be integers.
Then 
\[
\overline{w}_{(d-1)(k-\alpha(k))}(\xi_{\R^d,k})\neq 0\qquad\text{and}\qquad\overline{w}_{(d-1)(\ell-\alpha(\ell))}\big(\xi_{\R^d,\ell}^{\oplus (d+1)}\big)\neq 0,
\] 
and consequently $\overline{w}_{(d-1)(k-\alpha(k)+\ell-\alpha(\ell) )}\big(\xi_{\R^d,k}\times\xi_{\R^d,\ell}^{\oplus (d+1)}\big)\neq 0$.

\medskip 
{\bf (5)}
Let $d+1\geq 2$ be a power of $2$, and let $k\geq 1$ and $\ell\geq1$ be integers.
Then 
\[
\overline{w}_{(d-1)(k-\epsilon(k))/2}(\xi_{\R^d,k})\neq 0\qquad\text{and}\qquad\overline{w}_{0}\big(\xi_{\R^d,\ell}^{\oplus (d+1)}\big)\neq 0,
\] 
and consequently $\overline{w}_{(d-1)(k-\epsilon(k))/2}\big(\xi_{\R^d,k}\times\xi_{\R^d,\ell}^{\oplus (d+1)}\big)\neq 0$

\medskip
{\bf (6)}
Let $d\geq 5$ be an integer which is not a power of $2$, and in addition $d+1$ is a not power of $2$, and let $k\geq 1$ and $\ell\geq1$ be integers.
Then 
\[
\overline{w}_{(d-1)(k-\epsilon(k))/2}(\xi_{\R^d,k})\neq 0\qquad\text{and}\qquad\overline{w}_{(2^{\gamma(d)}-d-1)(\ell-\epsilon(\ell))/2}\big(\xi_{\R^d,\ell}^{\oplus (d+1)}\big)\neq 0,
\] 
and consequently $\overline{w}_{(d-1)(k-\epsilon(k))/2+(2^{\gamma(d)}-d-1)(\ell-\epsilon(\ell))/2}\big(\xi_{\R^d,k}\times\xi_{\R^d,\ell}^{\oplus (d+1)}\big)\neq 0$.

\medskip
{\bf (7)}
Let $d\geq 6$ be an even integer which is not a power of $2$, $2^{\gamma(d)}-d-1=2^{a_1}z$ where $a_1\geq 0$ is an integer and $z\geq 1$ is an odd integer, and let $k\geq 1$ and $\ell\geq 3$ be integers. 
Then
\[
\overline{w}_{d(k-\epsilon(k))/2-\alpha(k)+\epsilon(k)}(\xi_{\R^d,k})\neq 0
\qquad\text{and}\qquad
\overline{w}_{d(k-\epsilon(k))/2-\alpha(k)+\epsilon(k)}(\xi_{\R^d,k})\neq 0,
\]
and consequently
\[
\overline{w}_{
d(k-\epsilon(k))/2-\alpha(k)+\epsilon(k)+
(2^{\gamma(d)}-d-1+2^{a_1})(\ell-\epsilon(\ell) )/2-2^{a_1}\alpha(\ell)
}\big(\xi_{\R^d,k}\times\xi_{\R^d,\ell}^{\oplus (d+1)}\big)\neq 0.
\]

\end{proof}

\medskip
Like in the previous situation, we use Theorem \ref{thm : Correction of T.4.8} and the criterion from Lemma \ref{lem:dual:Whitney_regular_skew} to correct the result stated in \cite[Thm.\,4.1]{Blagojevic2016-01}.
In this way we completed corrections of the invalid results in \cite[Sec.\,4]{Blagojevic2016-01}.

\begin{theorem}
\label{thm : Correction of T.4.1}
Let $d\geq 1$, $k\geq1$, and $\ell\geq 1$ be integers.
\begin{compactenum}[ \ \rm (1)]
%--1--
\item
If $d=2$, and if $k\geq1$ and  $\ell\geq 1$ are integers, then there is no  $k$-regular-$\ell$-skew embedding\index{$k$-regular-$\ell$-skew embedding} $\R^2\longrightarrow\R^N$ for
\[
N\leq (d+1)\ell+2k-\alpha(k)+\ell-\alpha(l)-2.
\] 
%--2--
\item If $d\geq 2$ is a power of $2$, $\ell=2$, and if $k\geq 1$ is an integer, then there is no  $k$-regular-$2$-skew embedding $\R^d\longrightarrow\R^N$ for
\[
N\leq  (d+1)\ell+k-2+(d-1)(k-\alpha(k)+1) .
\] 
%--3--
\item If $d\geq 2$ is not a power of $2$, $\ell=2$, and if $k\geq 1$ is an integer, then there is no  $k$-regular-$2$-skew embedding $\R^d\longrightarrow\R^N$ for
\[
N\leq  (d+1)\ell+k-2+\frac12 (d-1)(k-\epsilon(k))+2^{\gamma(d)}-d-1 .
\] 
%--4--
\item  If $d\geq 2$ is not a power of $2$, and if $k\geq 1$ and $\ell\geq1$ are integers, then there is no  $k$-regular-$\ell$-skew embedding $\R^2\longrightarrow\R^N$ for
\[
N\leq  (d+1)\ell+k-2+(d-1)(k-\alpha(k)+\ell-\alpha(\ell)) .
\] 
%--5--
\item  If $d+1\geq 2$ is a power of $2$, and if $k\geq 1$ and $\ell\geq1$ are integers, then  there is no $k$-regular-$\ell$-skew embedding $\R^d\longrightarrow\R^N$ for
\[
N\leq  (d+1)\ell+k-2+\frac12 (d-1)(k-\epsilon(k)) .
\] 
%--6--
\item  If $d\geq 5$ is an integer which is not a power of $2$, and in addition $d+1$ is a not power of $2$, and if $k\geq 1$ and $\ell\geq1$ are integers, then there is no  $k$-regular-$\ell$-skew embedding\index{$k$-regular-$\ell$-skew embedding} $\R^d \longrightarrow\R^N$ for
\[
N\leq (d+1)\ell+k-2+\frac12 (d-1)(k-\epsilon(k))+\frac12\big(2^{\gamma(d)}-d-1\big)(\ell-\alpha(\ell)).
\] 

%--7--
\item If $d\geq 6$ be an even integer which is not a power of $2$, $2^{\gamma(d)}-d-1=2^{a_1}z$ where $a_1\geq 0$ is an integer and $z\geq 1$ is an odd integer, and if $k\geq 1$ and $\ell\geq 3$ are integers, then there is no  $k$-regular-$\ell$-skew embedding $\R^d \longrightarrow\R^N$ for
\begin{multline*}
N\leq  (d+1)\ell+k-2+\frac12 d(k-\epsilon(k))-\alpha(k)+\epsilon(k)+\\
 \frac12\big(2^{\gamma(d)}-d-1+2^{a_1}\big)(\ell-\epsilon(\ell) )-2^{a_1}\alpha(\ell).	
\end{multline*}

\end{compactenum}

\end{theorem}

This concludes all corrections of the paper \cite{Blagojevic2016-01}.

 %===========================
\subsection{Complex highly regular embeddings}
\label{sub : complex highly regular embeddings}
%===========================

A gap in the decomposition \cite[(4.7)]{Hung1990} created incorrectness that we have already discussed in the study of real highly regular embeddings \cite{Blagojevic2016-01}, which in turn implied two results on complex highly regular embeddings \cite[Thm.\,5.1, Thm.\,6.1]{Blagojevic2016-02} that now also need to be corrected.
For completeness we recall some basic notions about complex highly regular embeddings.

\medskip
First we introduce a notion of complex $k$-regular embedding in a similar way as in the case of real $k$-regular embedding, see Section \ref{sub : k-regular maps}. 
\begin{definition}
Let $k\geq 1$ be an integer, and let $X$ be a topological space.
A continuous map $f\colon X \longrightarrow\C^N$ is a {\bf complex $k$-regular embedding}\index{complex $k$-regular embedding} if for every 
$(x_1,\ldots,x_k)\in \conf(X,k)$ the  vectors $f(x_1),\ldots,f(x_n)$ of the complex vector space $\C^N$ are linearly independent.
\end{definition}

\medskip
Next we introduce a notion of complex $\ell$-regular embedding. 
For a real analogue consult Section \ref{sub : l-skew maps}.
A collection of complex affine subspaces $\{L_1,\ldots,L_{\ell}\}$ of the complex vector space $\C^N$ is {\bf affinely independent} if the following equality holds
\[
\dimaff^{\C}\spann (L_1\cup\dots\cup L_{\ell})
=
(\dimaff^{\C} L_1+1)+\cdots+(\dimaff^{\C} L_{\ell}+1)-1.
\]
Let $M$ be a complex $d$-dimensional manifold.
The associated complex tangent bundle of $M$ is denote by $TM$.
For a point $y\in M$ we denote by $T_yM$ the corresponding tangent space to $M$.
If $f \colon M\longrightarrow\C^N$ is a smooth complex map, then $df\colon TM\longrightarrow T\C^N$ denotes the complex differential map between tangent complex vector bundles induced by $f$.
Furthermore, let $\iota \colon T\C^N  \longrightarrow \C^N$ be the map which sends a tangent vector $v \in T_x\C^N$ at a point $x \in \C^N$ to the sum $x +v$, where the standard identification $T_x\C^N = \C^N$ is assumed.

\begin{definition} 
  Let $\ell\geq 1$ be an integer, and let $M$ be a smooth complex $d$-dimensional manifold.  
  A smooth complex embedding $f \colon M\longrightarrow\C^N$ is an {\bf complex $\ell$-skew embedding}\index{complex $\ell$-skew embedding} if for every $(y_1,\ldots,y_{\ell})\in \conf(M,\ell)$
  the collection of complex affine subspaces $\{ (\iota\circ df_{y_1})(T_{y_1}M),\ldots ,(\iota\circ df_{y_{\ell}})(T_{y_{\ell}}M)\}$ of $\C^N$ is affinely independent.
\end{definition}

\medskip
In the following we will work with complex vector bundles that are analogues of real vector bundles introduced in \eqref{R-vector-bundles}, see also \cite[Sec.\,4]{Blagojevic2016-02}.
Hence, consider the complex vector bundles\index{vector bundle}:
\begin{equation}\label{C-vector-bundles}
\xymatrix@R-2pc{
\xi_{X,k}^{\C} \colon & \C^k\ar[r] & \conf(X,k)\times_{\Sym_k}\C^k\ar[r]  & \conf(X,k)/\Sym_k,
  \\
 \zeta_{X,k}^{\C} \colon & W_k^C \ar[r] & \conf(X,k)\times_{\Sym_k}W_k^{\C}   \ar[r] & \conf(X,k)/\Sym_k,\\
  \tau_{X,k}^{\C} \colon & \C   \ar[r] &\conf(X,k)/\Sym_k\times\C     \ar[r] & \conf(X,k)/\Sym_k,
}
\end{equation}
where $W_k^{\C}=\{(b_1,\dots,b_k)\in\C^k : b_1+\dots +b_k=0\}$ is an $\Sym_k$-invariant subspace of $\C^k$.
It is obvious that on the level of underlying real vector bundles the following bundle isomorphisms hold:
\[
 \xi_{X,k}^{\C}\cong\xi_{X,k}^{\oplus 2},\qquad
 \zeta_{X,k}^{\C}\cong\zeta_{X,k}^{\oplus 2},\qquad
  \tau_{X,k}^{\C}\cong\tau_{X,k}^{\oplus 2}.
\]

\medskip
Like in the real case a criterion for an existence of complex $k$-regular embeddings  and of complex $\ell$-skew embeddings can be phrased in terms of the vector bundle $\xi_{X,k}^{\C}$.
Let us recall the relevant special cases of \cite[Lem.\,5.7]{Blagojevic2016-02} and \cite[Lem.\,6.6]{Blagojevic2016-02}.

\begin{lemma}\label{lem : complex criterion}
Let $k\geq 1$, $\ell\geq 1$, $d\geq 1$ and  $N\geq 1$ be integers.
\begin{compactenum}[\ \rm (1)]
\item If there exists a complex $k$-regular embedding  $\R^d\longrightarrow\C^N$, then the complex vector bundle $\xi_{\R^d,k}^{\C}$ admits an $(N-k)$-dimensional complex inverse. 
\item If there exists a complex $\ell$-skew embedding  $\C^d\longrightarrow\C^N$, then the complex vector bundle $(\xi_{\C^d,\ell}^{\C})^{\oplus (d+1)}$ admits an $(N-(d+1)\ell+1)$-dimensional complex inverse.
\end{compactenum}
\end{lemma}

\medskip
The result we prove next corrects \cite[Thm.\,5.1]{Blagojevic2016-02}.
In particular, we follow the outline of the proof of \cite[Thm.\,5.1]{Blagojevic2016-02} and alternate at a single place.

\begin{theorem}\label{theorem_complex_regular}
Let $d\geq 1$ be an integer. 
There is no complex $k$-regular embedding\index{complex $k$-regular embedding} $\R^d\longrightarrow\C^N$ for $N<\frac12(M+k)$,
where
\[
M:=
\begin{cases}
	(d-1)(k-\alpha(k)), & d\text{ is a power of }2,  \\
	(d-1)(k-\epsilon(k)), & d\text{ is not a power of }2.
\end{cases}
\]
\end{theorem}
\begin{proof}
Consider a complex $k$-regular embedding $\R^d\longrightarrow\C^N$.
From Lemma \ref{lem : complex criterion} the complex vector bundle $\xi_{\R^d,k}^{\CC}$ admits an $(N-k)$-dimensional complex inverse. 
Hence, the real vector bundle $\xi_{\R^d,k}^{\oplus 2}$ admits a $2(N-k)$-dimensional real inverse.
In particular, the real vector bundle $\xi_{\R^d,k}$ admits a $(2N-k)$-dimensional real inverse.
Since $\bar{w}_{M}(\xi_{\R^d,k})\neq 0$, Theorem \ref{th : Correctionog T 2.13}, we have that
\[
2N-k\geq M
\quad\Longleftrightarrow\quad 
N\geq \tfrac{1}{2}(M+k).\vspace{-10pt}
\]	
\end{proof}

\begin{remark}
Notice that, compared to \cite[Thm.\,5.1]{Blagojevic2016-02}, a correction was needed for the case when $d$ is not a power of $2$.	
\end{remark}
 
\medskip
Next we correct \cite[Thm.\,6.1]{Blagojevic2016-02}.
Again we proceed in the footsteps of the proof of \cite[Thm.\,6.1]{Blagojevic2016-02}.
\begin{theorem}
  \label{theorem:Main-02}
  Let $d\geq1$ and $\ell\geq 1$ be integers.  
  There is no complex $\ell$-skew embedding\index{complex $\ell$-skew embedding} $\C^d\longrightarrow\C^{N}$
  for 
  \[N\leq  d+\frac{1}{2}\big(M-\ell-2\big), \]
 where
 \[
 M:=\begin{cases}
 2^{\gamma(d)+1}-2d-1 , & \ell=2,\\
 (2d-1)(\ell-\alpha(\ell)) , & d\text{ is a power of }2,\\
(2^{\gamma(d)+1}-2d-1)(\ell-\epsilon(\ell) )/2, & d\geq 3\text{ is not  power of }2.
 \end{cases}
 \] 
\end{theorem}
\begin{proof}
Consider a complex $\ell$-skew embedding $\C^d\longrightarrow\C^N$.
According to Lemma~\ref{lem : complex criterion} the complex vector bundle $(\xi_{\C^d,\ell}^{\C})^{\oplus (d+1)}$ admits an $(N-(d+1)\ell+1)$-dimensional complex inverse.
Hence, the real vector bundle $\xi_{\R^{2d},\ell}^{\oplus 2(d+1)}$ admits a $2(N-(d+1)\ell+1)$-dimensional real inverse.
Consequently, the real vector bundle $\xi_{\R^{2d},\ell}^{\oplus (2d+1)}$ admits a $(2N-2(d+1)\ell+2+\ell)$-dimensional real inverse.
From Theorem \ref{thm : Correction of T.3.7} follows that $\bar{w}_{M}(\xi_{\R^{2d},\ell}^{\oplus (2d+1)})\neq 0$ and therefore:
\[
2N-2(d+1)\ell+2+\ell\geq M  
\quad\Longleftrightarrow\quad
N\geq\frac{1}{2}\big(2(d+1)-2-\ell+M\big).
\]
\end{proof}

\medskip
This completes the corrections of the paper \cite{Blagojevic2016-02}.

%%%%%%%%%%%%%%%%%%%%%%%%%%%%%%%%%%%%%%%%%%%%%%%%%%%%%%%%%%%%%%%%%%%%%%%%%%%%%%%%%%%%%
%%%%%%%%%%%%%%%%%%%%%%%%%%%%%%%%%%%%%%%%%%%%%%%%%%%%%%%%%%%%%%%%%%%%%%%%%%%%%%%%%%%%%
\section{More bounds for highly regular embeddings}
\label{sec : more bounds for regular embeddings}
%%%%%%%%%%%%%%%%%%%%%%%%%%%%%%%%%%%%%%%%%%%%%%%%%%%%%%%%%%%%%%%%%%%%%%%%%%%%%%%%%%%%%
%%%%%%%%%%%%%%%%%%%%%%%%%%%%%%%%%%%%%%%%%%%%%%%%%%%%%%%%%%%%%%%%%%%%%%%%%%%%%%%%%%%%%

In this section we present computations that yield additional bounds for the existence of highly regular embeddings.
For this we present an alternative approach in the study of Stiefel--Whitney classes of the vector bundle  $\xi_{\R^d,2^m}$.
In particular, we utilize a specific decomposition of the pull-back vector bundle $\rho_{d,2^m}^*\xi_{\R^d,2^m}$ of the vector bundle $\xi_{\R^d,2^m}$ via the map $\rho_{d,2^m}\colon \Sp(\R^d,2^m)/\Sy_{2^m}\longrightarrow\conf(\R^d,2^m)/\Sym_{2^m}$.

\medskip
First, we introduce several real $\Sy_{2^m}$-representations and consider properties of the associated vector bundles\index{vector bundle}.

%===========================
\subsection{Examples of $\Sy_{2^m}$-representations and associated vector bundles}
\label{sec : Examples of representations}
%===========================

In this text the real vector space  $\R^k$ is often considered as the real $k$-dimensional $\Sym_{k}$-representation with the action defined by
\[
\pi\cdot (a_1,\ldots,a_k):=(a_{\pi^{-1}(1)},\ldots,a_{\pi^{-1}(k)}),
\]
where $\pi\in\Sym_k$ and $(a_1,\ldots,a_k)\in\R^k$.
Furthermore, its $(n-1)$-dimensional vector subspace 
\[
W_k:=\{(a_1,\ldots,a_k)\in\R^k : a_1+\cdots+a_k=0\}
\] 
is an $\Sym_{k}$-subrepresentation.
For every subgroup $G$ of $\Sym_{k}$ both vector spaces $\R^k$ and $W_k$ become real $G$-representations via the corresponding inclusion homomorphism $G\hookrightarrow\Sym_{k}$.

%------------------------------------------------------
\subsubsection{Examples of $\Sy_{2^m}$-representations}
%------------------------------------------------------
Let $k=2^m$ for some integer $m\geq 1$. 
Recall that the group $\Sy_{2^m}$, as a subgroup of $\Sym_{2^m}$, was introduced in Definition \ref{def : Epicycles} and the corresponding inclusion homomorphism was denotes by $\iota_m\colon\Sy_{2^m}\longrightarrow \Sym_{2^m}$.
Inductively, we can describe the group $\Sy_{2^m}$ via the exact sequence of groups
\begin{equation}
\label{eq : exact sequence of groups representations}
\xymatrix{
1\ar[r] \ &\ \Sy_{2^{m-1}}\times\Sy_{2^{m-1}}\ar[r]^-{\vartheta_{m}}  \ & \ \Sy_{2^m}\ar[r]^-{\varsigma_m} \ &\ \Z_2\ar[r]\ &\ 1,
}	
\end{equation}
where the inclusion map $\vartheta_{m}$ is defined in the expected way.

\medskip
Now we define inductively a sequence of real $\Sy_{2^m}$-representations.
Let $M_m[m]$ be the $\Sy_{2^m}$-representation obtained as the pull-back of the real $1$-dimensional $\Z_2$-representation $W_2$ via the surjection $\varsigma_m$ from the exact sequence \eqref{eq : exact sequence of groups representations}.
This means that $M_m[m]=W_2$ as a vector space, and that $\pi\cdot v:= \varsigma_m(\pi)\cdot v$ for every $v\in M_m[m]$ and every $\pi\in \Sy_{2^m}$.

\medskip
Assume that for every integer $1\leq i\leq m-1$ we have defined the sequence of $\Sy_{2^i}$-representations $M_i[1],\dots,M_i[i]$ with $\dim(M_i[j])=2^{i-j}$, $1\leq j\leq i$.
Next, for $i=m$ and $1\leq j\leq m-1$, we define the $\Sy_{2^m}$-representation $M_m[j]$ to be, on the level of vector spaces, the direct sum $M_m[j]:=M_{m-1}[j]\oplus M_{m-1}[j]$.
The action of $\Sy_{2^m}$ on $M_m[j]$ is given by:
\begin{align*}
	(h_1,h_2)\cdot(v_1,v_2) &:= ( h_1\cdot v_1, h_2\cdot v_2),\\
	(h_1,h_2,\omega)\cdot(v_1,v_2) &:=( h_2\cdot v_2, h_1\cdot v_1),
\end{align*}
 where $(h_1,h_2)\in \Sy_{2^{m-1}}\times \Sy_{2^{m-1}}\subseteq (\Sy_{2^{m-1}}\times \Sy_{2^{m-1}})\rtimes \Z_2$, $\omega$ is the generator of the subgroup $\Z_2\subseteq  (\Sy_{2^{m-1}}\times \Sy_{2^{m-1}})\rtimes \Z_2$, and  $(v_1,v_2)\in  M_{m-1}[j]\oplus M_{m-1}[j]$. 
 It follows directly that $\dim(M_m[j])=2^{m-j}$ for all $1\leq j\leq m$.
 %TODO : Check what is pull-back representation
 
 \medskip
 The vector space $W_{2^m}$, now considered as an $\Sy_{2^m}$-representation, we denote by~$L_m$.
The following decomposition of $\Sy_{2^m}$-representations holds.
 
\begin{lemma}
\label{lem : iso of rep}
For every integer  $m\geq 1$ there is an isomorphism of real $\Sy_{2^m}$-representations
\[
L_m\cong M_m[1]\oplus\dots\oplus M_m[m].
\]
\end{lemma}

%------------------------------------------------------
\subsubsection{Associated vector bundles}
%------------------------------------------------------
Let $m\geq 1$ be an integer, and  let $d\geq 1$ be an integer, or $d=\infty$.
To every $\Sy_{2^m}$-representation introduced in the previous section we associate real vector bundles\index{vector bundle} in the following way:
\begin{equation}
\label{eq : associated bundles-01}
\xymatrix@1{
\lambda_{d,m} :\quad  L_m\ar[r] \ &\ \Sp(\R^d,2^m)\times_{\Sy_{2^m}}L_m\ar[r] \ &\ \Sp(\R^d,2^m)/ \Sy_{2^m},
}
\end{equation}

\begin{equation}
\label{eq : associated bundles-02}
\xymatrix@1{
\mu_{d,m}[j] : \quad M_m[j]\ar[r] \ &\ \Sp(\R^d,2^m)\times_{\Sy_{2^m}}M_m[j]\ar[r] \ &\  \Sp(\R^d,2^m)/ \Sy_{2^m},
}
\end{equation}
where $1\leq j\leq m$. 
Note that 
\begin{compactitem}[\ ---]
\item $\lambda_{1,d}=\mu_{1,d}[1]$ is the Hopf line bundle over $\RP^{d-1}\cong \Sp(\R^d,2)/ \Sy_{2}$,
\item $\mu_{d,m}[m]$ is the pull-back vector bundle\index{vector bundle} of the line bundle $\lambda_{d,1}$ via the map
\[
 \big(\Sp(\R^d,2^{m-1})/ \Sy_{2^{m-1}}\times \Sp(\R^d,2^{m-1}\big)/ \Sy_{2^{m-1}})\times_{\Z_2}\Sp(\R^d,2)\longrightarrow \Sp(\R^d,2)/ \Sy_{2},
\]
\item $\lambda_{d,m}$ and $\mu_{d,m}[j]$ are pull-backs of $\lambda_{\infty ,m}$ and $\mu_{\infty ,m}[j]$ along the map 
\[
\kappa_{d,2^m}/\Sy_{2^m}\colon \Sp(\R^d,2^m)/ \Sy_{2^m}\longrightarrow \Sp(\R^{\infty},2^m)/ \Sy_{2^m},
\]
\item $\lambda_{d,m}$ is the pull-back of $\zeta_{\R^d,2^m}$ via the map 
\[
\rho_{d,2^m}\colon \Sp(\R^d,2^m)/\Sy_{2^m}\longrightarrow\conf(\R^d,2^m)/\Sym_{2^m}.
\]
\end{compactitem}

\medskip
The vector bundles $\mu_{d,m}[j]$ for different values of $m$, are connected via the wreath square of bundles operations as follows:
\begin{equation}
\label{eq : connection of mu's}
	S^2\mu_{\infty ,m-1}[j]\cong \mu_{\infty,m}[j]
	\qquad\text{and}\qquad
	S^{2,d}\mu_{d,m-1}[j]\cong \mu_{d,m}[j],
\end{equation}
for all integers $d\geq 2$, $m\geq 1$, and all $1\leq j\leq m-1$.
For details about wreath square operations see Section \ref{sec : dyadic}.

\medskip
The decomposition from Lemma \ref{lem : iso of rep} implies the following Whitney sum decomposition on the vector bundle $\lambda_{d,m}$.

\begin{lemma}
\label{lem : iso of rep - bun}
For every integer  $m\geq 1$, and every integer $d\geq 1$  or $d=\infty$, there is an isomorphism of real $\Sy_{2^m}$-representations
\[
\lambda_{d,m}\cong \mu_{d,m}[1]\oplus\dots\oplus \mu_{d,m}[m].
\]
\end{lemma}

\medskip
Using the relation \eqref{eq : relation - 10} and the decomposition from the previous lemma we get the following equalities.
\begin{lemma}
\label{lem : iso of rep - sw classes}
For every integer  $m\geq 1$ and every integer $1\leq j\leq m$:
\[
w_{2^m-1}(\lambda_{m,\infty})=D_{m,0}\qquad\text{and}\qquad
w_{2^{m-j}}(\mu_{m,\infty}[j])=V_{m,m-j}.\index{Stiefel--Whitney classes}
\]
\end{lemma}

\medskip
From the fact that $w(\zeta_{\R^d,2^m})=w(\xi_{\R^d,2^m})$, Lemma \ref{lemma : corection power of 2}, Lemma \ref{lem : correction of 2.15},  and Theorem \ref{th : injection }, the injectivity of the homomorphism 
\[
\rho_{d,2^m}^*\colon H^*(\conf(\R^d,2^m)/\Sym_{2^m};\F_2)\longrightarrow H^*(\Sp(\R^d,2^m)/\Sy_{2^m};\F_2),
\]
we deduce the following facts.

\begin{lemma}
\label{lem : sw classes of pull-backs}
Let  $m\geq 1$ be an integer.
\begin{compactenum}[\rm  \ (1)]

\item If $d=2^a$ for an integer $a\geq 1$, then
\begin{multline*}
\overline{w}_{(d-1)(2^m-1)}(\lambda_{d,m})= 
\overline{w}_{(d-1)(2^m-1)}(\rho_{d,2^m}^*\zeta_{\R^d,2^m})=\\
\overline{w}_{(d-1)(2^m-1)}(\rho_{d,2^m}^*\xi_{\R^d,2^m})=
\rho_{d,2^m}^*\big(\overline{w}_{(d-1)(2^m-1)}(\xi_{\R^d,2^m})\big)\neq 0.	
\end{multline*}

\item If $d\geq 2$ is an integer, then
\begin{multline*}
\overline{w}_{(d-1)2^{m-1}}(\lambda_{d,m})= 
\overline{w}_{(d-1)2^{m-1}}(\rho_{d,2^m}^*\zeta_{\R^d,2^m})=\\
\overline{w}_{(d-1)2^{m-1}}(\rho_{d,2^m}^*\xi_{\R^d,2^m})=
\rho_{d,2^m}^*\big(\overline{w}_{(d-1)2^{m-1}}(\xi_{\R^d,2^m})\big)\neq 0.	
\end{multline*}
\end{compactenum}
\end{lemma}

Furthermore, from the fact that  $\mu_{d,m}[m]$ is the pull-back of the line bundle $\lambda_{d,1}$, and the description of the cohomology of $\Sp(\R^d,2^{m})/ \Sy_{2^{m}}$ given in Section \ref{sec : Equivariant cohomology of epicicles} we compute $\mu_{d,m}[m]$.

\begin{lemma}
\label{lem : SW class of mu[m]}
Let $d\geq1$ and $m\geq 1$ be integers.
Then
\[
w(\mu_{d,m}[m])=1+f,
\]
where the class $f\in H^1(\Sp(\R^d,2^{m})/ \Sy_{2^{m}};\F_2)$ was denoted by $V_{m,1}$ in the proof of Lemma \ref{lem : computation of algebra}. 
In particular $f^d=0$.
\end{lemma}

%------------------------------------------------------
\subsection{The Key Lemma and its consequences}
%------------------------------------------------------
Let $m\geq 1$ and $d\geq 1$ be integers.
For arbitrary non-negative integers $r_1,\dots,r_m$ we define the vector bundle $\psi_{d,m}[r_1,\dots,r_m]$ by:
\[
\psi_{d,m}[r_1,\dots,r_m]:=\mu_{d,m}[1]^{\oplus r_1}\oplus\cdots\oplus\mu_{d,m}[m]^{\oplus r_m}.
\]
Now we will prove the key technical lemma of this section.

\begin{lemma}[The Key Lemma]
\label{lem : key}\index{key lemma}%
Let $m\geq 1$, $d\geq 1$ and $\ell$ be integers, and let $r_1,\dots,r_m$ be non-negative integers. 
If the binomial coefficient
\begin{equation}
	\label{condition-binomial}
	{(r_12^{m-2}+\dots+r_{m-1}2^0)+r_m-(d-1+\ell)(2^{m-1}-1)-\ell} \choose {d-1}
\end{equation}
is odd, then
\[
w_{(d-1)(2^m-1)}(\lambda_{d,m}^{-\ell}\oplus \psi_{d,m}[r_1,\dots,r_m])\neq 0\index{Stiefel--Whitney classes}
\]
as an element of the cohomology group $H^{(d-1)(2^m-1)}(\Sp(\R^d,2^m)/\Sy_{2^m};\F_2)$.
Here $\lambda_{d,m}^{-\ell}$ denotes an inverse of the vector bundle\index{vector bundle} $\lambda_{d,m}^{\oplus\ell}$. 
\end{lemma}
\begin{proof}
We prove the claim of the lemma by induction on $m\geq 1$. 
In the case $m=1$ the condition	\eqref{condition-binomial} reads ${{r_1-\ell}\choose {d-1}}\neq 0$ in $\F_2$.
Furthermore, $\psi_{d,1}[r_1]=\mu_{d,1}[1]^{\oplus r_1}$ and $\lambda_{d,1}\cong\mu_{d,1}[1]$.
Consequently, $\lambda_{d,1}^{-\ell}\oplus\psi_{d,1}[r_1]\cong\mu_{d,1}[1]^{\oplus (r_1-\ell)}$, and so
\[
w(\lambda_{d,1}^{-\ell}\oplus\psi_{d,1}[r_1])=w(\mu_{d,1}[1]^{\oplus (r_1-\ell)})=w(\mu_{d,1}[1])^{r_1-\ell}=(1+f)^{r_1-\ell}
\]
in $H^*(\Sp(\R^d,2)/\Sy_{2};\F_2)\cong H^*(\RP^{d-1};\F_2)=\F_2[f]/\langle f\rangle $.
Here we use the fact that the vector bundle $\mu_{d,1}[1]$ is the Hopf line bundle, and so $w(\mu_{d,1}[1])=1+f$.
Hence, we have that
\[
w_{(d-1)(2^1-1)}(\lambda_{d,1}^{-\ell}\oplus\psi_{d,1}[r_1])=
w_{d-1}(\lambda_{d,1}^{-\ell}\oplus\psi_{d,1}[r_1])= {{r_1-\ell}\choose {d-1}}w_1^{d-1}\neq 0.
\]	

\medskip
Let $m\geq 2$, and assume, as an induction hypothesis, that 
\[
w_{(d-1)(2^{m-1}-1)}(\lambda_{d,m-1}^{-\ell}\oplus \psi_{d,m-1}[r_1,\dots,r_{m-1}])\neq 0
\]
as an element of $H^{(d-1)(2^{m-1}-1)}(\Sp(\R^d,2^{m-1})/\Sy_{2^{m-1}};\F_2)$.
Then there exists a vector bundle\index{vector bundle} $\varpi$ of dimension $(d-1)(2^{m-1}-1)$ over $\Sp(\R^d,2^{m-1})/\Sy_{2^{m-1}}$ with the property that for some integers $N_0\geq\ell$ and $N_1\geq 1$ there is an isomorphism of vector bundles
\begin{align}
\label{eq : isomorphism in IH}
\varpi\oplus\lambda_{d,m-1}^{\oplus\ell}\oplus\tau_{d,m-1}^{\oplus N_0} &\cong \psi_{d,m-1}[r_1,\dots,r_{m-1}]\oplus\tau_{d,m-1}^{\oplus N_1}\nonumber	\\
& \cong \mu_{d,m-1}[1]^{\oplus r_1}\oplus\cdots\oplus\mu_{d,m-1}[m-1]^{\oplus r_{m-1}}\oplus\tau_{d,m-1}^{\oplus N_1}.
\end{align}
%TODO : This needs explanation
Here $\tau_{d,m-1}$ denotes the trivial line bundle over $\Sp(\R^d,2^{m-1})/\Sy_{2^{m-1}}$.
Consequently,
\begin{equation}
\label{lem : IH SW class}
w_{(d-1)(2^{m-1}-1)}(\varpi)=w_{(d-1)(2^{m-1}-1)}(\lambda_{d,m-1}^{-\ell}\oplus \psi_{d,m-1}[r_1,\dots,r_{m-1}])\neq 0
\end{equation}
does not vanish --- by induction hypothesis.

\medskip
Now we apply $(d-1)$-partial wreath square\index{partial wreath square} $S^{2,d}$ to the isomorphism of vector bundles \eqref{eq : isomorphism in IH} and get the following isomorphism of vector bundles 
\begin{multline}
\label{eq : isomorphism in IH squared}
S^{2,d}\varpi\oplus\lambda_{d,m}^{\oplus\ell}\oplus\tau_{d,m}^{\oplus N_0}\oplus
\mu_{d,m}[m]^{\oplus N_0-\ell}
\cong \\
\mu_{d,m}[1]^{\oplus r_1}\oplus\cdots\oplus\mu_{d,m}[m-1]^{\oplus r_{m-1}}\oplus\tau_{d,m}^{\oplus N_1}\oplus\mu_{d,m}[m]^{\oplus N_1},
\end{multline}
over the base space $S^{2,d}\Sp(\R^d,2^{m-1})/\Sy_{2^{m-1}}\cong \Sp(\R^d,2^{m})/\Sy_{2^{m}}$.
Indeed, by direct inspection and using \eqref{eq : connection of mu's} and \eqref{eq : direct sum wsquare}, we get that
\[
S^{2,d}\tau_{d-1,m}\cong \tau_{d,m}\oplus\mu_{d,m}[m]
\qquad
\text{and}
\qquad
S^{2,d}\mu_{d,m-1}[j]\cong \mu_{d,m}[j]
\] 
for all $1\leq j\leq m-1$.
Consequently, from Lemma \ref{lem : iso of rep - bun} we get that
\[
S^2\lambda_{d,m-1}\cong S^2(\mu_{d,m-1}[1]\oplus\dots\oplus\mu_{d,m-1}[m-1])\cong\mu_{d,m}[1]\oplus\dots\oplus\mu_{d,m}[m-1].
\]
Collecting all these facts together we compute
\begin{multline*}
\varpi\oplus\lambda_{d,m-1}^{\oplus\ell}\oplus\tau_{d,m-1}^{\oplus N_0}=\\ \varpi\oplus (\mu_{d,m-1}[1]\oplus\dots\oplus\mu_{d,m-1}[m-1])^{\oplus\ell}\oplus\tau_{d,m-1}^{\oplus N_0} 
\\
\xymatrix{~\ar@{|->}[d]^{S^{2,d}}\\~} 
\\  
S^{2,d}\varpi\oplus (\mu_{d,m}[1]\oplus\dots\oplus\mu_{d,m}[m-1])^{\oplus\ell}\oplus\tau_{d,m}^{N_0}\oplus\mu_{d,m}[m]^{\oplus N_0}\\ 
\cong S^{2,d}\varpi\oplus\lambda_{d,m}^{\oplus\ell}\oplus\tau_{d,m}^{N_0}\oplus\mu_m[m]^{\oplus N_0-\ell}.
\end{multline*}
On the other hand directly follows that
\begin{multline*}
\mu_{d,m-1}[1]^{\oplus r_1}\oplus\cdots\oplus\mu_{d,m-1}[m-1]^{\oplus r_{m-1}}\oplus\tau_{d,m-1}^{\oplus N_1} \\
\xymatrix{~\ar@{|->}[d]^{S^{2,d}}\\~} 
\\
\mu_{d,m}[1]^{\oplus r_1}\oplus\cdots\oplus\mu_{d,m}[m-1]^{\oplus r_{m-1}}\oplus\tau_{d,m}^{\oplus N_1}\oplus\mu_{d,m}[m]^{\oplus N_1}.
\end{multline*}
Thus, the isomorphism \eqref{eq : isomorphism in IH squared} holds.

\medskip
For an arbitrary vector bundle $\eta$ over $\Sp(\R^d,2^{m})/\Sy_{2^{m}}$ the isomorphism of vector bundles\index{vector bundle} \eqref{eq : isomorphism in IH squared} yields the isomorphism
\begin{multline}
\label{eq : isomorphism in IH squared plus eta}
\eta\oplus S^{2,d}\varpi\oplus\lambda_{d,m}^{\oplus\ell}\oplus\tau_{d,m}^{\oplus N_0}\oplus
\mu_{d,m}[m]^{\oplus N_0-\ell}
\cong \\
\eta\oplus\mu_{d,m}[1]^{\oplus r_1}\oplus\cdots\oplus\mu_{d,m}[m-1]^{\oplus r_{m-1}}\oplus\tau_{d,m}^{\oplus N_1}\oplus\mu_{d,m}[m]^{\oplus N_1}.
\end{multline}
Take $\eta$ to be stable equivalent to the vector bundle $\mu_{d,m}[m]^{\oplus r_m+N_0-\ell-N_1}$.
Then the vector bundles
\[
\eta\oplus S^{2,d}\varpi
\qquad\text{and}\qquad
\lambda_{d,m}^{-\ell}\oplus\psi_{d,m}[r_1,\dots,r_m]
\]
are stable equivalent. 
Indeed, if $\eta$ is stable equivalent to $\mu_{d,m}[m]^{\oplus r_m+N_0-\ell-N_1}$, then there  exist integers $M_1,M_2\geq 0$, and an isomorphism
\[
\eta\oplus\tau_{d,m}^{\oplus M_1}\cong \mu_{d,m}[m]^{\oplus r_m+N_0-\ell-N_1}\oplus\tau_{d,m}^{\oplus M_2}.
\]
Now the isomorphism \eqref{eq : isomorphism in IH squared plus eta} implies that
\begin{multline*}
\label{eq : isomorphism in IH squared plus eta}
\eta\oplus S^{2,d}\varpi\oplus\lambda_{d,m}^{\oplus\ell}\oplus\tau_{d,m}^{\oplus N_0+M_1}\oplus
\mu_{d,m}[m]^{\oplus N_0-\ell}
\cong \\
(\mu_{d,m}[1]^{\oplus r_1}\oplus\cdots\oplus\mu_{d,m}[m-1]^{\oplus r_{m-1}})\oplus\mu_{d,m}[m]^{\oplus r_m}\oplus \\
\mu_{d,m}[m]^{\oplus N_0-\ell-N_1}\oplus\tau_{d,m}^{\oplus N_1+M_2}\oplus\mu_{d,m}[m]^{\oplus N_1}.
\end{multline*}
Hence,
\begin{multline*}
(\eta\oplus S^{2,d}\varpi)\oplus\lambda_{d,m}^{\oplus\ell}\oplus\tau_{d,m}^{\oplus N_0+M_1}\oplus
\mu_{d,m}[m]^{\oplus N_0-\ell}
\cong \\
\psi_{d,m}[r_1,\dots,r_m]\oplus\tau_{d,m}^{\oplus N_1+M_2}\oplus\mu_{d,m}[m]^{\oplus N_0-\ell},
\end{multline*}
and consequently the vector bundles $\eta\oplus S^{2,d}\varpi$ and $\lambda_{d,m}^{-\ell}\oplus\psi_{d,m}[r_1,\dots,r_m]$ are stable equivalent.
In particular,
\[
w(\lambda_{d,m}^{-\ell}\oplus\psi_{d,m}[r_1,\dots,r_m])=w(\eta\oplus S^{2,d}\varpi)=w(\mu_{d,m}[m]^{\oplus r_m+N_0-\ell-N_1}\oplus S^{2,d}\varpi).
\]

\medskip
Now we continue computation of the total Stiefel--Whitney class\index{Stiefel--Whitney classes} of the vector bundle\index{vector bundle} $\lambda_{d,m}^{-\ell}\oplus\psi_{d,m}[r_1,\dots,r_m]$ as follows:
\begin{eqnarray*}
w(\lambda_{d,m}^{-\ell}\oplus\psi_{d,m}[r_1,\dots,r_m])
&=&	
w(\mu_{d,m}[m]^{\oplus r_m+N_0-\ell-N_1}\oplus S^{2,d}\varpi) \nonumber \\
&=&
w(\mu_{d,m}[m])^{\oplus r_m+N_0-\ell-N_1}\cdot w(S^{2,d}\varpi)\nonumber \\
&\overset{\text{Lem.\,\ref{lem : SW class of mu[m]}}}{=} &(1+f)^{r_m+N_0-\ell-N_1}\cdot w(S^{2,d}\varpi)	\nonumber\\
&=&
(1+f)^{r_m+N_0-\ell-N_1}\cdot s^{2,d}(\varpi)\nonumber \\
&=&
\Big(\sum_{i=0}^{r_m+N_0-\ell-N_1}{ {r_m+N_0-\ell-N_1\choose i}}\,f^i\Big) \cdot s^{2,d}(\varpi)\nonumber \\
&\overset{f^d=0}{=}&
\Big(\sum_{i=0}^{d-1}{{r_m+N_0-\ell-N_1\choose i}} \,f^i\Big)\cdot s^{2,d}(\varpi).
\end{eqnarray*}
From Corollary \ref{th : SW class of partial doubling} applied to the $(d-1)(2^{m-1}-1)$-dimensional vector bundle $\varpi$ we have that
\begin{multline*}
s^{2,d}(\varpi)=
\sum_{0\leq r< s\leq N}T(w_r(\varpi)\otimes w_s(\varpi))+\\
\sum_{0\leq r\leq N}\,
\sum_{0\leq j\leq \min\{N-r,d-1\}}
{N-r \choose j}  P(w_r(\varpi)) f^{j},
\end{multline*}
where $N=(d-1)(2^{m-1}-1)$.

\medskip
Combining the last two equations we get
\begin{eqnarray*}
w(\lambda_{d,m}^{-\ell}\oplus\psi_{d,m}[r_1,\dots,r_m])
&=&
\Big(\sum_{i=0}^{d-1}{{r_m+N_0-\ell-N_1\choose i}} \,f^i\Big)\cdot\nonumber	\\
& &\Big(\sum_{0\leq r< s\leq N}T(w_r(\varpi)\otimes w_s(\varpi))+\nonumber \\
& &\sum_{0\leq r\leq N}\,
\sum_{0\leq j\leq \min\{N-r,d-1\}}
{N-r \choose j}  P(w_r(\varpi)) f^{j}\Big)\nonumber \\
&\overset{\eqref{eq : multiplication Q and t-a}}{=}&
\sum_{0\leq r< s\leq N}T(w_r(\varpi)\otimes w_s(\varpi))+\nonumber \\
& &\sum_{i=0}^{d-1}\sum_{0\leq r\leq N}\sum_{0\leq j\leq \min\{N-r,d-1\}}\nonumber \\
& &
{{r_m+N_0-\ell-N_1\choose i}}{N-r \choose j}  P(w_r(\varpi)) \,f^{i+j}.
\end{eqnarray*}
Note that 
\begin{compactitem}[\ ---]
\item $\deg(T(w_r(\varpi)\otimes w_s(\varpi)))\leq 2N-1=(d-1)(2^{m}-2)-1$, and
\item $\deg(P(w_r(\varpi))f^{i+j})\leq 2N+d-1=(d-1)(2^{m}-1)$.	
\end{compactitem}
Consequently, we have that
\[
w_{(d-1)(2^{m}-1)}(\lambda_{d,m}^{-\ell}\oplus\psi_{d,m}[r_1,\dots,r_m])=
{{r_m+N_0-\ell-N_1\choose d-1}} P(w_N(\varpi)) \,f^{d-1}.
\]

\medskip
Before making final arguments that the Stiefel--Whitney class\index{Stiefel--Whitney classes} 
\[
w_{(d-1)(2^m-1)}(\lambda_{d,m}^{-\ell}\oplus \psi_{d,m}[r_1,\dots,r_m])
\]
does not vanish let us review all the assumptions we have:
\begin{compactitem}[\ ---]
\item ${{(r_12^{m-2}+\dots+r_{m-1}2^0)+r_m-(d-1+\ell)(2^{m-1}-1)-\ell} \choose {d-1}}\neq 0\in\F_2$ --- the assumption \eqref{condition-binomial} from the statement of the lemma, and
\item $w_{(d-1)(2^{m-1}-1)}(\varpi)\neq 0$ --- the induction hypothesis \eqref{lem : IH SW class}.
\end{compactitem}
Hence, by the induction hypothesis we know that the class $P(w_N(\varpi)) \,f^{d-1}$ does not vanish.
For the binomial coefficient ${{r_m+N_0-\ell-N_1\choose d-1}}$ note that the isomorphism of vector bundles \eqref{eq : isomorphism in IH}, by evaluating dimensions, implies that
\[
N_0-N_1=(r_12^{m-2}+\dots+r_{m-1}2^0)-(d-1+\ell)(2^{m-1}-1),
\]
or in other word
\[
r_m+N_0-\ell-N_1=(r_12^{m-2}+\dots+r_{m-1}2^0)+r_m-(d-1+\ell)(2^{m-1}-1)-\ell.
\]
Therefore, by the assumption \eqref{condition-binomial}
\[
{{r_m+N_0-\ell-N_1\choose d-1}}\neq 0\in\F_2.
\]
This completes the proof that
\[
w_{(d-1)(2^m-1)}(\lambda_{d,m}^{-\ell}\oplus \psi_{d,m}[r_1,\dots,r_m])\neq 0.
\]
\end{proof}

\begin{remark}
Observe that the proof of the previous lemma actually yields the following equivalence:
\[
{{(r_12^{m-2}+\dots+r_{m-1}2^0)+r_m-(d-1+\ell)(2^{m-1}-1)-\ell} \choose {d-1}}\neq 0\in\F_2
\]
if and only if
\[
w_{(d-1)(2^m-1)}(\lambda_{d,m}^{-\ell}\oplus \psi_{d,m}[r_1,\dots,r_m])\neq 0.
\]
\end{remark}

\begin{remark}
The proof of Lemma \ref{lem : key} can be simplified as follows. 
In the step when made a choice of the vector bundle $\eta$, to be stable equivalent to the vector bundle $\mu_m[m]^{\oplus r_m+N_0-\ell-N_1}$, we could have in addition asked that $\eta$ is in addition $(d-1)$-dimensional.
(Note that  $\mu_m[m]$ is a pull-back vector bundle of the vector bundle $\lambda_{d,1}$ over $\RP^{d-1}$.)
Then the Stiefel--Whitney class\index{Stiefel--Whitney classes} 
\[
w_{(d-1)(2^m-1)}(\lambda_{d,m}^{-\ell}\oplus \psi_{d,m}[r_1,\dots,r_m])=
w_{(d-1)(2^m-1)}(\eta\oplus S^{2,d}\varpi)
\]
is actually the mod $2$ Euler class  $\mathfrak{e}(\eta\oplus S^{2,d}\varpi)$ of the vector bundle $\eta\oplus S^{2,d}\varpi$.
Indeed, $\dim(\eta\oplus S^{2,d}\varpi)=(d-1)(2^m-1)$.
Thus, using the product formula for Euler classes\index{Euler class} \cite[Prop.\,9.6]{Milnor1974} we have that
\[
w_{(d-1)(2^m-1)}(\eta\oplus S^{2,d}\varpi)=\mathfrak{e}(\eta\oplus S^{2,d}\varpi)=
\mathfrak{e}(\eta)\cdot\mathfrak{e}( S^{2,d}\varpi).
\]
Here $\mathfrak{e}(\cdot)$ denotes the mod $2$ Euler class of the respected vector bundle.
Now from Corollary \ref{cor : mod 2 Euler class} and the description of the cohomology $H^{*}(\Sp(\R^d,2^{m-1})/\Sy_{2^{m-1}};\F_2)$ of the base space we have that
\begin{multline*}
w_{(d-1)(2^m-1)}(\eta\oplus S^{2,d}\varpi)=
\mathfrak{e}(\eta)\cdot\mathfrak{e}( S^{2,d}\varpi)=\\	
{{r_m+N_0-\ell-N_1\choose d-1}} P(w_N(\varpi)) \,f^{d-1}\neq 0.
\end{multline*}
Thus, there was no need to evaluate the total Stiefel--Whitney class of the vector bundle\index{vector bundle} $\eta\oplus S^{2,d}\varpi$ in full.
\end{remark}

\medskip
After proving Lemma \ref{lem : key} we want to discuss how to utilize it. 
In other words, for which integer parameters $d\geq 2$, $m\geq 1$, $r_1,\dots,r_m\geq 0$ the assumption \eqref{condition-binomial} is satisfied.

\begin{lemma}
\label{when Key Lemma can be applied}
Let $d\geq 2$, $m\geq 1$ and $\ell$ be integers, and let $d=2^t+e$ for some integers $t\geq 1$ and $0\leq e\leq 2^t-1$.
If
\begin{compactenum}[\rm \ (1)]
\item $\ell=1$, $r_1=0$, and $r_2=\dots=r_m=2e$, or	
\item $\ell=d+1$, and  $r_1=\dots=r_m=2e$, or
\item $\ell=-(d-1+k2^{t+1})$ for some integer $k$, and $r_1=\dots=r_m=0$,
\end{compactenum}
then
\[
{{(r_12^{m-2}+\dots+r_{m-1}2^0)+r_m-(d-1+\ell)(2^{m-1}-1)-\ell} \choose {d-1}}=1\in\F_2.
\]
\end{lemma}
\begin{proof}
{\bf (1)} Let $\ell=1$, $r_1=0$, and $r_2=\dots=r_m=2e$.
Then
\begin{multline*}
(r_12^{m-2}+\dots+r_{m-1}2^0)+r_m-(d-1+\ell)(2^{m-1}-1)-\ell=	\\
2e(2^{m-3}+\dots+2^0)+2e-d(2^{m-1}-1)-1=2^{m-1}(e-d)+2e+d-1=\\
2^{m-1}(e-d)+d-1=
\begin{cases}
	d-1-2^{m-1+t}, & m\geq 2,\\
	-1, & m=1.
\end{cases}
\end{multline*}
Now, ${-1 \choose d-1}=_{\F_2}1\in\F_2$ and ${d-1-2^{m-1+t}\choose d-1}=_{\F_2}{d-1\choose d-1}=_{\F_2}1\in\F_2$, because for $m\geq 2$ we have that $m-1+t\geq t+1$  and $d-1< 2^{m+1}$.

\medskip
{\bf (2)} Let  $\ell=d+1$, and $r_1=\dots=r_m=2e$.
Then
\begin{multline*}
(r_12^{m-2}+\dots+r_{m-1}2^0)+r_m-(d-1+\ell)(2^{m-1}-1)-\ell=	\\
2e(2^{m-2}+\dots+2^0)+2e-2d(2^{m-1}-1)-d-1=2^m(e-d)+d-1\\
=d-1-2^{m+t}.
\end{multline*}
Now ${d-1-2^{m+t}\choose d-1}=_{\F_2}1\in\F_2$, because $d-1<2^{t+1}\leq 2^{t+m}$, and in the ring of formal power series $\F_2[[T]]$ the following equality holds
\begin{multline*}
	(1+T)^{d-1-2^{m+t}}=(1+T)^{d-1}(1+T)^{-2^{m+t}}= \\(1+T)^{d-1}(1+T^{2^{m+t}})^{-1}= 
	(1+T)^{d-1}\sum_{j\geq 0}T^{j2^{m+t}}.
\end{multline*}

\medskip
{\bf (3)} Let $\ell=-(d-1+j2^{t+1})$ for some integer $k\geq 1$, and $r_1=\dots=r_m=0$.
Then
\begin{multline*}
(r_12^{m-2}+\dots+r_{m-1}2^0)+r_m-(d-1+\ell)(2^{m-1}-1)-\ell=	\\
-(d-1+\ell)(2^{m-1}-1)-\ell=d-1+k2^{m+t}
\end{multline*}
Hence, ${d-1+k2^{m+t}\choose d-1}=_{\F_2}1\in\F_2$, because $d-1<2^{m+t}$, and in the ring $\F_2[[T]]$ the following equality holds 
\begin{multline*}
	(1+T)^{d-1+k2^{m+t}}=(1+T)^{d-1}(1+T)^{k2^{m+t}}= \\(1+T)^{d-1}(1+T^{2^{m+t}})^{k}= 
	(1+T)^{d-1}\sum_{j=0}^k{k\choose j}T^{j2^{m+t}}.
\end{multline*}
\end{proof}

\medskip
Now from Lemma \ref{lem : key} and Lemma \ref{when Key Lemma can be applied} we get the following particular results.

\begin{corollary}
\label{cor 1 of Key Lemma}	
Let $d\geq 2$, $m\geq 1$ and $k$ be integers, and let $d=2^t+e$ for some integers $t\geq 1$ and $0\leq e\leq 2^t-1$.
Then
\begin{compactenum}[\rm \ (1)]
\item $w_{(d-1)(2^m-1)}(\lambda_{d,m}^{-1}\oplus (\mu_{d,m}[2]\oplus\cdots\oplus\mu_{d,m}[m])^{\oplus 2e})\neq 0$,\index{Stiefel--Whitney classes} 
\item  $w_{(d-1)(2^m-1)}(\lambda_{d,m}^{-(d+1)}\oplus (\mu_{d,m}[1]\oplus\cdots\oplus\mu_{d,m}[m])^{\oplus 2e})\neq 0$, and 
\item $w_{(d-1)(2^m-1)}(\lambda_{d,m}^{d-1+k2^{t+1}})\neq 0$. 

\end{compactenum}
\end{corollary}

\begin{remark}
\label{remark-connections}
In the case (1) of the previous corollary for $e=0$ we get an alternative proof of Lemma \ref{lemma : corection power of 2}.
On the other hand, in the case (2) for $e=0$ we have a particular case of Theorem \ref{thm : Correction of T.3.7}\,(3).	
\end{remark}

\medskip

%------------------------------------------------------
\subsection{Additional bounds for the existence of highly regular embeddings}
%------------------------------------------------------

In this section we use consequences of the Key Lemma to derive further bounds for the existence of highly regular embeddings.

\medskip
First, we use Corollary \ref{cor 1 of Key Lemma} to get specific results which are relevant for the study of highly regular embeddings.	

\begin{corollary}
\label{cor 2 of Key Lemma}
Let $d\geq 2$ and $m\geq 1$ be integers, and let $d=2^t+e$ for some integers $t\geq 1$ and $0\leq e\leq 2^t-1$.
Then exist integers $a$ and $b$ with the property that 
\[
(d-1)(2^m-1)-2(2^{m-1}-1)e\leq a\leq (d-1)(2^m-1),
\]
\[
(d-1)(2^m-1)-2(2^{m}-1)e\leq b\leq (d-1)(2^m-1),
\]
and in addition, the Stiefel--Whitney classes\index{Stiefel--Whitney classes} 
\begin{equation}
\label{eq : estimate}
w_a(\lambda_{d,m}^{-1})\neq 0
\qquad\text{and}\qquad
w_b(\lambda_{d,m}^{-(d+1)})\neq 0
\end{equation}
do not vanish.
\end{corollary}
\begin{proof}
From Corollary \ref{cor 1 of Key Lemma}\,(i) we have that	
\begin{equation}
\label{eq-a-1}
w_{(d-1)(2^m-1)}(\lambda_{d,m}^{-1}\oplus (\mu_{d,m}[2]\oplus\cdots\oplus\mu_{d,m}[m])^{\oplus 2e})\neq 0.
\end{equation}
On the other hand
\begin{multline*}
w_{(d-1)(2^m-1)}(\lambda_{d,m}^{-1}\oplus (\mu_{d,m}[2]\oplus\cdots\oplus\mu_{d,m}[m])^{\oplus 2e})=\\
\sum_{i=0}^{(d-1)(2^m-1)}
w_i(\lambda_{d,m}^{-1})\cdot
w_{(d-1)(2^m-1)-i}((\mu_{d,m}[2]\oplus\cdots\oplus\mu_{d,m}[m])^{\oplus 2e}),
\end{multline*}
where
\begin{multline}\label{eq-a-2}
\dim (\mu_{d,m}[2]\oplus\cdots\oplus\mu_{d,m}[m])^{\oplus 2e}=2e\cdot\dim(\mu_{d,m}[2]\oplus\cdots\oplus\mu_{d,m}[m])=\\
2e(2^{m-2}+\dots+2^0)=2e(2^{m-1}-1). 	
\end{multline}
Now \eqref{eq-a-1} and \eqref{eq-a-2} imply that at lease one Stiefel--Whitney class $w_i(\lambda_{d,m}^{-1})\neq 0$ does not vanish for
\[
(d-1)(2^m-1)-i\leq 2e(2^{m-1}-1)=2e\cdot\dim (\mu_{d,m}[2]\oplus\cdots\oplus\mu_{d,m}[m])
\]
and $i\leq (d-1)(2^m-1)$.
Hence we proved the first claim of \eqref{eq : estimate}.

\medskip
For the second part of \eqref{eq : estimate} recall Corollary \ref{cor 1 of Key Lemma}\,(ii):
\begin{equation}\label{eq-b-1}
w_{(d-1)(2^m-1)}(\lambda_{d,m}^{-(d+1)}\oplus (\mu_{d,m}[1]\oplus\cdots\oplus\mu_{d,m}[m])^{\oplus 2e})\neq 0.
\end{equation}
Similarly,
\begin{multline*}
w_{(d-1)(2^m-1)}(\lambda_{d,m}^{-(d+1)}\oplus (\mu_{d,m}[1]\oplus\cdots\oplus\mu_{d,m}[m])^{\oplus 2e})=\\
\sum_{i=0}^{(d-1)(2^m-1)}
w_i(\lambda_{d,m}^{-(d+1)})\cdot
w_{(d-1)(2^m-1)-i}((\mu_{d,m}[1]\oplus\cdots\oplus\mu_{d,m}[m])^{\oplus 2e}),
\end{multline*}
where
\begin{multline}\label{eq-b-2}
\dim (\mu_{d,m}[1]\oplus\cdots\oplus\mu_{d,m}[m])^{\oplus 2e}=2e\cdot\dim(\mu_{d,m}[1]\oplus\cdots\oplus\mu_{d,m}[m])=\\
2e(2^{m-1}+\dots+2^0)=2e(2^{m-1}-1). 	
\end{multline}
In the same way, now \eqref{eq-b-1} and \eqref{eq-b-2}, imply that there exists at lease one Stiefel--Whitney class $w_i(\lambda_{d,m}^{-(d+1)})\neq 0$ which does not vanish where
\[
(d-1)(2^m-1)-i\leq 2e(2^{m}-1)=2e\cdot\dim (\mu_{d,m}[1]\oplus\cdots\oplus\mu_{d,m}[m])
\]
and $i\leq (d-1)(2^m-1)$.
Thus, we proved the second part of \eqref{eq : estimate}.
\end{proof}

\begin{remark}
Since the vector bundle $\lambda_{d,m}$ is the pull-back of the vector bundle\index{vector bundle} $\zeta_{\R^d,2^m}$, and $w(\zeta_{\R^d,2^m})=w(\xi_{\R^d,2^m})$, the previous corollary implies, under identical assumptions on parameters, that
\[
\overline{w}_a(\xi_{\R^d,2^m})\neq 0
\qquad\text{and}\qquad
\overline{w}_b(\xi_{\R^d,2^m}^{\oplus d+1})\neq 0.
\]
	
\end{remark}

\medskip
Next, we use Corollary \ref{cor 2 of Key Lemma} to get further insights on the dual Stiefel--Whitney classes\index{Stiefel--Whitney classes}  $\overline{w}(\xi_{\R^d,k})$ and $\overline{w}(\xi_{\R^d,k}^{\oplus d+1})$, but this time without restricting to the case when $k$ is a power of~$2$.

\begin{corollary}
\label{cor 3 of Key Lemma}
Let $d\geq 2$ and $k\geq 1$ be integers, and let $d=2^t+e$ for some integers $t\geq 1$ and $0\leq e\leq 2^t-1$.
Then exist integers $A$ and $B$ with the property that 
\[
 (d-e-1)(k-\alpha(k))+e(\alpha(k)-\epsilon(k))\leq A\leq (d-1)(k-1),
\]
\[
(d-2e-1)(k-\alpha(k))\leq B\leq (d-1)(k-1),
\]
and in addition, the dual Stiefel--Whitney classes 
\begin{equation}
\label{eq : further estimate}
\overline{w}_A(\xi_{\R^d,k})\neq 0
\qquad\text{and}\qquad
\overline{w}_B(\xi_{\R^d,k}^{\oplus d+1})\neq 0
\end{equation}
do not vanish.
Recall that $\epsilon(k)=1$ for $k$ odd, and $\epsilon(k)=0$ for $k$ even. 
\end{corollary}
\begin{proof}
Let $r:=\alpha(k)$ be the number of $1$s in the binary presentation of the integer $k\geq 1$, and let $k=2^{k_1}+\dots+2^{k_r}$ where $0\leq k_1<k_2<\dots<k_r$.
 Like in the proof of Lemma \ref{lemma : corection not a power of 2}, consider a morphism between vector bundles $\prod_{i=1}^{r}\xi_{\R^d,2^{k_i}}$ and $\xi_{\R^d,k}$ where the following commutative square is a pullback diagram:
\[
\xymatrix@1{\prod_{i=1}^{r}\xi_{\R^d,2^{k_i}} \ar[rr]^-{\Theta}  \ar[d] \
& & \ \xi_{\R^d,k}   \ar[d] 
\\
%& & \\
\prod_{i=1}^{r} \conf(\R^d,2^{k_i})/\Sym_{2^{k_i}}  \ar[rr]^-{\theta} \  & & \ 
 \conf(\R^d,k)/{\Sym_k}.
}
\]
Recall, that the map $\theta$ is induced, up to an equivariant homotopy, from a restriction of the little cubes operad\index{little cubes operad} structural map.

\medskip
Thus, we have that $\theta^*\xi_{\R^d,k} \cong  \prod_{i=1}^{r}\xi_{\R^d,2^{k_i}}$, and consequently
\begin{equation}
\label{eq:two-pull-backs}
\theta^*(\overline{w}(\xi_{\R^d,k}))=\overline{w}\Big(\prod_{i=1}^{r}\xi_{\R^d,2^{k_i}}\Big)
\qquad\text{and}\qquad
\theta^*(\overline{w}(\xi_{\R^d,k}^{\oplus d+1}))=\overline{w}\Big(\prod_{i=1}^{r}\xi_{\R^d,2^{k_i}}^{\oplus d+1}\Big).	
\end{equation}
The product formula for Stiefel--Whitney classes\index{Stiefel--Whitney classes}  \cite[Pr.\,4-A, p.\,54]{Milnor1974} implies that
\[
 \overline{w}\Big(\prod_{i=1}^{r}\xi_{\R^d,2^{k_i}}\Big)=\overline{w}(\xi_{\R^d,2^{k_1}})\times\cdots\times\overline{w}(\xi_{\R^d,2^{k_r}}),
\]
and similarly 
\[
 \overline{w}\Big(\prod_{i=1}^{r}\xi_{\R^d,2^{k_i}}^{\oplus d+1}\Big)=\overline{w}(\xi_{\R^d,2^{k_1}}^{\oplus d+1})\times\cdots\times\overline{w}(\xi_{\R^d,2^{k_r}}^{\oplus d+1}).
\] 

\medskip
From Corollary \ref{cor 2 of Key Lemma} and for every $0\leq k_1<k_2<\dots<k_r$ there exist integers $a_1,\dots, a_r$, and integers $b_1,\dots, b_r$, with the property that 
\[
\begin{cases}
	a_1=0, & k_1=0,\\
	a_1\geq (d-1)(2^{k_1}-1)-2(2^{k_1-1}-1)e, & k_1\geq 1,\\
	a_i\geq (d-1)(2^{k_i}-1)-2(2^{k_i-1}-1)e, & 2\leq i\leq r,
\end{cases}
\]
and 
\[
\begin{cases}
	b_1=(d-1)(2^{k_1}-1)-2(2^{k_1}-1)e=0, & k_1=0,\\
	b_1\geq (d-1)(2^{k_1}-1)-2(2^{k_1}-1)e, & k_1\geq 1,\\
	b_i\geq (d-1)(2^{k_i}-1)-2(2^{k_i}-1)e, & 2\leq i\leq r,
\end{cases}
\]
and in addition the dual Stiefel-- Whitney classes
\[
\overline{w}_{a_i}(\xi_{\R^d,2^{k_i}})\neq 0
\qquad\text{and}\qquad
\overline{w}_{b_i}(\xi_{\R^d,2^{k_i}}^{\oplus d+1})\neq 0.
\]
do not vanish.

\medskip
Let denotes the sums of integers $a_1,\dots, a_r$ and $b_1,\dots, b_r$ by
\[
A:=\sum_{i=0}^ra_i\geq (d-e-1)(k-\alpha(k))+e(\alpha(k)-\epsilon(k)),
\]
and
\[
B:=\sum_{i=0}^rb_i\geq (d-2e-1)(k-\alpha(k)).
\]
Now, from the product formula for Stiefel--Whitney classes and the K\"unneth formula we conclude that
\[
 \overline{w}_A\Big(\prod_{i=1}^{r}\xi_{\R^d,2^{k_i}}\Big)\neq 0
 \qquad\text{and}\qquad
  \overline{w}_B\Big(\prod_{i=1}^{r}\xi_{\R^d,2^{k_i}}^{\oplus d+1}\Big)\neq 0.
\]
(For details of the last argument consult for example the proof of Lemma \ref{lemma : corection not a power of 2}.)
Finally, from the pull-backs \eqref{eq:two-pull-backs} we conclude that 
\[
\overline{w}_A(\xi_{\R^d,k})\neq 0
\qquad\text{and}\qquad
\overline{w}_B(\xi_{\R^d,k}^{\oplus d+1})\neq 0.
\]
\end{proof}

Combining the results of Corollary \ref{cor 3 of Key Lemma} and the approach used in the proof of Theorem \ref{thm : Correction of T.4.8} we get the following result.

\begin{corollary}
\label{cor 4 of Key Lemma}	
Let $d\geq 2$, $k\geq 1$ and $\ell\geq 1$ be integers, and let $d=2^t+e$ for some integers $t\geq 1$ and $0\leq e\leq 2^t-1$.
Then exists an integer $C$ with the property that 
\begin{multline*}
	 (d-e-1)(k-\alpha(k))+e(\alpha(k)-\epsilon(k))+(d-2e-1)(\ell-\alpha(\ell))\leq \\
	 C\leq (d-1)(k+\ell-2),
\end{multline*}

and in addition, the dual Stiefel--Whitney class
\begin{equation}
\label{eq : further estimate-02}
\overline{w}_C(\xi_{\R^d,k}\times \xi_{\R^d,\ell}^{\oplus d+1})\neq 0
\end{equation}
does not vanish.
\end{corollary}
\begin{proof}
We start as in the proof of Theorem 	\ref{thm : Correction of T.4.8}.
To compute the dual Stiefel--Whitney class\index{Stiefel--Whitney classes} of the product vector bundle\index{vector bundle} $\xi_{\R^d,k}\times \xi_{\R^d,\ell}^{\oplus (d+1)}$ we apply the product formula~\cite[Problem~4-A, page~54]{Milnor1974}:  
\[
 \overline{w}\big(\xi_{\R^d,k}\times\xi_{\R^d,\ell}^{\oplus (d+1)}\big)= \overline{w}(\xi_{\R^d,k})\times\overline{w}\big(\xi_{\R^d,\ell}^{\oplus (d+1)}\big).
\]
Hence, for fixed $r\geq 0$ we have that
\begin{multline*}
\overline{w}_{r}\big(\xi_{\R^d,k}\times \xi_{\R^d,\ell}^{\oplus (d+1)}\big)=
\sum_{i+j=r}
\overline{w}_{i}(\xi_{\R^d,k})\times \overline{w}_{j}\big(\xi_{\R^d,\ell}^{\oplus (d+1)}\big) \\
\ \in \
H^r(\conf(\R^d,k)/\Sym_k\times\conf(\R^d,\ell)/\Sym_{\ell};\F_2 )
.	
\end{multline*}	
From the K\"unneth formula\index{K\"unneth formula} we get that each of the terms $\overline{w}_{i}(\xi_{\R^d,k})\times \overline{w}_{j}\big(\xi_{\R^d,\ell}^{\oplus (d+1)}\big)$ in the previous sum belongs to a different direct summand of the  cohomology
\begin{multline*}
H^r(\conf(\R^d,k)/\Sym_k\times\conf(\R^d,\ell)/\Sym_{\ell};\F_2 )\cong \\
\bigoplus_{i+j=r}H^i(\conf(\R^d,k)/\Sym_k ;\F_2 )\otimes H^j(\conf(\R^d,\ell)/\Sym_{\ell};\F_2 ).
\end{multline*}	
Therefore, the following equivalence holds:
\[
\overline{w}_{r}\big(\xi_{\R^d,k}\times \xi_{\R^d,\ell}^{\oplus (d+1)}\big)  \neq 0 \Longleftrightarrow 
\overline{w}_{i}(\xi_{\R^d,k})\times \overline{w}_{j}\big(\xi_{\R^d,\ell}^{\oplus (d+1)}\big) \neq 0
\; \text{for some} \; i+j=r.
\]
Now, from Corollary  \ref{cor 3 of Key Lemma} we know that there are integers $A$ and $B$ such that 
\[
(d-e-1)(k-\alpha(k))+e(\alpha(k)-\epsilon(k))\leq A\leq (d-1)(k-1),
\]
\[
(d-2e-1)(\ell-\alpha(\ell))\leq B\leq (d-1)(\ell-1),
\]
and 
\[
\overline{w}_A(\xi_{\R^d,k})\neq 0
\qquad\text{and}\qquad
\overline{w}_B(\xi_{\R^d,\ell}^{\oplus d+1})\neq 0.
\]
Consequently, for $C=A+B$ we have that $\overline{w}_{A}(\xi_{\R^d,k})\times \overline{w}_{B}\big(\xi_{\R^d,\ell}^{\oplus (d+1)}\big) \neq 0$ and so 
$\overline{w}_C(\xi_{\R^d,k}\times \xi_{\R^d,\ell}^{\oplus d+1})\neq 0$.
Obviously, we have that
\begin{multline*}
 (d-e-1)(k-\alpha(k))+e(\alpha(k)-\epsilon(k))+(d-2e-1)(\ell-\alpha(\ell))\leq \\
  C\leq (d-1)(k+\ell-2),	
\end{multline*}
and we have completed the proof of the corollary.
\end{proof}

Finally, using the criteria in Lemma \ref{lemma:criterion}, Lemma \ref{lem:dual:Whitney_skew} and Lemma \ref{lem:dual:Whitney_regular_skew} in combination with Corollary \ref{cor 3 of Key Lemma} and Corollary \ref{cor 4 of Key Lemma} we get the strongest lower bounds for the existence of highly regular embeddings.

\begin{theorem}
\label{th : additional bounds}
Let $d\geq 2$, $k\geq 1$ and $\ell\geq 1$ be integers, and let $d=2^t+e$ for some integers $t\geq 1$ and $0\leq e\leq 2^t-1$.	
Then
\begin{compactenum}[\rm \ (1)]

\item\label{additional-bounds-1} there is no $k$-regular embedding\index{$k$-regular embedding} $\R^d\longrightarrow\R^N$ if
\[
N\leq  (d-e-1)(k-\alpha(k))+e(\alpha(k)-\epsilon(k))+k-1,
\]
\item\label{additional-bounds-2} there is no $\ell$-skew embedding\index{$\ell$-skew embedding} $\R^d\longrightarrow\R^N$ if
\[
N\leq (d-2e-1)(\ell-\alpha(\ell))+(d+1)\ell-2,
\]
\item\label{additional-bounds-3} there is no $k$-regular-$\ell$-skew embedding\index{$k$-regular-$\ell$-skew embedding} $\R^d\longrightarrow\R^N$ if
\begin{multline*}
N\leq (d-e-1)(k-\alpha(k))+e(\alpha(k)-\epsilon(k))+\\
(d-2e-1)(\ell-\alpha(\ell))+(d+1)\ell+k-2.	
\end{multline*}	
\end{compactenum}
\end{theorem}

\medskip
In order to explain the strength of the previous theorem we demonstrate that Theorem \ref{th : Correctionog T 2.1}, Theorem \ref{th : Correctionog T 3.1} and Theorem 	\ref{thm : Correction of T.4.1} are consequences of Theorem \ref{th : additional bounds}.

\begin{corollary}\label{cor:better1}
Let $d\geq 2$ and $k\geq 1$ be integers, and let $d=2^t+e$ for some integers $t\geq 1$ and $0\leq e\leq 2^t-1$.	
Then	
\begin{center}{\normalfont
Theorem \ref{th : additional bounds}\eqref{additional-bounds-1}
\qquad$\Longrightarrow$\qquad
Theorem \ref{th : Correctionog T 2.1}. 	}
\end{center}
\end{corollary}
\begin{proof}
We check that Theorem \ref{th : additional bounds}\eqref{additional-bounds-1} implies all three cases of Theorem \ref{th : Correctionog T 2.1} independently.

\medskip
{\bf (1) }
In Theorem \ref{th : Correctionog T 2.1}\eqref{th : Correctionog T 2.1-1} it is stated that for $d$ being a power of $2$ there is no $k$-regular embedding\index{$k$-regular embedding} $\R^d\longrightarrow\R^N$ when $N\leq d(k-\alpha(k))+\alpha(k)-1$.
The same results follows from  Theorem \ref{th : additional bounds}\eqref{additional-bounds-1} by taking $e=0$ and observing that in this case
\begin{multline*}
(d-e-1)(k-\alpha(k))+e(\alpha(k)-\epsilon(k))+k-1=	(d-1)(k-\alpha(k))+k-1 =\\
(d-1)(k-\alpha(k))+k-\alpha(k)+\alpha(k)-1 = d(k-\alpha(k))+\alpha(k)-1 .
\end{multline*}	

\medskip
{\bf (2) }
Next, in Theorem \ref{th : Correctionog T 2.1}\eqref{th : Correctionog T 2.1-2} it is claimed that for $d$ being not a power of $2$ there is no $k$-regular embedding $\R^d\longrightarrow\R^N$ when $N\leq  \frac12(d-1)(k-\epsilon(k))+k-1$.
Since $d$ is not a power of $2$ we have that $e>0$.
Set that $k=2k'+\epsilon(k)$.
Hence, $\alpha(k)=\alpha(k')+\epsilon(k)$.
Now, the claim of Theorem \ref{th : Correctionog T 2.1}\eqref{th : Correctionog T 2.1-2} follows from Theorem \ref{th : additional bounds}\eqref{additional-bounds-1} since the following difference is non-negative:
\begin{multline*}
\big((d-e-1)(k-\alpha (k))+e(\alpha (k)-\epsilon (k))+k-1\big)
-\big(\frac12(d-1)(k-\epsilon (k))+k-1\big)=\\
(2^t-1-\frac12(2^t+e-1))k-(2^t-1-e)\alpha (k)+\big(\frac12(2^t+e-1)-e\big)\epsilon (k)=\\
\frac12(2^t-e-1)k-(2^t-1-e)\alpha (k)+\frac12\big(2^t-e-1\big)\epsilon (k)=\\
\frac12\big(2^t-e-1\big)(k-2\alpha (k)+\epsilon (k))=
\big(2^t-e-1\big)(k'-\alpha (k'))\geq 0.
\end{multline*}

\medskip
{\bf (3)}
Finally, Theorem \ref{th : Correctionog T 2.1}\eqref{th : Correctionog T 2.1-3} says that for $d$ being even and not a power of $2$ there is no $k$-regular embedding $\R^d\longrightarrow\R^N$ when $N\leq  \frac12d(k-\epsilon(k))+k-\alpha(k)+\epsilon(k) -1$.
We see that this result is also a consequence of Theorem \ref{th : additional bounds}\eqref{additional-bounds-1} by showing that following difference is non-negative:
\begin{multline*}
\big((d-e-1)(k-\alpha (k))+e(\alpha (k)-\epsilon (k))+k-1\big)-
\big(\frac12d(k-\epsilon (k))+k-\alpha (k)+\epsilon (k)-1\big)=\qquad\\
(2^t-1)k-(2^t-1)\alpha (k)+e\alpha (k)-e\epsilon (k)
-\frac12\big(2^t+e\big)k+\frac12\big(2^t+e\big)\epsilon (k)+\alpha (k)-\epsilon (k)=\\
\frac12\big(2^t-e-2\big)k-(2^t-2-e)\alpha (k)+\frac12\big(2^t-e-2\big)\epsilon (k)=\\
\frac12\big(2^t-e-2\big)(k+\epsilon (k)-2\alpha (k))=
\big(2^t-e-2\big)(k'-\alpha (k'))\geq 0.
\end{multline*}
Here we have that $e\leq 2^t-2$ and $\epsilon (k)=0$, because $e$ is even.
As in the previous computation we set $k=2k'+\epsilon(k)=2k'$, and so $\alpha(k)=\alpha(k')+\epsilon(k)=\alpha(k')$.
\end{proof}

\begin{corollary}\label{cor:better2}
Let $d\geq 2$ and $\ell\geq 1$ be integers, and let $d=2^t+e$ for some integers $t\geq 1$ and $0\leq e\leq 2^t-1$.	
Then	
\begin{center}
{\normalfont
Theorem \ref{th : additional bounds}\eqref{additional-bounds-2}
\qquad$\Longrightarrow$\qquad
Theorem \ref{th : Correctionog T 3.1}.}
\end{center}
\end{corollary}
\begin{proof}
We discuss each part of Theorem \ref{th : Correctionog T 3.1} separately.	

\medskip
{\bf (1)}
In Theorem \ref{th : Correctionog T 3.1}\eqref{th : Correctionog T 3.1-1} it is stated that for $d=2$ and $\ell\geq 2$  there is no $\ell$-skew embedding\index{$\ell$-skew embedding} $\R^2\longrightarrow\R^N$ when $N\leq  4\ell-\alpha(\ell)-2$.
The same conclusion follows from Theorem \ref{th : additional bounds}\eqref{additional-bounds-2} because in this case $d=2^t+e=2 \Leftrightarrow t=1,e=0$, and the following difference vanishes: 
\begin{multline*}
	\big((d-2e-1)(\ell-\alpha(\ell))+(d+1)\ell-2\big) - \big(4\ell-\alpha(\ell)-2\big)=\\
(\ell-\alpha(\ell)+3\ell-2)-(4\ell-\alpha(\ell)-2)=0.
\end{multline*}

\medskip
{\bf (2)}
In Theorem \ref{th : Correctionog T 3.1}\eqref{th : Correctionog T 3.1-2} we have claimed that for $d\geq 1$ and $\ell=2$ there is no $2$-skew embedding $\R^d\longrightarrow\R^N$ when $N\leq  2^{\gamma(d)}+d-1$.
Since we have assumed that $d=2^t+e$ and $0\leq e\leq 2^t-1$ note that $2^{\gamma(d)}=2^{t+1}$. 
Again, the case result can be deduced from Theorem \ref{th : additional bounds}\eqref{additional-bounds-2} because the following difference vanishes:
\begin{multline*}
	\big((d-2e-1)(\ell-\alpha(\ell))+(d+1)\ell-2\big) - \big(  2^{\gamma(d)}+d-1\big)=\\
	d-2e-1+2d+2-2- 2^{\gamma(d)}-d-1=2d- 2^{\gamma(d)}-2e=\\
	2^{t+1}+2e-2^{t+1}-2e=0.
\end{multline*}

\medskip
{\bf (3)}
 In Theorem \ref{th : Correctionog T 3.1}\eqref{th : Correctionog T 3.1-3} we considered the case when $d=2^t+0\geq 2$ is a power of $2$ and $\ell\geq 2$.
We proved that there is no $\ell$-skew embedding $\R^d\longrightarrow\R^N$ when $N\leq  2d\ell-(d-1)\alpha(\ell)-2$.
To see that Theorem \ref{th : additional bounds}\eqref{additional-bounds-2} implies this  result we show that the following difference vanishes:
\begin{multline*}
\big((d-2e-1)(\ell-\alpha(\ell))+(d+1)\ell-2\big) - \big( 2d\ell-(d-1)\alpha(\ell)-2 \big)=\\
(d-1)\ell-(d-1)\alpha(\ell)+(d+1)\ell-2d\ell+(d-1)\alpha(\ell)=0.
\end{multline*}

 \medskip
{\bf (4)}
In Theorem \ref{th : Correctionog T 3.1}\eqref{th : Correctionog T 3.1-4} we analysed  the case when $d+1\geq 2$ is a power of~$2$ and  $\ell\geq 2$.
We showed that the method based of the computation of dual Stiefel--Whitney class\index{Stiefel--Whitney classes} does not produce any non-trivial result about the existence of $\ell$-skew embeddings $\R^d\longrightarrow\R^N$.
In this situation also Theorem \ref{th : additional bounds}\eqref{additional-bounds-2} does not give any relevant result.
  
 \medskip
{\bf (5)}
In Theorem \ref{th : Correctionog T 3.1}\eqref{th : Correctionog T 3.1-6} we studied the case where both $d\geq 5$ and $d+1$ are not powers of $2$ and $\ell\geq 3$.
We showed that there cannot be $\ell$-skew embedding $\R^d\longrightarrow\R^N$ for $N\leq  \frac{1}{2}(2^{\gamma(d)}-d-1)(\ell-\epsilon(\ell)) +(d+1)\ell-2$.
In order to see that Theorem \ref{th : additional bounds}\eqref{additional-bounds-2} implies this result we show that the following difference is non-negative:
\begin{multline*}
\big((d-2e-1)(\ell-\alpha(\ell))+(d+1)\ell-2\big) - \big( \frac{1}{2}(2^{\gamma(d)}-d-1)(\ell-\epsilon(\ell)) +(d+1)\ell-2  \big)=\\
(2^t-e-1)\Big(\ell-\alpha(\ell)-\frac{1}{2}\big(2^{t+1}-2^t-e-1\big)(\ell-\epsilon(\ell)\Big)=\\
\frac{1}{2}\big(2^t-e-1\big)(\ell-\alpha(\ell)+\epsilon(\ell))\geq 0.
\end{multline*}

 \medskip
{\bf (6)}
In the last case, Theorem \ref{th : Correctionog T 3.1}\eqref{th : Correctionog T 3.1-7}, we have that $d\geq 5$ and $d+1$ are not powers of two, $2^{\gamma(d)}-d-1=2^{a_1}z$ where $a_1\geq 0$  and $z\geq 1$ is  odd, and $\ell\geq 3$.
We proved that in this case there is no $\ell$-skew embedding $\R^d\longrightarrow\R^N$ for
\[
	N\leq  \frac{1}{2}\big(2^{\gamma(d)}-d-1+2^{a_1}\big)(\ell-\epsilon(\ell) )-2^{a_1}\alpha(\ell)+(d+1)\ell-2.
\] 
In order to see that Theorem \ref{th : additional bounds}\eqref{additional-bounds-2} implies this result we first note that in this case 
\[
d=2^t+e,\quad
1\leq e\leq 2^t-2,\quad
2^{a_1}z=2^{\gamma(d)}-d-1=2^t-e-1.
\]
Now, as in the previous situations we consider the following difference and show that it is non-negative:
\begin{multline*}
\big((d-2e-1)(\ell-\alpha(\ell))+(d+1)\ell-2\big) -\\
\Big(
 \frac{1}{2}\big(2^{\gamma(d)}-d-1+2^{a_1}\big)(\ell-\epsilon(\ell) )-2^{a_1}\alpha(\ell)+(d+1)\ell-2
\Big)=
\end{multline*}
\begin{multline*}
(2^t-e-1)(\ell-\alpha(\ell)) \\- \frac{1}{2}\big(2^t-e-1\big)(\ell-\epsilon(\ell))-2^{a_1-1}(\ell-\epsilon(\ell))+2^{a_1}\alpha(\ell)=
\end{multline*}
\begin{multline*}
2^{a_1-1}z(\ell-2\alpha(\ell)+\epsilon(\ell) )-2^{a_1-1}(\ell-2\alpha(\ell)-\epsilon(\ell))=\\
2^{a_1-1}\big((z-1)(\ell-2\alpha(\ell))+(z+1)\epsilon(\ell)\big)\geq 0.
\end{multline*}
This concludes the proof of the corollary.
\end{proof}
  
In the same way as previous two corollaries we can show the following.
 
\begin{corollary}\label{cor:better3}
Let $d\geq 2$, $k\geq 1$ and $\ell\geq 1$ be integers, and let $d=2^t+e$ for some integers $t\geq 1$ and $0\leq e\leq 2^t-1$.	
Then	
\begin{center}
{\normalfont
Theorem \ref{th : additional bounds}\eqref{additional-bounds-3}
\qquad$\Longrightarrow$\qquad
Theorem \ref{thm : Correction of T.4.1}.}
\end{center}
 \end{corollary}

 %------------------------------------------------------
\subsection{Additional bounds for the existence of complex highly regular embeddings}
%------------------------------------------------------
 In this section we derive additional consequence of Lemma \ref{lem : key} which yields further bounds for the existence of complex highly regular embeddings.
 In particular, we will improve bounds given in Theorem \ref{theorem_complex_regular} and Theorem \ref{theorem:Main-02}.
 
  \medskip
Let us denote by $\lambda_{d,m}^{\C}:=\C\otimes\lambda_{d,m}$ and $\psi^{\C}_{d,m}[r_1^{\C},\dots,r_m^{\C}]:=\C\otimes\psi_{d,m}$ complex versions of the real vector bundles\index{vector bundle} $\lambda_{d,m}$ and $\psi_{d,m}[r_1,\dots,r_m]$.
Note that there is an isomorphism of real vector bundles $\lambda_{d,m}^{\C}\cong \lambda_{d,m}^{\oplus 2}$ and 
\[
\psi^{\C}_{d,m}[r_1^{\C},\dots,r_m^{\C}]\cong \psi_{d,m}[2r_1^{\C},\dots,2r_m^{\C}].
\]  
The following lemma is a consequence of Lemma \ref{lem : key}.

\begin{lemma}
\label{lem : key -- complex}
Let $m\geq 1$, $d^{\C}\geq 1$ and ${\ell}^{\C}\geq 1$ be integers, and let $r_1^{\C},\dots,r_m^{\C}$ be non-negative integers.
If the binomial coefficient
\begin{equation}
	\label{condition-binomial-complex}
	{(r_1^{\C}2^{m-2}+\dots+r_{m-1}^{\C}2^0)+r_m^{\C}-(d^{\C}-1+\ell^{\C})(2^{m-1}-1)-\ell^{\C}} \choose {d^{\C}-1}
\end{equation}
is odd, then
\[
w_{2(d^{\C}-1)(2^m-1)}((\lambda_{2d^{\C}-1,m}^{\C})^{-\ell^{\C}}\oplus \psi_{2d^{\C}-1,m}^{\C}[r_1^{\C},\dots,r_m^{\C}])\neq 0,
\]
as an element of the cohomology group $H^{2(d^{\C}-1)(2^m-1)}(\Sp(\R^d,2^m)/\Sy_{2^m};\F_2)$.
Here $(\lambda_{2d^{\C}-1,m}^{\C})^{-\ell^{\C}}$ denotes an inverse of the vector bundle $(\lambda_{2d^{\C}-1,m}^{\C})^{\oplus\ell}$. 
\end{lemma}
\begin{proof}
There is isomorphism of real vector bundles
\[
(\lambda_{2d^{\C}-1,m}^{\C})^{-\ell^{\C}}\oplus \psi_{2d^{\C}-1,m}^{\C}[r_1^{\C},\dots,r_m^{\C}]\cong 
\lambda_{2d^{\C}-1,m}^{-2\ell^{\C}}\oplus \psi_{2d^{\C}-1,m}[2r_1^{\C},\dots,2r_m^{\C}].
\]
Thus, it suffices to prove that 
\[
w_{2(d^{\C}-1)(2^m-1)}(\lambda_{2d^{\C}-1,m}^{-2\ell^{\C}}\oplus \psi_{2d^{\C}-1,m}[2r_1^{\C},\dots,2r_m^{\C}])\neq 0.
\]

\medskip
If we set that $d=2d^{\C}-1$, $\ell=2\ell^{\C}$, and $r_1:=2r_1^{\C},\dots,r_1:=2r_m^{\C}$, then by Lemma \ref{lem : key} we have that the Stiefel--Whitney class\index{Stiefel--Whitney classes}
\begin{multline*}
w_{2(d^{\C}-1)(2^m-1)}(\lambda_{2d^{\C}-1,m}^{-2\ell^{\C}}\oplus \psi_{2d^{\C}-1,m}[2r_1^{\C},\dots,2r_m^{\C}])=\\
w_{(d-1)(2^m-1)}(\lambda_{d,m}^{-\ell}\oplus \psi_{d,m}[r_1,\dots,r_m])\neq 0
\end{multline*} 
does not vanish if and only if the binomial coefficient
\begin{multline*}
{	{(r_12^{m-2}+\dots+r_{m-1}2^0)+r_m-(d-1+\ell)(2^{m-1}-1)-\ell} \choose {d-1}}=\\
	{{2\big((r_1^{\C}2^{m-2}+\dots+r_{m-1}^{\C}2^0)+r_m^{\C}-(d^{\C}-1+\ell^{\C})(2^{m-1}-1)-\ell^{\C}\big)} \choose {2(d^{\C}-1)}}
\end{multline*}
is odd.
Since the binomial coefficient $\binom{2a}{2b}$ is odd if and only if the binomial coefficient $\binom{a}{b}$ is odd, then the assumption \eqref{condition-binomial-complex} implies that the binomial coefficient
\[
{{2\big((r_1^{\C}2^{m-2}+\dots+r_{m-1}^{\C}2^0)+r_m^{\C}-(d^{\C}-1+\ell^{\C})(2^{m-1}-1)-\ell^{\C}\big)} \choose {2(d^{\C}-1)}}
\]
is odd.
Hence, 
$w_{2(d^{\C}-1)(2^m-1)}((\lambda_{2d^{\C}-1,m}^{\C})^{-\ell^{\C}}\oplus \psi_{2d^{\C}-1,m}^{\C}[r_1^{\C},\dots,r_m^{\C}])\neq 0.$
\end{proof}

 Let us denote by $\mu_{d,m}^{\C}[j]:=\C\otimes\mu_{d,m}[j]$, for all $1\leq j\leq m$, the  complex version of the real vector bundle\index{vector bundle} $\mu_{d,m}[j]$.
There is an isomorphism of real vector bundles $\mu_{d,m}^{\C}[j]\cong \mu_{d,m}[j]^{\oplus 2}$.
Now, like in the case of Lemma \ref{when Key Lemma can be applied} and Corollary \ref{cor 1 of Key Lemma}, we deduce the following consequence of Lemma \ref{lem : key -- complex}.

\begin{corollary}
\label{cor 1 of complex Key Lemma}
Let $d^{\C}\geq 2$ and $m\geq 1$ be integers, and let $d^{\C}=2^t+e$ for some integer $t\geq 1$ and $0\leq e \leq 2^t-1$.
If $d=2d^\C-1$, then
\begin{compactenum}[\rm \ (1)]
\item $w_{2(d^\C-1)(2^m-1)}\big((\lambda^\C_{d,m})^{-1}\oplus (\mu^\C_{d,m}[2]\oplus\cdots\oplus\mu^\C_{d,m}[m])^{\oplus 2e}\big)\not=0$,\index{Stiefel--Whitney classes}
\item $w_{2(d^\C -1)(2^m-1)}((\lambda^\C_{d,m})^{-(d^\C+1)}\oplus (\mu^\C_{d,m}[1]\oplus\cdots\oplus\mu^\C_{d,m}[m])^{\oplus 2e})\not=0$.
\end{compactenum}
\end{corollary}
\begin{proof}
The proof is almost identical to the proof of Lemma \ref{when Key Lemma can be applied}. 
For the sake of completeness we present the proof in detail.

\medskip
{\bf (1)} Let $\ell^\C=1$, $r_1^\C=0$, and $r_2^\C=\dots=r_m^\C=2e$.
Then 
\begin{multline*}
(r_1^{\C}2^{m-2}+\dots+r_{m-1}^{\C}2^0)+r_m^{\C}-(d^{\C}-1+\ell^{\C})(2^{m-1}-1)-\ell^{\C}=\\
2e(2^{m-3}+\dots+2^0)+2e-d^\C(2^{m-1}-1)-1=
2^{m-1}(e-d^\C)+2e+d^\C-1=\\
2^{m-1}(e-d^\C)+d^\C-1=\\	
\begin{cases}
	d^\C-1-2^{m-1+t}, & m\geq 2,\\
	-1, & m=1.
\end{cases}
\end{multline*}
Since ${{-1} \choose {d^\C-1}}=_{\F_2}1\in\F_2$ and ${d^\C-1-2^{m-1+t}\choose d^\C-1}=_{\F_2}{d^\C-1\choose d^\C-1}=_{\F_2}1\in\F_2$, the assumption of Lemma \ref{lem : key -- complex} is satisfied for the chosen parameters.
Consequently, 
\[
w_{2(d^\C-1)(2^m-1)}\big((\lambda^\C_{d,m})^{-1}\oplus (\mu^\C_{d,m}[2]\oplus\cdots\oplus\mu^\C_{d,m}[m])^{\oplus 2e}\big)\not=0.
\]

\medskip
{\bf (2)} Let $\ell^\C=d+1$, $r_1^\C=0$, and $r_1^\C=\dots=r_m^\C=2e$.
Hence
\begin{multline*}
(r_1^{\C}2^{m-2}+\dots+r_{m-1}^{\C}2^0)+r_m^{\C}-(d^{\C}-1+\ell^{\C})(2^{m-1}-1)-\ell^{\C}=	\\
2e(2^{m-2}+\dots+2^0)+2e-2d^\C(2^{m-1}-1)-d^\C-1=\\
2^m(e-d^\C)+d^\C-1=d^\C-1-2^{m+t}.
\end{multline*}
Thus, ${d^\C-1-2^{m+t}\choose d^\C-1}=_{\F_2}1\in\F_2$, because $d^\C-1<2^{t+1}\leq 2^{t+m}$.
Consequently, the assumption of Lemma \ref{lem : key -- complex} is satisfied for the chosen parameters, and so
\[
w_{2(d^\C -1)(2^m-1)}((\lambda^\C_{d,m})^{-(d^\C+1)}\oplus (\mu^\C_{d,m}[1]\oplus\cdots\oplus\mu^\C_{d,m}[m])^{\oplus 2e})\not=0.
\]
\end{proof}

Just as in the proof of Corollary \ref{cor 3 of Key Lemma} we can derive the following estimate for non-vanishing of the relevant dual Stiefel--Whitney classes.

\begin{corollary}
\label{cor 3 of complex Key Lemma}	
Let $d^{\C}\geq 2$ and $m\geq 1$ be integers, and let $d^{\C}=2^t+e$ for some integer $t\geq 1$ and $0\leq e \leq 2^t-1$.
Set $d=2d^\C-1$.
Then there exist integers 
$a$ and $b$ with the property that
\[
(d^\C -1)(2^m-1)-2(2^{m-1}-1)e\leq a\leq (d^\C -1)(2^m-1),
\]
\[
(d^\C -1)(2^m-1)-2(2^m-1)e\leq b\leq (d^\C -1)(2^m-1),
\]
and in addition, the dual Stiefel--Whitney classes \index{Stiefel--Whitney classes}
\[
\overline{w}_{2a}(\lambda^\C_{d,m})\not=0
\qquad\text{and}\qquad
\overline{w}_{2b}((\lambda^\C_{d,m})^{d^\C +1})\not=0,
\]
do not vanish.
\end{corollary}

Now, using the criteria for the existence of complex $k$-regular embeddings and complex $\ell$-skew embeddings given in Lemma \ref{lem : complex criterion}, we can directly derive the following estimates --- a complex analogue of Theorem \ref{th : additional bounds}.

\begin{theorem}
\label{th : additional complex bounds}
Let $d^{\C}\geq 2$, $m\geq 1$, $k\geq 1$ and $\ell\geq 1$ be integers, and let $d^{\C}=2^t+e$ for some integer $t\geq 1$ and $0\leq e \leq 2^t-1$.
Set $d=2d^\C-1$.
Then 
\begin{compactenum}[\rm \ (i)]
\item there is no complex $k$-regular embedding\index{$k$-regular embedding} $\R^d\longrightarrow \C^N$ if
\[
N\leq (d^\C -1-e)(k-\alpha(k)) + e(\alpha_2(k)-\epsilon (k)) +k-1,
\]
\item there is no complex $\ell$-skew embedding\index{$\ell$-skew embedding} $\C^{d^\C}\longrightarrow\C^{N}$
if
\[
N\leq (d^\C -1-2e)(\ell-\alpha(\ell))+(d^\C +1)\ell-2.
\]
\end{compactenum}
\end{theorem}

In a similar way as in Corollary \ref{cor:better1}, Corollary \ref{cor:better2} and Corollary \ref{cor:better3} it can be verified that Theorem \ref{th : additional complex bounds} implies Theorem \ref{theorem_complex_regular} and Theorem \ref{theorem:Main-02}.

%%%%%%%%%%%%%%%%%%%%%%%%%%%%%%%%%%%%%%%%%%%%%%%%%%%%%%%%%%%%%%%%%%%%%%%%%%%%%%%%%%%%%
%%%%%%%%%%%%%%%%%%%%%%%%%%%%%%%%%%%%%%%%%%%%%%%%%%%%%%%%%%%%%%%%%%%%%%%%%%%%%%%%%%%%%	
\section{Appendix}
\label{sec : appendix}
%%%%%%%%%%%%%%%%%%%%%%%%%%%%%%%%%%%%%%%%%%%%%%%%%%%%%%%%%%%%%%%%%%%%%%%%%%%%%%%%%%%%%
%%%%%%%%%%%%%%%%%%%%%%%%%%%%%%%%%%%%%%%%%%%%%%%%%%%%%%%%%%%%%%%%%%%%%%%%%%%%%%%%%%%%%	

%--------------------
\subsection{Operads}
\label{sub : operads}
%--------------------

In this section we recall basic notions from the theory of operads that we use. 
We follow the framework developed by May in \cite{May1972}, with some slight modifications in the notation.

\medskip
Let $\Ctop$ be the category of compactly generated weak Hausdorff spaces\index{category of compactly generated weak Hausdorff spaces} with continuous maps as morphisms, and let $\Ctoppt$ denote the category of compactly generated weak Hausdorff spaces with non-degenerate base points\index{category of compactly generated weak Hausdorff spaces with non-degenerate base points} where the morphisms are base point preserving continuous maps. 
Furthermore, assume that all the products as well as function spaces are endowed with compactly generated topology.

%----------------------------------------------------
\subsubsection{Definition and basic example} 
%----------------------------------------------------
The notion of an operad and its first formal definition appeared in work of May in 1970's.
For more details consult the original publication \cite[Sec.\,1]{May1972}.

\begin{definition}
	\label{def : operad}
	An {\bf operad}\index{operad} $\OOO$ is given by a family of topological spaces in $\Ctop$
	\[
	\OOO:=\{\OOO(n):n\geq 0\}\qquad\text{where}\qquad \OOO(0):=\{\pt\}
	\]
	together with a family of continuous maps
	\[
	\mu\colon (\OOO(n_1)\times\cdots\times\OOO(n_k))\times \OOO(k) \longrightarrow \OOO(n)
	\]
	where $k, n_1,\ldots,n_k\geq 0$ and $n=n_1+\cdots+n_k$,  such that the following axioms are satisfied:
	\begin{compactenum}[ \ \rm (1)]
		\item \label{def : operad - multiplication}
		For every $a\in\OOO(k)$, $b_1\in\OOO(n_1),\ldots,b_k\in\OOO(n_k)$, $c_1\in\OOO(m_1),\ldots,c_n\in\OOO(m_n)$
		\[
		\mu(c_1,\ldots,c_n;\mu(b_1,\ldots,b_k;a))=\mu(d_1,\ldots,d_k;a)
		\]
		where for $1\leq i\leq k$
		\[
		d_i:=\begin{cases}
			\mu(c_{n_1+\cdots+n_{i-1}+1},\ldots,c_{n_1+\cdots+n_{i}};b_i), 	& n_i\neq 0\\
			\pt ,															& n_i=0.
		\end{cases}
		\]
		Here $n=n_1+\cdots+n_k$.
		\item \label{def : operad - existence of unit}
		There exists an element $\IIII\in\OOO(1)$ with the property that for every $a\in\OOO(n)$ and every $b\in\OOO(k)$
		\[
		\mu(a;\IIII)=a,\qquad\qquad\qquad \mu(\IIII,\ldots,\IIII;b)=b.
		\]
		\item \label{def : operad -action}
		Every space $\OOO(n)$ is endowed with a right action of the symmetric group $\Sym_n$ that fulfills
		\begin{eqnarray*}
		\mu(b_1,\ldots,b_k;a\cdot\pi) & = & \mu(b_{\pi^{-1}(1)},\ldots,b_{\pi^{-1}(k)};a)\cdot\pi_{n_1,\ldots,n_k}	,\\
		\mu(b_1\cdot\pi_1,\ldots,b_k\cdot\pi_k;a)& = &\mu(b_1,\ldots,b_k;a)\cdot(\pi_1,\ldots,\pi_k),
		\end{eqnarray*}
		where the permutation
		\begin{compactitem}[ \ ---]
		\item $\pi_{n_1,\ldots,n_k}\in\Sym_n$ is given by permuting $k$ blocks 
		 \[
		 (1,\ldots, n_1)(n_1+1,\ldots,n_1+n_2)\cdots(n_1+\cdots	n_{k-1}+1,\ldots,n_1+\cdots	n_{k})
		 \]
		 as the permutation $\pi\in\Sym_k$ permutes $(1,\ldots,k)$, and
		\item $(\pi_1,\ldots,\pi_k)\in\Sym_n$ denotes the image of $(\pi_1,\ldots,\pi_k)\in\Sym_{n_1}\times\cdots\times\Sym_{n_k}$ via the inclusion $\Sym_{n_1}\times\cdots\times\Sym_{n_k}\hookrightarrow \Sym_n$.
		\end{compactitem}
	\end{compactenum}
	If in addition each action of the symmetric group $\Sym_n$ on $\OOO(n)$ is free the operad\index{operad} $\OOO$ is called a {\bf  $\Sym$-free operad}\index{$\Sym$-free operad}.
	The map $\mu_{\OOO}:=\mu$ is called the {\em structural map} of the operad $\OOO$.
\end{definition}

\medskip
The defining requirement  on the action of the symmetric group\index{symmetric group} $\Sym_n$ on $\OOO(n)$ for $n\geq 0$,  Definition \ref{def : operad}\eqref{def : operad -action}, directly implies the following property of the structural map.

\begin{lemma}
\label{lemma : map mu is equivariant}
	Let $k\geq 1$ be an integer, and let $n_1,\ldots,n_k$ be integers such that
	\[
	n_1=\cdots=n_{i_1}< n_{i_1+1}=\cdots=n_{i_2}<\cdots < n_{i_{r-1}+1}=\cdots=n_{i_{r}}
	\]
	where $1\leq i_1<i_2<\cdots<i_{r-1}<i_{r}=k$.
	Then
	\begin{compactenum}[ \ \rm (1)]
	\item The product $(\OOO(n_1)\times\cdots\times\OOO(n_k))\times\OOO(k)$ is endowed with a 
		\begin{multline*}
		\Sy_{n_1,\ldots,n_k;k}= \\
		( \Sym_{n_{i_1}}^{i_1} \rtimes \Sym_{i_1})\times
		(\Sym_{n_{i_2}}^{i_2-i_1+1}\rtimes \Sym_{i_2-i_1+1} )\times\cdots\times
		(  \Sym_{n_{i_r}}^{i_r-i_{r-1}+1}\rtimes \Sym_{i_r-i_{r-1}+1})	
		\end{multline*}
	right action.
	\item The structural map 
	\[
	\mu\colon (\OOO(n_1)\times\cdots\times\OOO(n_k))\times \OOO(k)\longrightarrow \OOO(n),
	\]
	with $n:=n_1+\cdots+n_k$, is an equivariant map with respect to the natural inclusion map of the groups $\Sy_{n_1,\ldots,n_k;k}\hookrightarrow\Sym_n$.
	\end{compactenum}
\end{lemma}

\medskip
The first and the central example of an operad is the endomorphism operad associated with any topological space in $\Ctoppt$.
The importance of the endomorphism operad will be become apparent after the definition of an action of an operad on a topological space.

\begin{example}
\label{ex : endomorphism operad}
	Let $X$ be a topological space in $\Ctoppt$.
	The {\bf endomorphism operad}\index{endomorphism operad} $\End_X$ associated to $X$ is defined as follows:
	\begin{compactitem}[ \ ---]
		\item $\End_X(n):=\mor_{\Ctoppt}(X^n,X)$ where $n\geq 0$ and $X^0:=\pt$,
		\item $\mu(g_1,\ldots,g_k;f):=f\circ (g_1\times\cdots\times g_k)$ for $f\in\End_X(k)$, $g_1\in\End_X(n_1),\ldots,g_k\in\End_X(n_k)$, and 
		\item $(f\cdot\pi)(x_1,\ldots,x_k):=f(x_{\pi^{-1}(1)},\ldots,x_{\pi^{-1}(k)})$ where $f\in\End_X(k)$, $\pi\in\Sym_k$, and $(x_1,\ldots,x_k)\in X^k$.
	\end{compactitem}
	The endomorphism operad is not a  $\Sym$-free operad.
\end{example}

\medskip
In order to have a category of operads we introduce a notion of a morphism between operads.

\begin{definition}
	\label{def : morphism of operads}
	Let $\OOO$ and $\DD$ be operads.
	A {\bf morphism of operads}\index{morphism of operads} $\Phi\colon\OOO\longrightarrow\DD$ is a family of $\Sym_n$-equivariant maps $\Phi_n\colon\OOO(n)\longrightarrow\DD(n)$ such that the following diagram commute:
	\[
	\xymatrix{
	 (\OOO(n_1)\times\cdots\times\OOO(n_k))\times \OOO(k)\ar[rr]^-{\mu_{\OOO}}\ar[dd]^-{(\Phi_{n_1}\times\cdots\times\Phi_{n_k})\times \Phi_k} \ & &\ \OOO(n)\ar[dd]^-{\Phi_n}\\
	 & &\\
	 (\DD(n_1)\times\cdots\times\DD(n_k))\times \DD(k)\ar[rr]^-{\mu_{\DD}} \ & &\ \DD(n)
	}
	\]
	for every collection of integers $k,n_1,\ldots,n_k\geq 0$ where $n:=n_1+\cdots+n_k$.
	
	\medskip
	The {\bf category of operads}\index{category of operads}  $\Op$ consists of operads as objects and morphisms of operads as morphisms.
\end{definition}
%
%----------------------------------------------------
\subsubsection{$\OOO$-space} 
%----------------------------------------------------

An action of an operad on a topological with a base point is defined as follows.

\begin{definition}
	\label{def : action of an operad}
	Let $X$ be a space in $\Ctoppt$ and let $\OOO$ be an operad.
	An {\bf action of the operad}\index{action of operad} $\OOO$ on the space $X$ is a morphism of operads $\Theta\colon\OOO\longrightarrow\End_X$.
	The pairs $(X,\Theta)$ is called an {\bf $\OOO$-space}\index{$\OOO$-space}.
	
	\noindent
	A {\bf morphism of $\OOO$-spaces}\index{morphism of $\OOO$-spaces} $(X,\Theta)$ and $(X',\Theta')$ is a continuous based map $f\colon X\longrightarrow X'$ in $\Ctoppt$ such that for every $n\geq 0$ and every $a\in\OOO(n)$ the following diagram commutes
	\[
	\xymatrix{
	X^n\ar[rr]^-{\Theta_n(a)}\ar[d]^-{f^n} \  & & \ X \ar[d]^-{f}\\
	Y^n\ar[rr]^-{\Theta_n'(a)} \  & &\ Y.
	}
	\]
\end{definition}

\medskip
An equivalent definition of the action of an operad on a topological space is formulated in the following elementary lemma.

\begin{lemma}
	\label{lemma : equivalent definition of action of an operad}
	Let $X$ be a space in $\Ctoppt$, and let $\OOO$ be an operad\index{operad}.
	An action $\Theta\colon\OOO\longrightarrow\End_X$\index{endomorphism operad} of the operad $\OOO$ on $X$ determines and is determined by the family of continuous maps
	\[
	 \Theta_n\colon X^n\times\OOO(n)\longrightarrow X,
	\]
	where $\Theta_0\colon\pt\longrightarrow X$ is the inclusion of the base point, that satisfy the following properties:
	\begin{compactenum}[ \ \rm (1)]
		\item For every $k, n_1,\ldots,n_k\geq 0$ and $n:=n_1+\cdots+n_k$ the following diagram commutes 
		\[\hspace{-15pt}
		{\small
		\xymatrix@C=0.9em{
		X^n\times (\OOO(n_1)\times\cdots\times\OOO(n_k))\times \OOO(k) \ar[rrrr]^-{\id\times\mu}\ar[dd]^-{  s} \ & & & & \ X^n\times\OOO(n)\ar[drr]^-{\Theta_n} & &\\
		& & & & & &X\\
		\big(X^{n_1}\times\OOO(n_1) \big)\times\cdots\times\big( X^{n_k}\times\OOO(n_k)\big)\times \OOO(k)\ar[rrrr]^-{\Theta_{n_1}\times\cdots\times\Theta_{n_k}\times\id}\ & & & & \ X^k\times\OOO(k)\ar[urr]^-{\Theta_k} & &
		}}
		\]
		where 
		\begin{multline*}
		\qquad s\colon X^n \times  (\OOO(n_1)\times\cdots\times\OOO(n_k))\times\OOO(k) \longrightarrow \\ 	
		\big( X^{n_1}\times\OOO(n_1)\big)\times\cdots\times\big(X^{n_k}\times \OOO(n_k)\big)\times\OOO(k)
		\end{multline*}
		is the obvious shuffle homeomorphism.
		\item For every $x\in X$
		\[ \Theta_1(x;\IIII)=x.\]
		\item For every $a\in\OOO(n)$, $\pi\in\Sym_n$, and every $(x_1,\ldots,x_n)\in X^n$
		\[\Theta_n(x_1,\ldots,x_n;a\cdot\pi) = \Theta_n(x_{\pi^{-1}(1)},\ldots,x_{\pi^{-1}(n)};a). \]
	\end{compactenum}
	Furthermore, if $f\colon (X,\Theta)\longrightarrow (X',\Theta')$ is a morphism of $\OOO$-spaces, then for every $n\geq 0$ the following diagram commutes
	\[
	\xymatrix{
	 X^n\times\OOO(n)\ar[rr]^-{\Theta_n}\ar[d]^-{f^n\times\id } \ & & \ X\ar[d]^-{f}\\
	(X')^n\times \OOO(n)\ar[rr]^-{\Theta'_n} \ & &\ X'.
	}
	\]
\end{lemma}
%
%----------------------------------------------------
\subsubsection{Little cubes operad\index{little cubes operad}}
\label{subsub : little cube operad}
%----------------------------------------------------

In this section we give an example of an operad\index{operad} whose structural map is studied in the central part of this {work}.

\medskip
Let $I:=[0,1]\subseteq\R$ be the unit interval, and then $I^d\subseteq\R^d$ is the associated $d$-cube.
Denote by $p:=(\tfrac12,\ldots,\tfrac12)\in I^d$ the centre of the $d$-cube. 
A little $d$-cube is simply an embedding of a cube into $I^d$ in such a way that then corresponding edges are parallel, see illustration in Figure \ref{fig : little cubes}.

\begin{definition}
\label{def : little cube}
Let $d\geq 1$ be an integer.
A {\bf little $d$-cube}\index{little $d$-cube}  is an orientation preserving affine embedding $\vec{c}\colon I^d\longrightarrow I^d$ of a $d$-cube that can be presented as the product map $\vec{c}=c_{1}\times\cdots\times c_{d}$ where each $c_{i}\colon I\longrightarrow I$, $1\leq i\leq d$, is an orientation preserving affine embedding. 
\end{definition}

\begin{figure}
\centering
\includegraphics[scale=0.65]{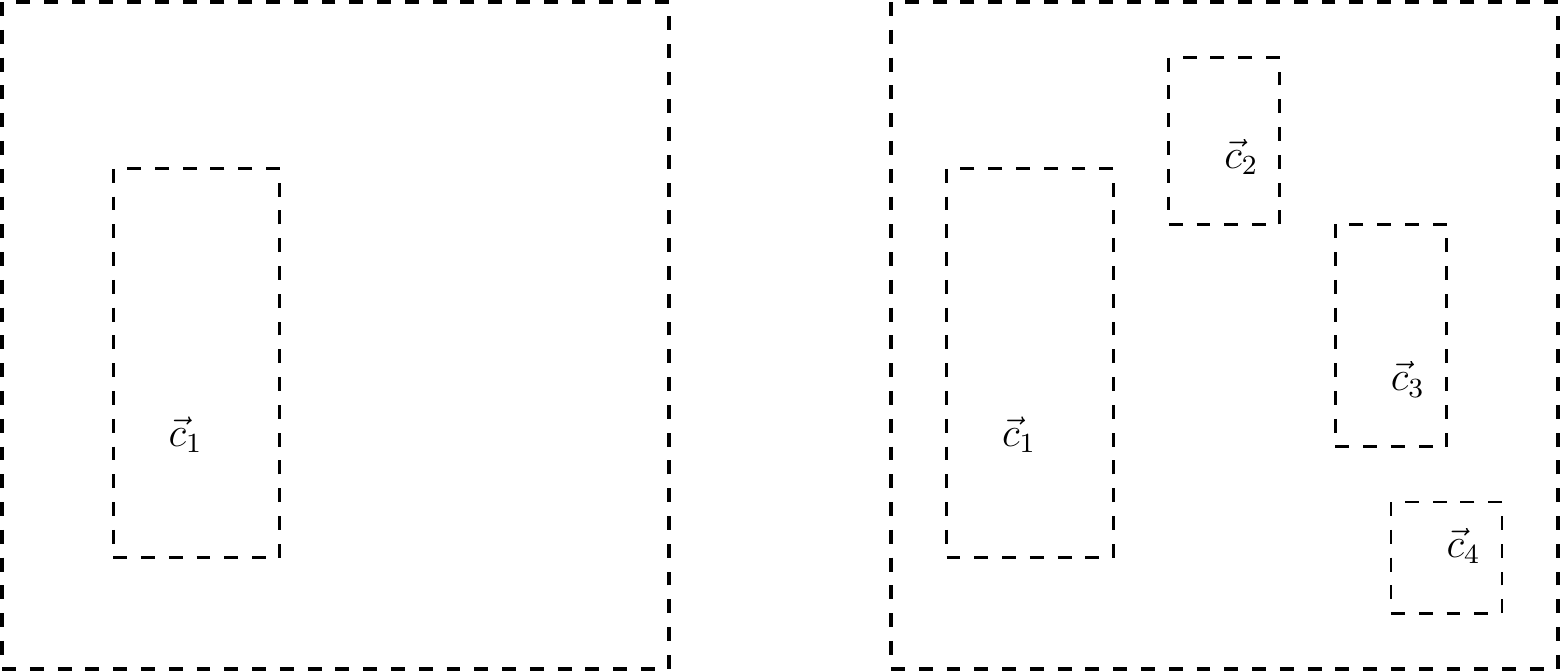}
\caption{\small Illustration of a $2$-cube and an element of the little $2$-cube operad.}
\label{fig : little cubes}
\end{figure}

\medskip
Now, the little cubes operad is defined as follows.  

\begin{definition}
\label{def : little cube operad}
Let $d\geq 1$ be an integer.
The {\bf little $d$-cubes operad}\index{little cubes operad} $\CC_d$, for $n\geq 1$, is defined by the family of topological spaces of all $n$-tuples of little $d$-cubes whose interiors are pairwise disjoint, that is
\[
\CC_d(n):=\{\alpha:=(\vec{c}_1,\ldots, \vec{c}_n) : 
\vec{c}_i(\inter I^d)\cap  \vec{c}_j(\inter I^d)=\emptyset\text{ for }i\neq j\}.
\]
The space $\CC_d(0)$ is the point, containing the unique ``embedding'' of the empty set into the cubde $I^d$.
The $n$-tuple $\alpha=(\vec{c}_1,\ldots, \vec{c}_n)$ of little $d$-cubes can be seen also as a continuous map $\alpha\colon \coprod_{j=1}^n I^d\longrightarrow I^d$, where ``$\coprod$'' denotes the disjoint union.
The space $\CC_d(0)$ is assumed to be a point interpreted as the unique embedding of the empty set into $I^d$.
The space $\CC_d(n)$ is equipped with the subspace topology induced from the space of all continuous maps $\coprod_{j=1}^n I^d\longrightarrow I^d$.

\medskip
The remaining necessary ingredients for the definition the little $d$-cube operad are given as follows.
\begin{compactenum}[ \ \rm (1)]
\item The structural map of the little $d$-cube operad
\[
\mu\colon (\CC_d(n_1)\times\cdots\times\CC_d(n_k))\times  \CC_d(k)\longrightarrow \CC_d(n),
\]
$n=n_1+\cdots+n_k$, is defined as a composition
\[
\alpha\circ (\beta_1\uplus\cdots\uplus\beta_k)\colon \Big(\coprod_{j=1}^{n_1} I^d\Big)\amalg \cdots\amalg \Big(\coprod_{j=1}^{n_k} I^d\Big)\longrightarrow 
\Big(\coprod_{j=1}^{k} I^d\Big)\longrightarrow I^d 
\]
where $\alpha\in\CC(k)$ and $\beta_1\in\CC_d(n_1),\ldots,\beta_k\in\CC_d(n_k)$.
\item The element $\IIII\in\CC_d(1)$ is the identity map $\id\colon I^d\longrightarrow I^d$.
\item The right action of the symmetric group $\Sym_n$ on the space $\CC_d(n)$ is given by
\[
(\vec{c}_1,\ldots, \vec{c}_n)\cdot\pi:= (\vec{c}_{\pi(1)},\ldots, \vec{c}_{\pi(n)}).
\] 
\end{compactenum}
\end{definition}

\noindent
The little $d$-cubes operad\index{little cubes operad} is an $\Sym$-free operad since for each $n$ the symmetric group $\Sym_n$ acts freely on $\CC_d(n)$.  

\medskip
There exists an $\Sym_n$-equivariant map of little $d$-cube operad space $\CC_d(n)$ into the configuration space $\conf(\R^d,n)$ given by evaluating each cube at the centre $p=(\tfrac12,\ldots,\tfrac12)\in I^d$:
\begin{equation}
	\label{eq : definition of ev map}
	\ev_{d,n}\colon\CC_d(n)\longrightarrow\conf(\R^d,n), \qquad 
(\vec{c}_1,\ldots, \vec{c}_n)\longmapsto (\vec{c}_1(p),\ldots, \vec{c}_n(p)).
\end{equation}
This map is an $\Sym_n$-equivariant homotopy equivalence, see {\cite[Thm.\,4.8]{May1972}}.

\begin{lemma}
	\label{lemma : little cube -- configuration space}
	For integers $d\geq 1$ and $n\geq 1$ the evaluation at centers of cubes map $\ev_{d,n}\colon\CC_d(n)\longrightarrow\conf(\R^d,n)$ is an $\Sym_n$-equivariant homotopy equivalence of the spaces $\CC_d(n)$ and $\conf(\R^d,n)$.
\end{lemma}
%
%----------------------------------------------------
\subsubsection{$\CC_d$-spaces, an example}
%----------------------------------------------------
In this section we give an example of a family of $\CC_d$-spaces\index{$\CC_d$-space}.
More precisely, for a pointed topological space $X$ we define an action of the little cubes operad  $\CC_d$ on its $d$-fold loop space $\Omega^dX$.
For more details consult for example \cite[Thm.\,5.1]{May1972}.

\medskip
In order to define a $\CC_d$-action on the $d$-fold loop space $\Omega^dX$ we use Lemma \ref{lemma : equivalent definition of action of an operad} and give a family of functions
\[
\Theta_{d,n} \colon (\Omega^dX)^n\times\CC_d(n) \longrightarrow \Omega^dX.
\]
For $\omega_1,\ldots,\omega_n\in\Omega^dX$  and $(\vec{c}_1,\ldots, \vec{c}_n)\in \CC_d(n)$ the loop
\[
\Theta_{d,n} (\omega_1,\ldots,\omega_n;\vec{c}_1,\ldots, \vec{c}_n)\colon I^d\longrightarrow X
\]
is defined, when $x\in I^d$, by
\[
\Theta_{d,n} (\omega_1,\ldots,\omega_n;\vec{c}_1,\ldots, \vec{c}_n)(x):=
\begin{cases}
	\omega_i(\vec{c}_i^{-1}(x)),  & \text{if } x\in \vec{c}_i(I^d)\text{ for some }0\leq i\leq n,\\
	 \pt ,	& \text{otherwise},
\end{cases}
\] 
where $\pt\in X$ is the base point. 
In the case $d=1$ the elements $\omega_1,\ldots,\omega_n$ are pointed loops and the the map $\Theta_{1,n}(\omega_1,\ldots,\omega_n;\vec{c}_1,\ldots, \vec{c}_n)$ is just a concatenation of loops $\omega_1,\ldots,\omega_n$ specified by the collection of, pairwise interior disjoint, intervals $(\vec{c}_1,\ldots, \vec{c}_n)\in \CC_1(n)$.
For an illustration see Figure \ref{fig:loop}.
By direct inspection it can be verified that just defined family of functions satisfies all the necessary conditions of Lemma \ref{lemma : equivalent definition of action of an operad}.
Thus, for a  space $X$ in $\Ctoppt$ the $d$-fold loop space $\Omega^dX$ is a $\CC_d$-space\index{$\CC_d$-space}.

\begin{figure}[ht]
\centering
\includegraphics[scale=1.1]{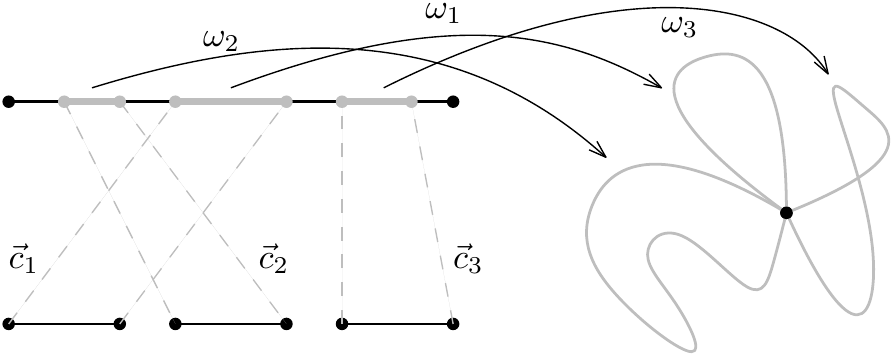}
\caption{\small An illustration of the loop $\Theta_{1,3}(\omega_1,\omega_2,\omega_1;\vec{c}_1,\vec{c}_2,\vec{c}_3)$.}
\label{fig:loop}
\end{figure}

%----------------------------------------------------
\subsubsection{$\CC_d$-spaces, a free $\CC_d$-space over $X$}
\label{subsub : free C_d space}
%----------------------------------------------------
Let $X$ be a pointed space with the base point $\pt\in X$, and let $d\geq 1$ be an integer.
Then we define the $\CC_d$-space\index{$\CC_d$-space} associated to $X$ to be the quotient space
\[
\CC_d(X):=\Big(\coprod_{m\geq 0}\CC_d(m)\times_{\Sym_m}X^m\Big)/_{\approx}
\]
where for $(\vec{c}_1,\ldots, \vec{c}_m)\in \CC_d(m)$ and $(x_1,\ldots,x_{m-1},\pt)\in X^m$ we define equivalence relation generated by
\[
((\vec{c}_1,\ldots, \vec{c}_{m-1}, \vec{c}_m), (x_1,\ldots,x_{m-1},\pt))\approx
((\vec{c}_1,\ldots, \vec{c}_{m-1}), (x_1,\ldots,x_{m-1})).
\]
The $\CC_d$-action on $\CC_d(X)$ is induced from the structural maps of the little cubes operad $\CC_d$. 
For more details on the definition of $\CC_d(X)$, as a monad associated to the operad $\CC_d$, see \cite[Constr.\,2.4]{May1972}.

\medskip
The $\CC_d$-space $\CC_d(X)$, associated to the pointed space $X$, can be called the {\bf free $\CC_d$-space generated by $X$}\index{free $\CC_d$-space}, because for every $\CC_d$-space $Y$ there exists a bijective correspondence 
\begin{multline}
	\label{correspondence-01}
	(\text{morphisms of }\CC_d\text{-space from }\CC_d(X)\text{ to }Y)
\longleftrightarrow \\
(\text{morphisms of topological space from }X\text{ to }Y).
\end{multline}
Consult \cite[Prop\,2.8 and Lem.\,2.9]{May1972}.

\medskip
Let $X\longrightarrow\Omega^d\Sigma^dX$ be the map associated to the identity map $\id\colon\Sigma^dX\longrightarrow\Sigma^dX$ along the adjunction relation 
\[
[A,\Omega^dB]_{\pt}\longleftrightarrow [\Sigma^dA,B]_{\pt}.
\]
Here $[A,B]_{\pt}$ denotes the set of all homotopy classes of pointed maps $A\longrightarrow B$ between the pointed spaces $A$ and $B$.
Next, let $\alpha_d\colon\CC_d(X)\longrightarrow\Omega^d\Sigma^dX$ denotes the morphism of $\CC_d$-space associated to the map $X\longrightarrow\Omega^d\Sigma^dX$ with respect to the correspondence \eqref{correspondence-01}.
Now the Approximation theorem\index{Approximation theorem} \cite[Thm.\,2.7 and Thm.\,6.1]{May1972} of May states the following.

\begin{theorem}
	\label{th : approximation}
	Let $d\geq 1$ be an integer, or let $d=\infty$.
	If $X$ is a path-connected space in $\Ctoppt$, then  
	\[
	\alpha_d\colon\CC_d(X)\longrightarrow\Omega^d\Sigma^dX
	\]
	is a weak homotopy equivalence.
\end{theorem}

%
%----------------------------------------------------
\subsubsection{Araki--Kudo--Dyer--Lashof homology operations}
\label{sub : Araki--Kudo--Dyer--Lashof operations}
%----------------------------------------------------

Following analogy with Dyer \& Lashof \cite[Def.\,2.2]{Dyer1962} we review basic properties of the Araki--Kudo--Dyer--Lashof homology operations as defined by Cohen \cite[Def.\,5.6]{Cohen1976LNM533}.
In the case $p=2$ the homology operations were first introduced by Araki \& Kudo \cite{Araki1956}.

\medskip
Let $d\geq 1$ be an integer, or let $d=\infty$.
Let $Y$ be a compactly generated weak Hausdorff space with non-degenerate base point endowed with an action of the little cubes operad $\CC_d$\index{little cubes operad}.
Then there exists a sequence of maps 
\[
Q_i\colon H_j(Y;\F_2)\longrightarrow H_{i+2j}(Y;\F_2), 
\qquad\qquad
0\leq i\leq d-1,
\]
called {\bf Araki--Kudo--Dyer--Lashof homology operations}\index{Araki--Kudo--Dyer--Lashof homology operations}.
Properties of these operations are listed in the next proposition, see also \cite[Thm.\,2.2, Cor.\,1]{Dyer1962} \cite[Sec.\,1]{Cohen1976LNM533}.
In the following $\lambda_{d-1}\colon H_i(Y;\F_2)\otimes H_j(Y;\F_2)\longrightarrow H_{i+j+d-1}(Y;\F_2)$ denotes the Browder operation, consult \cite{Browder1960} \cite[Thm.\,1.2]{Cohen1976LNM533}.

\begin{proposition}
\label{prop:DL-operations}
Let $Y$ and $Z$ be $\CC_d$-spaces\index{$\CC_d$-space}.
The following properties hold:
\begin{compactenum}[ \ \rm (1)]
\item $Q_0(y)=y^2$ for every $x\in H_*(Y;\F_2)$.
\item $Q_i$ is a homomorphism for every $0\leq i\leq d-2$.
\item $Q_{d-1}$ is not homomorphism in general and
		\[
		Q_{d-1}(y_1+y_2)=Q_{d-1}(y_1)+Q_{d-1}(y_2)+ \lambda_{d-1}(y_1,y_2)
		\]
	 	for $y_1,y_2\in H_*(Y;\F_2)$ with $\deg(y_1)=\deg(y_2)$.
\item If $Y$ is connected and $1\in H_0(Y;\F_2)$, then $Q_i(1)=0$ for $0<i\leq d-1$.
\item The operations $Q_i$ are natural with respect to the morphisms of $\CC_d$-spaces.
\item If  $y_r\in H_r(Y;\F_2)$,  $Z_s\in H_s(Z;\F_2)$ and $0\leq i\leq d-1$, then in $H_*(Y\times Z;\F_2)$ holds
 	       \[
 	      Q_{i}(y_r\otimes z_s)=\sum_{j=0}^{i}Q_{j}(y_r)\otimes Q_{i-j}(z _s)+\varepsilon_n (y_r, z_s),
 	      \]
 	      where $\varepsilon_n (y_r\otimes z_s)$ is the ``error term.''
 	      For example, in the case when $Y=Z=\Omega^d\Sigma^dS^L$, $d\geq 2$, and $L\geq 1$, the ``error term'' vanishes. 
\item If  $y_r\in H_r(Y;\F_2)$,  $y_s'\in H_s(Y;\F_2)$ and $0\leq i\leq d-1$, then in $H_*(Y;\F_2)$ holds	      
		  \[
	      Q_{i}(y_r\cdot y_s')=\sum_{j=0}^{i}Q_{j}(y_r)\cdot Q_{i-j}(y _s')+\varepsilon_n' (y_r, y_s'),
	      \]
	      where ``$\cdot$'' denotes the Pontryagin product and $\varepsilon_n' (y_r, y_s')$ stands for the error term.
	      For example, in the case when $Y=\Omega^d\Sigma^dS^L$, $d\geq 2$, and $L\geq 1$, the error term vanishes.
\end{compactenum}
\end{proposition}

\medskip
Assume that $Y$ is still a  $\CC_d$-space\index{$\CC_d$-space}.
Then the operations 
\[
Q_i\colon H_j(Y;\F_2)\longrightarrow H_{i+2j}(Y;\F_2)
\] 
are defined in \cite[Def.\,5.6]{Cohen1976LNM533} as follows.
The $\CC_d$-action on $Y$ yields a map 
\[
\Theta_2\colon \CC_d(2) \times (Y\times Y) \longrightarrow Y
\] 
that behaves naturally with respect to the acton of $\Sym_2\cong\Z_2$.
Consequently, it induces the quotient map (denoted in the same way)
 \[
 \Theta_2\colon \CC_d(2) \times_{\Z_2} (Y\times Y) \longrightarrow Y
 \] 
where the action on the product $Y\times Y$ is given by interchanging factors.
The space $\CC_d(2)$ is $\Z_2$ equivariantly homotopic to the sphere $S^{d-1}$ equipped with the antipodal actions.
Hence, the cohomology and homology  of $\CC_d(2) \times_{\Z_2} (Y\times Y)$, with $\F_2$-coefficients, is completely described in Section \ref{sub : cohomology of wreath product}.
Using the notation from Theorem \ref{th : homology of wreath product} we define
\[
Q_i(y):= (\Theta_2)_* \big( (y\otimes y)\otimes_{\Z_2}f_i\big)
\]
for $y\in H_*(Y;\F_2)$ where $0\leq i\leq d-1$. 
%
%===========================
\subsection{The Dickson algebra}
%===========================
In this section we present all the facts about the Dickson algebra\index{Dickson algebra} that are relevant for the calculations in this paper.
For this we rely on the sources \cite[Sec.\,III.2]{Adem2004}, \cite{Campbell1996-02},  and \cite{Wilkerson1983}, where some of the results
we present appeared already in the original paper of Dickson \cite{Dickson1911}.

%----------------------------------------------------
\subsubsection{Rings of invariants}
%----------------------------------------------------
Let $m\geq 1$ be an integer, let $V$ be an $m$-dimensional vector space over the field $\F_2$, and let $\Salg(V)$ denotes the symmetric algebra of $V$ over $\F_2$.
Then for a choice of a basis $(x_1,\ldots,x_m)$ of $V$ there is an isomorphism $\Salg(V)\cong\F_2[x_1,\ldots,x_m]$.
The algebra $\Salg(V)$ is graded by setting $\deg(x_1)=\cdots=\deg(x_m)=1$.
The degree of a monomial is defined in the usual way by $\deg (x_1^{\alpha_1}\cdots x_m^{\alpha_m}):=\alpha_1+\cdots+\alpha_m$.
The set of all homogeneous polynomials of degree $n$ is denoted by $\Salg^n(V)$ where $\Salg^0(V)\cong\F_2$.
Consequently, $\Salg(V)\cong\bigoplus_{n\geq 0}\Salg^n(V)$.

\medskip
The general linear group $\GL(V)\cong\GL_m(\F_2)$ on the vector space $V$ acts from the left on $\Salg(V)$ and preserves the introduced grading. 
If $G$ is a subgroup of $\GL_m(\F_2)$, then $\Salg(V)^G$ denotes the ring of $G$-invariants.
In this section we will discuss the invariants only of the full general linear group $\GL_m(\F_2)$ and the subgroup $\U_m(\F_2)\subseteq\GL_m(\F_2)$ of all lower triangular matrices with $1$'s on the main diagonal.
It is a known fact that $\U_m(\F_2)$ is a Sylow $2$-subgroup\index{Sylow $2$-subgroup} of $\GL_m(\F_2)$.
In fact, the order of the group  $\U_m(\F_2)$ is $2^{\frac{m(m-1)}{2}}$ while its index in $\GL_m(\F_2)$ is $(2m-1)!!=1\cdot 3\cdot 5\cdots (2m-1)$.

\medskip
Let $(x_1,\ldots,x_m)$ be a fixed basis of the vector space $V$, and let us introduce a complete flag of subspaces in $V$ by
\[
\{0\}=V_0\subseteq V_1\subseteq\cdots\subseteq V_{m-1}\subseteq V_m=V
\]
where $V_i:=\spann\{x_{m},\ldots,x_{m-i+1}\}$, for $1\leq i\leq m$.
Thus the complete flag we consider is
\begin{multline*}
\{0\}\subseteq\spann\{x_m\}\subseteq\spann\{x_m,x_{m-1}\}\subseteq \\ \cdots\subseteq\spann\{x_m,x_{m-1},\dots,x_2\}\subseteq\spann\{x_m,x_{m-1},\dots,x_2,x_1\}.	
\end{multline*}
The rings of invariants of the polynomial ring $S(V)\cong\F_2[x_1,\ldots,x_m]$ with respect to the groups $\U_m(\F_2)$ and $\GL_m(\F_2)$ are denoted as follows:
\[
\mathfrak{H}_m:=\Salg(V)^{\U_m(\F_2)}\cong\F_2[x_1,\ldots,x_m]^{\U_m(\F_2)},
\]
and
\[
\mathfrak{D}_m:=\Salg(V)^{\GL_m(\F_2)}\cong\F_2[x_1,\ldots,x_m]^{\GL_m(\F_2)}.
\]

\medskip
Assuming the previously introduced notation a result of M\`{u}i \cite[Thm.\,6.4]{Mui1975} gives the following described the ring of invariants $\mathfrak{H}_m$, see also \cite[Thm.\,3.1]{Campbell1996-02}. 

\begin{theorem}
	\label{th : Mui ring of invariants of upper triangular matrices}
	Let $h_{i}:=\prod_{v\in V_{i-1}}(x_{m-i+1}+v)$ for $1\leq i\leq m$, a polynomial of degree $2^{i-1}$ in $\Salg(V)\cong\F_2[x_1,\ldots,x_m]$.
	Then
	\[
	\mathfrak{H}_m=\F_2[h_1,\ldots,h_m].
	\]
\end{theorem}

\noindent
For example, if $m=3$ then
\[
h_1=x_3,\qquad h_2=x_{2}(x_{2}+x_3), \qquad h_3=x_{1}(x_{1}+x_{2})(x_{1}+x_3)(x_{1}+x_{2}+x_3).
\] 

\medskip
For the ring of invariants $\mathfrak{D}_m$ we rely on a result of Dickson \cite{Dickson1911} and have the following presentation, see also \cite[Thm.\,1.2]{Wilkerson1983}.

\begin{theorem}
	\label{th : Dickson ring of invariants}
	Let $d_{m,0},\ldots,d_{m,m-1}, d_{m,m}$ be the polynomials in $\F_2[x_1,\ldots,x_m]$ defined as the coefficients of the polynomial
	\[
		f_m(T):=\prod_{v\in V} (T+v)=\sum_{i=0}^{m} d_{m,i}\,T^{2^i}
	\]
	in $\F_2[x_1,\ldots,x_m][T]$. 
	Then $\deg(d_{m,i})=2^m-2^i$ for $0\leq i\leq m$, $d_{m,m}=1$, and 
	\[
	\mathfrak{D}_m=\F_2[d_{m,0},\ldots,d_{m,m-1}].
	\]
\end{theorem}

\medskip
The polynomials $d_{m,0},\ldots,d_{m,m-1}$ are called the {\bf Dickson invariants}\index{Dickson invariants} of $\Salg(V)$ or of the polynomial ring $\F_2[x_1,\ldots,x_m]$.
For example, if $m=2$ then
\[
f_2(T)=T(T+x_1)(T+x_2)(T+x_1+x_2)=
x_1x_2(x_1+x_2)T+(x_1x_2+x_1^2+x_2^2)T^2.
\]
Hence,
\[
d_{2,0}=x_1x_2(x_1+x_2)
\qquad\text{and}\qquad
d_{2,1}=x_1x_2+x_1^2+x_2^2.
\]
On the other hand $h_1=x_2$ and $h_2=x_1(x_1+x_2)$, and consequently
\[
d_{2,0}=h_1h_2
\qquad\text{and}\qquad
d_{2,1}=1\cdot h_2+\chi_2h_1^2,
\]
where $\chi_2\in\GL_2(F_2)$ is the variable substitution given by the matrix 
$
\begin{pmatrix}
	0 & 1\\
	1 & 0
\end{pmatrix} 
$.

\medskip
The formula connecting generators of the rings of invariants $\mathfrak{H}_m$ and $\mathfrak{D}_m$ is as follows, see for more details \cite[Prop.\,1.3\,(b)]{Wilkerson1983}.

\begin{proposition}
\label{prop : Dickson invariants recursive formula}
For $0\leq i \leq m-1$ and setting $d_{m-1,-1}=0$ we have that
\begin{equation}
\label{relation - recurrence Dickson - 01}
d_{m,i}=(\chi_m d_{m-1,i})\,h_m + (\chi_m d_{m-1,i-1})^2,	
\end{equation}
where $\chi_m\in\GL_m(\F_2)$ is the change of variables given by the matrix
\[
\begin{pmatrix}
	0 & 0 & \cdots & 0 & 1\\
	0 & 0 & \cdots & 1 & 0\\
	  &   & \cdots &   &  \\
	1 & 0 & \cdots & 0 & 0
\end{pmatrix} .
\]
\end{proposition}

\medskip
If we set $k_j:=\chi_mh_j$ for $1\leq j\leq m$ then the relation \eqref{relation - recurrence Dickson - 01} becomes
\begin{equation}
\label{relation - recurrence Dickson - 02} 
d_{m,i}=d_{m-1,i}k_m +  d_{m-1,i-1}^2,
\end{equation}
where $0\leq i \leq m-1$.
It is not hard to see that $k_1,\ldots,k_m$ are generators of the ring of invariants with respect to the subgroup $\mathrm{U}_m(\F_2)$ of $\GL_m(\F_2)$ that consists of all upper triangular matrices with $1$'s on the main diagonal.
More precisely, 
\[
\Salg(V)^{\mathrm{U}_m(\F_2)}\cong\F_2[x_1,\ldots,x_m]^{\mathrm{U}_m(\F_2)}=\F_2[k_1,\ldots,k_m].
\]

\medskip
From the relation \eqref{relation - recurrence Dickson - 02} we can derive a presentation of each Dickson invariant\index{Dickson invariants} $d_{m,i}$ in term of $k_1,\ldots,k_m$.
Indeed, let $0\leq r\leq m-1$ be an integer, and let $J=(j_1,\ldots,j_r)\in [m]^r$ with $1\leq j_1<\cdots< j_r\leq m$. 
We denote particular monomials in $k_1,\ldots,k_n$ as follows:
\[
k[J]:= (k_1\cdots k_{j_1-1})^{2^r}\, (k_{j_1+1}\cdots k_{j_2-1})^{2^{r-1}}\cdots  (k_{j_r+1}\cdots k_{m})^{2^{0}}.
\]
Then as in \cite[Thm.\,4.3]{Campbell1996-02} we have that
\begin{eqnarray}
\label{relation - recurrence Dickson - 03}
d_{m,r} 
&=&\sum_{J\in [m]^{r}\, :\, j_1<\cdots< j_r}k[J]\\
&=&\sum_{J\in [m]^{r}\, :\, j_1<\cdots< j_r}
(k_1\cdots k_{j_1-1})^{2^r}\, (k_{j_1+1}\cdots k_{j_2-1})^{2^{r-1}}\cdots  (k_{j_r+1}\cdots k_{m})^{2^{0}}.\nonumber 	
\end{eqnarray}
In particular,
\begin{equation}
\label{relation - recurrence Dickson - 04}
d_{m,0}=k_1\cdots k_m
\qquad\text{and}\qquad 
d_{m,m-1}=k_1^{2^{m-1}}+k_2^{2^{m-2}}+\cdots+k_m^{2^{0}}.	
\end{equation}

%----------------------------------------------------
\subsubsection{The Dickson invariants as characteristic classes}
\label{subsub : Dickson invariants as characteristic classes}
%----------------------------------------------------

Let $m\geq1$ be an integer.
Consider the sequence of group embeddings
\[
\xymatrix{
\EE_m:=\Z_2^{\oplus m}\ar[r]^-{(\mathrm{reg})} & \Sym_{2^m} \ar[r]^-{\iota_{2^m}}& \OO(2^m),
}
\]
where 
\begin{compactitem}[\quad$\bullet$]
\item $(\mathrm{reg})\colon \Z_2^{\oplus m}\longrightarrow \Sym_{2^m}$ is the regular embedding
that is given  by the left translation action of $\Z_2^{\oplus m}$ on itself, consult \cite[Ex.\,III.2.7]{Adem2004}, and
\item $\iota_{2^m}\colon \Sym_{2^m}\longrightarrow \OO(2^m)$ 
is the embedding given by the permutation representation.	
%TODO : Check the name permutation representation
\end{compactitem}
This sequence of embeddings induces a sequence of maps of the corresponding classifying spaces
\[
\xymatrix{
\B \EE_m \ar[r]^-{\B(\mathrm{reg})} & \B\Sym_{2^m} \ar[r]^-{\B(\iota_{2^m})}& \B\OO(2^m).
}
\]

\medskip
Let $\gamma_{2^m}$ denotes the tautological vector bundle\index{vector bundle} over $\B\OO(2^m)$, and let us denote the pullbacks as follows
\[
\xi_{2^m}:=\B(\iota_{2^m})^*\gamma_{2^m}
\qquad\text{and}\qquad
\nu_{2^m}:= (\B(\iota_{2^m})\circ \B(\mathrm{reg}))^*\gamma_{2^m}.
\]
Then the bundles $\gamma_{2^m}$, $\xi_{2^m}$ and $\nu_{2^m}$ induce the following commutative diagram of vector bundle morphisms:
\[
\xymatrix@1{
\EEE\EE_m\times_{\EE_m} \R^{2^m}\ar[rr]\ar[d]^{\nu_{2^m}} \ &  & \ \EEE\Sym_{2^m} \times_{\Sym_{2^m}} \R^{2^m}\ar[rr]\ar[d]^{\xi_{2^m}} \ & &\ \EEE\OO (2^m)\times_{\OO(2^m)}\R^{2^m}\ar[d]^{\gamma_{2^m}}\\
%& & & &\\
\B \EE_m \ar[rr]^-{\B(\mathrm{reg})} \ & & \ \B\Sym_{2^m} \ar[rr]^-{\B(\iota_{2^m})} \ & &\ \B\OO(2^m).
}
\]

\medskip
Recall, that the cohomology ring of the elementary abelian group $\EE_m$ with coefficients in the field $\F_2$ is a polynomial ring on $m$ generators in degree $1$, that is $H^*(\EE_m;\F_2)\cong\F_2[y_1,\ldots,y_m]$ with $\deg(y_1)=\cdots=\deg(y_m)=1$.

\medskip
The relationship of the introduced bundles and the Dickson invariants\index{Dickson invariants} is given by the following theorem, see \cite[Lem.\,3.26]{Madsen1979}.

\begin{theorem}\label{thm : SW classes of regular rep. over e.a.g.}
	The total Stiefel--Whitney class\index{Stiefel--Whitney classes} of the vector bundle $\nu_{2^m}$ is
	\[
	w(\nu_{2^m})=1+d_{m,m-1}+d_{m,m-2}+\cdots+d_{m,0},
	\]
	where $d_{m,m-1},\ldots,d_{m,0}$ are the Dickson invariants\index{Dickson invariants} of $F_2[y_1,\ldots,y_m]\cong H^*(\EE_m;\F_2)$.
	This means that
	\[
	w_i(\nu_{2^m})=
	\begin{cases}
		d_{m,j}, & i=2^m-2^j, \ 0\leq j\leq m-1, \\
		   1 ,   & i=0,\\
		   0 ,   & \text{\rm  otherwise}.   
	\end{cases}
	\]
\end{theorem}

%===========================
\subsection{The Stiefel--Whitney classes of the wreath square of a vector bundle}
\label{sec : dyadic}
%===========================

In this section we introduce notions called the wreath square of a vector bundle \cite[Sec.\,3]{BaralicBlagojevicKarasecVucic2018} and the $(d-1)$-partial   wreath square of a vector bundle\index{vector bundle}.
Furthermore, we collect all necessary facts we use in Section \ref{sec : more bounds for regular embeddings}.
Our presentation partially follows \cite[Sec.\,3]{BaralicBlagojevicKarasecVucic2018} and is given in the generality necessary for our computations.

%----------------------------
\subsubsection{The wreath square\index{wreath square} and the $(d-1)$-partial wreath square\index{partial wreath square} of a vector bundle}
%----------------------------
Let $X$ be a CW-complex which a priori is not finite.
The product $X\times X$ has a natural action of the group $\Z_2$ given by interchanging the copies, that is  $\omega\cdot (x_1,x_2):=(x_2,x_1)$, where $\omega$ generates $\Z_2$ and $(x_1,x_2)\in X\times X$.
The product spaces $(X\times X)\times \EEE\Z_2$ and $(X\times X)\times S^{d-1}$ with the diagonal $\Z_2$-actions are free $\Z_2$-spaces.
The action on the sphere $S^{d-1}$ is assumed to be antipodal.
Thus, the projection maps 
\[
p_1\colon (X\times X)\times \EEE\Z_2\longrightarrow X\times X
\qquad\text{and}\qquad
p_1\colon (X\times X)\times S^{d-1}\longrightarrow X\times X
\] 
given by $(x_1,x_1,e)\longmapsto (x_1,x_2)$ are $\Z_2$-map.
Here the model for $\EEE\Z_2$ is assumed to be the infinite sphere $S^{\infty}:=\colim_{d\to\infty} S^{d-1}$ equipped with the antipodal action inherited from the action on $S^{d-1}$.

\medskip
Now, we define the functors 
\[
S^2\colon\CWtop\longrightarrow\CWtop
\qquad\text{and}\qquad
S^{2,d}\colon\CWtop\longrightarrow\CWtop
\]
that are on the objects given by  
\[
S^2X:=((X\times X)\times \EEE\Z_2)/\Z_2=(X\times X)\times_{\Z_2} \EEE\Z_2,
\]
and
\[
S^{2,d}X:=((X\times X)\times S^{d-1})/\Z_2=(X\times X)\times_{\Z_2} S^{d-1},
\]
where $X$ is a CW-complex and $d\geq 1$ integer.
For a morphism $h\colon X\longrightarrow Y$ of CW-complexes, a continuous map, we set 
\[
 S^2h:=(h\times h)\times_{\Z_2}\id \colon (X\times X)\times_{\Z_2}\EEE\Z_2\longrightarrow (Y\times Y)\times_{\Z_2}\EEE\Z_2,
\]
and 
\[
 S^{2,d}h:=(h\times h)\times_{\Z_2}\id \colon (X\times X)\times_{\Z_2}S^{d-1}\longrightarrow (Y\times Y)\times_{\Z_2}S^{d-1},
\]
to be the maps induced by the product maps
\[
(h\times h)\times\id \colon (X\times X)\times\EEE\Z_2\longrightarrow (Y\times Y)\times\EEE\Z_2
\]
and
\[
(h\times h)\times\id \colon (X\times X)\times S^{d-1}\longrightarrow (Y\times Y)\times S^{d-1}
\]
by passing to the $\Z_2$-orbits.

\medskip
Consider now a real $n$-dimensional vector bundle $\xi:=(E(\xi)\overset{p_{\xi}}{\longrightarrow} B(\xi))$ over the CW-complex  $B(\xi)$ whose fiber is $F(\xi)$.
Here we denote by $E(\xi)$ the total space of  $\xi$, and by  $p_{\xi}$ the corresponding projection map.
The pull-back vector bundle $p_1^*(\xi\times\xi)$ of the product vector bundle\index{vector bundle} $\xi\times\xi$ along the $p_1$:
\[
\xymatrix{
E(p_1^*(\xi\times\xi))\ar[rr]\ar[d] \ &  & \  E(\xi\times\xi)\cong E(\xi)\times E(\xi)\ar[d] \\
(B(\xi)\times B(\xi))\times \EEE\Z_2\ar[rr]^-{p_1} \ &  & \ B(\xi\times\xi)\cong B(\xi)\times B(\xi),
}
\]
is equipped with a free $\Z_2$-action.
Moreover, the projection map of the pull-back bundle $E(p_1^*(\xi\times\xi))\longrightarrow (B(\xi)\times B(\xi))\times \EEE\Z_2$ is a $\Z_2$-map.
Hence, after passing to $\Z_2$-orbits we get the $2n$-dimensional vector bundle $S^2\xi$ over $S^2B(\xi)$:
\[
E(p_1^*(\xi\times\xi))/\Z_2\longrightarrow (B(\xi)\times B(\xi))\times_{\Z_2} \EEE\Z_2.
\]
The bundle $S^2\xi$ is called the {\bf  wreath square}\index{wreath square} of the vector bundle $\xi$.

\medskip
Next we  consider the following pull-back diagram induced by the $\Z_2$-inclusion $i\colon S^{d-1}\longrightarrow S^{\infty}$: 
\[
\xymatrix{
E(((\id\times\id)\times_{\Z_2}i)^*S^2\xi)\ar[rr]\ar[d] \ & & \ E(p_1^*(\xi\times\xi))\ar[d]/\Z_2=E(S^2\xi)\\
(B(\xi)\times B(\xi))\times_{\Z_2}S^{d-1}\ar[rr]^-{(\id\times\id)\times_{\Z_2}i} \ & & \ (B(\xi)\times B(\xi))\times_{\Z_2} \EEE\Z_2=B(S^2\xi).
}
\]
The pull-back vector bundle $((\id\times\id)\times_{\Z_2}i)^*S^2\xi$ is called the {\bf $(d-1)$-partial wreath square}\index{partial wreath square} of the vector bundle $\xi$ and is denoted by $S^{2,d}\xi$.
In other words,
\[
S^{2,d}\xi:=((\id\times\id)\times_{\Z_2}i)^*S^2\xi,
\]
with the base space $S^{2,d}B(\xi)$.

\medskip
The wreath square\index{wreath square}, as well as the $(d-1)$-partial wreath square\index{partial wreath square}, of a vector bundle is natural with respect to morphisms of vector bundles\index{vector bundle}.
Indeed, a morphism between vector bundles $\xi\longrightarrow\eta$ induces morphisms between associated wreath squares $S^2\xi\longrightarrow S^2\eta$ and $(d-1)$-partial wreath squares $S^{2,d}\xi\longrightarrow S^{2,d}\eta$.
These morphisms satisfy all expected properties with respect to the composition of morphisms.  
Furthermore, the wreath square and the $(d-1)$-partial wreath square of a vector bundle behave naturally with respect to the Whitney sum of vector bundles.
This means that for arbitrary vector bundles $\xi$ and $\eta$ there are isomorphisms of vector bundles 
\begin{equation}
\label{eq : direct sum wsquare}
S^2(\xi\oplus\eta)\cong S^2(\xi)\oplus S^2(\eta)
\qquad\text{and}\qquad
S^{2,d}(\xi\oplus\eta)\cong S^{2,d}(\xi)\oplus S^{2,d}(\eta).	
\end{equation}

%----------------------------
\subsubsection{Cohomology of $B(S^2\xi)=S^2B(\xi)$}
\label{subsec : cohomology of S^2}
%----------------------------
In this section, based on the material presented in Section \ref{sub : cohomology of wreath product}, we describe the cohomology of a typical base space of the wreath square of a vector bundle.
 
\medskip 
Let $X$ be the base space of the vector bundle $\xi$, that is $X=B(\xi)$. 
Assume that $X$ in addition is a CW-complex.
Then the base space of the vector bundle $S^2\xi$ is $S^2X$, and also the total space of the fiber bundle
\begin{equation}
\label{eq : fib-001-a}
\xymatrix{
X\times X\ar[r] \ &\ (X\times X)\times_{\Z_2}\EEE\Z_2\ar[r] \ & \ \B \Z_2.
}
\end{equation}
Note that similarly the base space of the vector bundle $S^{2,d}\xi$ is the total space of the fiber bundle
\begin{equation}
\label{eq : fib-001-b}
\xymatrix{
X\times X\ar[r] \ & \ (X\times X)\times_{\Z_2}S^{d-1}\ar[r] \ & \ \RP^{d-1}.
}
\end{equation}
The Serre spectral sequence\index{Serre spectral sequence} associated to the fiber bundle \eqref{eq : fib-001-a} has $E_2$-term given by
\begin{equation}
	\label{eq : fib-005-a}
	E_{2}^{i,j}(X)=H^i(\B\Z_2;\mathcal{H}^j(X\times X;\F_2))\cong H^i(\Z_2;H^j(X\times X;\F_2)).
\end{equation}
As discussed in Section \ref{sub : cohomology of wreath product} this spectral sequence collapses at the $E_2$-term, that means $E_{2}^{i,j}(X)\cong E_{\infty}^{i,j}(X)$ for all $i,j\in\Z$.
For more details consult Theorem \ref{th : E_2 collapses} or \cite[Thm.\,IV.1.7]{Adem2004}.

\medskip
In the description of the total Stiefel--Whitney classes\index{Stiefel--Whitney classes} of the wreath square of a vector bundle the following maps turn out to be very useful.
At first, consider the map (not a homomorphism)
\[
P\colon H^j(X;\F_2)\longrightarrow H^{2j}(X\times X;\F_2)^{\Z_2}\cong E^{0,2j}_2\cong E^{0,2j}_{\infty}
\]
given by 
\[
P(u):= u\otimes u,
\] 
for $u\in H^j(X;\F_2)$ and $j\geq 0$ an integer.
By a direct inspection we see that the map $P$ is not an additive map, but it is a multiplicative map.
The second map we consider is 
\[
T\colon H^{j}(X\times X;\F_2) \longrightarrow H^{j}(X\times X;\F_2)^{\Z_2}
\]
defined by 
\[
T(u\otimes u'):= u\otimes u' + u'\otimes u,
\] 
where $u\otimes u'\in H^{j}(X\times X;\F_2)$ and $j\geq 0$ is an integer.
The map $T$ is an additive map, but not a multiplicative map.

\medskip
With the help of just introduced maps $P$ and $T$, based on Lemma \ref{lem : 02} and Theorem \ref{th : E_2 collapses}, we can describe the $E_{\infty}$-term of the Serre spectral sequence \eqref{big diagram 02} as follows:
\begin{compactitem}[\ ---]
\item $E^{*,0}_2\cong E^{*,0}_{\infty}\cong H^*(\Z_2;\F_2)\cong\F_2[f]$, $\deg(f)=1$,
\item $E^{0,*}_2 \cong E^{0,*}_{\infty} \cong  H^{*}(X\times X;\F_2)^{\Z_2}$,
\item $E^{0,j}_2\cong E^{0,j}_{\infty} \cong P(H^{j/2}(X;\F_2))\oplus T(H^{j}(X\times X;\F_2))$ for $j\geq 2$ even,
\item $E^{i,j}_2\cong E^{i,j}_{\infty}\cong P(H^{j/2}(X;\F_2))\otimes H^*(\Z_2;\F_2)$ for $j\geq 2$ even, and $i\geq 1$.
\end{compactitem}
Furthermore, we have that the generator $f$ of $H^*(\Z_2;\F_2)$ annihilates the image of the map $T$, that is
\begin{equation}
	\label{eq : multiplication Q and t-a}
	T(H^{*}(X\times X;\F_2))\cdot  f=0.
\end{equation}

\medskip
Note that the Serre spectral sequence\index{Serre spectral sequence} associated to the fiber bundle \eqref{eq : fib-001-b} can be describe in a similar way --- as discussed in Section \ref{sub : cohomology of wreath product}.

%----------------------------
\subsubsection{The total Stiefel--Whitney class of the wreath square of a vector bundle}
%----------------------------

Let $\xi$ be a real $n$-dimensional vector bundle\index{vector bundle} over the CW-complex $B(\xi)$.
To the vector bundle $\xi$ we associate characteristic classes
\[
s^2(\xi):=w(S^2\xi)\in H^*(B(S^2\xi);\F_2),
\]
and
\[
s^{2,d}(\xi):=w(S^{2,d}\xi)\in H^*(B(S^{2,d}\xi);\F_2).
\]
These are the total Stiefel--Whitney classes\index{Stiefel--Whitney classes} of the real $2n$-dimensional vector bundles $S^2\xi$ and $S^{2,d}\xi$.
The assignments 
\[
\xi \longmapsto s^2(\xi)
\qquad\text{and}\qquad
\xi \longmapsto s^{2,d}(\xi)
\] 
are natural with respect to the continuous maps.
This means that for a continuous map $h\colon K\longrightarrow B(\xi)$ from a CW-complex $K$  into the base space $B(\xi)$ of $\xi$ the following equalities hold
\[
(S^2h)^*(s^2(\xi))=s^2(h^*\xi)
\qquad\text{and}\qquad
(S^{2,d}h)^*(s^{2,d}(\xi))=s^{2,d}(h^*\xi)
\]
Here $h^*\xi$ denotes the pull-back vector bundle of $\xi$ along the map $h$.

\medskip
The characteristic class $s^2(\xi)$ was calculated, in term of Stiefel--Whitney classes\index{Stiefel--Whitney classes} of $\xi$, in \cite[Thm.\,3.4]{BaralicBlagojevicKarasecVucic2018} and is given in the following theorem.

\begin{theorem}
\label{th : SW class of doubling}
Let $\xi$ be a real $n$-dimensional vector bundle\index{vector bundle} over a $CW$-complex.
Then
\begin{multline*}
s^2(\xi)=w(S^2\xi)=\sum_{0\leq r< s\leq n}T(w_r(\xi)\otimes w_s(\xi))+
\sum_{0\leq r\leq n} P(w_r(\xi))\cdot(1+f)^{n-r}
\end{multline*}
In particular, $\xi$ is $1$ dimensional, then
\begin{multline*}
s^2(\xi)	= w(S^2\xi)= T(w_0(\xi)\otimes w_1(\xi))+P(w_0(\xi))\cdot(1+f)+P(w_1(\xi))=\\
1+\big(f+ 1\otimes w_1(\xi) + w_1(\xi)\otimes 1 \big)+w_1(\xi)\otimes w_1(\xi).
\end{multline*}
\end{theorem}

\medskip
From the computations in Section \ref{sub : cohomology of wreath product} and the fact that by definition $S^{2,d}\xi=((\id\times\id)\times_{\Z_2}i)^*S^2\xi$ we get the following consequence of Theorem \ref{th : SW class of doubling}.

\begin{corollary}
\label{th : SW class of partial doubling}
Let $\xi$ be a real $n$-dimensional vector bundle over a $CW$-complex, and let $d\geq 2$ be an integer. 
Then
\begin{multline*}
s^{2,d}(\xi)=w(S^{2,d}\xi)=\\
\sum_{0\leq r< s\leq n}T(w_r(\xi)\otimes w_s(\xi))+
\sum_{0\leq r\leq n}\,
\sum_{0\leq j\leq \min\{n-r,d-1\}}
{n-r \choose j}  P(w_r(\xi)) f^{j}.
\end{multline*}	
\end{corollary}

\medskip
In particular, Theorem \ref{th : SW class of doubling} and Corollary \ref{th : SW class of partial doubling} give formulas for the evaluation of mod $2$ Euler classes\index{Euler class} (top Stiefel--Whitney class\index{Stiefel--Whitney classes}) of the vector bundles $S^2\xi$ and $S^{2,d}\xi$ in term of the mod $2$ Euler class of the vector bundle $\xi$. 
In the following we present these formulas and show them with a direct proof which does not rely on the previous claims.

\begin{corollary}
\label{cor : mod 2 Euler class}
Let $\xi$ be a real $n$-dimensional vector bundle over a $CW$-complex.
Then
\[
\mathfrak{e}(S^2\xi)=P(\mathfrak{e}(\xi))
\qquad\text{or}\qquad
w_{2n}(S^2\xi)=P(w_n(\xi)),
\]
and
\[
\mathfrak{e}(S^{2,d}\xi)=P(\mathfrak{e}(\xi))
\qquad\text{or}\qquad
w_{2n}(S^{2,d}\xi)=P(w_n(\xi)).
\]
In particular, if $\mathfrak{e}(\xi)\neq 0$, then $\mathfrak{e}(S^2\xi)\neq 0$ and $\mathfrak{e}(S^{2,d}\xi)\neq 0$.
\end{corollary} 
\noindent
Here $\mathfrak{e}(S^2\xi)$, $\mathfrak{e}(S^{2,d}\xi)$ and $\mathfrak{e}(\xi)$ denote mod $2$ Euler classes of the vector bundles $S^2\xi$, $S^{2,d}\xi$ and $\xi$, respectively.
Note the abuse of notation: the map $P$ is not the same in the formulas for $\mathfrak{e}(S^2\xi)$ and $\mathfrak{e}(S^{2,d}\xi)$, since it operates on different spectral sequences. 
\begin{proof}
Let $u\in H^n(D(\xi),S(\xi);\F_2)$ be the Thom class\index{Thom class} of the $n$-dimensional vector bundle $\xi$.
Here $D(\xi)$ and $S(\xi)$ denote respectively the disk and sphere bundles associate to the vector bundle $\xi$.
The mod $2$ Euler class of the vector bundle $\xi$, by definition, equals to $\mathfrak{e}(\xi)=i_{\xi}^*(u)$ where $i_{\xi}\colon (E(\xi),\emptyset)\longrightarrow (D(\xi),S(\xi))$ is the zero section.

\medskip
Consider the following commutative diagram
\[{\small
\xymatrix@1{
H^n(D(\xi),S(\xi))\ar[rr]^-{i_{\xi}^*}\ar[d]^-{P} \ & &\ H^n(E(\xi))\ar[d]^-{P}\\
H^{2n}(((D(\xi),S(\xi))\times (D(\xi),S(\xi)))\times_{\Z_2}\EEE\Z_2)\ar[rr]^-{(Pi_{\xi})^*}\ar[d]\ar@/^1.66pc/[dd] \ & & \ H^{2n}((E(\xi)\times E(\xi))\times_{\Z_2}\EEE\Z_2)\ar[d] \ar@/^1.66pc/[dd]\\
H^{2n}(((D(\xi),S(\xi))\times (D(\xi),S(\xi)))\times_{\Z_2}S^{d-1})\ar[rr]^-{(Pi_{\xi})^*}\ar@/_5.5pc/[dd] \ & & \  H^{2n}((E(\xi)\times E(\xi))\times_{\Z_2}S^{d-1})\ar@/_5.5pc/[dd] \\
H^{2n}(D(S^2\xi),S(\xi))\ar[d]\ar[rr]^-{i_{S^2\xi}^*} \ & & \  H^{2n}((E(\xi)\times E(\xi))\times_{\Z_2}\EEE\Z_2)\ar[d]\\
H^{2n}(D(S^{2,d}\xi),S(\xi))\ar[rr]^-{i_{S^{2,d}\xi}^*} \ & & \  H^{2n}((E(\xi)\times E(\xi))\times_{\Z_2}S^{d-1})
}}
\]
where the coefficients are assumed to be in the field $\F_2$.
Note that the curly left arrows are isomorphism while the right ones are equalities.

\medskip
We claim that $P(u)$, appropriately interpreted, is the Thom class\index{Thom class} of the vector bundle  $S^2\xi$, respectively $S^{2,d}\xi$. 
To check this it is enough to show that it restricts to the generator in each fibre. 
Thus we can reduce to the case in which $E(\xi)$ and $\B\Z_2$, respectively $\RP^{d-1}$, are just points.
Then it is an elementary fact that the square map $P$:
\[
\big(H^n(D(\R^n),S(\R^n);\F_2)\cong \F_2 \big)\longrightarrow  \big( H^{2n}(D(\R^{2n}),S(\R^{2n});\F_2) \cong \F_2\big)
\]
maps the generator to the generator.

\medskip
The assertion now follows from the commutativity of the correspond diagram:
\[
\mathfrak{e}(S^2\xi)= i_{S^{2}\xi}^*(P(u))=P(i_{\xi}^*(u))=P(\mathfrak{e}(\xi)),
\]
and similarly
\[
\mathfrak{e}(S^{2,d}\xi)= i_{S^{2,d}\xi}^*(P(u))=P(i_{\xi}^*(u))=P(\mathfrak{e}(\xi)).
\]
\end{proof}

%%%%%%%%%%%%%%%%%%%%%%%%%%%%%%%%%%%%%%%%%%%%%%%%%%%%%%%%%%%%%%%%%%%%%%%%%%%%%%%%%%%%%
\subsection{Miscellaneous calculations}
%%%%%%%%%%%%%%%%%%%%%%%%%%%%%%%%%%%%%%%%%%%%%%%%%%%%%%%%%%%%%%%%%%%%%%%%%%%%%%%%%%%%%
In order give a smoother presentation of the main computational components in the main body of the paper in this section we present details of some auxiliary computations.

%===========================
\subsubsection{Detecting group cohomology}
\label{sec : detection}
%===========================

In this section we review some basic facts from classical work of Quillen \cite{Quillen1971} in generality we need in this paper. 
For more details consult for example \cite[Sec.\,IV.4\,and\,VI.1]{Adem2004}.

\medskip
Let $G$ be a finite group, and let $p$ be a prime.
The family of subgroups $\{H_i : i\in I\}$ of the group $G$ {\bf detects the cohomology} of $G$ modulo $\F_p$, or is a {\bf detecting family of subgroups}\index{detecting family of subgroups}, if the homomorphism
\[
H^*(G;\F_p)\longrightarrow \prod_{i\in I} H^*(H_i;\F_p)
\]
induced by the restrictions $\res^G_{H_i}$ is a monomorphism.
If $G^{(p)}$ is Sylow $p$-subgroup\index{Sylow $p$-subgroup} of $G$, then the restriction $\res^G_{G^{(p)}}$ is a monomorphism, and consequently $G^{(p)}$ detects the cohomology of $G$ modulo $\F_p$.

\medskip
First we recall an auxiliary lemma that is particularly useful for us, see \cite[Lem.\,3.22]{Madsen1979}.

\begin{lemma}
	\label{lem : detection of wreath product}
	Let $G$ be a finite group, and let $p$ be a prime. 
	Then $G\wr\Z_p$ is detected by $G\times \Z_p$ and $\underbrace{G\times\cdots\times G}_{p\text{ times}}$.
\end{lemma} 

\medskip
Next, let $p$ be a prime, and let $n=p^m$ for an integer $m\geq 1$. 
Consider symmetric group\index{symmetric group} $\Sym_n=\Sym_{p^m}$ as a group of permutation of the set $\Z_p^{\oplus m}$.
Then the elementary abelian group\index{elementary abelian group} $\EE_m$ of all translations (seen as permutations) of the vector space $\Z_p^{\oplus m}$, the so called regular embedded subgroup \cite[Ex.\,III.2.7]{Adem2004}, is isomorphic to the elementary abelian group $\Z_p^{\oplus m}$.
Let 
\[\Sy_{p^m}=E_{m,1}\wr\cdots\wr E_{m,m} \cong \underbrace{\Z_p\wr\cdots\wr\Z_p}_{m\text{ times}}\] 
denotes the Sylow $p$-subgroup\index{Sylow $p$-subgroup} of $\Sym_{p^m}$ containing $\EE_m$. 
Here $E_{m,i}\cong\Z_p$ is the subgroup of $\EE_m$ generated by the $i$th basis element $(0,\ldots,0,1,0,\ldots,0)$ of the vector space $\Z_p^{\oplus m}$.

\medskip
Now we proceed with the following detection property of the cohomology of the symmetric group $\Sym_{p^m}$ and its Sylow $p$-subgroup $\Sy_{p^m}$, consult \cite[Cor.\,VI.1.4]{Adem2004}.

\begin{theorem}
	\label{thm : symmetric group  p^n detected}
	Let $p$ be a prime, and let $n=p^m$ for an integer $m\geq 1$. 
	\begin{compactenum}[\rm \ (1)]
	\item 	The cohomology $H^*(\Sym_{p^m};\F_p)$ of the symmetric group\index{symmetric group} $\Sym_{p^m}$ is detected modulo $\F_p$ by the elementary abelian subgroup $\EE_m$ and the product subgroup 
	\[
	\underbrace{\Sym_{p^{m-1}}\times\cdots\times\Sym_{p^{m-1}}}_{p\text{ times}}.
	\]
	\item The cohomology $H^*(\Sy_{p^m};\F_p)$ of the Sylow $p$-subgroup\index{Sylow $p$-subgroup} $\Sy_{p^m}$ of the symmetric group $\Sym_{p^m}$ is detected modulo $\F_p$ by the elementary abelian subgroup $\EE_m$ and the product subgroup 
	\[
	\underbrace{\Sy_{p^{m-1}}\times\cdots\times \Sy_{p^{m-1}}}_{p\text{ times}}.
	\]
	\end{compactenum}
\end{theorem}

Finally we state the classical result of Quillen about detection with the family of elementary abelian subgroups, see \cite[Thm.\,VI.1.2]{Adem2004}.

\begin{theorem}
	\label{thm : symmetric group detected by e. abelian}
	Let $n\geq 1$ be an integer, and let $p$ be a prime. 
	The cohomology $H^*(\Sym_n;\F_p)$ of the symmetric group $\Sym_n$ with $\F_p$ coefficients is detected by the family of its elementary abelian $p$-subgroups.
\end{theorem}

%
%----------------------------------------------------
\subsubsection{The image of a restriction homomorphism}
\label{subsub : Image of a restriction}
%----------------------------------------------------

Let $G$ be a finite group and let $H$ a subgroup of $G$.
In this section we want to give a description of the image of the restriction homomorphism\index{restriction homomorphism} 
\[
\im \big(\res^G_H \colon H^*(G;M)\longrightarrow H^*(H;M)\big)
\]
where $M$ is a trivial $G$-module.

\medskip
Let $G$ be a (finite) group.
Any contractible free $G$-CW complex equipped with the right $G$ cellular action is called a model for an $\EEE G$ space.
The Milnor model is given by $\EEE G=\colim_{n\in\N}G^{*n}$ where $G$ stands for a $0$-dimensional free $G$-simplicial complex whose vertices are indexed by the elements of the group $G$ and the action on $G$ is given by the right translation, and $G^{*n}$ is an $n$-fold join of the $0$-dimensional simplicial complex with induced diagonal (right) action.
A typical point in $\EEE G$ can be presented as follows
\[
\sum_{i\geq 1}\lambda_ig_i\equiv (\lambda_1g_1,\lambda_2g_2,\lambda_3g_3,\ldots),
\] 
where $g_i\in G$ and $\lambda_i\geq 0$ for all $i\geq 1$, the set $I:=\{\lambda_i\neq 0 : i\geq 1\}$ is finite,   and $\sum_{i\in I}\lambda_i=1$. 
The quotient space $\B G=\EEE G/G$ is called a classifying space a the group $G$.
For a trivial $G$-module $M$ the group cohomology $H^*(G;M)$ can be defined as a singular cohomology $H^*(\B G;M)$. 

\medskip
Let $H$ and $G$ be (finite) groups, and let $f\colon H\longrightarrow G$ be a homomorphism.
Then $f$ induces the following $G$-equivariant map $\EEE(f)\colon\EEE H\longrightarrow \EEE G$ by
\begin{multline*}
\sum_{i\geq 1}\lambda_ih_i\equiv (\lambda_1h_1,\lambda_2h_2,\lambda_3h_3,\ldots) 
\longmapsto \\
\sum_{i\geq 1}\lambda_i g(h_i)\equiv (\lambda_1 f(h_1),\lambda_2 f(h_2),\lambda_3 f(h_3),\ldots)	
\end{multline*}
where the $H$-action on $G$ is induced by the homomorphism $f$.
Since the map $\EEE(f)$ is $H$-equivariant map it induces a map between quotient spaces $\B(f)\colon\B H\longrightarrow \B G$.
In particular, if $H$ is a subgroup of $G$, $i\colon H\longrightarrow G$ the inclusion map, then the induced homomorphism in cohomology $\B(i)^*$ by definition is the restriction homomorphism\index{restriction homomorphism} $\res^G_H$, that is $\B(i)^*=\res^G_H$.

\medskip
We prove an auxiliary lemma following \cite[Prop.\,I.6.14]{TomDieck1987} and \cite[Thm.\,II.1.9]{Adem2004}.

\begin{lemma}
	\label{lem : aux-01}
	Let $G$ be a finite group.
	\begin{compactenum}[\rm \ (i)]
	\item Any two continuous $G$-equivariant maps $f,g\colon\EEE G\longrightarrow \EEE G$ are $G$-homotopic.
	\item Let $a\in G$, and let $k_a\colon G\longrightarrow G$ be the conjugation homomorphism $k_a(g):=aga^{-1}$.
		Then the induced map $\B(k)\colon \B G\longrightarrow\B G$ is homotopic to the identity $\id\colon \B G\longrightarrow\B G$.
	\end{compactenum}
\end{lemma}
\begin{proof}
(i) This is a presentation of the proof of \cite[Prop.\,I.6.14]{TomDieck1987}.
For $e\in\EEE G$ let us denote images $f(e)$ and $g(e)$ as follows 
\[
f(e)=(\lambda_1(e) f_1(e),\lambda_2(e) f_2(e), \ldots)
\quad\text{and}\quad
g(e)=(\mu_1(e) g_1(e),\mu_2(e) g_2(e), \ldots).
\] 
We define two additional maps $\bar{f},\bar{g}\colon\EEE G\longrightarrow \EEE G$ by
\[
\bar{f}(e)=(\lambda_1(e) f_1(e),0,\lambda_2(e) f_2(e),0,\ldots)
\]
and
\[
\bar{g}(e)=(0,\mu_1(e) g_1(e),0,\mu_2(e) g_2(e),0, \ldots).
\] 
The pairs of maps $f$, $\bar{f}$ and $g$, $\bar{g}$ are $G$-homotopic, that is $f\simeq_G\bar{f}$ and $g\simeq_G\bar{g}$.
In order to construct a $G$-homotopy, for example, between  $f$ and $\bar{f}$ we proceed as follows.
Let $j\geq1$ be an integer and let the $G$-homotopy $H_j\colon \EEE G\times I\longrightarrow\EEE G$ be defined by
\begin{multline*}
H_j(e,t):=\big(\lambda_1(e)f_1(e),\ldots,\lambda_j(e)f_j(e),\\
t\lambda_{j+1}(e)f_{j+1}(e),(1-t)\lambda_{j+1}(e)f_{j+1}(e),\\
t\lambda_{j+2}(e)f_{j+2}(e),(1-t)\lambda_{j+2}(e)f_{j+2}(e),  \ldots  \big).	
\end{multline*}
Starting with $\bar{f}(e)=H_1(e,0)$ and applying consecutively $H_1,H_2,\ldots$ we reach $f(e)$.
Hence we have define a $G$-homotopy between $\bar{f}$ and $f$.
Since in the presentation of every point $e=(\lambda_1g_1,\lambda_2g_2,\ldots)\in\EEE G$ there are only finitely many non-zero $\lambda$'s the definition we just gave is correct and the map is moreover continuous.  
Therefore, $f\simeq_G\bar{f}$ and $g\simeq_G\bar{g}$.

\medskip
It suffices to prove that $\bar{f}\simeq_G\bar{g}$.
For this we give the $G$-homotopy $H\colon \EEE G\times I\longrightarrow\EEE G$ by
\[
H(e,t):=\big((1-t)\lambda_1(e)f_1(e), \, t\mu_1(e)g_1(e), \, (1-t)\lambda_2(e)f_2(e), \, t\mu_2(e)g_2(e),\ldots\big).
\]
This concludes the proof of part (i).

\smallskip
(ii) The homomorphisms $\id\colon G\longrightarrow G$ and $k_a\colon G\longrightarrow G$ induce $G$-equivariant maps 
\[
\EEE(\id)=\id\colon\EEE G\longrightarrow\EEE G
\qquad\text{and}\qquad
\EEE(k_a)\colon \EEE G\longrightarrow\EEE G.
\]
From the part (i) of this lemma we have that $\EEE(\id)=\id$ and $\EEE(k_a)$ are $G$-homotopic. 
Consequently, $\B(\id)=\id$ and $\B(k_a)$ are homotopic.
\end{proof}

\medskip
Now we are ready to give a description of the image a restriction that we use in the central part of the paper. 
Consult \cite[Lem.\,II.3.1]{Adem2004}.

\begin{lemma}
	\label{lem : image of restriction}
	Let $G$ be a finite group, $H$ its subgroup, $N_G(H)$ the normalizer\index{normalizer of a subgroup} of $H$ in $G$, $W_G(H):=N_G(H)/H$ the corresponding Weyl group\index{Weyl group}, and $M$ a trivial $G$-module.
	There is an action of $W_G(H)$ on $H^*(H;M)$ such that
	\[
	\im \big(\res^G_H \colon H^*(G;M)\longmapsto H^*(H;M)\big)\subseteq H^*(H;M)^{W_G(H)}.
	\] 
\end{lemma}

\begin{proof}
First we introduce an action of the normalizer $N_G(H)$ on $\B H$.
Let $a\in N_G(H)$, and let $k_a\colon H\longrightarrow H$ be a homomorphism defined by $k_a(h):=aha^{-1}$ for $h\in H$.
For $a,b\in N_G(H)$ the following relation obviously holds $k_{ab}=k_a\circ k_b$.
Applying the functor $\B$ on the maps $k_a$, for every $a\in N_G(H)$, we get an action of $N_G(H)$ on $\B H$.
This action naturally extends to an action on the cohomology 
\[
N_G(H)\times H^*(H;M)\longrightarrow H^*(H;M).
\]
In the case when $a\in H\subseteq N_G(H)$ from Lemma \ref{lem : aux-01} we get that the homomorphism  $\B(k_a)^*\colon H^*(H;M)\longrightarrow H^*(H;M)$ is the identity.
Consequently the action of $N_G(H)$ factors through an action of the Weyl group $W_G(H)$, that gives us a commutative diagram
\[
\xymatrix{
N_G(H)\times H^*(H;M)\ar[rr]\ar[dr] \ &  & \ H^*(H;M)\\
&W_G(H)\times H^*(H;M).\ar[ur]&
}
\]
The group $N_G(H)$ acts on $\B G$ in the identical way. 
Now for $a\in N_G(H)$ we set $k_a\colon G\longrightarrow G$ to be again defined by $k_a(g):=aga^{-1}$ for $g\in G$.
Applying the functor $\B$ we now get an action of $N_G(H)$ on $\B G$ and consequently on $H^*(G;M)$.
In this case Lemma \ref{lem : aux-01} implies that each map $\B(k_a)$ is homotopic to the identity. 
Hence, the induced action of $N_G(H)$ on $H^*(G;M)$ is a trivial action.

\noindent
Since the action of $N_G(H)$ on $\B G$ is an extension of the action of $N_G(H)$ on $\B H$ we have the following commutative diagram
\begin{equation}
\label{diagram-01}
\xymatrix{
\B H\ar[rr]^-{\B(i)}\ar[d]_{\B(k_a)} \ & & \ \B G \ar[d]_{\B(k_a)} \\
\B H\ar[rr]^-{\B(i)}                \ & & \ \B G
}	
\end{equation}
for every $a\in N_G(H)$ where $i\colon H\longrightarrow G$ is the inclusion.
Applying the cohomology functor on the commutative diagram \eqref{diagram-01} we get that
\[
\B(i)^*\circ \B(k_a)^*=\B(k_a)^*\circ \B(i)^*.
\]
Since $\B(i)^*=\res^G_H$, and $\B(k_a)^*\colon H^*(G;M)\longrightarrow H^*(G;M)$ is the identity for every $a\in N_G(H)$, we get that $\res^G_H=\B(k_a)^*\circ\res^G_H$, and consequently
\[
\im \big(\res^G_H \colon H^*(G;M)\longmapsto H^*(H;M)\big)\subseteq H^*(H;M)^{N_G(H)}=H^*(H;M)^{W_G(H)}.
\]
\end{proof}
%
%----------------------------------------------------
\subsubsection{Weyl groups of an elementary abelian group\index{elementary abelian group}}
\label{subsub : Weyl groups of an elementary abelian group}
%----------------------------------------------------

In this section we identify Weyl groups\index{Weyl group} $W_{\Sy_{2^m}}(\EE_m)$ and $W_{\Sym_{2^m}}(\EE_m)$ for every $m\geq 1$.
First, following M\`{u}i {\cite[Proof of Lem.\,II.5.1]{Mui1975}} we prove the following fact.

\begin{lemma}
\label{lem : Weyl group of E_m}
	Let $m\geq 0$ be an integer.
	Then
	\[
	W_{\Sy_{2^m}}(\EE_m)\cong \U_m(\F_2).
	\]
\end{lemma}
\begin{proof}
Let $\kappa\colon N_{\Sy_{2^m}}(\EE_m)\longrightarrow\mathrm{Aut}(\EE_m)$ be the homomorphism defined by $\kappa(a):=k_a$, where as before $k_a\colon \EE_m\longrightarrow \EE_m$ is the conjugation automorphism $k_a(e)=aea^{-1}$,  $e\in \EE_m$.
The kernel of the homomorphism $\kappa$ is the centralizer of $\EE_m$ in $\Sy_{2^m}$, which is $\ker(\kappa)=C_{\Sy_{2^m}}(\EE_m)$.
Furthermore, in our situation, $C_{\Sy_{2^m}}(\EE_m)=\EE_m$.
Thus, there is an exact sequence of groups
\begin{equation}
	\label{ex : 02}
\xymatrix{
1\ar[r]\ &\ \EE_m\ar[r]\ & \N_{\Sy_{2^m}}(\EE_m)\ar[r]\ & \ \im (\kappa)\ar[r] &\ 1.
}
\end{equation}
In other words $N_{\Sy_{2^m}}(\EE_m)$ is a semi-direct product $\EE_m\rtimes\im (\kappa)$.
Consequently, 
\[
W_{\Sy_{2^m}}(\EE_m)=N_{\Sy_{2^m}}(\EE_m)/\EE_m\cong\im(\kappa).
\]
Since $\mathrm{Aut}(\EE_m)\cong\GL_m(\F_2)$ we can say that $\im(\kappa)\subseteq \GL_m(\F_2)$ and therefore the Weyl\index{Weyl group} group $W_{\Sy_{2^m}}(\EE_m)$ can be seen as a subgroup of $\GL_m(\F_2)$.
Thus it remains to identify $\im(\kappa)$, and this will be done in two steps.

\medskip
First we prove that $\U_m(\F_2)\subseteq N_{\Sy_{2^m}}(\EE_m)$ using the induction on $m\geq 1$.
Notice that obviously $\U_m(\F_2)\subseteq N_{\Sym_{2^m}}(\EE_m)$.
%TODO : Not so obvious -- explain
For $m=1$ the group $\U_1(\F_2)$ is trivial group while $\EE_1=\Sy_2=N_{\Sy_{2}}(\EE_1)\cong\Z_2$.
Let us assume that for $m\geq 2$ the following inclusions $\U_{m-1}(\F_2)\subseteq N_{\Sy_{2^{m-1}}}(\EE_{m-1})\subseteq \Sy_{2^{m-1}}$ holds.
Consider the embedding $\U_{m-1}(\F_2)\longrightarrow \U_m(\F_2)$ of group given by
\[
A\longmapsto
\left(
\begin{array}{cc}
	1 & 0\\
	0 & A\\
\end{array}
\right)
\]
for $A\in\U_{m-1}(\F_2)$.
Note the ``difference'' between zeroes in the upper matrix.
Since for every $i\in\F_2$ and $e\in\F_2^{m-1}$ 
\[
\left(
\begin{array}{cc}
	1 & 0\\
	0 & A\\
\end{array}
\right)
\left(
\begin{array}{c}
	i\\
	e
\end{array}
\right)
=
\left(
\begin{array}{c}
	i\\
	Ae
\end{array}
\right)
\]
we have that $\U_{m-1}(\F_2)$ is a subgroup of $\delta(\Sy_{2^{m-1}})\subseteq \Sy_{2^{m-1}}\times \Sy_{2^{m-1}}$ where $\delta$ denotes the diagonal embedding.
An arbitrary element of the group $\U_{m}(\F_2)$ can be presented in the form 
\[
\left(
\begin{array}{cc}
	1 & 0\\
	a & A\\
\end{array}
\right)
\]
where $A\in\U_{m-1}(\F_2)$ and $a\in\F_2^{m-1}$.
Then for  $i\in\F_2$ and $e\in\F_2^{m-1}$ we have
\[
\left(
\begin{array}{cc}
	1 & 0\\
	a & A\\
\end{array}
\right)
\left(
\begin{array}{c}
	i\\
	e
\end{array}
\right)
=
\left(
\begin{array}{c}
	i\\
	Ae+ai
\end{array}
\right).
\]
Consequently, $\U_{m}(\F_2)\subseteq (\U_{m-1}(\F_2)\cdot \EE_{m-1})\times (\U_{m-1}(\F_2)\cdot \EE_{m-1})$.
From the induction hypothesis we have more
\[
\U_{m}(\F_2)\subseteq (\U_{m-1}(\F_2)\cdot\EE_{m-1})\times (\U_{m-1}(\F_2)\cdot\EE_{m-1})\subseteq \Sy_{2^{m-1}}\times \Sy_{2^{m-1}}\subseteq \Sy_{2^{m}}.
\]
As we have seen $\U_m(\F_2)\subseteq N_{\Sym_{2^m}}(\EE_m)$ and thus
\[
\U_m(\F_2)\subseteq N_{\Sy_{2^m}}(\EE_m),
\]
which concludes the induction.

\medskip
Now we continue identification of $\im(\kappa)$.
Since $\U_m(\F_2)\cap \EE_m=\{1\}$ we have that the exact sequence \eqref{ex : 02} gives us an embedding of $\U_m(\F_2)$ into $\im(\kappa)$.
Hence we have the following inclusions of $2$-groups (with the obvious abuse of notation)
\[
\U_m(\F_2)\subseteq \im(\kappa)\subseteq \GL_{m}(\F_2).
\]
The image $\im(\kappa)$ is a $2$-group as an image of the $2$-group $N_{\Sy_{2^m}}(\EE_m)$. 
Because $\U_m(\F_2)$ is a Sylow $2$-subgroup\index{Sylow $2$-subgroup} of $\GL_{m}(\F_2)$ we have that $\U_m(\F_2)=\im(\kappa)$ and consequently
\[
W_{\Sy_{2^m}}(\EE_m)\cong\im(\kappa)\cong \U_m(\F_2).
\]
\end{proof}

\medskip
Next, adapting the proof of the previous lemma and following \cite[Ex.\,III.2.7]{Adem2004} we determine the Weyl group of $\EE_m$ now inside the symmetric group $\Sym_{2^m}$.

\begin{lemma}
\label{lem : Weyl group of E_m - 2}
	Let $m\geq 0$ be an integer.
	Then
	\[
	W_{\Sym_{2^m}}(\EE_m)\cong \GL_m(\F_2).
	\]
\end{lemma}
\begin{proof}
Let $\kappa\colon N_{\Sym_{2^m}}(\EE_m)\longrightarrow\mathrm{Aut}(\EE_m)$ be the homomorphism defined by $\kappa(a):=k_a$, where as before $k_a\colon \EE_m\longrightarrow \EE_m$ is the conjugation automorphism $k_a(e)=aea^{-1}$,  $e\in \EE_m$.
As in the proof of the previous lemma $\ker(\kappa)=C_{\Sym_{2^m}}(\EE_m)=\EE_m$.
Hence, we get an exact sequence of groups
\[
\xymatrix{
1\ar[r] \ &\ \EE_m\ar[r] \ & \ N_{\Sym_{2^m}}(\EE_m)\ar[r] \ & \ \im (\kappa)\ar[r] & 1.
}
\]
Consequently,  $N_{\Sym_{2^m}}(\EE_m)$ is a semi-direct product $\EE_m\rtimes\im (\kappa)$ implying that 
\[
W_{\Sym_{2^m}}(\EE_m)=N_{\Sym_{2^m}}(\EE_m)/\EE_m\cong\im(\kappa).
\]
Since $\mathrm{Aut}(\EE_m)\cong\GL_m(\F_2)$, in order to complete the proof of the lemma, it suffices to prove that $\im(\kappa)=\mathrm{Aut}(\EE_m)$.
For that fix $\alpha\in \mathrm{Aut}(\EE_m)$, and denote by $\bar{\alpha}\in\Sym_{2^m}$ the permutation given with $e\longmapsto\alpha(e)$ for $e\in \EE_m$.
Here we use the fact that $\Sym_{2^m}$ is group of permutations of the set $\EE_m$. 
Then $\bar{\alpha}\in  N_{\Sym_{2^m}}(\EE_m)$ and $\kappa( \bar{\alpha})=\alpha$.
Indeed, for $e,h\in \EE_m$ holds:
\[
\bar{\alpha}e\bar{\alpha}^{-1}(h)= (\bar{\alpha}e) (\alpha^{-1}(h))= \bar{\alpha}(e+\alpha^{-1}(h))=\alpha(e+\alpha^{-1}(h))=\alpha(e)+h.
\]
consequently we have an equality of permutations $\bar{\alpha}e\bar{\alpha}^{-1}=\alpha(e)\in\EE_m$ that implies $\bar{\alpha}\in  N_{\Sym_{2^m}}(\EE_m)$ and $\kappa( \bar{\alpha})=\alpha$.
This concludes the proof of surjectivity of $\kappa$ and of the lemma.
\end{proof}

%----------------------------------------------------
\subsubsection{Cohomology of the real projective space with local coefficients}
\label{subsub : computation of local coefficients}
%----------------------------------------------------

For $d\geq 2$ we compute the cohomology of the projective space
$H^*(\RP^{d-1};\mathcal{M})$
where the local coefficient system $\mathcal{M}$ is additively $\F_2\oplus\F_2$, and the action of $\pi_1(\RP^{d-1})=\langle t\rangle$ on $\mathcal{M}$ is given by $t\cdot (a_1,a_2)=(a_2,a_1)$ where $(a_1,a_2)\in \F_2\oplus\F_2$.
\medskip
We consider two separate cases: $d=2$ when $\pi_1(\RP^{d-1})\cong\Z$, and $d\geq 3$ when $\pi_1(\RP^{d-1})\cong\Z_2$.
First we recall the definition of the cohomology with local coefficients, consult for example \cite[Sec.\,3.H]{Hatcher2002}.

\medskip
Let $X$ be a path-connected CW-complex, $\widetilde{X}$ its universal cover, and let $\pi:=\pi_1(X)$ be its fundamental group.
Denote by $\mathcal{L}$ a local coefficient system on $X$, that is a $\Z[\pi]$-module.
Assume that we are given a structure of $\Z[\pi]$-CW-complex on $\widetilde{X}$ with associated cellular chain complex
\[
\xymatrix@1{
\cdots \ \ar[r] \ & \  C_{n+1}(\widetilde{X})\ar[r]^-{d_{n+1}} \ & \ C_{n}(\widetilde{X})\ar[r]^-{d_{n}} \ & \ C_{n-1}(\widetilde{X})\ar[r]^-{d_{n-1}} \ & \  \cdots \\ 
& & \cdots \ar[r]^-{d_{2}} \ & \ C_{1}(\widetilde{X})\ar[r]^-{d_{1}} \ & \ C_{0}(\widetilde{X})\ar[r] \ & \ 0 \ ,
}
\]
where $C_{n}(\widetilde{X})$ is a $\Z[\pi]$-module, and $d_n\colon C_{n}(\widetilde{X})\longrightarrow C_{n-1}(\widetilde{X})$ is a $\Z[\pi]$-module homomorphism, for every $n\in\Z$.
The cohomology $H^*(X;\mathcal{L})$ of $X$ with local coefficients in $\mathcal{L}$ is cohomology of the cochain complex
\[
%\xymatrix@R-1.5pc{
\xymatrix@C=1.37em{
\cdots \  &\ \ar[l] \hom_{\Z[\pi]}(C_{n+1}(\widetilde{X});\mathcal{L}) \ &\ \ar[l]\hom_{\Z[\pi]}(C_{n}(\widetilde{X});\mathcal{L})&\ar[l]\hom_{\Z[\pi]}(C_{n-1}(\widetilde{X});\mathcal{L}) \\
\cdots \ &\ \ar[l]\hom_{\Z[\pi]}(C_{1}(\widetilde{X});\mathcal{L}) \ &\ \ar[l] \hom_{\Z[\pi]}(C_{0}(\widetilde{X});\mathcal{L}) \ &\ \ar[l]0 \ . \qquad\qquad\qquad
}
\]
This means that
\[
H^n(X;\mathcal{M}):=\frac{\ker\big( \hom_{\Z[\pi]}(C_{n}(\widetilde{X});\mathcal{L}))\longrightarrow \hom_{\Z[\pi]}(C_{n+1}(\widetilde{X});\mathcal{L})\big)}{\im\big(\hom_{\Z[\pi]}(C_{n-1}(\widetilde{X});\mathcal{L})\longrightarrow   \hom_{\Z[\pi]}(C_{n}(\widetilde{X});\mathcal{L})\big)} \ .
\]

\medskip
For $d=2$ the universal cover of the projective space $\RP^1$ is the real line $\R^1$, and the fundamental group $\pi=\pi_1(\RP^1)=\langle t\rangle$ is the infinite cyclic group. 
An associated $\Z[\pi]$-CW-complex of the universal cover $\widetilde{\RP^1}=\R^1$ is defined as follows: The $0$-cells are all integers $\{x_0^{i}:=\{i\} : i\in\Z\}$, while the $1$-cells are intervals $\{x_1^{i}:=[i,i+1] : i\in\Z\}$.
The action of the generator $t$ of the fundamental group on the cells of the $\Z[\pi]$-CW-complex is given by $t\cdot x_{j}^{i}=x_{j}^{i+1}$ for $j\in\{0,1\}$.
Hence, the induced cellular chain complex of $\Z[\pi]$-modules is given by
\[
C_{0}(\widetilde{\RP^1})=\Z[\Z]=_{\Z[\Z]\text{-module}}\langle x_{0}^{1} \rangle
\qquad\text{and}\qquad
C_{1}(\widetilde{\RP^1})=\Z[\Z]=_{\Z[\Z]\text{-module}}\langle x_{1}^{1} \rangle
\]
where the only non-trivial boundary homomorphism $d_1$ is defined on the generator of $C_{1}(\widetilde{\RP^1})$ by $d_1( x_{1}^{1}):= x_{0}^{1}- x_{0}^{0}=(t-1)\cdot x_{0}^{0} $.
Thus, the cellular chain complex we consider in this case is
\[
\xymatrix{
0\ar[r] \ & \ \Z[\Z]\ar[r]^-{(t-1)\cdot} \ & \ \Z[\Z]\ar[r] & \ 0.
}
\]
After applying the functor $\hom_{\Z[\Z]}(\cdot,\mathcal{M})$ we get the cochain complex
\[
\xymatrix{
0 \ &\ \ar[l] \mathcal{M} \ & \ \ar[l]_-{(t-1)\cdot}\mathcal{M} \ & \ \ar[l] \ 0.
}
\]
If we recall that additively $\mathcal{M}=\F_2\oplus\F_2$ by direct inspection we see that
\[
H^r(\RP^{1};\mathcal{M})\cong
\begin{cases}
	\ker\big(\mathcal{M}\overset{(t-1)\cdot}{\longrightarrow} \mathcal{M} \big)\cong  \F_2 , &  r=0 ,\\
	\mathcal{M}/\im\big(\mathcal{M}\overset{(t-1)\cdot}{\longrightarrow} \mathcal{M} \big)\cong  \F_2 ,&  r=1 ,\\
	0,& \text{otherwise}.
\end{cases}
\]
For further use we denote the generator of the group  $H^1(\RP^{d-1};\mathcal{M})$ by $z_1$.

\medskip
Now, let $d\geq 3$ be an integer. 
Then the universal cover of the projective space $\RP^{d-1}$ is the sphere $S^{d-1}$, and the fundamental group is $\pi:=\pi_1(\RP^{d-1})\cong\Z_2=\langle t\rangle$. 
We associate an $\Z[\pi]$-CW-complex to the universal cover $\widetilde{\RP^{d-1}}=S^{d-1}$ as follows: 
In each dimension $i$, where $0\leq i\leq d-1$, there are two cells $x_{i}^{0}$ and $x_{i}^{1}$. 
The generator $t$ of the fundamental group $\pi$ acts on the cells by $t\cdot x_{j}^{0}=x_{j}^{1}$ and $t\cdot x_{j}^{1}=x_{j}^{0}$ for $0\leq j\leq d-1$.
Thus, the induced cellular chain complex of $\Z[\pi]$-modules is given by
\[
C_{n}(\widetilde{\RP^{d-1}})=\Z[\Z_2]=_{\Z[\Z_2]\text{-module}}\langle x_{n}^{0}\rangle,
\qquad\qquad\text{for }
0\leq n\leq d-1,
\] 
where the boundary homomorphism $d_n\colon C_{n}(\widetilde{\RP^{d-1}})\longrightarrow C_{n-1}(\widetilde{\RP^{d-1}})$ on the generator $x_{n}^{0}$ is 
\begin{align*}
d_n(x_{n}^{0})&:=x_{n-1}^{1}-x_{n-1}^{0}=(t-1)\cdot x_{n-1}^{0},
\qquad\qquad\text{for }
1\leq n\leq d-1\text{ odd},	\\
d_n(x_{n}^{0})&:=x_{n-1}^{1}+x_{n-1}^{0}=(t+1)\cdot x_{n-1}^{0},
\qquad\qquad\text{for }
1\leq n\leq d-1\text{ even},	
\end{align*}
and otherwise zero.
The cellular chain complex we obtained in this case is
\[
\xymatrix{
0\ar[r] \ &\ \Z[\Z_2]\ar[r]^-{(t+1)\cdot} \ &\ \Z[\Z_2]\ar[r]^-{(t-1)\cdot} \ & \ \cdots \ar[r]^-{(t+1)\cdot} \ & \ \Z[\Z_2]\ar[r]^-{(t-1)\cdot} \ & \  \Z[\Z_2]\ar[r] \ & \ 0,
}
\]
when $d$ is odd, and 
\[
\xymatrix{
0\ar[r] \ & \ \Z[\Z_2]\ar[r]^-{(t-1)\cdot}\ & \ \Z[\Z_2]\ar[r]^-{(t+1)\cdot} \ & \ \cdots \ar[r]^-{(t+1)\cdot} \ & \ \Z[\Z_2]\ar[r]^-{(t-1)\cdot} \ & \  \Z[\Z_2]\ar[r] \ & \ 0,
}
\]
when $d$ is even.

\medskip
Applying the functor $\hom_{\Z[\Z_2]}(\cdot,\mathcal{M})$ we get the cochain complex which is isomorphic to the following cochain complex
\[
\xymatrix{
0 \ & \ \ar[l] \mathcal{M} \ & \ \ar[l]_-{(t+1)\cdot}\mathcal{M} \ &\ \ar[l]_-{(t-1)\cdot}\cdots \ & \ \ar[l]_-{(t+1)\cdot}\mathcal{M} \ &\ \ar[l]_-{(t-1)\cdot}\mathcal{M} \ & \ \ar[l]0,
}
\]
when $d$ is odd, and 
\[
\xymatrix{
0 \ &\ \ar[l] \mathcal{M} \ & \ \ar[l]_-{(t-1)\cdot}\mathcal{M} \ & \ \ar[l]_-{(t+1)\cdot}\cdots \ & \ \ar[l]_-{(t+1)\cdot}\mathcal{M}\ &\ \ar[l]_-{(t-1)\cdot}\mathcal{M}\ &\ \ar[l]0,
}
\]
when $d$ is even.
Since $\mathcal{M}=\F_2\oplus\F_2$, and the homomorphisms $\mathcal{M}\overset{(t-1)\cdot}{\longrightarrow} \mathcal{M} $ and $\mathcal{M}\overset{(t+1)\cdot}{\longrightarrow} \mathcal{M}$ coincide on $\mathcal{M}$, we conclude that
\[
H^r(\RP^{d-1};\mathcal{M})\cong
\begin{cases}
	\ker\big(\mathcal{M}\overset{(t-1)\cdot}{\longrightarrow} \mathcal{M} \big)\cong  \F_2 , &  r=0 ,\\
	\mathcal{M}/\im\big(\mathcal{M}\overset{(t-1)\cdot}{\longrightarrow} \mathcal{M} \big)\cong  \F_2 ,&  r=d-1 ,\\
	0,& \text{otherwise}.
\end{cases}
\]
We denote the generator of the group  $H^{d-1}(\RP^{d-1};\mathcal{M})$ by $z_{d-1}$.

%----------------------------------------------------
\subsubsection{Homology of the real projective space with local coefficients}
\label{subsub : homology computation of local coefficients}
%----------------------------------------------------

For $d\geq 2$, along the lines of the previous section, we compute the homology of the projective space
$H_*(\RP^{d-1};\mathcal{M})$
where the local coefficient system $\mathcal{M}$ is additively $\F_2\oplus\F_2$, and the action of $\pi_1(\RP^{d-1})=\langle t\rangle$ on $\mathcal{M}$ is given by $t\cdot (a_1,a_2)=(a_2,a_1)$ where $(a_1,a_2)\in \F_2\oplus\F_2$.

\medskip
Like in the case of cohomology we consider two separate cases: $d=2$ when $\pi_1(\RP^{d-1})\cong\Z$, and $d\geq 3$ when $\pi_1(\RP^{d-1})\cong\Z_2$.
First we recall the definition of the homology with local coefficients, consult for example \cite[Sec.\,3.H]{Hatcher2002}.

\medskip
For a path-connected CW-complex $X$ let $\widetilde{X}$ denote its universal cover, and let $\pi:=\pi_1(X)$ be its fundamental group.
Denote by $\mathcal{L}$ a local coefficient system on $X$, that is a $\Z[\pi]$-module.
Assume that we are given a structure of $\Z[\pi]$-CW-complex on $\widetilde{X}$ with associated cellular chain complex
\[
\xymatrix@1{
\cdots \ \ar[r] \ & \ C_{n+1}(\widetilde{X})\ar[r]^-{d_{n+1}} \ &\ C_{n}(\widetilde{X})\ar[r]^-{d_{n}} \ & \ C_{n-1}(\widetilde{X})\ar[r]^-{d_{n-1}} \ & \  \cdots \\ 
&  &\cdots \ \ar[r]^-{d_{2}} \ &\ C_{1}(\widetilde{X})\ar[r]^-{d_{1}} \ &\ C_{0}(\widetilde{X})\ar[r]\ &\ 0 \ ,
}
\]
where $C_{n}(\widetilde{X})$ is a $\Z[\pi]$-module, and $d_n\colon C_{n}(\widetilde{X})\longrightarrow C_{n-1}(\widetilde{X})$ is a $\Z[\pi]$-module homomorphism, for every $n\in\N$.

The homology $H_*(X;\mathcal{L})$ of $X$ with local coefficients in $\mathcal{L}$ is homology of the chain complex
\[
\xymatrix@C=1.37em{
\cdots\ar[r] \ &\ C_{n+1}(\widetilde{X})\otimes_{\Z[\pi]}\mathcal{L}\ar[r] \ &\ C_{n}(\widetilde{X})\otimes_{\Z[\pi]}\mathcal{L}\ar[r] \ & \  C_{n-1}(\widetilde{X})\otimes_{\Z[\pi]}\mathcal{L}\ar[r] \ & \  \cdots \\
\cdots\ar[r] \ & \ C_{1}(\widetilde{X})\otimes_{\Z[\pi]}\mathcal{L}\ar[r] \ & \ C_{0}(\widetilde{X})\otimes_{\Z[\pi]}\mathcal{L}\ar[r] \ & \  0 \ .\qquad \qquad &
}
\]
This means that
\[
H_n(X;\mathcal{M}):=\frac{\ker\big( C_{n}(\widetilde{X})\otimes_{\Z[\pi]}\mathcal{L}\longrightarrow C_{n-1}(\widetilde{X})\otimes_{\Z[\pi]}\mathcal{L}\big)}{\im\big(C_{n+1}(\widetilde{X})\otimes_{\Z[\pi]}\mathcal{L}\longrightarrow  C_{n}(\widetilde{X})\otimes_{\Z[\pi]}\mathcal{L}\big)} \ .
\]

\medskip
First we consider the case $d=2$.
The universal cover of the projective space $\RP^1$ is the real line $\R^1$, and the fundamental group $\pi=\pi_1(\RP^1)=\langle t\rangle$ is the infinite cyclic group. 
An associated $\Z[\pi]$-CW-complex of the universal cover $\widetilde{\RP^1}=\R^1$ is defined as follows: The $0$-cells are all integers $\{x_0^{i}:=\{i\} : i\in\Z\}$, while the $1$-cells are intervals $\{x_1^{i}:=[i,i+1] : i\in\Z\}$.
The action of the generator $t$ of the fundamental group on the cells of the $\Z[\pi]$-CW-complex is given by $t\cdot x_{j}^{i}=x_{j}^{i+1}$ for $j\in\{0,1\}$.
Hence, the induced cellular chain complex of $\Z[\pi]$-modules is given by
\[
C_{0}(\widetilde{\RP^1})=\Z[\Z]=_{\Z[\Z]\text{-module}}\langle x_{0}^{1} \rangle
\qquad\text{and}\qquad
C_{1}(\widetilde{\RP^1})=\Z[\Z]=_{\Z[\Z]\text{-module}}\langle x_{1}^{1} \rangle
\]
where the only non-trivial boundary homomorphism $d_1$ is defined on the generator of $C_{1}(\widetilde{\RP^1})$ by $d_1( x_{1}^{1}):= x_{0}^{1}- x_{0}^{0}=(t-1)\cdot x_{0}^{0} $.
Thus, the cellular chain complex we consider in this case is
\[
\xymatrix{
0\ar[r] \ & \ \Z[\Z]\ar[r]^-{(t-1)\cdot} \ & \ \Z[\Z]\ar[r] \ & \ 0.
}
\]
After applying the functor $\cdot\otimes_{\Z[\Z]}\mathcal{M}$ we get the chain complex
\[
\xymatrix{
0\ar[r] \ & \ \mathcal{M}   \ar[r]^-{(t-1)\cdot}\ & \  \mathcal{M}\ar[r] \ & \ 0.
}
\]
Since additively $\mathcal{M}=\F_2\oplus\F_2$ by direct inspection we see that
\[
H_r(\RP^{1};\mathcal{M})\cong
\begin{cases}
	\mathcal{M}/\im\big(\mathcal{M}\overset{(t-1)\cdot}{\longrightarrow} \mathcal{M} \big)\cong  \F_2 , &  r=0 ,\\
	 \ker\big(\mathcal{M}\overset{(t-1)\cdot}{\longrightarrow} \mathcal{M} \big)\cong  \F_2 , &  r=1 ,\\
	0,& \text{otherwise}.
\end{cases}
\]
For further use we denote the generator of the group  $H_1(\RP^{d-1};\mathcal{M})$ by $h_1$.

\medskip
When $d\geq 3$ the universal cover of the projective space $\RP^{d-1}$ is the sphere $S^{d-1}$, and the fundamental group is $\pi:=\pi_1(\RP^{d-1})\cong\Z_2=\langle t\rangle$. 
An $\Z[\pi]$-CW-complex is associated to the universal cover $\widetilde{\RP^{d-1}}=S^{d-1}$ as follows: 
In each dimension $i$, where $0\leq i\leq d-1$, there are two cells $x_{i}^{0}$ and $x_{i}^{1}$. 
The generator $t$ of the fundamental group $\pi$ acts on the cells by $t\cdot x_{j}^{0}=x_{j}^{1}$ and $t\cdot x_{j}^{1}=x_{j}^{0}$ for $0\leq j\leq d-1$.
Thus, the induced cellular chain complex of $\Z[\pi]$-modules is given by
\[
C_{n}(\widetilde{\RP^{d-1}})=\Z[\Z_2]=_{\Z[\Z_2]\text{-module}}\langle x_{n}^{0}\rangle,
\qquad\qquad\text{for }
0\leq n\leq d-1,
\] 
where the boundary homomorphism $d_n\colon C_{n}(\widetilde{\RP^{d-1}})\longrightarrow C_{n-1}(\widetilde{\RP^{d-1}})$ on the generator $x_{n}^{0}$ is 
\begin{align*}
d_n(x_{n}^{0})&:=x_{n-1}^{1}-x_{n-1}^{0}=(t-1)\cdot x_{n-1}^{0},
\qquad\qquad\text{for }
1\leq n\leq d-1\text{ odd},	\\
d_n(x_{n}^{0})&:=x_{n-1}^{1}+x_{n-1}^{0}=(t+1)\cdot x_{n-1}^{0},
\qquad\qquad\text{for }
1\leq n\leq d-1\text{ even},	
\end{align*}
and otherwise zero.
The cellular chain complex we obtained in this case is
\[
\xymatrix{
0\ar[r] \ & \ \Z[\Z_2]\ar[r]^-{(t+1)\cdot} \ &\ \Z[\Z_2]\ar[r]^-{(t-1)\cdot} \ &\ \cdots \ar[r]^-{(t+1)\cdot} \ &\ \Z[\Z_2]\ar[r]^-{(t-1)\cdot} \ &\ \Z[\Z_2]\ar[r] \ &\  0 ,
}
\]
when $d$ is odd, and 
\[
\xymatrix{
0\ar[r] \ & \ \Z[\Z_2]\ar[r]^-{(t-1)\cdot} \ &\ \Z[\Z_2]\ar[r]^-{(t+1)\cdot} \ &\ \cdots \ar[r]^-{(t+1)\cdot} \ &\ \Z[\Z_2]\ar[r]^-{(t-1)\cdot} \ & \ \Z[\Z_2]\ar[r] \ & \ 0,
}
\]
when $d$ is even.

\medskip
Applying the functor $\cdot\otimes_{\Z[\Z]}\mathcal{M}$ yields a chain complex isomorphic to the chain complex:
\[
\xymatrix{
0\ar[r] \ &\ \mathcal{M}   \ar[r]^-{(t+1)\cdot} \ & \ \mathcal{M} \ar[r]^-{(t-1)\cdot} \ & \ \cdots \ar[r]^-{(t+1)\cdot} \ & \ \mathcal{M}\ar[r]^-{(t-1)\cdot} \ & \ \mathcal{M}\ar[r] \ & \ 0,
}
\]
when $d$ is odd, and 
\[
\xymatrix{
0\ar[r] \ & \ \mathcal{M}   \ar[r]^-{(t-1)\cdot} \ &\ \mathcal{M} \ar[r]^-{(t+1)\cdot}\ &\ \cdots \ar[r]^-{(t+1)\cdot}\ & \   \mathcal{M}\ar[r]^-{(t-1)\cdot} \ &\ \mathcal{M}\ar[r]\ & \ 0,
}
\]
when $d$ is even.
Since $\mathcal{M}=\F_2\oplus\F_2$,  and the homomorphisms $\mathcal{M}\overset{(t-1)\cdot}{\longrightarrow} \mathcal{M} $ and $\mathcal{M}\overset{(t+1)\cdot}{\longrightarrow} \mathcal{M}$ coincide on $\mathcal{M}$, we have that 
\[
H_r(\RP^{d-1};\mathcal{M})\cong
\begin{cases}
\mathcal{M}/\im\big(\mathcal{M}\overset{(t-1)\cdot}{\longrightarrow} \mathcal{M} \big)\cong  \F_2, &  r=0 ,\\
	\ker\big(\mathcal{M}\overset{(t-1)\cdot}{\longrightarrow} \mathcal{M} \big)\cong  \F_2  ,&  r=d-1 ,\\
	0,& \text{otherwise}.
\end{cases}
\]
We denote the generator of the group $H_{d-1}(\RP^{d-1};\mathcal{M})$ by $h_{d-1}$, and the generator of the group $H_{0}(\RP^{d-1};\mathcal{M})$ by $h_{0}$.
%%%%%%%%%%%%%%%%%%%%%%%%%%%%%%%%%%%%%%%%%%%%%%%%%%%%%%%%%%%%%%%%%%%%%%%%%%%%%%%%%%%%%
%%%%%%%%%%%%%%%%%%%%%%%%%%%%%%%%%%%%%%%%%%%%%%%%%%%%%%%%%%%%%%%%%%%%%%%%%%%%%%%%%%%%%	

%%%%%%%%%%%%%%%%%%%%%%%%%%%%%%%%%%%%%%%%%%%%%%%%%%%%%%%%%%%%%%%%%%%%%%%%%%%%%%%%%%%%%
%%%%%%%%%%%%%%%%%%%%%%%%%%%%%%%%%%%%%%%%%%%%%%%%%%%%%%%%%%%%%%%%%%%%%%%%%%%%%%%%%%%%%	

%%%%%%%%%%%%%%%%%%%%%%%%%%%%%%%%%%%%%%%%%%%%%%%%%%%%%%%%%%%%%%%%%%%%%%%%%%%%%%%%%%%%%
%%%%%%%%%%%%%%%%%%%%%%%%%%%%%%%%%%%%%%%%%%%%%%%%%%%%%%%%%%%%%%%%%%%%%%%%%%%%%%%%%%%%%	
\newpage

%{
%\small
%\bibliography{references-hung}{}
%\bibliographystyle{amsplain}
%}

\providecommand{\noopsort}[1]{}

%%%%%%%%%%%%%%%%%%%%%%%%%%%%%%%%%%%%%%%%%%%%%%%%%%%%%%%%%%%%%%%%%%%%%%%%%%%%%%%%%%%%%
%%%%%%%%%%%%%%%%%%%%%%%%%%%%%%%%%%%%%%%%%%%%%%%%%%%%%%%%%%%%%%%%%%%%%%%%%%%%%%%%%%%%%

%\printindex

\end{document}